\documentclass[11pt]{article}

\usepackage[left=1.00in,top=1.00in,right=1.00in,bottom=1.00in]{geometry}
\usepackage{amsfonts}
\usepackage{amssymb}
\usepackage{amsmath}
\usepackage{amsthm}
\usepackage{amscd}
\usepackage[pdftex]{graphicx}
\usepackage{color}
\usepackage{enumerate}
\usepackage{subfigure}
\usepackage[titletoc,title]{appendix}

\graphicspath{{./}{figures/}{Figures/}}

\newtheorem{definition}{Definition}[section]
\newtheorem{remark}{Remark}[section]
\newtheorem{theorem}{Theorem}[section]
\newtheorem{lemma}[theorem]{Lemma}
\newtheorem{proposition}[theorem]{Proposition}
\newtheorem{corollary}[theorem]{Corollary}

\newtheorem{assumption}{Assumption}

\DeclareMathOperator{\Ima}{Im}

\newcommand{\curl}{{\bf curl}\,}
\newcommand{\rot}{{\bf rot}\,}
\newcommand{\curls}{{\rm curl}\,}
\def\div{{\rm div}}
\newcommand{\grad}{{\bf grad}\,}

\newcommand{\supp}{\mathrm{supp}}

\newcommand{\bv}{{\bf v}}
\newcommand{\bw}{{\bf w}}

\newcommand{\M}{{\cal M}}
\newcommand{\MG}{{\cal G}}

\newcommand{\Xhath}[1]{\,\hspace{-1pt}\widehat{X}^{#1}_h\hspace{-1pt}\,}
\newcommand{\Xhathb}[1]{\,\hspace{-1pt}\widehat{X}^{#1}_{h,0}\hspace{-1pt}\,}
\newcommand{\Zhath}[1]{\,\hspace{-1pt}\widehat{Z}^{#1}_h\hspace{-1pt}\,}
\newcommand{\Zhathb}[1]{\,\hspace{-1pt}\widehat{Z}^{#1}_{h,0}\hspace{-1pt}\,}

\newcommand{\Xhatlh}[1]{\,\hspace{-1pt}\widehat{X}^{#1}_{\ell,h}\hspace{-1pt}\,}

\newcommand{\What}{\hat W}
\newcommand{\Xhat}{\hat X}

\let\hat\widehat
\let\tilde\widetilde

\newcommand{\nh}{n_{holes}}
\newcommand{\nc}{n_{comp}}

\newcommand{\hole}{\eta}

\begin{document}
\title{Hierarchical B-spline Complexes of Discrete Differential Forms}
\author{John A. Evans, Michael A. Scott, Kendrick Shepherd, Derek Thomas, and Rafael V\'azquez}
\maketitle

\begin{abstract}
In this paper, we introduce the hierarchical B-spline complex of discrete differential forms for arbitrary spatial dimension.  This complex may be applied to the adaptive isogeometric solution of problems arising in electromagnetics and fluid mechanics.  We derive a sufficient and necessary condition guaranteeing exactness of the hierarchical B-spline complex for arbitrary spatial dimension, and we derive a set of local, easy-to-compute, and sufficient exactness conditions for the two-dimensional setting.  We examine the stability properties of the hierarchical B-spline complex, and we find that it yields stable approximations of both the Maxwell eigenproblem and Stokes problem provided that the local exactness conditions are satisfied.  We conclude by providing numerical results showing the promise of the hierarchical B-spline complex in an adaptive isogeometric solution framework.
\end{abstract}

\section{Introduction} \label{sec:introduction}
Many systems of partial differential equations exhibit important underlying geometric structure, often in the form of topological constraints, conservation or balance laws, symmetries, or invariants. For a large collection of these systems, the underlying geometric structure is best expressed in the language of exterior calculus of differential forms. For instance, the primary quantities of interest in electromagnetics, namely potential fields, electric fields, magnetic fields, and current densities, may be described as differential 0, 1, 2, and 3-forms, respectively, and the symmetry and invariance properties of Maxwell's equations are directly embedded in the classical de Rham complex of differential forms \cite{BOS88}.  Consequently, there has been significant research in recent years dedicated to the development of numerical methods based on discrete analogues of exterior calculus, including discrete exterior calculus \cite{DEC-arxiv,Hirani-darcy,Hirani} and finite element exterior calculus \cite{AFW06,AFW-2}.

Isogeometric discrete differential forms were recently introduced to extend finite element exterior calculus to the emerging framework of isogeometric analysis \cite{BRSV11,Buffa_Sangalli_Vazquez}.  Isogeometric discrete differential forms are constructed via an exact sequence of spline spaces of mixed degree, and they can be understood as generalizations of nodal, edge, face, and volume finite elements to splines of high regularity.  Isogeometric discrete differential forms were first constructed based on tensor-product B-splines, though generalizations based on analysis-suitable T-splines \cite{BSV14} and LR B-splines \cite{Johannessen15} have also been proposed.  It was shown in \cite{BRSV11} that isogeometric discrete differential forms based on tensor-product B-splines conform to a commutative de Rham diagram in conjunction with specially defined projection operators, thus they are a suitable candidate for use in computational electromagnetics.  Indeed, isogeometric discrete differential forms have been applied to the stable and convergent discretization of Maxwell's equations \cite{BRSV11,Corno20161}, and they have also been used in the development of divergence-conforming discretization methods for incompressible fluid flow \cite{Buffa_deFalco_Sangalli,Evans12,EvHu12,EvHu12-2,EvHu12-3,Van17}.  In these works, it is demonstrated that isogeometric discrete differential forms are more accurate on a per degree of freedom basis than classical finite element methods for electromagnetic scattering and incompressible fluid flow.

In this paper, we extend the definition of isogeometric discrete differential forms to the setting of hierarchical B-splines.  As opposed to tensor-product B-splines, hierarchical B-splines allow for local refinement of a B-spline space through the select addition of basis functions belonging to a hierarchy of nested B-spline spaces \cite{Grinspun02,Kraft,Vuong_giannelli_juttler_simeon}.  Therefore, they may be employed in an adaptive isogeometric solution framework in conjunction with suitable error estimators and refinement strategies \cite{Kuru14,Buffa16-1,Buffa16-2,BGi17}.  Hierarchical B-splines have been successfully applied to problems arising in fluid mechanics \cite{Dettmer16}, structural mechanics \cite{Schillinger12}, and fluid-structure interaction \cite{Kadapa16}, though they have not yet been applied to the stable approximation of electromagnetic scattering or the divergence-conforming approximation of incompressible fluid flow.  These latter applications require the use of hierarchical B-spline spaces comprising a discrete de Rham complex, inspiring the current work.

There are many challenges associated with extending isogeometric discrete differential forms to hierarchical B-splines, including the specification of conditions guaranteeing exactness and stability of the hierarchical B-spline complex of isogeometric discrete differential forms.  In this paper, we not only define isogeometric discrete differential forms based on hierarchical B-splines, but we also derive a sufficient and necessary condition guaranteeing exactness of the hierarchical B-spline complex for arbitrary spatial dimension.  Moreover, we derive a set of local, easy-to-compute, and sufficient exactness conditions for the two-dimensional setting, and we demonstrate through numerical example that these conditions also yield stable approximations of both the Maxwell eigenvalue problem and the Stokes problem.  Though the focus of this paper is on hierarchical B-splines, the theory presented in this paper also applies to other spaces of discrete differential forms with hierarchical structure including hierarchical T-splines \cite{Evans15} and wavelets \cite{Cohen01}.

An outline of the paper is as follows.  In Section~\ref{sec:derham}, we discuss mathematical preliminaries, including mathematical notation as well as the definition of the de Rham complex of differential forms.  In Section~\ref{sec:splinecomplex}, we present the B-spline complex of isogeometric discrete differential forms, and we discuss exactness and stability properties exhibited by the B-spline complex.  In Section~\ref{sec:hierarchicalcomplex}, we introduce the hierarchical B-spline complex of isogeometric discrete differential forms.  In Section~\ref{sec:exactness}, we examine the cohomological structure of the hierarchical B-spline complex, and we establish a set of conditions guaranteeing exactness of the complex.  In Section~\ref{sec:stability}, we examine the stability properties of the hierarchical B-spline complex in the context of the Maxwell eigenvalue problem and the Stokes problem.  In Section~\ref{sec:tests}, we apply the hierarchical B-spline complex to the numerical solution of canonical electromagnetics and creeping flow problems of interest.  Finally, in Section~\ref{Sec:conclusion}, we draw conclusions.

\section{Mathematical Preliminaries} \label{sec:derham}


In this section, we introduce the de Rham complex of differential forms and present basic mathematical notation.
Before explicitly defining the de Rham complex, we first introduce the more general concept of a Hilbert complex.  Our presentation largely follows \cite[Sections~3 and 4]{AFW-2}, so the interested reader is referred to that work for a more thorough introduction.

\subsection{Hilbert Complexes and Subcomplexes}

A {\bf \emph{Hilbert complex}} $\left( V, d \right)$ is a sequence of Hilbert spaces $\left\{V^k\right\}_{k=0}^n$ connected by a set of closed linear maps $d^k: V^k \rightarrow V^{k+1}$ such that the composition of any two consecutive maps is zero (that is, $d^k \circ d^{k-1} = 0$ for $k = 0, \ldots, n-1$).  A Hilbert complex is typically written as
\begin{equation} \label{eq:complex}
\begin{CD}
V^0 @>d^0>> V^1 @>d^1>> V^2 @>d^2>> \ldots @>d^{n-1}>> V^n.
\end{CD}
\end{equation}
The elements of the kernel of $d^k$ are called {\bf \emph{closed forms}}, and we denote $Z^k(V) \equiv Z^k(V,d) = \ker\left(d^k\right)$.  The elements of the range of $d^{k-1}$ are called {\bf \emph{exact forms}}, and we denote $B^k(V) \equiv B^k(V,d) = \Ima\left(d^{k-1}\right)$.  
The $k^{\textup{th}}$ {\bf \emph{cohomology group}} is given by the quotient space $H^k(V) \equiv H^k(V,d) = Z^k(V)/B^k(V)$.  When the cohomology groups are all trivial, we say that the sequence is {\bf \emph{exact}}.  We will deal with exact Hilbert complexes extensively in this work.

A Hilbert complex is said to be {\bf \emph{bounded}} if, for each $k$, $d^k$ is a bounded operator from $V^k$ to $V^{k+1}$, and a Hilbert complex is said to be {\bf \emph{closed}} if, for each $k$, the range of $d^k$ is closed in $V^{k+1}$.  Finite-dimensional Hilbert complexes are bounded and closed by construction, but infinite-dimensional Hilbert complexes are generally neither closed nor bounded.  Most of the Hilbert complexes presented in this paper, including the de Rham complex, are bounded and closed.

Utilizing the inner product structure of the Hilbert space $V^k$, we define the space of {\bf \emph{harmonic forms}} as $\mathfrak{h}^k(V) = Z^k(V) \cap \left(B^k(V)\right)^\perp$ where $\left(B^k(V)\right)^\perp$ is the {\bf \emph{orthogonal complement}} to $B^k(V)$ with respect to the inner-product defined on $V^k$.  For a closed Hilbert complex, the space of harmonic forms $\mathfrak{h}^k(V)$ is isomorphic to the cohomology group $H^k(V)$, and we have the {\bf \emph{Hodge decomposition}}:
\begin{equation*}
V^k = B^k(V) \oplus \left(Z^k(V)\right)^\perp \oplus \mathfrak{h}^k(V).
\end{equation*}
As most of the Hilbert complexes presented in this paper are closed, we will employ cohomology groups and harmonic forms interchangeably in our mathematical analysis.

We finish here by defining the notion of a Hilbert {\bf \emph{subcomplex}}.  We say that $(V_h,d)$ is a subcomplex of $(V,d)$ if $V_h^k \subset V^k$ for $k = 0, \ldots, n$ and $d^k(V_h^k) \subset V_h^{k+1}$ for $k = 0, \ldots, n-1$.  An important fact to note is that a subcomplex of a given Hilbert complex may not share its cohomology structure.  More specifically, the cohomology groups of the subcomplex may not be equal to the cohomology groups of the parent Hilbert complex, and in fact, the cohomology groups of the subcomplex may not even be isomorphic to the cohomology groups of the parent Hilbert complex.  Consequently, a subcomplex of an exact Hilbert complex may not be exact.

\subsection{The de Rham Complex of Differential Forms}

We are now in a position to define the {\bf \emph{de Rham complex}} of {\bf \emph{differential forms}}.  As previously mentioned, the de Rham complex is an example of a bounded and closed Hilbert complex, and to distinguish between the de Rham complex and the more general notion of a Hilbert complex, we employ the notation $\left(X(D), d\right)$ where $D$ is a domain of interest.

Before proceeding, let us briefly review the concepts of {\bf \emph{differential $\textbf{k}$-forms}} and {\bf \emph{exterior derivative operators}}.  We restrict our discussion to the Euclidean setting, but the foundational ideas also extend to the setting of smooth manifolds (see \cite[Section~4.1]{AFW-2} for more details).  With this in mind, let $D \subset \mathbb{R}^n$ be a given open domain.  For such a domain, a differential $0$-form is simply a function $\omega: D \rightarrow \mathbb{R}$, and a differential $1$-form is a function
\begin{equation*}
\omega = \sum_{\alpha=1}^{n} \omega_{\alpha} dx_{\alpha},
\end{equation*}
where $\omega_{\alpha}: D \rightarrow \mathbb{R}$ for $\alpha = 1, \ldots, n$ and $dx_{\alpha}: \mathbb{R}^n \rightarrow \mathbb{R}$ is the linear form which associates to a vector its $\alpha^{\textup{th}}$ coordinate.  Note that we may also view $dx_{\alpha}$ as a differential length element or, more specifically, the differential change in the $\alpha^{\textup{th}}$ spatial coordinate.  With this in mind, we can also define differential surface elements $dx_{\alpha_1} \wedge dx_{\alpha_2}$ and differential volume elements $dx_{\alpha_1} \wedge dx_{\alpha_2} \wedge dx_{\alpha_3}$ where it is assumed that $\alpha_1 < \alpha_2 < \alpha_3$\footnote{It should be noted that under the assumption that $\alpha_1 < \alpha_2$, the differential surface elements $dx_{\alpha_1} \wedge dx_{\alpha_2}$ may not be positively oriented for the standard orientation on $\mathbb{R}^n$, and this is also true for general $k$-forms $dx_{\alpha_1} \wedge dx_{\alpha_2} \wedge \cdots \wedge dx_{\alpha_k}$ with $\alpha_1 < \alpha_2 < \ldots < \alpha_k$.  For instance in the three-dimensional setting, $dx_1 \wedge dx_2$ and $dx_2 \wedge dx_3$ are positively oriented while $dx_1 \wedge dx_3$ is negatively oriented.  That being said, the assumption of ordered multi-indices dramatically simplifies our later presentation.}, and these elements comprise a basis for differential $2$-forms and differential $3$-forms respectively.  Generally, for $k > 1$, a differential $k$-form is a function
\begin{equation*}
\omega = \sum_{\boldsymbol \alpha \in \mathcal{I}_k} \omega_{\boldsymbol \alpha} dx_{\alpha_1} \wedge dx_{\alpha_2} \wedge \cdots \wedge dx_{\alpha_k},
\end{equation*}
where $\omega_{\boldsymbol \alpha}: D \rightarrow \mathbb{R}$, $\boldsymbol \alpha = (\alpha_1, \alpha_2, \ldots, \alpha_k)$ is a multi-index, and $\mathcal{I}_k$ is the set of ordered multi-indices
\begin{equation*}
\mathcal{I}_k = \left\{ \boldsymbol \alpha = (\alpha_1, \alpha_2, \ldots, \alpha_k): 1 \leq \alpha_1 < \alpha_2 < \ldots < \alpha_k \leq n \right\}.
\end{equation*}
Up to this point, we have assumed that $\alpha_1 < \alpha_2 < \ldots < \alpha_k$ when defining the differential basis $k$-form $dx_{\alpha_1} \wedge dx_{\alpha_2} \wedge \cdots \wedge dx_{\alpha_k}$.  However, we may break this assumption if we define
\begin{equation*}
dx_{\alpha_1} \wedge dx_{\alpha_2} \wedge \cdots \wedge dx_{\alpha_k} = \textup{sgn}(\sigma) dx_{\sigma(\alpha_1)} \wedge dx_{\sigma(\alpha_2)} \wedge \ldots \wedge dx_{\sigma(\alpha_k)},
\end{equation*}
where $\sigma: \left\{\alpha_1, \alpha_2, \ldots, \alpha_k\right\} \rightarrow \left\{\alpha_1, \alpha_2, \ldots, \alpha_k\right\}$ is an arbitrary permutation of the set $\left\{\alpha_1, \alpha_2, \ldots, \alpha_k\right\}$ and $\textup{sgn}(\sigma)$ is its sign (+1 if there is an even number of pairs $\alpha_i < \alpha_j$ such that $\sigma(\alpha_i) > \sigma(\alpha_j)$ and -1 otherwise).  Note that this implies
\begin{equation*}
dx_{\alpha_1} \wedge dx_{\alpha_2} \wedge \cdots \wedge dx_{\alpha_k} = 0
\end{equation*}
if $\alpha_i = \alpha_j$ for some $i \neq j$.  With these conventions in mind, we define the {\bf \emph{wedge product}} between a differential $k$-form and a differential $l$-form as
\begin{equation*}
\omega \wedge \eta = \sum_{\boldsymbol \alpha \in \mathcal{I}_k} \sum_{\boldsymbol \beta \in \mathcal{I}_l} \omega_{\boldsymbol \alpha} \eta_{\boldsymbol \beta} dx_{\alpha_1} \wedge \ldots \wedge dx_{\alpha_k} \wedge dx_{\beta_1} \wedge \ldots \wedge dx_{\beta_l},
\end{equation*}
and the result is a differential $(k+l)$-form. We note that the wedge product generalizes the notion of cross product to arbitrary dimension. 
We are now finally ready to define the exterior derivative operator for a differential $k$-form.  The exterior derivative of a differential $0$-form, which is also referred to as the {\bf \emph{differential}}, is defined as
\begin{equation*}
d\omega = \sum_{\alpha=1}^n \frac{\partial \omega}{\partial x_{\alpha}} dx_\alpha,
\end{equation*}
while the exterior derivative of a differential $k$-form for $k > 0$ is defined as
\begin{equation*}
d^k \omega = \sum_{\boldsymbol \alpha \in \mathcal{I}_k} d\omega_{\boldsymbol \alpha} \wedge dx_{\alpha_1} \wedge dx_{\alpha_2} \wedge \cdots \wedge dx_{\alpha_k},
\end{equation*}
where we note $d\omega_{\boldsymbol \alpha}$ is the differential of the differential $0$-form $\omega_{\boldsymbol \alpha}$.  For ease of notation, we interchangeably write the exterior derivative of a differential $k$-form $\omega$ as $d\omega$ and $d^k\omega$.  Note that, by construction, the exterior derivative of a differential $k$-form is a differential $(k+1)$-form, and we also have that the composition of two exterior derivatives is identically zero (i.e, $d^k \circ d^{k-1} = 0$).  

To construct the de Rham complex, we need to endow differential $k$-forms with Hilbert space structure.  To do so, we first define the space of {\bf \emph{smooth}} differential $k$-forms as $\Lambda^k(D)$, and we define an $L^2$-inner product between two smooth differential $k$-forms as
\begin{equation*}
\left( \omega, \eta \right)_{L^2 \Lambda^k(D)} =  \sum_{\boldsymbol \alpha \in \mathcal{I}_k} \int_D \omega_{\boldsymbol \alpha} \eta_{\boldsymbol \alpha} dD.
\end{equation*}
If we take the completion of $\Lambda^k(D)$ with respect to the above inner-product, we obtain a Hilbert space which we denote as $L^2 \Lambda^k(D)$.  Now, let us define
\begin{equation*}
X^k(D) = \left\{ \omega \in L^2\Lambda^k(D) : d^k \omega \in L^2\Lambda^{k+1}(D) \right\},
\end{equation*}
which is a Hilbert space equipped with the inner-product
\begin{equation*}
\left( \omega, \eta \right)_{X^k(D)} = \left( \omega, \eta \right)_{L^2 \Lambda^k(D)} + \left( d^k \omega, d^k \eta \right)_{L^2 \Lambda^{k+1}(D)}.
\end{equation*}
With the above notation established, we see that we have arrived at a bounded and closed Hilbert complex of the form:
\begin{equation} \label{eq:derhamcomplex}
\begin{CD}
X^0(D) @>d^0>> X^1(D) @>d^1>> X^2(D) @>d^2>> \ldots @>d^{n-1}>> X^n(D).
\end{CD}
\end{equation}
We refer to the above as the $L^2$ de Rham complex of differential forms or simply the de Rham complex for short.

The meaning of the de Rham complex is most evident in the two- and three-dimensional settings.  Here, all differential $k$-forms may be viewed as either scalar or vector {\bf \emph{proxy}} fields  (see \cite[Section~4.2]{AFW-2} for more details).  Via proxy fields, the exterior derivatives coincide with standard differential operators of calculus.  For a three-dimensional domain $D \subset \mathbb{R}^3$, we find that
\begin{equation*}
d^0 = \grad, \hspace{15pt} d^1 = \curl, \hspace{15pt} d^2 = \div,
\end{equation*}
and the de Rham complex becomes:
\begin{equation} \label{eq:derhamcomplex_3D}
\begin{CD}
H^1(D) @>\grad>> {\bf H}(\curl, D) @>\curl>> {\bf H}(\div, D) @>\div>> L^2(D).
\end{CD}
\end{equation}
For a two-dimensional domain $D \subset \mathbb{R}^2$, we find that
\begin{equation*}
d^0 = \grad, \hspace{15pt} d^1 = \curls,
\end{equation*}
and the de Rham complex becomes:
\begin{equation} \label{eq:derhamcomplex_2D}
\begin{CD}
H^1(D) @>\grad>> {\bf H}(\curls, D) @>\curls>> L^2(D).
\end{CD}
\end{equation}
Consequently, the de Rham complex provides an algebraic framework for vector calculus in the two- and three-dimensional settings.  We will often work with proxy fields in the paper as they are both easier to intuit and to visualize.

In the two-dimensional setting, one can also consider the {\bf \emph{rotated de Rham complex}}, that is, the standard de Rham complex rotated by $\pi/2$, that takes the form:
\begin{equation} \label{eq:stokescomplex_2D}
\begin{CD}
H^1(D) @>\rot>> {\bf H}(\div,D) @>\div>> L^2(D),
\end{CD}
\end{equation}
where $\rot$ is the {\bf \emph{rotor}} or {\bf \emph{perpendicular gradient}} operator.  The rotated de Rham complex is particularly useful in the development of discretization methodologies for the mixed form of the Laplacian operator, and as we will later see, it also provides a natural framework for incompressible fluid flow.

In this work, we will primarily work with de Rham complexes of differential forms subject to {\bf \emph{homogeneous}} or {\bf \emph{vanishing boundary conditions}}.  Such complexes are built from spaces of differential $k$-forms with {\bf \emph{compact support}} in exactly the same manner as that above.  We denote de Rham complexes with vanishing boundary boundary conditions by $(X_0(D),d)$, and we write:
\begin{equation} \label{eq:derhamcomplex_homogeneous}
\begin{CD}
X^0_0(D) @>d^0>> X^1_0(D) @>d^1>> X^2_0(D) @>d^2>> \ldots @>d^{n-1}>> X^n_0(D).
\end{CD}
\end{equation}
Through the use of proxy fields, the de Rham complex $(X_0(D), d)$ is comprised of the spaces
\begin{equation*}
X^0_0(D) = H^1_0(D), \; X^1_0(D) = {\bf H}_0(\curls, D), \; X^2_0(D) = L^2(D),
\end{equation*}
for a two-dimensional domain $D \subset \mathbb{R}^2$, and
\begin{equation*}
X^0_0(D) = H^1_0(D), \; X^1_0(D) = {\bf H}_0(\curl, D), \; X^2_0(D) = {\bf H}_0(\div, D), \; X^3_0(D) = L^2(D),
\end{equation*}
for a three-dimensional domain $D \subset \mathbb{R}^3$.

In the setting when the domain $D$ is {\bf \emph{contractible}}, the de Rham complexes presented above have special structure.  Specifically, when the domain $D$ is contractible, the {\bf \emph{extended de Rham complexes}}
\begin{equation*}
\begin{CD}
0 @>>> \mathbb{R} @>\subset>> X^0(D) @>d^0>> X^1(D) @>d^1>> X^2(D) @>d^2>> \ldots @>d^{n-1}>> X^n(D) @>>> 0,
\end{CD}
\end{equation*}
and
\begin{equation*}
\begin{CD}
0 @>>> X^0_0(D) @>d^0>> X^1_0(D) @>d^1>> X^2_0(D) @>d^2>> \ldots @>d^{n-1}>> X^n_0(D) @>\int>> \mathbb{R} @>>> 0
\end{CD}
\end{equation*}
are exact.  Consequently, in this setting, all of the cohomology groups of the de Rham complex $(X(D),d)$ are trivial except for $H^0(X(D)) = \mathbb{R}$, and all of cohomology groups of the de Rham complex $(X_0(D),d)$ are trivial except for $H^n(X_0(D)) = \mathbb{R}$.  When the domain $D$ is not contractible, the extended de Rham complexes presented above are no longer exact.  However, the structure of the cohomology groups is well-understood in this more general setting.  Namely, $\dim H^k(X(D)) = \dim H^{n-k}(X_0(D)) = \beta_k$, where $\beta_k$ is the $k^{\textup{th}}$ {\bf \emph{Betti number}} for the domain $D$.  Recall the zeroth Betti number is the number of connected components in the domain, the first Betti number is the number of one-dimensional or ``circular'' holes in the domain, and the second Betti number is the number of two-dimensional ``voids'' or ``cavities'' in the domain.  Betti numbers exist for more general  topological spaces including graphs, and the topological meaning of the Betti number will be implicitly exploited throughout this paper in establishing theoretical results for discrete spaces of differential forms.

\subsection{The Parametric Domain, The Physical Domain, and Pullback Operators}
\label{subsec:pullback}
Throughout this paper, we will work with a particular open domain, $\hat \Omega = (0,1)^n$, which we refer to as the {\bf \emph{parametric domain}}.  It is most natural in the isogeometric setting to first define discretizations over the parametric domain $\hat \Omega$ and then map these to a {\bf \emph{physical domain}} $\Omega \subset \mathbb{R}^n$ using a parametric mapping $\textbf{F}: \hat \Omega \rightarrow \Omega$.  Given the frequency with which we will work with the parametric domain, we will denote the de Rham complex corresponding to the parametric domain by $(\hat{X}(\hat \Omega),d)$ or simply $(\hat{X},d)$.

We relate differential $k$-forms in the parametric domain to differential $k$-forms in the physical domain using a set of {\bf \emph{pullback operators}} $\iota^k: X^k(\Omega) \rightarrow \hat{X}^k(\hat{\Omega})$.  The pullback operator for differential $0$-forms takes the form $\iota^0(\omega) = \omega \circ \textbf{F}$ for all $\omega \in X^0(\Omega)$, while the pullback operator for differential $n$-forms takes the form $\iota^n(\omega) = \det(D\textbf{F}) \left(\omega \circ \textbf{F}\right)$ for all $\omega \in X^n(\Omega)$ where $D\textbf{F}$ is the Jacobian matrix of the mapping $\textbf{F}$.  The exact form of the pullback operator for other differential $k$-forms is more involved, and the interested reader is referred to \cite[Section~2]{AFW06} for more details.  The pullback operators are commutative in the sense that $d^k \circ \iota^k = \iota^{k+1} \circ d^k$.  Consequently, the following commutative de Rham diagram is satisfied:
\begin{equation}
\begin{CD}
\Xhat^0(\hat \Omega) @>d^0>> \Xhat^1(\hat \Omega) @>d^1>> \ldots @>d^{n-1}>> \Xhat^n(\hat \Omega) \\
@A\iota^0AA @A\iota^1AA @. @A\iota^nAA  \\
X^0(\Omega) @>d^0>> X^1(\Omega) @>d^1>> \ldots @>d^{n-1}>> X^n(\Omega).
\end{CD}
\end{equation}
The commutative property of the pullback operators will be exploited when defining the spline complex of discrete differential forms in the physical domain.

\subsection{Application of the de Rham Complex in Electromagnetics and the Stokes Complex for Incompressible Fluid Mechanics}
The most prominent application of differential forms in numerical analysis is in computational electromagnetism. Bossavit formulated Maxwell's equations in terms of differential forms in \cite{BOS88}, and he established a link between {\bf\emph{Whitney forms}}, developed in the computational geometry community, and low-order finite element methods used in computational electromagnetics. Specifically, he showed the electric and magnetic field intensities may be represented as differential 1-forms which can be further approximated through edge finite elements, while the magnetic induction and the electric displacement may be represented as differential 2-forms which can be approximated with face finite elements. The computational electromagnetics community has further developed this connection, and it is now common to introduce finite element schemes as discrete differential forms (see \cite{AFW-2,HIP02a} and the references therein). Coincidentally, preservation of the structure of the de Rham complex at the discrete level is crucial for the development of stable and convergent approximation schemes in computational electromagnetics, and in particular for the solution of Maxwell's eigenproblem \cite{Boffi01} as we will see later in this paper.


The scalar and vector fields appearing in incompressible fluid mechanics may also be represented as differential forms, although additional smoothness is required in the context of viscous flow.  The {\bf \emph{Stokes complex}} naturally accommodates such smoothness requirements \cite{Buffa_deFalco_Sangalli,EvHu12,EvHu12-2,EvHu12-3}.  For a two-dimensional domain $D \subset \mathbb{R}^2$, the Stokes complex is a subcomplex of the rotated de Rham complex and takes the form:
\begin{equation} \label{eq:stokescomplex_2D_regular}
\begin{CD}
H^2(D) @>\rot>> {\bf H}^1(D) @>\div>> L^2(D),
\end{CD}
\end{equation}
where
\begin{equation*}
{\bf H}^1(D) = \left(H^1(D)\right)^n.
\end{equation*}
For a three-dimensional domain $D \subset \mathbb{R}^3$, the Stokes complex is a subcomplex of the standard de Rham complex and takes the form:
\begin{equation} \label{eq:stokescomplex}
\begin{CD}
H^1(D) @>\grad>> \boldsymbol{\Phi}(D) @>\curl>> {\bf H}^1(D) @>\div>> L^2(D),
\end{CD}
\end{equation}
where
\begin{equation*}
\boldsymbol{\Phi}(D) = \left\{ \boldsymbol{\phi} \in {\bf H}(\curl, D): \curl \boldsymbol{\phi} \in {\bf H}^1(D) \right\}.
\end{equation*}
The last two spaces of the Stokes complex, namely ${\bf H}^1(D)$ and $L^2(D)$, are natural spaces for the velocity and pressure fields in incompressible fluid mechanics.  Preservation of the Stokes complex at the discrete level yields stable finite element methods which return point-wise divergence-free velocity fields, a property which is highly desirable in multi-physics applications such as coupled flow-transport \cite{Matthies_Tobiska07} and fluid-structure interaction \cite{Kamensky17}.

\section{The B-spline Complex of Discrete Differential Forms} \label{sec:splinecomplex}


In this section, we present the tensor-product B-spline complex of isogeometric discrete differential forms, which we refer to henceforth as simply the B-spline complex, extending the construction in \cite{BRSV11} to arbitrary dimension.  This complex is built using the tensor-product nature of multivariate B-splines as its name implies.  To fix notation, we first present univariate B-splines and multivariate B-splines before introducing the B-spline complex in both the parametric and physical domain.  We then discuss exactness and stability properties exhibited by the B-spline complex, and we finish with a geometric interpretation of the complex.

\subsection{Univariate B-splines} \label{sec:univariate}
To begin, let us briefly recall the definition of {\bf \em{univariate B-splines}}, and we refer the reader to \cite{deBoor} for further details.  Let $p$ denote the polynomial degree of the univariate B-splines, and let $m$ denote the number of B-spline basis functions in the space.  To define the B-spline basis functions, we first introduce a {\bf \em{($p$-)open knot vector}} $\Xi = \{\xi_{1}, \ldots , \xi_{m+p+1} \}$, where
\begin{equation*}
0 = \xi_{1} = \ldots = \xi_{p+1} < \xi_{p+2} \le \ldots \le \xi_m < \xi_{m+1} = \ldots = \xi_{m+p+1} = 1.
\end{equation*}
Using the well known Cox-de Boor formula, we can then define B-spline basis functions $\left\{B_{i,p}\right\}_{i=1}^m$.  We denote by $S_p(\Xi)$ the space they span, which is the space of piecewise polynomials of degree $p$, with the number of continuous derivatives at each knot $\xi_i$ given by $p-r_i$, where $r_i$ is the multiplicity of the knot.

We can actually define the $i^{\textup{th}}$ B-spline basis function $B_{i,p}$ as the unique B-spline basis function defined from the Cox-de Boor algorithm and the local knot vector $\{\xi_i, \ldots, \xi_{i+p+1}\}$, and its {\bf \emph{support}} is given by the interval $(\xi_i,\xi_{i+p+1})$. We remark that the support of functions will be always taken as an open set\footnote{This is the standard choice in the literature for hierarchical B-splines.} throughout the paper. From the internal knots of the local knot vector, we associate to each basis function a {\bf \emph{Greville site}} \cite[Chapter~9]{deBoor}:
\begin{equation} \label{eq:greville-point}
\gamma_i = \frac{\xi_{i+1} + \ldots + \xi_{i+p}}{p}.
\end{equation}
We will make heavy use of Greville sites in our later mathematical analysis.

Assuming the multiplicity of each of the internal knots is less or equal to $p$ (i.e., the B-spline functions are at least continuous), the derivative of a B-spline belonging to the space  $S_{p}(\Xi)$ is a B-spline belonging to the space $S_{p-1}(\Xi')$, where $\Xi' = \{\xi_2, \ldots, \xi_{m+p}\}$ is defined from $\Xi$ by removing the first and last repeated knots. We define the normalized basis functions $D_{i,p-1} = \frac{p}{\xi_{i+p} - \xi_i}B_{i,p-1}$, also called {\bf \em{Curry-Schoenberg B-splines}}, and note the following formula holds:
\begin{equation*}
B'_{i,p}(\zeta) = D_{i,p-1}(\zeta) - D_{i+1,p-1}(\zeta)
\end{equation*}
for $i = 1, \ldots, m$, where we are considering the convention that $D_{1,p-1}(\zeta) = D_{m+1,p-1}(\zeta) = 0$ for any $\zeta \in (0,1)$.  Analogous to before, the support of the $i^{\textup{th}}$ Curry-Schoenberg B-spline $D_{i,p}$ is $(\xi_i, \ldots, \xi_{i+p})$, and we can uniquely define the Curry-Schoenberg B-splines from their local knot vectors.

\subsection{Multivariate Tensor-Product B-splines}

{\bf \em{Multivariate tensor-product B-splines}}, or simply multivariate B-splines, are constructed from a tensor-product of univariate B-splines.  Namely, given a vector of polynomial degrees ${\bf p} = (p_1, \ldots, p_n)$ and a set of open knot vectors ${\boldsymbol \Xi} = \left\{ \Xi_q \right\}_{q=1}^n$, we define the multivariate B-spline basis as
\[
{\cal B}_{\bf p}(\boldsymbol \Xi) := \{ B_{\bf i,p}({\boldsymbol \zeta}) = B_{i_1,p_1}(\zeta_1) \ldots B_{i_n,p_n}(\zeta_n) \} \; \textup{ for ${\boldsymbol \zeta} \in (0,1)^n$},
\]
and the corresponding space they span as
\begin{equation*}
S_{\bf p}(\boldsymbol \Xi) \equiv S_{p_1,\ldots,p_n}(\Xi_1,\ldots,\Xi_n) = \otimes^n_{q=1} S_{p_q}\left(\Xi_q\right) = {\rm span} ({\cal B}_{\bf p}(\boldsymbol \Xi)),
\end{equation*}
where we have introduced the shorthand $S_{\bf p}(\boldsymbol \Xi)$ for ease of notation. 

For the definition of the B-spline complex, it is convenient to define an alternate set of multivariate B-splines as the product of a suitable choice of standard B-splines and Curry-Schoenberg B-splines. To do so we first introduce, for a given multi-index $\boldsymbol \alpha = \left(\alpha_1,\alpha_2,\ldots,\alpha_k\right)$, the notation
\begin{equation*}
\tilde{p}_q \equiv \tilde{p}_q(\boldsymbol \alpha) = \left\{
\begin{array}{cl}
p_q - 1 & \textup{ if } q = \alpha_j \textup{ for some } j \in \left\{1, \ldots, k \right\}, \\
p_q & \textup{ otherwise, }
\end{array} \right.
\end{equation*}
and
\begin{equation*}
\tilde{\Xi}_q \equiv \tilde{\Xi}_q(\boldsymbol \alpha) = \left\{
\begin{array}{cl}
\Xi'_q & \textup{ if } q = \alpha_j \textup{ for some } j \in \left\{1, \ldots, k \right\}, \\
\Xi_q & \textup{ otherwise. }
\end{array} \right.
\end{equation*}
With this notation, we define
\[
{\cal B}_{\bf p}(\boldsymbol \Xi; \boldsymbol \alpha) := \{\tilde B_{\bf i,\tilde p}({\boldsymbol \zeta}) = \tilde{B}_{i_1,\tilde p_1}(\zeta_1) \ldots \tilde{B}_{i_n,\tilde p_n}(\zeta_n) \} \; \textup{ for ${\boldsymbol \zeta} \in (0,1)^n$},
\]
where
\[
\tilde B_{i_q,\tilde p_q} \equiv \tilde B_{i_q,\tilde p_q}(\boldsymbol \alpha) := 
\left \{
\begin{array}{cl}
D_{i_q,p_q-1} & \textup{ if } q = \alpha_j \textup{ for some } j \in \left\{1, \ldots, k \right\}, \\
B_{i_q,p_q} & \textup{ otherwise, }
\end{array}
\right.
\]
and we denote the space these basis functions span by
\begin{equation*}
S_{\bf p}(\boldsymbol \Xi;\boldsymbol \alpha) = \otimes^n_{q=1} S_{\tilde{p}_q}\left(\tilde{\Xi}_q\right) = {\rm span}({\cal B}_{\bf p}(\boldsymbol \Xi; \boldsymbol \alpha)). 
\end{equation*}
Note that in constructing the multivariate spline space $S_{\bf p}(\boldsymbol \Xi;\boldsymbol \alpha)$, we have reduced the polynomial degree and continuity in each of the directions indicated by the multi-index $\boldsymbol \alpha$. In these directions, standard univariate B-splines are replaced by Curry-Schoenberg B-splines in order to construct multivariate basis functions.


\subsection{The B-spline Complex}\label{sec:spline-complex}
We are now in a position to define the {\bf \em{B-spline complex of isogeometric discrete differential forms}}.  We first define this complex in the parametric domain before mapping it onto the physical domain.  To begin, let us assume that we are given a set of polynomial degrees $\left\{p_q\right\}_{q=1}^n$ and open knot vectors $\left\{ \Xi_q \right\}_{q=1}^n$, and let us further assume that multiplicity of the internal knots of $\Xi_q$ is never greater than $p_q$ for $q = 1, \ldots, n$.  This implies that the functions in $S_{\bf p}(\boldsymbol \Xi)$ are continuous and hence $S_{\bf p}(\boldsymbol \Xi) \subset H^1(\hat \Omega)$.  Then, we define the tensor-product spline space of isogeometric discrete differential $0$-forms as $\Xhath{0} := S_{\bf p}(\boldsymbol \Xi)$, and denote the corresponding basis by ${\cal B}^0 := {\cal B}_{\bf p}(\boldsymbol \Xi)$. For $k>0$, we define the basis of B-splines differential $k$-forms as
\begin{equation*}
{\cal B}^k = \bigcup_{\boldsymbol \alpha \in {\cal I}_k} \left\{ \tilde B_{\bf i, \tilde p} d\hat{x}_{\alpha_1} \wedge \ldots \wedge d\hat{x}_{\alpha_k} : \, \tilde B_{\bf i, \tilde p} \in {\cal B}_{\bf p}(\boldsymbol \Xi;\boldsymbol \alpha) \right\},
\end{equation*}
and we denote the space they span, which we call the tensor-product spline space of isogeometric discrete differential $k$-forms, by
\begin{equation*}
\Xhath{k} = \left\{ \hat{\omega}^h = \sum_{\boldsymbol \alpha \in \mathcal{I}_k} \hat{\omega}^h_{\boldsymbol \alpha} d\hat{x}_{\alpha_1} \wedge \ldots \wedge d\hat{x}_{\alpha_k} : \, \hat{\omega}^h_{\boldsymbol \alpha} \in S_{\bf p}(\boldsymbol \Xi;\boldsymbol \alpha), \hspace{5pt} \forall \boldsymbol \alpha \in \mathcal{I}_k \right\} = {\rm span}({\cal B}^k).
\end{equation*}
Then, it is easily shown that the above spaces comprise a discrete de Rham complex of the form:
\begin{equation} \label{eq:cd-iga-tp}
\begin{CD}
0 @>>> \mathbb{R} @>\subset>> \Xhath{0} @>d^0>> \Xhath{1} @>d^1>> \ldots @>d^{n-1}>> \Xhath{n} @>>> 0,
\end{CD}
\end{equation}
which we call the B-spline complex.

As in the infinite-dimensional setting, the meaning of the B-spline complex is most evident in the two- and three-dimensional settings, since the exterior derivatives coincide with differential operators of calculus via scalar or vector proxy fields.  In the two-dimensional case, we find that:
\begin{align}
\Xhath{0} &= S_{\bf p}(\boldsymbol \Xi) = S_{p_1,p_2}(\Xi_1,\Xi_2) \nonumber \\
\Xhath{1} &= S_{\bf p}\left(\boldsymbol \Xi; 1\right) \times S_{\bf p}\left(\boldsymbol \Xi; 2\right) = S_{p_1-1,p_2}(\Xi'_1,\Xi_2) \times S_{p_1,p_2-1}(\Xi_1,\Xi'_2) \nonumber \\
\Xhath{2} &= S_{\bf p}\left(\boldsymbol \Xi; (1, 2) \right) = S_{p_1-1,p_2-1}(\Xi'_1,\Xi'_2), \nonumber
\end{align}
and in the three-dimensional case, we find that:
\begin{align}
\Xhath{0} &= S_{\bf p}(\boldsymbol \Xi) \nonumber \\
&= S_{p_1,p_2,p_3}(\Xi_1,\Xi_2,\Xi_3) \nonumber \\
\Xhath{1} &= S_{\bf p}\left(\boldsymbol \Xi; 1\right) \times S_{\bf p}\left(\boldsymbol \Xi; 2\right) \times S_{\bf p}\left(\boldsymbol \Xi; 3\right) \nonumber \\
&= S_{p_1-1,p_2,p_3}(\Xi'_1,\Xi_2,\Xi_3) \times S_{p_1,p_2-1,p_3}(\Xi_1,\Xi'_2,\Xi_3) \times S_{p_1,p_2,p_3-1}(\Xi_1,\Xi_2,\Xi'_3) \nonumber \\
\Xhath{2} &= S_{\bf p}\left(\boldsymbol \Xi; (2,3)\right) \times S_{\bf p}\left(\boldsymbol \Xi; (1,3)\right) \times S_{\bf p}\left(\boldsymbol \Xi; (1,2)\right) \nonumber \\
&= S_{p_1,p_2-1,p_3-1}(\Xi_1,\Xi'_2,\Xi'_3) \times S_{p_1-1,p_2,p_3-1}(\Xi'_1,\Xi_2,\Xi'_3) \times S_{p_1-1,p_2-1,p_3}(\Xi'_1,\Xi'_2,\Xi_3) \nonumber \\
\Xhath{3} &= S_{\bf p}\left(\boldsymbol \Xi; (1, 2, 3) \right) \nonumber \\
&= S_{p_1-1,p_2-1,p_3-1}(\Xi'_1,\Xi'_2,\Xi'_3). \nonumber
\end{align}
Note that when all of the polynomial degrees are equal to one, the above spaces collapse onto the spaces of the {\bf \em{lowest-order finite element complex}}.  For instance, in the three-dimensional case, the discrete $0$-forms become continuous piecewise trilinear finite elements, the discrete $1$-forms 
become lowest-order edge N\'{e}d\'{e}lec finite elements, the discrete $2$-forms become lowest-order face N\'{e}d\'{e}lec finite elements, and the discrete $3$-forms become discontinuous piecewise constant finite elements.  Later, we will exploit a relationship between the B-spline complex and the lowest-order finite element complex to develop a geometric interpretation of the B-spline complex.

Throughout this paper, we will work with spaces subject to vanishing boundary conditions.  With this in mind, we define $\Xhathb{k} = \Xhath{k} \cap \Xhat^k_0$ to be the tensor-product spline space of isogeometric discrete differential $k$-forms subject to vanishing boundary conditions for all integers $k \geq 0$, and it can again be easily shown that these spaces comprise a discrete de Rham complex of the form:
\begin{equation} \label{eq:cd-iga-tp-bc}
\begin{CD}
0 @>>> \Xhathb{0} @>d^0>> \Xhathb{1} @>d^1>> \ldots @>d^{n-1}>> \Xhathb{n} @>\int>> \mathbb{R} @>>> 0.
\end{CD}
\end{equation}
Now that we have defined B-spline complexes with and without vanishing boundary conditions, it remains to define the corresponding B-spline complexes in the physical domain.  This is simply done using the pullback operators defined in Subsection \ref{subsec:pullback}.  Specifically, we define
\begin{equation*}
X_h^k\left(\Omega\right) := \left\{ \omega : \iota^k(\omega) \in \Xhath{k} \right\},
\end{equation*}
and
\begin{equation*}
X_{h,0}^k\left(\Omega\right) := \left\{ \omega : \iota^k(\omega) \in \Xhathb{k} \right\}
\end{equation*}
for all integers $k \geq 0$.  Using the definition of the pullback operators, it is easily shown that the above spaces comprise discrete de Rham complexes in the physical domain.

In the two-dimensional setting, a \textbf{\textit{rotated B-spline complex}} may easily be constructed by rotating the orientation of the discrete 1-forms by $\pi/2$ in parametric space and employing appropriate push-forward operators to map the rotated complex to the physical domain.  The rotated B-spline complex is ideally suited for the design of divergence-conforming discretizations for incompressible fluid flow.  See \cite{Evans12,EvHu12,EvHu12-2,EvHu12-3} for more details.

\subsection{Exactness and Stability of the B-spline Complex}

We have the following result regarding the exactness and stability of the B-spline complexes introduced in the preceding section.  The result was proven in three dimensions in \cite{BRSV11}, though the method of proof extends naturally to an arbitrary number of spatial dimensions.

\begin{proposition}
\label{proposition:commuting_parametric}
There exist $L^2$-stable projection operators $\Pi^k: X^k\left(\Omega\right) \rightarrow X^k_h\left(\Omega\right)$ and $\Pi^k_0 : X^k_0\left(\Omega\right) \rightarrow X^k_{h,0}\left(\Omega\right)$ for $k = 0, \ldots, n$ such that both
\begin{equation}
\begin{CD}
0 @>>> \mathbb{R} @>\subset>> X^0\left(\Omega\right) @>d^0>> X^1\left(\Omega\right) @>d^1>> \ldots @>d^{n-1}>> X^n\left(\Omega\right) @>>> 0 \\
& & & & @V \Pi^0 VV @V \Pi^1 VV & & @V \Pi^n VV \\
0 @>>> \mathbb{R} @>\subset>> X^0_h\left(\Omega\right) @>d^0>> X^1_h\left(\Omega\right) @>d^1>> \ldots @>d^{n-1}>> X^n_h\left(\Omega\right) @>>> 0
\end{CD}
\end{equation}
and
\begin{equation}
\begin{CD}
0 @>>> X_0^0\left(\Omega\right) @>d^0>> X_0^1\left(\Omega\right) @>d^1>> \ldots @>d^{n-1}>> X_0^n\left(\Omega\right) @>\int>> \mathbb{R} @>>> 0 \\
& & @V \Pi^0_0 VV @V \Pi^1_0 VV & & @V \Pi^n_0 VV & &\\
0 @>>> X^0_{h,0}\left(\Omega\right) @>d^0>> X^1_{h,0}\left(\Omega\right) @>d^1>> \ldots @>d^{n-1}>> X^n_{h,0}\left(\Omega\right) @>\int>> \mathbb{R} @>>> 0
\end{CD}
\end{equation}
commute. Furthermore, the continuity constants associated with the commutative projection operators only depend on the polynomial degrees of the B-spline discretization, the shape regularity of the B\'{e}zier mesh, and the parametric mapping.
\end{proposition}

The above result has two very important consequences regarding the B-spline complex.  First of all, the result states that the B-spline complex has the same cohomological structure as the de Rham complex.  Notably, when the domain $\Omega$ is contractible, the extended B-spline complexes, defined in a similar manner as the extended de Rham complexes, are exact.  Second of all, the existence of a set of commutative projection operators ensures that a discretization of the Maxwell eigenproblem using the B-spline complex is stable and convergent, and it also ensures that a discretization of the Stokes problem using the B-spline complex is inf-sup stable.

\subsection{Geometric Interpretation of the B-spline Complex} \label{sec:geometric_interpretation}

Now that we have properly defined the B-spline complex and examined its topological properties, we are ready to present a geometric interpretation of the complex.  This will prove especially useful when defining and analyzing the hierarchical B-spline complex.  To proceed forward, we need to establish the notion of a {\bf \emph{Greville grid}}.  We define the Greville grid in the parametric domain, but the Greville grid may also be defined in the physical domain using pushforward operators.

Recall from our previous discussion that we associate to each univariate B-spline basis function a Greville site.  Since multivariate B-splines are built from a tensor-product of univariate B-splines, we can associate to each basis function in $\Xhath{0}$ a {\bf \emph{Greville point}} defined as the Cartesian product of univariate Greville sites.  By connecting the Greville points, we obtain a Cartesian grid which we refer to as the Greville grid and we denote by $\MG$.  The Greville grid is a convenient means of visualizing the degrees of freedom of the B-spline complex, and we can in fact associate the basis functions of the spaces of the B-spline complex to the geometrical entities of the Greville grid as explained in \cite{BSV14}. By construction, the basis functions in $\Xhath{0}$ are uniquely identified with the {\bf \emph{vertices}} of the Greville grid $\MG$. Similarly, the basis functions of $\Xhath{1}$ and $\Xhath{2}$ are uniquely identified with the {\bf \emph{edges}} and the {\bf \emph{faces}} of $\MG$, respectively, while the basis functions of $\Xhath{3}$ are uniquely identified by the {\bf \emph{cells}} of $\MG$ in the three-dimensional setting. The matrices representing the differential operators $\{d^k\}_{k=0}^n$ between the spaces of the B-spline complex in terms of their respective basis functions are precisely the {\bf \emph{incidence matrices}} of the Greville grid.

Using tools from finite element exterior calculus, we can establish a more formal identification between the B-spline complex and the Greville grid.  In this direction, let us define $\{\Zhath{k}\}_{k=0}^n$ to be spaces of the {\bf \emph{lowest-order finite element complex}} defined on the Greville grid $\MG$.  The degrees of freedom for the spaces $\Zhath{0}$, $\Zhath{1}$, and $\Zhath{2}$ are associated to the vertices, the edges, and the faces of $\MG$, respectively, while the degrees of freedom of $\Zhath{3}$ are associated to the cells of $\MG$ in the three-dimensional setting. Exploiting the structure of the Greville grid, it was shown in \cite{BSV14} (see also \cite[Section~5]{BBSV-acta}) that we can define isomorphisms $\hat I^k_h: \Xhath{k} \longrightarrow \Zhath{k}$ that commute with the differential operators, that is,
\begin{equation*}
\hat I^{k+1}_h d^k = d^k \hat I^{k}_h, \text{ for } k = 0, \ldots, n-1,
\end{equation*}
and as a consequence we have the following de Rham commutative diagram:
\begin{equation} \label{eq:cd-iga-fem}
\begin{CD}
\Xhath{0} @>d^0>> \Xhath{1} @>d^1>> \ldots @>d^{n-1}>> \Xhath{n} \\
@V\hat I^0_hVV @V\hat I^1_hVV @. @V\hat I^n_hVV\\
\Zhath{0} @>d^0>> \Zhath{1} @>d^1>> \ldots @>d^{n-1}>> \Zhath{n}.
\end{CD}
\end{equation}
This indicates that the topological structure of the B-spline complex is identical to that of the finite element complex.  Moreover, as the cohomological structure of the finite element complex is tightly linked to the Greville grid, so is the cohomological structure of the B-spline complex.  We will exploit this property later when analyzing the hierarchical B-spline complex.

As was mentioned earlier, we will frequently work with spline spaces with vanishing boundary conditions in this work.  It is a straight-forward exercise to build a lowest-order finite element complex which commutes with such a B-spline complex using the Greville grid, so we do not discuss details herein for the sake of brevity.  We simply remark that we use the zero subindex (e.g., $\Zhathb{k}$) to identify spaces with vanishing boundary conditions.

We will need one more geometric concept in our later mathematical analysis.  Namely, we will need the concept of a {\bf \emph{B\'ezier mesh}}.  The B\'ezier mesh is simply the Cartesian grid generated by the knot vectors of multivariate B-splines, neglecting knot repetitions, and we denote this mesh by $\M$.  Note that the B\'ezier mesh and Greville grid only align in the setting when $p_1 = p_2 = \ldots = p_n = 1$.

To ground the concepts of B\'ezier mesh and Greville grid in an example, consider a two-dimensional B-spline complex with polynomial degrees $p_1 = 2$ and $p_2 = 2$ and knot vectors $\Xi_1 = \{0,0,0,1/6,2/6,3/6,4/6,5/6,1,1,1 \}$ and $\Xi_2 = \{ 0,0,0,1/5,2/5,3/5,4/5,1,1,1 \}$.  The B\'ezier and Greville grids for this particular example are displayed in Figure~\ref{fig:meshes}.  Note that as $p_1 \neq 1$ and $p_2 \neq 1$, the B\'ezier and Greville grids do not align as expected.  The finite element spaces $\Zhath{0}, \Zhath{1}$ and $\Zhath{2}$ are the spaces of {\bf \emph{bilinear nodal elements}}, {\bf \emph{lowest-order edge elements}}, and {\bf \emph{piecewise constant elements}}, respectively, defined on the Greville grid $\MG$ with degrees of freedom associated with the vertices, edges, and cells of the grid.
\begin{figure}[tp]
\begin{subfigure}[B\'ezier mesh]{
\centering
\includegraphics[width=3in]{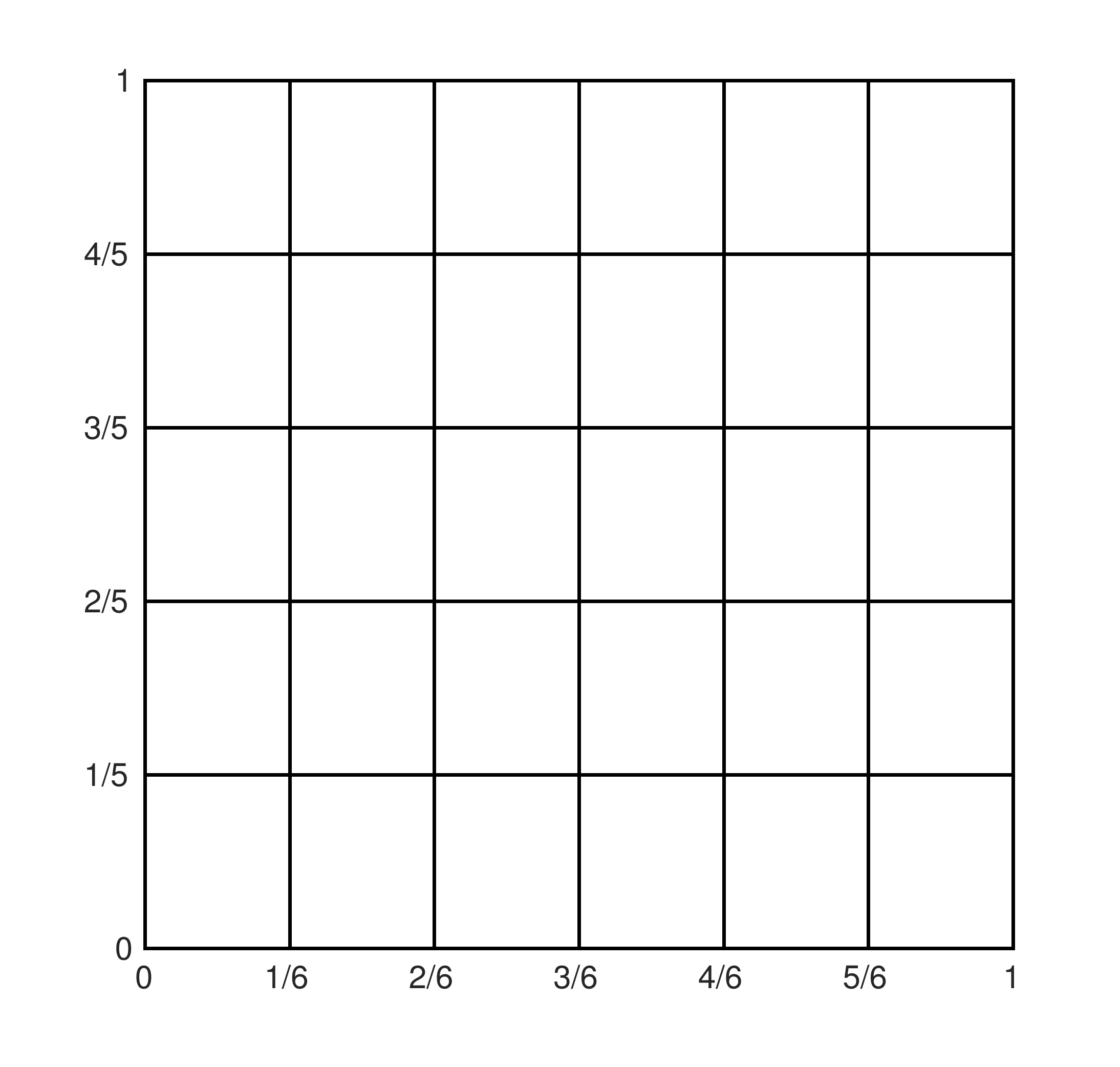}
}
\end{subfigure}
\begin{subfigure}[Greville grid]{
\centering
\includegraphics[width=3in]{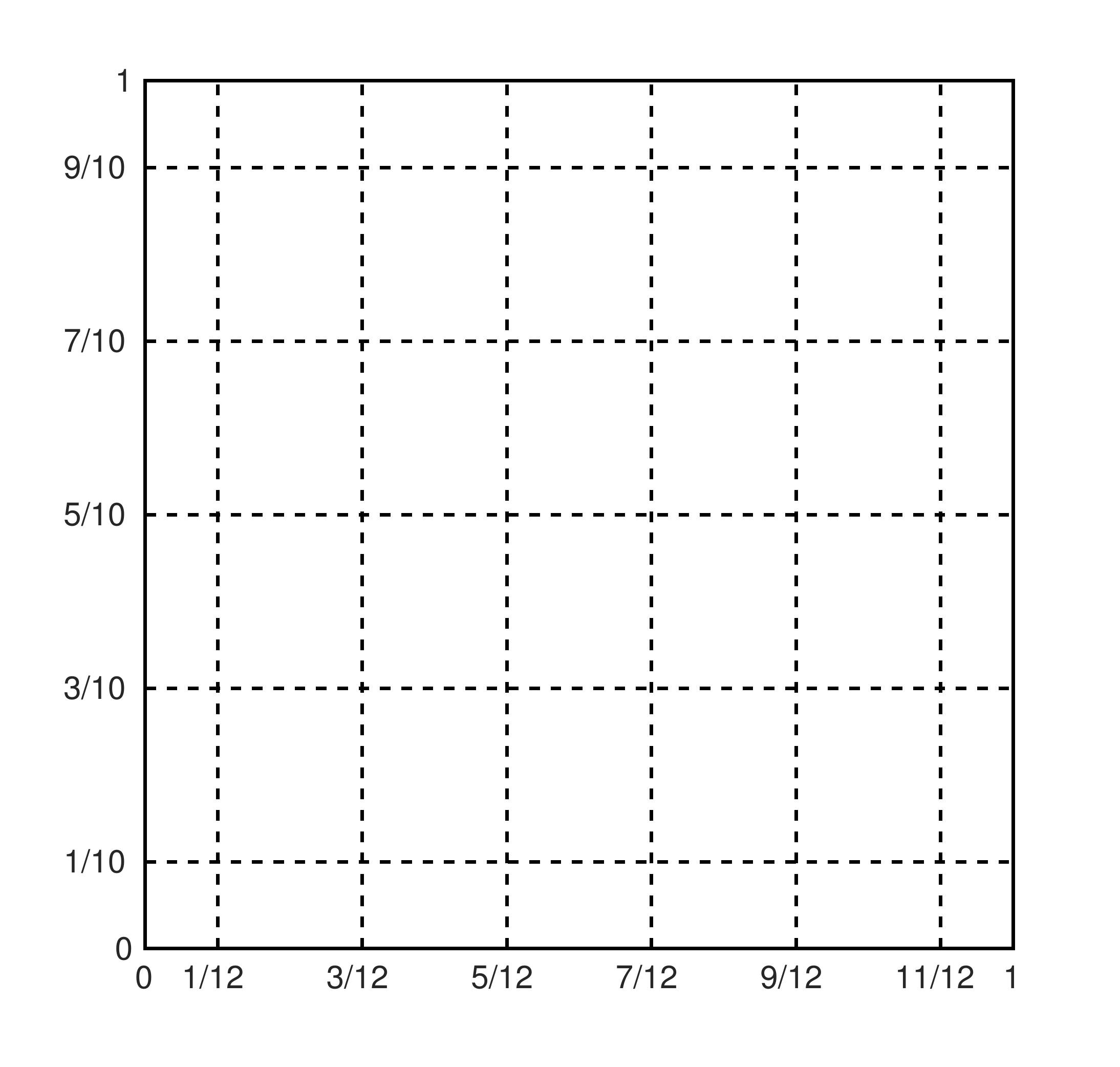}
}
\end{subfigure}
\caption{B-spline meshes associated with polynomial degrees $p_1 = 2$ and $p_2 = 2$ and knot vectors $\Xi_1 = [0,0,0,1/6,2/6,3/6,4/6,5/6,1,1,1]$ and $\Xi_2 = [0,0,0,1/5,2/5,3/5,4/5,1,1,1]$.}
\label{fig:meshes}
\end{figure}

\subsection{Extension to Multi-Patch Geometries} \label{sec:multipatch}
Most geometries of scientific and engineering interest cannot be represented as the image of the unit domain.  Instead, the {\bf \emph{multi-patch concept}} must be invoked.  In this direction, we assume that there exist $n_p$ (corresponding to the number of patches) sufficiently smooth parametric mappings $\textbf{F}^j : \hat \Omega = \left(0,1\right)^n \rightarrow \mathbb{R}^n$ such that the subdomains 
$$\Omega^j = \textbf{F}^j(\widehat{\Omega}), \hspace{3pt} j = 1,\ldots,n_p$$ 
are non-overlapping and
\begin{equation}
\overline{\Omega} = \cup_{j=1}^{n_p} \overline{\Omega^j}. \nonumber
\end{equation}
We refer to each subdomain $\Omega^j$ (and its pre-image) as a {\bf \emph{patch}}.  For a visual depiction of a multi-patch construction in $\mathbb{R}^2$, see Figure \ref{fig:stokes_multipatch}.  We build B-spline complexes over each patch $\Omega^j, \hspace{3pt} j = 1,\ldots,n_p$ in the same manner as described in previous subsections, and we denote the tensor-product spline spaces for these complexes as $\left\{ X^k_h(\Omega^j) \right\}_{k=1}^n$.

\begin{figure}[tp]
\centering
\includegraphics[width=4.50in]{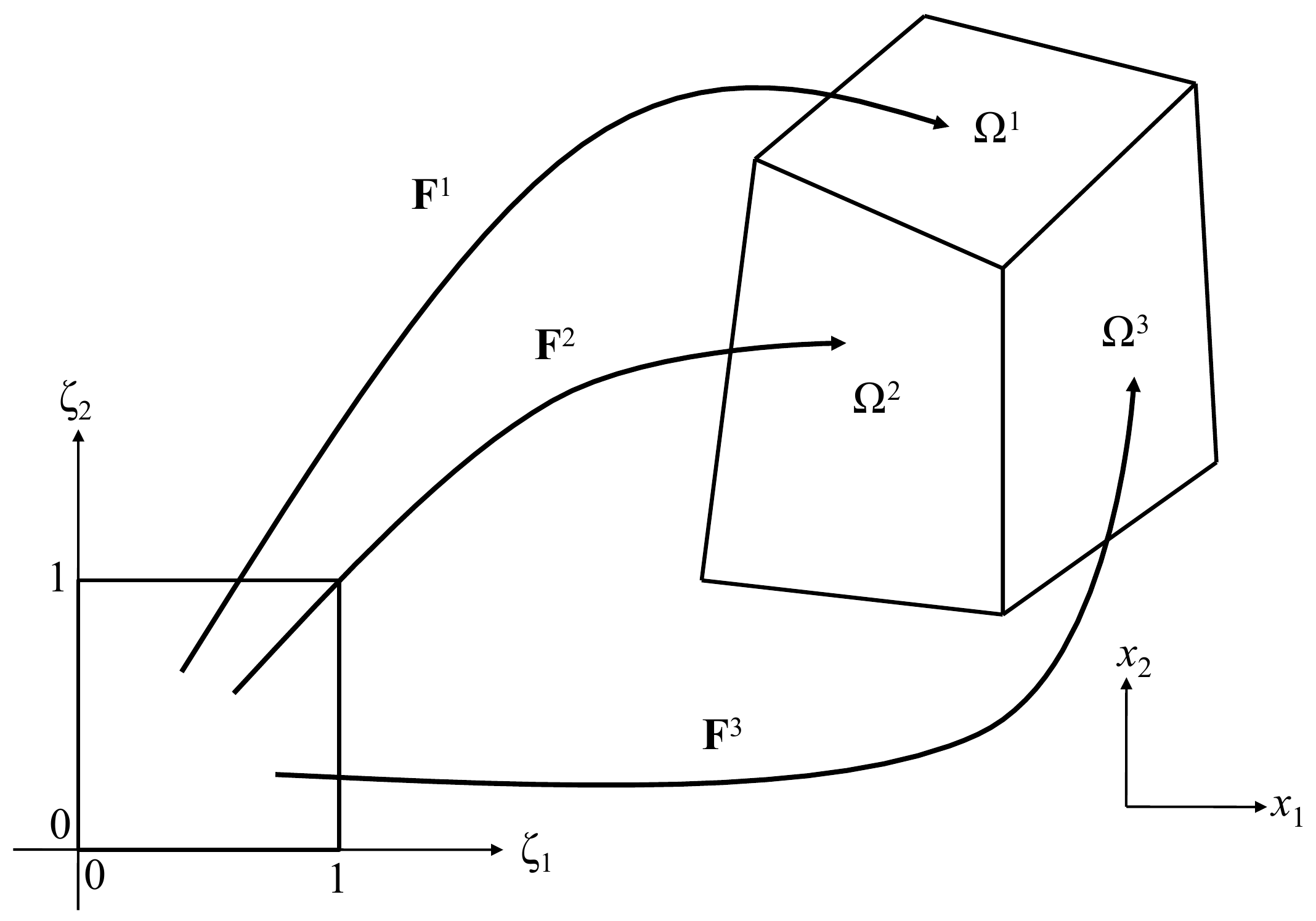} 
\caption[Multi-patch construction]{Example multi-patch construction in $\mathbb{R}^2$.}
\label{fig:stokes_multipatch}
\end{figure}

To proceed further, we must make some assumptions.  First of all, we assume that if two disjoint patches $\Omega^{j_1}$ and $\Omega^{j_2}$ have the property that $\partial \Omega^{j_1} \cap \partial \Omega^{j_2} \neq \emptyset$, then this intersection consists strictly of $k$-faces with $k < n$ (e.g., vertices, edges, and faces in the three-dimensional setting).  More succinctly, two patches cannot intersect along an isolated portion of a $k$-face.  Second, we assume that if two patches $\Omega^{j_1}$ and $\Omega^{j_2}$ share a $k$-face, the B-spline meshes associated with the patches are identical along that $k$-face.  This guarantees our mesh is conforming.  Finally, we assume that the mappings $\left\{ \textbf{F}^j \right\}_{j=1}^{n_p}$ are compatible in the following sense: if two patches $\Omega^{j_1}$ and $\Omega^{j_2}$ share a $k$-face, then $\textbf{F}^{j_1}$ and $\textbf{F}^{j_2}$ parametrize that $k$-face identically up to changes in orientation.  The above three assumptions automatically hold if we employ a conforming NURBS multi-patch construction.  See, for example, \cite[Chapter~2]{IGA-book}.

With the above assumptions in mind, the spaces of the de Rham complex over the multi-patch geometry are simply defined as
\begin{equation} \label{eq:spline_complex_multipatch}
X^k_h = \left\{ \omega \in X^k(\Omega) : \omega|_{\Omega^j} \in X^k_h\left(\Omega^j\right),  \textup{ for } j = 1, \ldots, n_p \right\} \nonumber
\end{equation}
for $k = 1, \ldots, n$.  Basis functions for these spaces are easily constructed from their patch-wise counterparts.  For discrete 0-forms, we equivalence the coefficients of any basis functions whose traces are nonzero and equal in magnitude along shared hyper-faces between patches.  For discrete $k$-forms with $0 < k < n$, we use a similar procedure to equivalence basis functions between adjacent patches except that we must take into account the {\bf \textit{orientation}} of the basis functions.  The orientations are directly associated to the entities of the Greville grid as described in the preceding subsection.  For example, the orientations of discrete $2$-forms are associated with the faces of the Greville grid (and, in particular, the normal vectors to said faces) in the three-dimensional setting.  Thus, along the intersection between adjacent patches, we (i) equivalence the coefficients of any basis functions whose normal components are nonzero and equal in magnitude and direction and (ii) set the coefficient of one to minus the coefficient of another for any basis functions whose normal components are nonzero, equal in magnitude, and opposite in direction.  This is in direct analogy with finite element differential forms.  Finally, discrete $n$-forms are piecewise discontinuous and hence we do not need to equivalence basis functions between adjacent patches.  For more details on the construction of the B-spline complex in the multi-patch setting, the reader is referred to \cite[Section~5.4]{BBSV-acta}.\\

\section{The Hierarchical B-spline Complex} \label{sec:hierarchicalcomplex}
In this section, we finally present the hierarchical B-spline complex of isogeometric discrete differential forms.  We start with a short introduction of hierarchical B-splines, where we follow the construction in \cite{Kraft} and \cite{Vuong_giannelli_juttler_simeon}, and we present an auxiliary lemma that will be useful for further study of the hierarchical B-spline complex. We then introduce a simple definition of the hierarchical B-spline complex, and we provide a counterexample showing that the complex is not always exact.  We next introduce the concept of Greville subgrids which will prove useful for the analysis of exactness in Section~\ref{sec:exactness}. We end the section with an explanation on how to construct the complex in the physical domain as well as for multi-patch domains.

For the sake of simplicity, from now on we only use spaces with homogeneous boundary conditions. The zero subindex corresponding to these spaces is removed to alleviate notation. For ease of reading, a given level of the multi-level mesh associated with a hierarchical B-spline basis is indicated by the subindex $\ell$, while the position in the complex for $k$-forms is indicated by the superindex $k$. 

\subsection{Hierarchical B-splines} \label{sec:HB-splines}

Recall that the parametric domain is given by $\hat \Omega = (0,1)^n$. Let us assume that we have a sequence of tensor-product spline spaces of 0-forms
\begin{equation}\label{eq:subspaces}
\Xhat_{0,h} \subset \Xhat_{1,h} \subset \ldots \subset \Xhat_{N,h}.
\end{equation}
More precisely, for each {\bf\emph{level}} $\ell = 0, \ldots, N$, the $n$-variate tensor product spline space $\Xhat_{\ell,h} \subset H^1(\hat \Omega)$ is defined as the tensor product of univariate spline spaces $S_{p_{j,\ell}}(\Xi_{j,\ell})$, each of degree $p_{j,\ell}$ and built using a knot vector $\Xi_{j,\ell}$.  That is, $\Xhat_{\ell,h} = \otimes_{j=1}^n S_{p_{j,\ell}}(\Xi_{j,\ell})$. We denote the set of B-spline basis functions for each level by ${\cal B}_\ell$, and we denote by $\M_\ell$ the corresponding B\'ezier mesh, that is, the Cartesian grid generated by the knot vectors without repetitions. 
Let moreover
\begin{equation}\label{eq:subdomains}
\hat \Omega = \Omega_0 \supset \Omega_1 \supset \ldots \supset \Omega_{N} \supset \Omega_{N+1} = \emptyset
\end{equation}
denote a sequence of nested subdomains, where we assume that (the closure of) each subdomain $\Omega_\ell$ is (the closure of) a union of elements of $\M_{\ell-1}$ for $\ell = 1, \ldots, N$. This is the same condition as the {\bf \emph{strong condition}} in \cite{Vuong_giannelli_juttler_simeon}.  With the above notation, \textbf{\textit{hierarchical B-splines}} are defined in a recursive manner as follows.

\begin{definition} \label{def:hierarchical}
We define the space of hierarchical splines $\What_N = {\rm span} \{{\cal H}_N\}$, with ${\cal H}_N$ the hierarchical basis constructed with the following recursive algorithm \cite{Kraft}:
\begin{enumerate}
\item Initialization: ${\cal H}_0 = {\cal B}_0$.
\item Recursively construct ${\cal H}_{\ell+1}$ from ${\cal H}_{\ell}$ by setting
\begin{equation*}
{\cal H}_{\ell+1} = {\cal H}_{\ell+1,c} \cup {\cal H}_{\ell+1,f}, \text{ for } \ell = 0, \ldots N-1,
\end{equation*}
with 
\begin{equation*}
\begin{array}{l}
{\cal H}_{\ell+1,c} = \{\beta \in {\cal H}_{\ell} : \supp(\beta) \not \subset \Omega_{\ell+1}\}, \\
{\cal H}_{\ell+1,f} = \{\beta \in {\cal B}_{\ell+1} : \supp(\beta) \subset \Omega_{\ell+1}\}.
\end{array}
\end{equation*}
\end{enumerate}
\end{definition}
\noindent That is, at step $\ell$ of the recursive algorithm we add the basis functions of level $\ell+1$ completely supported in $\Omega_{\ell+1}$ to the basis and remove the functions of previous levels completely supported in $\Omega_{\ell+1}$ from the basis. B-spline basis functions which are included in ${\cal H}_N$ are referred to as \textbf{\textit{active}} basis functions, and in particular, ${\cal H}_N \cap {\cal B}_\ell$ is the set of active basis functions of level $\ell$, for an integer $0 \leq \ell \leq N$.

Associated with the space of hierarchical splines $\What_N$ is a \textit{\textbf{hierarchical B\'{e}zier mesh}}
\begin{equation}
\mathcal{M}_\mathcal{H} := \bigcup_{\ell=0}^N \{Q \in {\cal M}_\ell \, \text{ such that } \, Q \subset \Omega_\ell \, \text{ and } \, Q \not \subset \Omega_{\ell+1} \}.
\end{equation}
The elements of the hierarchical B\'{e}zier mesh are precisely those included in a formation and assembly routine for isogeometric analysis \cite{Hennig16, Garau16}. An example of a two-dimensional hierarchical B\'ezier mesh, for a simple configuration with only two levels, is depicted in Figure~\ref{fig:hierarchical_bezier}. We have identified the Greville points (see Subsection~\ref{sec:geometric_interpretation}) associated with the active basis functions for both levels in Figure~\ref{fig:0forms_two_levels} assuming the basis functions are maximally continuous and of degree $p_1 = p_2 = 2$.

\begin{figure}
\centerline{\includegraphics[width=0.4\textwidth, trim = 11cm 2cm 11cm 2cm, clip]{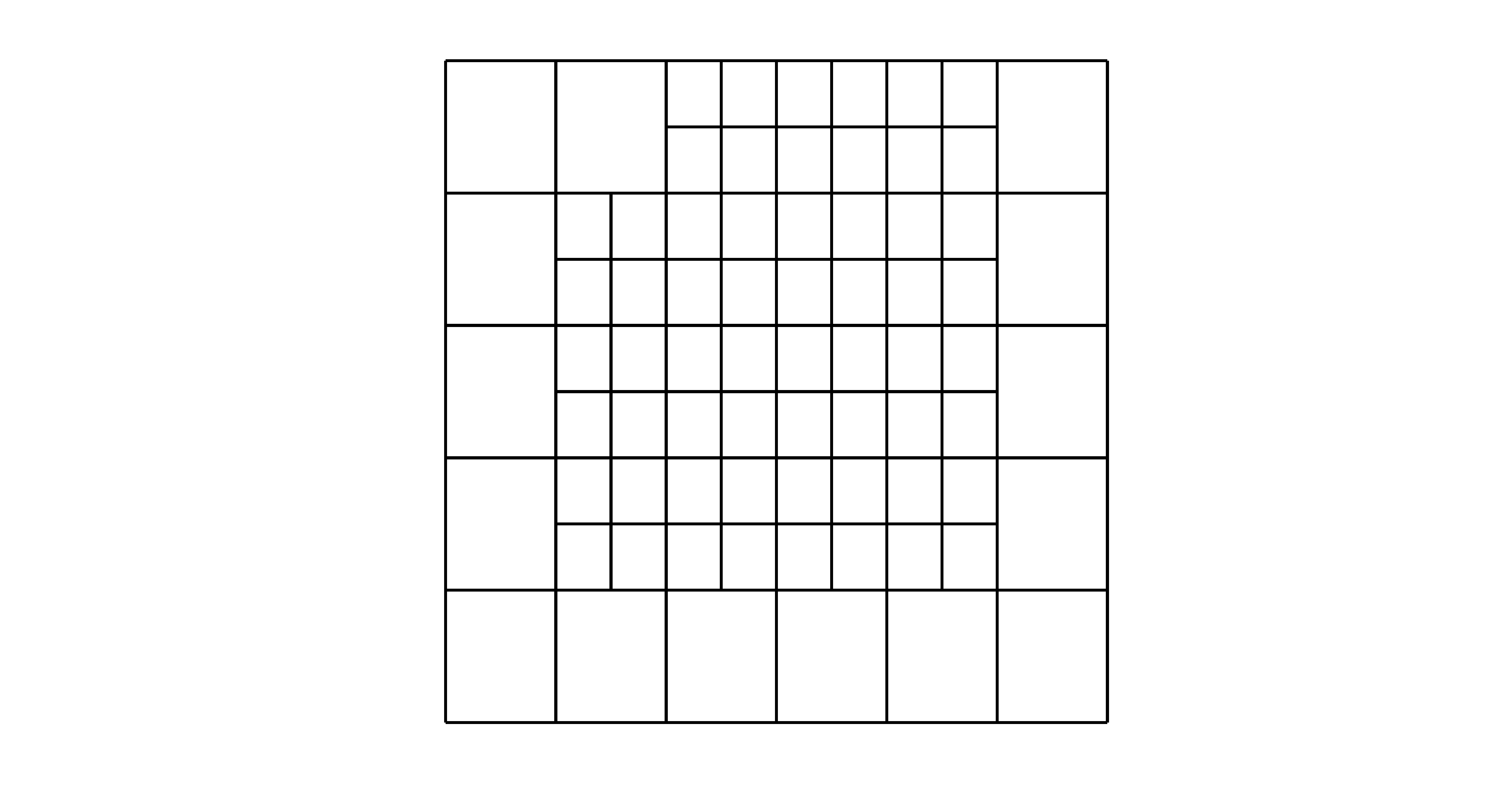}}
\caption{Example of a hierarchical B\'ezier mesh with two levels.}
\label{fig:hierarchical_bezier}
\end{figure}

\begin{figure}
\centerline{\includegraphics[width=0.4\textwidth, trim = 11cm 2cm 11cm 1cm, clip]{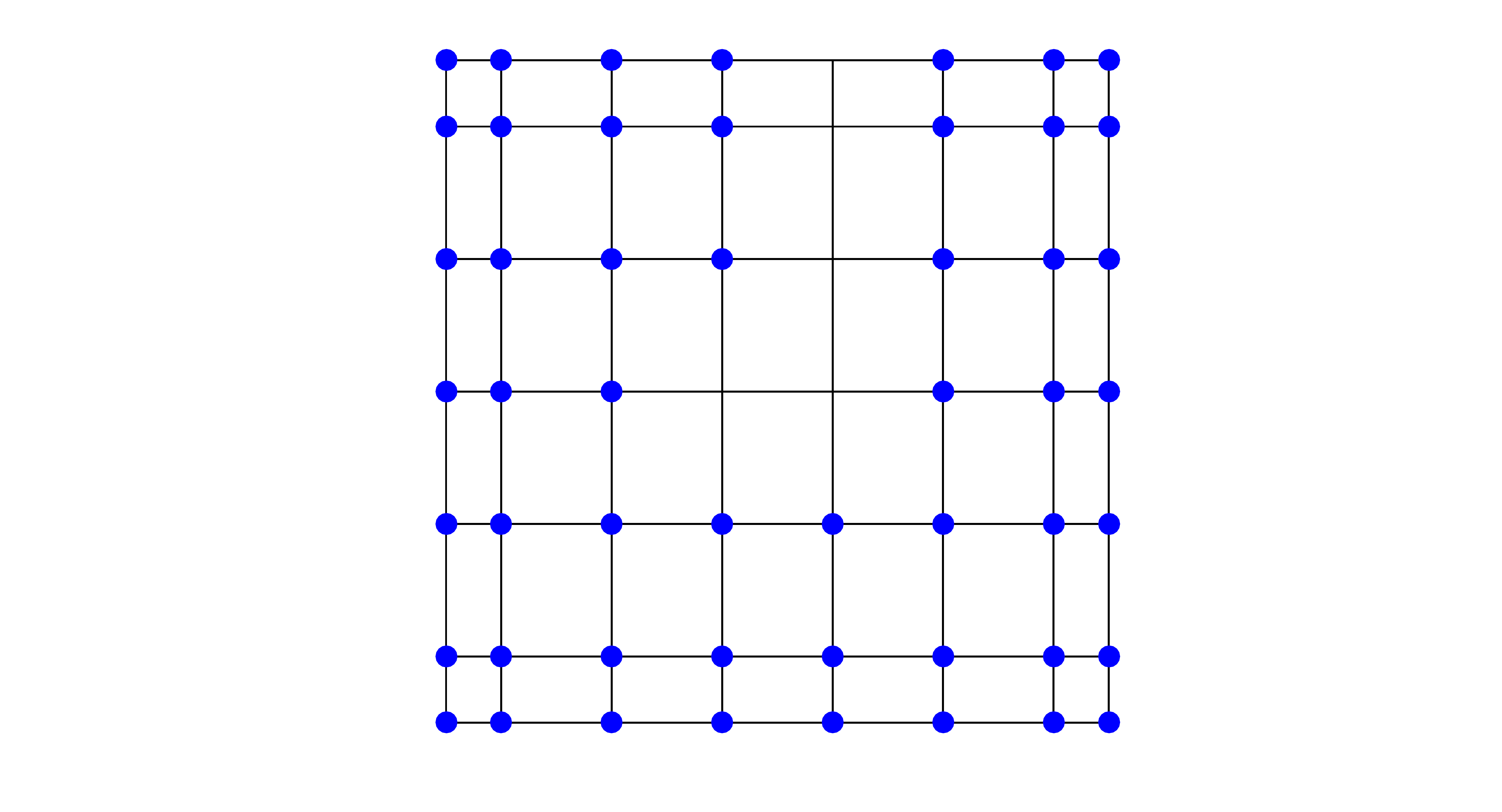}
\includegraphics[width=0.4\textwidth, trim = 11cm 2cm 11cm 1cm, clip]{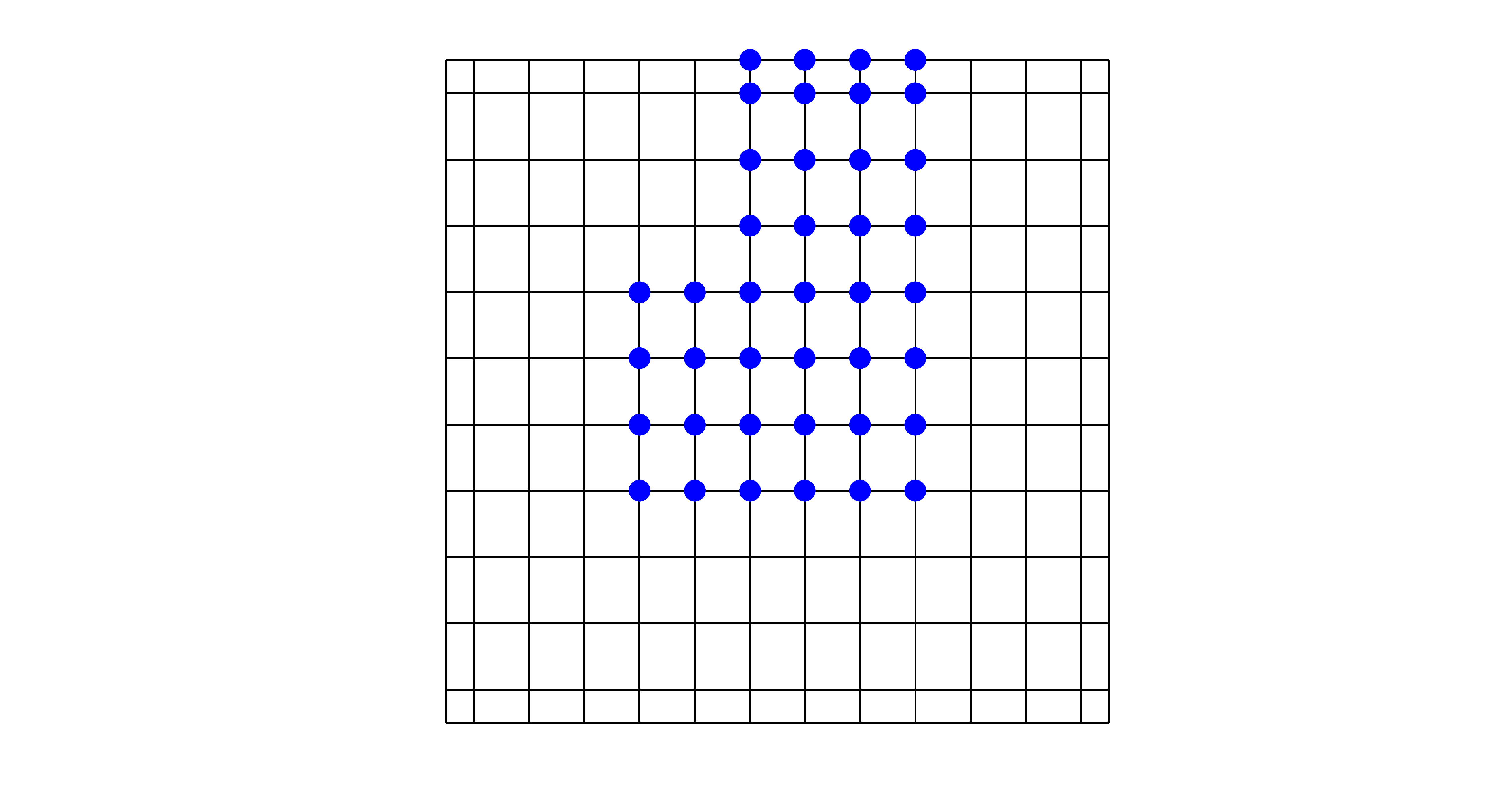}}
\caption{Greville points associated with the active basis functions for levels $\ell = 0$ (left) and $\ell = 1$ (right) in the hierarchical B\'ezier mesh depicted in Figure~\ref{fig:hierarchical_bezier} assuming the basis functions are maximally continuous and of degree $p_1 = p_2 = 2$.}
\label{fig:0forms_two_levels}
\end{figure}

It is useful to define, for $0 \le \ell \le \ell' \le N$, the set of basis functions of level $\ell$ with support contained in $\Omega_{\ell'}$, that is 
\begin{equation} \label{eq:B_lk}
{\cal B}_{\ell,\Omega_{\ell'}} \equiv {\cal B}_{\ell,\ell'} := \{ \beta \in {\cal B}_\ell : \supp (\beta) \subset \Omega_{\ell'} \}.
\end{equation}
Notice that, in Definition \ref{def:hierarchical}, we have ${\cal H}_{\ell+1,f} = {\cal B}_{\ell+1, {\ell+1}}$. We will also make use of the following auxiliary spaces, which are closely related to the recursive construction of the hierarchical spline space.
\begin{definition} \label{def:X_ell}
For $0 \le \ell \le \ell' \le N$, we define the spaces:
\begin{align*}
& \What_\ell := {\rm span} \{ {\cal H}_\ell \}, \\
& \Xhat_{\ell,\Omega_{\ell'}} \equiv \Xhat_{\ell,\ell'} := {\rm span} \{ {\cal B}_{\ell,\ell'} \} = \{ \beta \in \hat X_{\ell,h} : \supp (\beta) \subset \Omega_{\ell'} \},
\end{align*}
where the equality holds thanks to the local linear independence of B-splines.
\end{definition}
\noindent Note that we are using the $\Xhat$ notation for spaces spanned by basis functions of one single level, and we are using $\What$ for spaces spanned by functions of multiple levels. In particular, in the previous definition, the space $\What_\ell$ represents the space of hierarchical splines when applying the recursive algorithm up to level $\ell$. The set of auxiliary spaces denoted with the $\Xhat$ notation is helpful to represent the spaces spanned by the basis functions that have to be added ($\Xhat_{\ell+1,{\ell+1}}$) and removed ($\Xhat_{\ell,{\ell+1}}$) at each level in the recursive algorithm, as stated in the following lemma.

\begin{lemma} \label{lemma:cupcap}
For $0 \le \ell \le N-1$ it holds that
\begin{enumerate}[i)]
\item $\What_{\ell+1} = \What_{\ell} + \Xhat_{\ell+1,{\ell+1}} := \{u = v + w: v \in \What_\ell, \, w \in \Xhat_{\ell, {\ell+1}} \}$.
\item $\What_\ell \cap \Xhat_{\ell+1, {\ell+1}} = \Xhat_{\ell, {\ell+1}}$.
\end{enumerate}
\end{lemma}
\begin{proof}
To prove the result, we show that inclusions in both directions hold. The first inclusion in i), that is, $\What_{\ell+1} \subset \What_{\ell} + \Xhat_{\ell+1,{\ell+1}}$, follows from Definitions~\ref{def:hierarchical} and~\ref{def:X_ell} since ${\cal H}_{\ell+1} = {\cal H}_{\ell+1,c} \cup {\cal H}_{\ell+1,f} \subset {\cal H}_\ell \cup {\cal H}_{\ell+1,f}$ and $\Xhat_{\ell+1,{\ell+1}} = {\rm span}\{{\cal H}_{\ell+1,f}\}$. For the second inclusion, it is obvious that $\Xhat_{\ell+1,{\ell+1}} \subset \What_{\ell+1}$, and $\What_\ell \subset \What_{\ell+1}$ is a consequence of their definition and the nestedness provided by~\eqref{eq:subspaces}.

For the first inclusion in ii), we have that $u \in \What_\ell = {\rm span}\{{\cal H}_\ell\} \subset {\rm span\{\cal B}_\ell\}$, and moreover $u \in \Xhat_{\ell+1,{\ell+1}}$ implies that $\supp (u) \subset \Omega_{\ell+1}$, which along with the local linear independence of \mbox{B-splines} yields $u \in \Xhat_{\ell, {\ell+1}}$. For the second inclusion we first use the nestedness of the subdomains given by~\eqref{eq:subdomains} and then the result from i) to obtain $\Xhat_{\ell,{\ell+1}} \subset \Xhat_{\ell,\ell} \subset \What_\ell$. Then, the nestedness of the spline spaces~\eqref{eq:subspaces}, together with the local linear independence of B-splines, immediately gives $\Xhat_{\ell, {\ell+1}} \subset \Xhat_{\ell+1, {\ell+1}}$, which completes the proof.
\end{proof}

\subsection{The Hierarchical B-spline Complex}
The construction of the hierarchical B-spline complex is based on a simple idea: given a multi-level mesh, construct a sequence of hierarchical B-splines of mixed degree, using for each level the B-spline complex in Section~\ref{sec:spline-complex}. We will see, however, that this construction does not, in general, yield an exact complex of hierarchical B-splines.  As was done for the tensor-product B-spline complex, we first define the hierarchical B-spline complex in the parametric domain before mapping it onto the physical domain.

Assume we are given a sequence of nested tensor-product spline spaces
\begin{equation*}
\Xhat_{0,h} \subset \Xhat_{1,h} \subset \ldots \subset \Xhat_{N,h},
\end{equation*}
as in~\eqref{eq:subspaces} and a sequence of nested subdomains as in~\eqref{eq:subdomains}. For each level $\ell = 0, \ldots, N$, we select $\Xhatlh{0} = \Xhat_{\ell,h}$ as the discrete space of 0-forms, and we construct a discrete de Rham complex using tensor-product splines of mixed degree as explained in Section~\ref{sec:spline-complex}. This forms a complex $(\Xhatlh{},d)$ for each level:
\begin{equation} \label{eq:spline_complex}
\begin{CD}
\Xhatlh{0} @>d^0>> \Xhatlh{1} @>d^1>> \ldots @>d^{n-1}>> \Xhatlh{n} .
\end{CD}
\end{equation}
Moreover, with each space of the sequence, $\Xhatlh{k}$, we associate the same set of basis functions introduced in Section~\ref{sec:spline-complex}, and we denote this set of basis functions by ${\cal B}^k_\ell$. Recall that we can define isomorphisms from our B-spline complex to the low-order finite element complex defined on the Greville grid using these bases.

Now, for each $k =0, \ldots, n$, we can construct the hierarchical B-splines bases ${\cal H}^k_N$ and the corresponding hierarchical spline spaces $\What^k_N = {\rm span} \{ {\cal H}^k_N \}$ following the same procedure as in Definition~\ref{def:hierarchical}.  In practice, for spaces of vector-valued splines, the hierarchical B-spline basis can be determined using the recursive algorithm for each component. Adopting the notation introduced in Section \ref{sec:derham} for Hilbert complexes, the \textbf{\textit{hierarchical B-spline complex}} is then simply $\left(\What_N,d\right)$.  As an example, let us recall consider a maximally continuous hierarchical B-spline complex with $p_1 = p_2 = 3$ defined on the hierarchical B\'ezier mesh depicted in Figure~\ref{fig:hierarchical_bezier}.  The Greville nodes associated with active basis functions for $0$-forms are displayed in Figure~\ref{fig:0forms_two_levels}.  All 0-form basis functions are biquadratic, that is, of degree $(2,2)$.  The Greville edges associated with active 1-form basis functions are displayed in Figure~\ref{fig:1forms_two_levels}, where the horizontal arrows correspond to rightward oriented basis functions of degree $(1,2)$ and the vertical arrows to upward oriented basis functions of degree $(2,1)$.  The Greville cells associated with active 2-form basis functions are displayed in Figure~\ref{fig:2forms_two_levels}.  All 2-form basis functions are bilinear, that is, of degree $(1,1)$.

As in \eqref{eq:B_lk}, it will be useful to define, for $0 \le \ell \le \ell' \le N$ and $0 \le k \le n$, the set of basis functions of level $\ell$ that are completely contained in $\Omega_{\ell'}$, that is
\begin{equation} \label{eq:Bik}
{\cal B}^k_{\ell,\Omega_{\ell'}} \equiv {\cal B}^k_{\ell,\ell'} := \{ \beta^k \in {\cal B}^k_\ell : \supp (\beta^k) \subset \Omega_{\ell'} \}.
\end{equation}
Moreover, we can also apply Definition~\ref{def:X_ell} for each $k$ to define
\begin{align}
& \What^k_\ell := {\rm span}\{{\cal H}^k_\ell\}, \\
& \Xhat^k_{\ell,\Omega_{\ell'}} \equiv \Xhat^k_{\ell,\ell'} := {\rm span} \{{\cal B}^k_{\ell,\ell'}\} =  \{ \beta^k \in \Xhatlh{k}  : \supp (\beta^k) \subset \Omega_{\ell'} \},
\end{align}
and the results of Lemma~\ref{lemma:cupcap} also hold. These spaces also form spline complexes, as proved in the next lemma.

\begin{figure}
\centerline{\includegraphics[width=0.4\textwidth, trim = 11cm 2cm 11cm 1cm, clip]{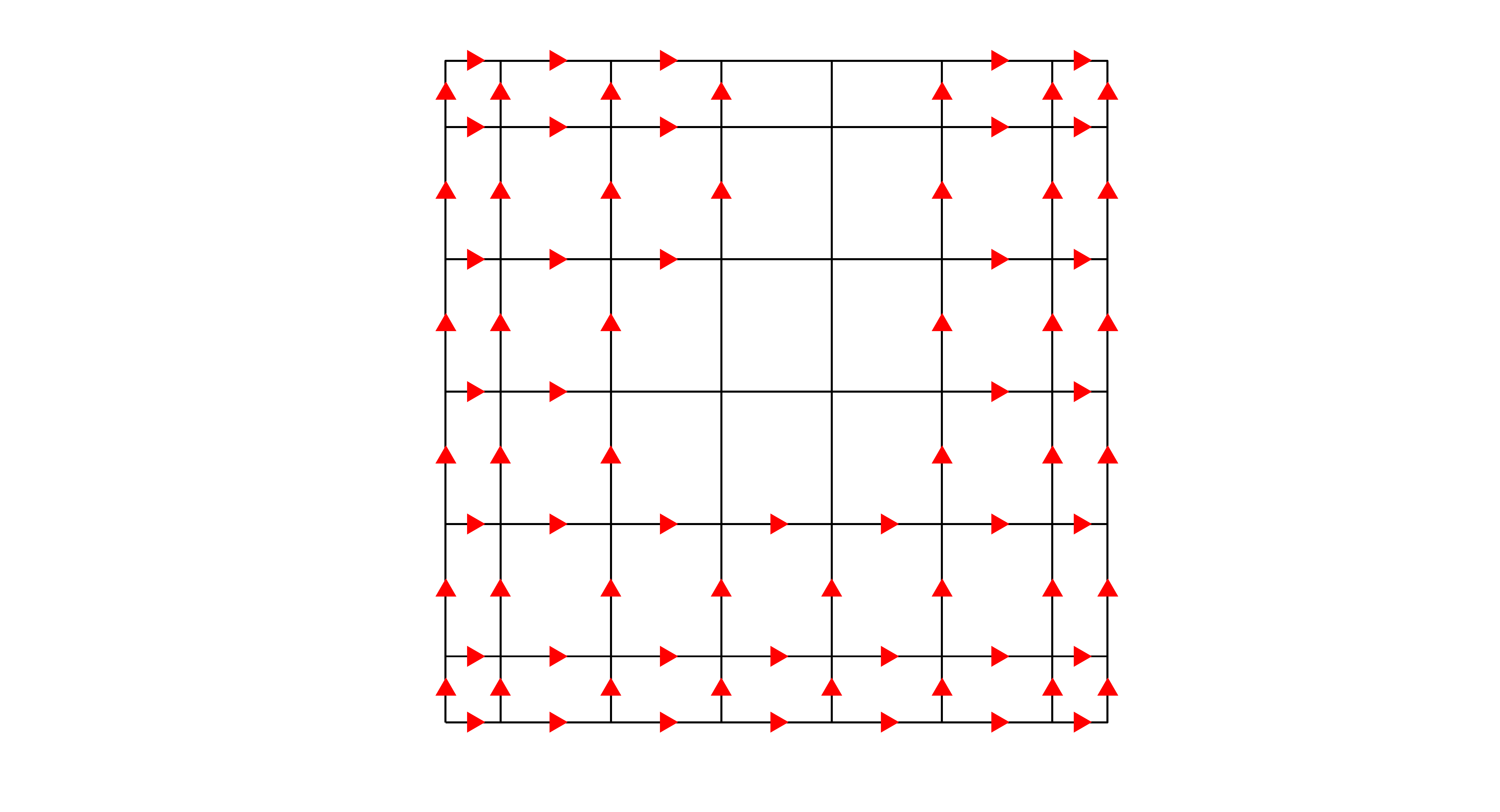}
\includegraphics[width=0.4\textwidth, trim = 11cm 2cm 11cm 1cm, clip]{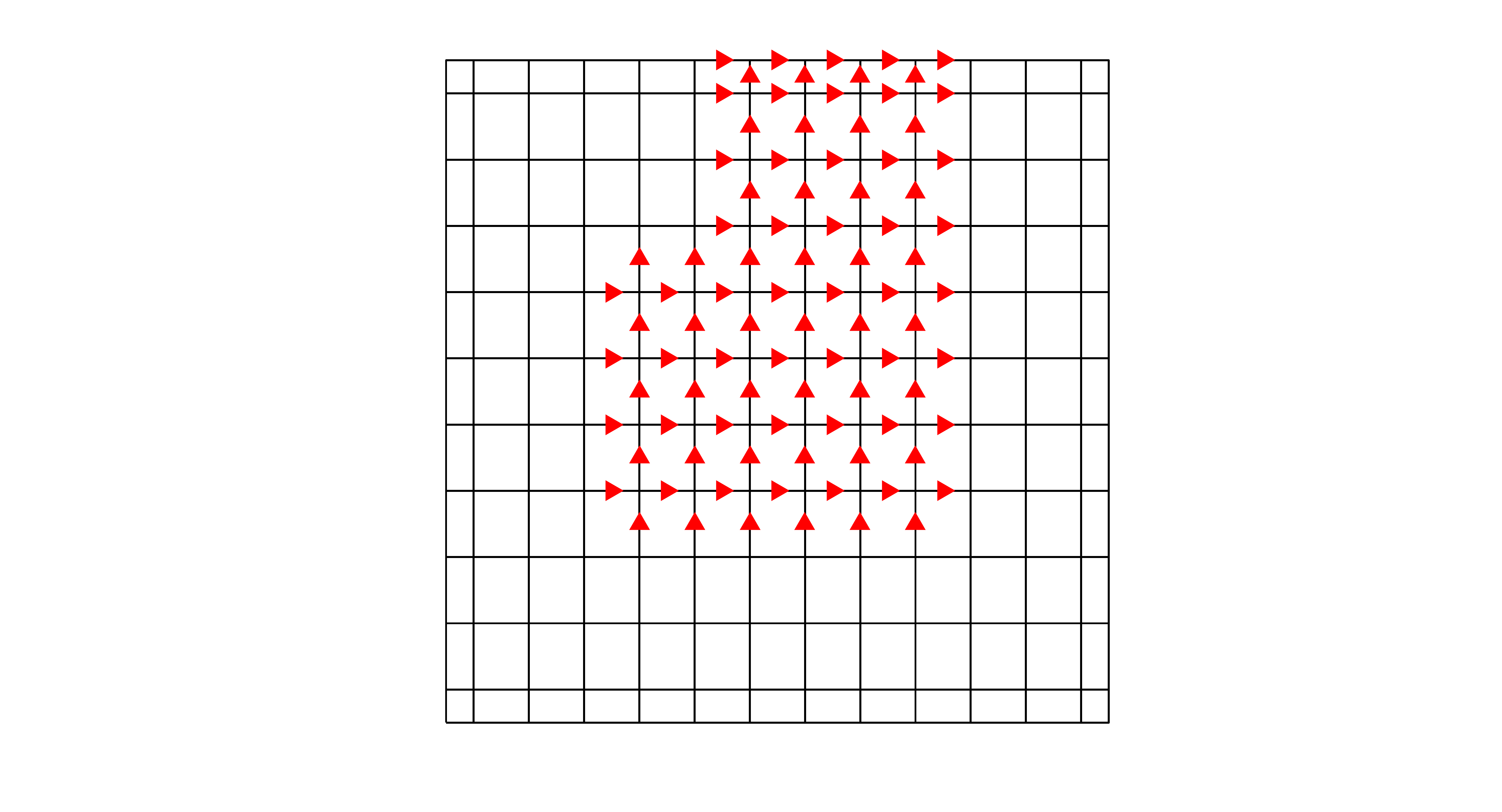}}
\caption{Greville edges associated with the active 1-form basis functions for levels $\ell = 0$ (left) and $\ell = 1$ (right) in the hierarchical B\'ezier mesh depicted in Figure~\ref{fig:hierarchical_bezier} assuming the 0-form basis functions are maximally continuous and of degree $p_1 = p_2 = 2$.}
\label{fig:1forms_two_levels}
\end{figure}

\begin{figure}[t!]
\centerline{\includegraphics[width=0.4\textwidth, trim = 11cm 2cm 11cm 1cm, clip]{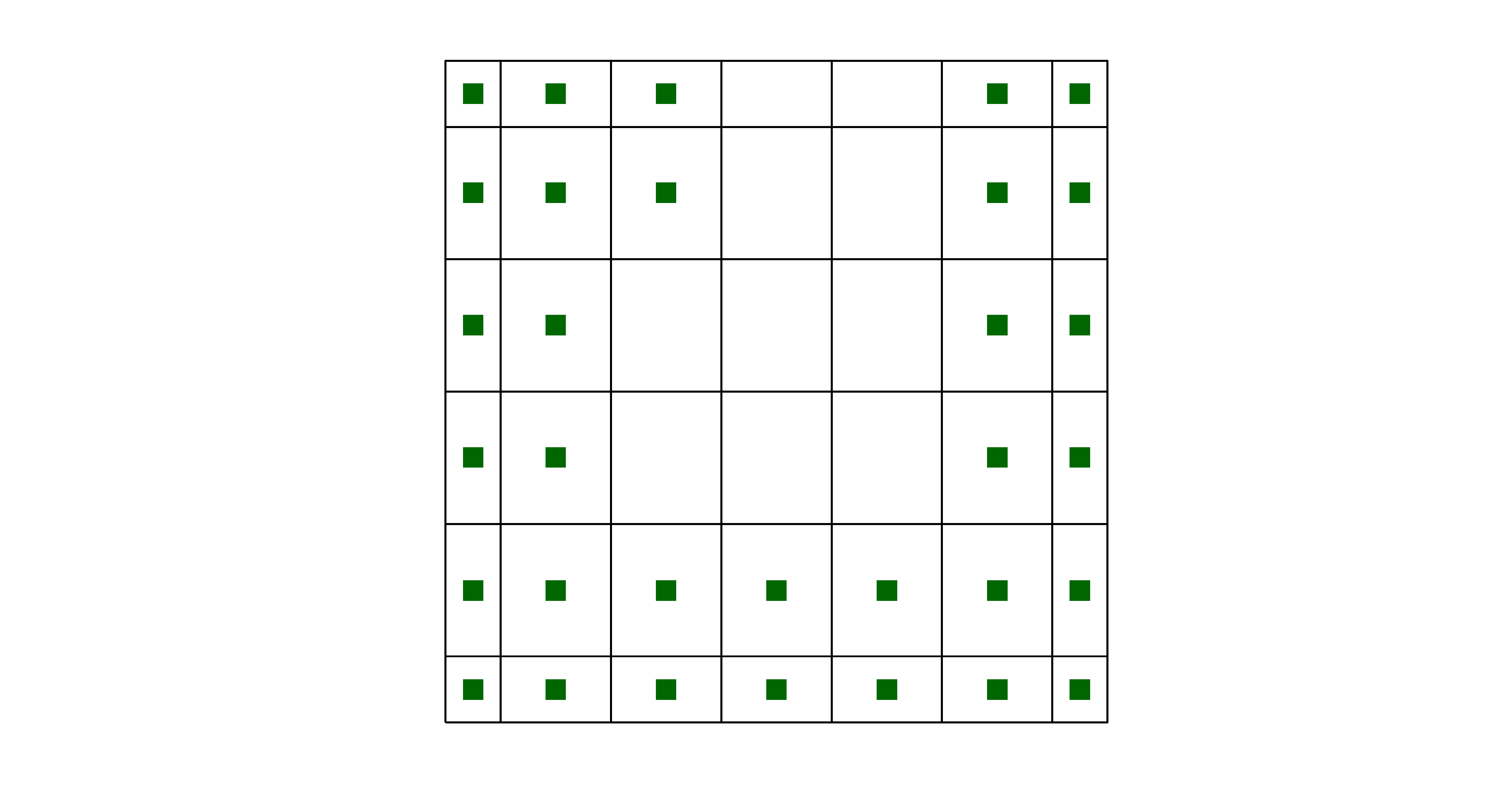}
\includegraphics[width=0.4\textwidth, trim = 11cm 2cm 11cm 1cm, clip]{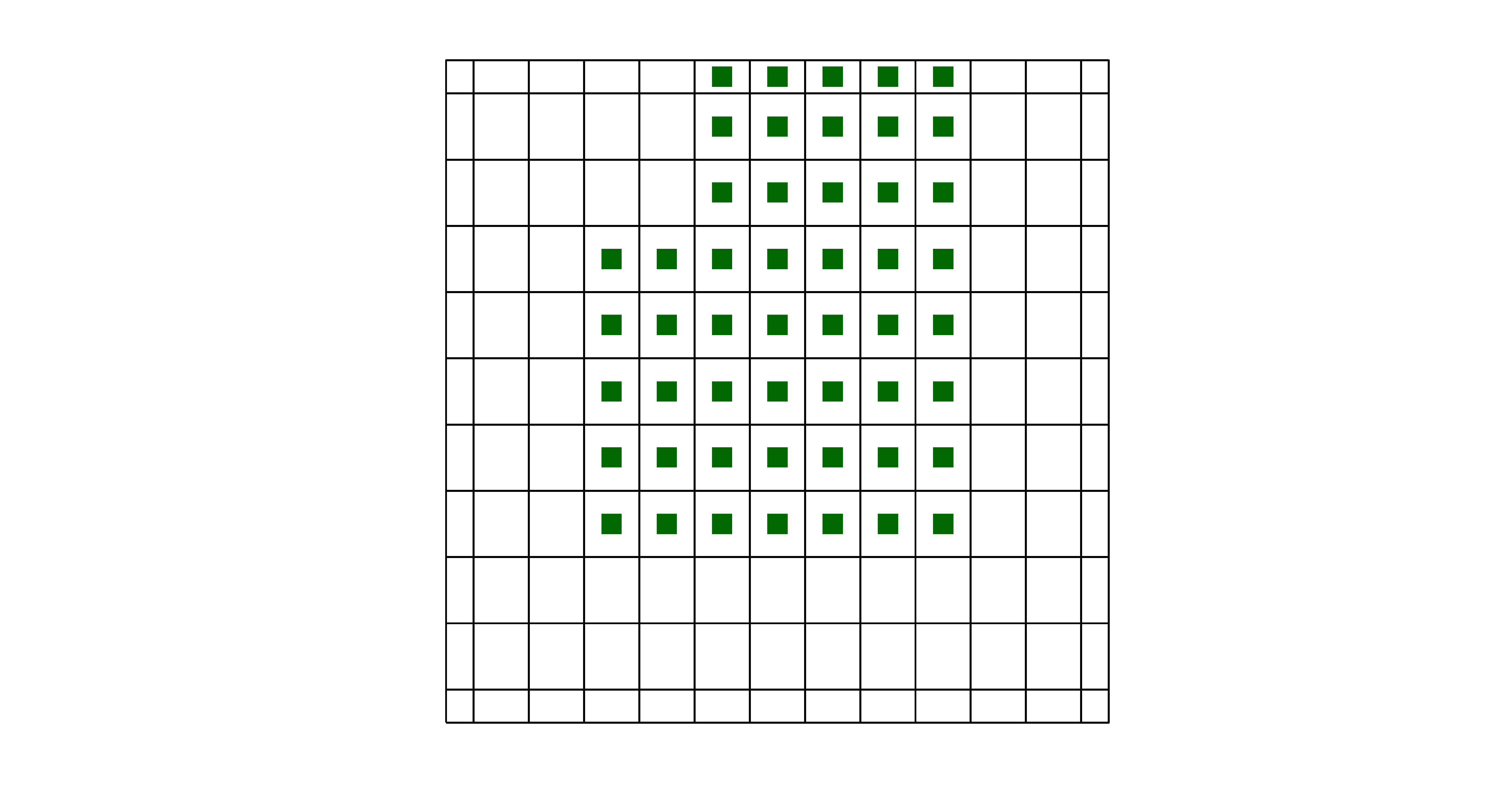}}
\caption{Greville faces associated with the active 2-form basis functions for levels $\ell = 0$ (left) and $\ell = 1$ (right) in the hierarchical B\'ezier mesh depicted in Figure~\ref{fig:hierarchical_bezier} assuming the 0-form basis functions are maximally continuous and of degree $p_1 = p_2 = 2$.}
\label{fig:2forms_two_levels}
\end{figure}

\begin{lemma}\label{lemma:subcomplex}
For $0 \le \ell \le \ell' \le N$, $(\Xhat_{\ell,\ell'},d)$ and $(\What_\ell,d)$ are Hilbert complexes, not necessarily exact.
\end{lemma}
\begin{proof}
The result for $\Xhat^k_{\ell,\ell'}$ is trivial, because clearly $d^k(\Xhat^k_{\ell,\ell'}) \subset \Xhat^{k+1}_{\ell,\ell'}$ and $(\Xhat_{\ell,\ell'},d)$ is a subcomplex of $(\Xhat_{\ell,h},d)$. For $\What^k_\ell$, since the spaces are contained in $\Xhatlh{k}$, the differential operators $d^k$ are well defined, and to prove the result we proceed by induction. The result is true for $\ell =0$, since $\What^k_0 = \hat X^k_{0,h}$ are the tensor-product spline spaces in Section~\ref{sec:spline-complex}. Assuming that $(\What_\ell,d)$ forms a Hilbert complex, the result holds for $(\What_{\ell+1},d)$, because $\What^k_{\ell+1} = \What^k_{\ell} + \Xhat^k_{\ell+1,{\ell+1}}$ by Lemma~\ref{lemma:cupcap}, and it is immediate to see that $d^k (\What^k_{\ell} + \Xhat^k_{\ell+1,{\ell+1}}) \subset (\What^{k+1}_{\ell} + \Xhat^{k+1}_{\ell+1,{\ell+1}})$.
\end{proof}

The result of Lemma~\ref{lemma:subcomplex} for $\ell = N$ guarantees that the hierarchical spline spaces $\What^k_N$ form a Hilbert complex, but in general the complex {\bf \emph{may not be exact}}. Actually, it is relatively easy to find counterexamples where the hierarchical spline spaces do not form an exact sequence. For instance, in the two-dimensional case a necessary (but not sufficient) condition for exactness is
\begin{equation*}
\dim (\What^0_N) + \dim (\What^2_N) = \dim (\What^1_N) + 1.
\end{equation*}
We can see that the condition is not fulfilled for the maximally continuous hierarchical B-spline complex with $p_1 = p_2 = 3$ defined on the hierarchical B\'{e}zier mesh depicted in Figure~\ref{fig:2x2}. In this case, a simple computation gives that the dimensions of the hierarchical spaces with boundary conditions are $\dim (\What^0_N) = 147$, $\dim (\What^1_N) = 328$, $\dim(\What^2_N) = 181$, which violates the previous condition.

\begin{figure}[t]
\centerline{\includegraphics[width=0.45\textwidth]{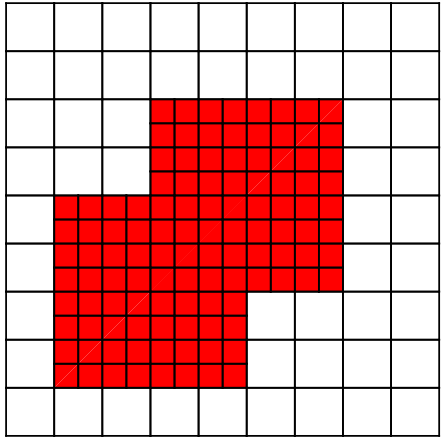}}
\caption{A hierarchical B\'{e}zier mesh for which the sequence is not exact for $p_1 = p_2 = 3$.}
\label{fig:2x2}
\end{figure}

Exactness of the hierarchical B-spline complex is crucial for the correct behavior of discretization schemes based on the complex, as we will see later in the examples of Section~\ref{sec:stability}. In Section~\ref{sec:exactness}, we develop conditions guaranteeing exactness of the complex based on an analysis of the spaces $\Xhat^k_{\ell,\ell}$ and $\Xhat^k_{\ell,\ell+1}$ appearing in the construction of the hierarchical B-spline complex. In this analysis, we invoke the concept of a Greville subgrid which we now describe in detail.

\begin{remark}
In the two-dimensional setting, a \textbf{rotated hierarchical B-spline complex} may easily be constructed by rotating the orientation of the discrete 1-forms by $\pi/2$ in parametric space, just as was the case for tensor-product B-splines.
\end{remark}

\subsection{The Spline Complex Restricted to a Submesh and the Greville Subgrid}\label{sec:submesh}
Consider a particular level $\ell$, and consider a subdomain $\Omega_{\ell'}$ with $0 \le \ell \le \ell' \le N-1$. From the strong condition in Section~\ref{sec:HB-splines}, we know that $\Omega_{\ell'}$ is a union of elements in $\M_\ell$ for $\ell' \in \{\ell,\ell+1\}$. 
As we already proved in Lemma~\ref{lemma:subcomplex}, the spaces $\Xhat^k_{\ell,\ell'}$ form the complex $(\Xhat_{\ell,\ell'},d)$, but the sequence they form is not necessarily exact.

Employing the geometrical entities in the Greville grid $\MG_\ell$ associated with the basis functions in ${\cal B}^k_{\ell,\ell'}$, we build a {\bf \emph{Greville subgrid}} that we denote by $\MG_{\ell,\ell'} \subset \MG_\ell$. With some abuse of notation, we will also denote by $\MG_{\ell,\ell'}$ (the interior of the closure of) the union of the cells in $\MG_{\ell,\ell'}$. In Figure~\ref{fig:submeshes} we have included, for the same spaces of tensor-product B-splines depicted in Figure~\ref{fig:meshes}, an example of a subdomain $\Omega_{\ell'}$ and the corresponding Greville subgrid. We notice that, in general, the subdomain defined by $\MG_{\ell,\ell'}$ is different from $\Omega_{\ell'}$, and in fact it will depend on the degree of the spline space. We also notice that, since the support of the B-spline basis functions is reduced with the degree, the basis functions associated to the entities on the boundary of $\MG_{\ell,\ell'}$ do not belong to ${\cal B}^k_{\ell,\ell'}$.

\begin{figure}[tp]
\begin{subfigure}[The subdomain $\Omega_{\ell'}$ as a B\'ezier submesh of $\M_\ell$]{
\centering
\includegraphics[width=3.0in]{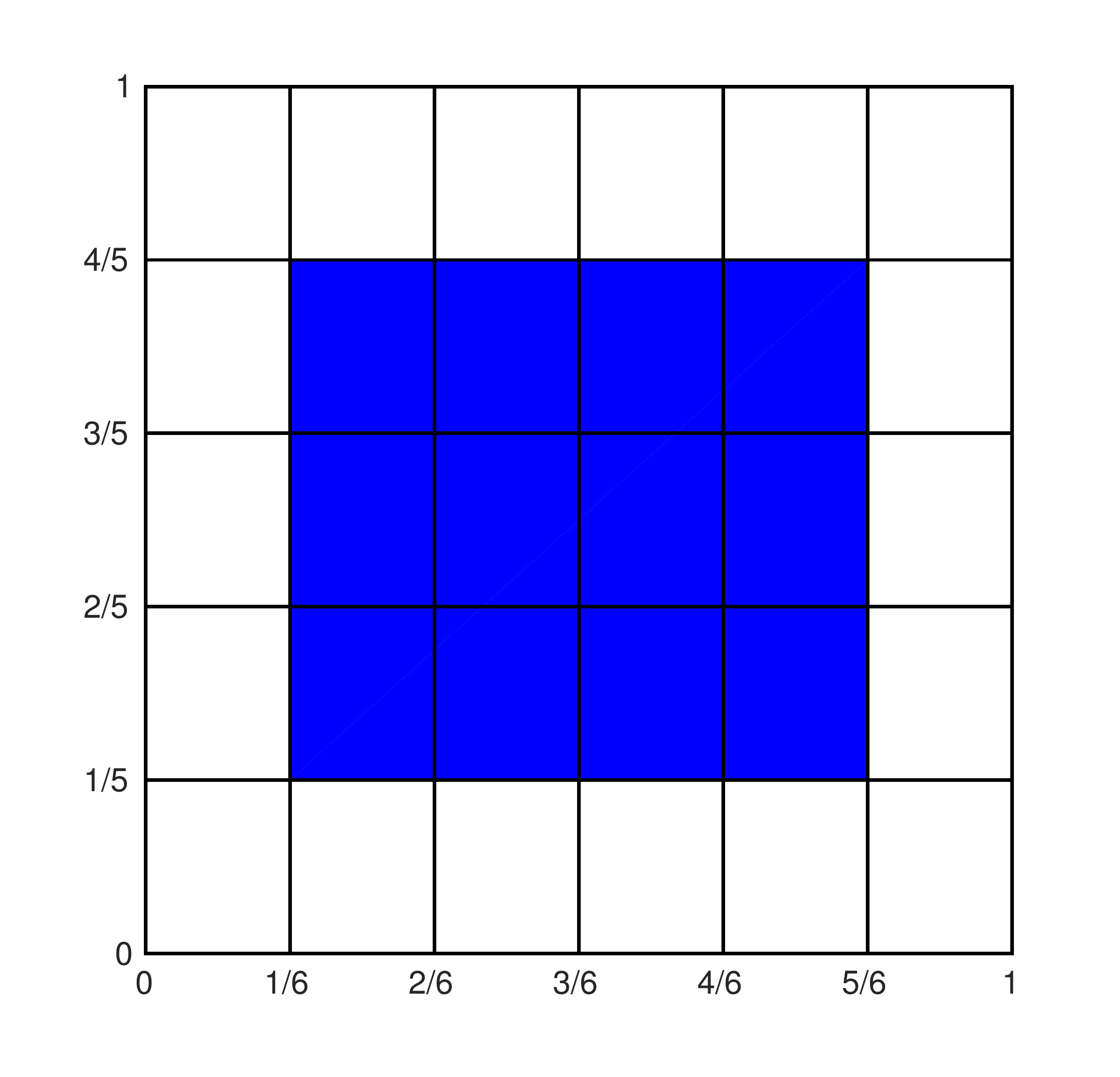}
\label{fig:subbezier}
}
\end{subfigure}
\begin{subfigure}[The corresponding Greville subgrid $\MG_{\ell,\ell'}$]{
\includegraphics[width=3.0in]{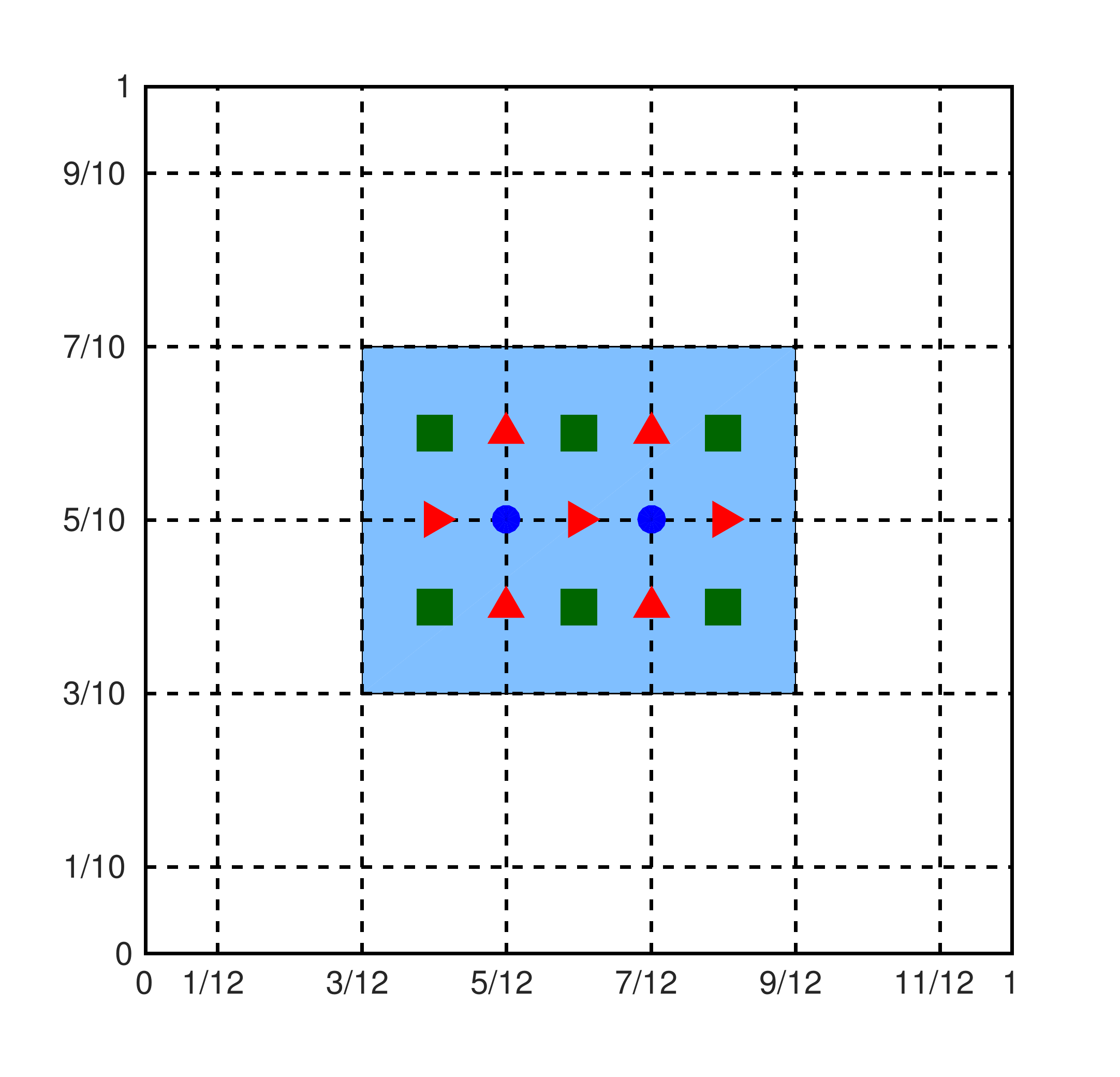}
\label{fig:subgreville}
}
\end{subfigure}
\caption{Construction of the Greville subgrid $\MG_{\ell,\ell'}$. Considering the same degrees and knot vectors as in Figure~\ref{fig:meshes} to represent $\M_\ell$ and $\MG_\ell$, we choose the subdomain  $\Omega_{\ell'} = (1/6,5/6) \times (1/5,4/5)$ in (a). The Greville subgrid $\MG_{\ell,\ell'}$ is formed by the darkened elements in (b), while the geometric entities corresponding to basis functions in ${\cal B}^k_{\ell,\ell'}$ are marked by blue dots (0-forms), red arrows (1-forms) and green squares (2-forms).}
\label{fig:submeshes}
\end{figure}

Finally, for each level $\ell$, we can define isomorphisms as in the diagram \eqref{eq:cd-iga-fem} that relate the B-spline complex to the low-order finite element complex defined on the Greville grid. Considering the subdomain $\Omega_{\ell'}$, the same isomorphisms relate the spline subspaces $\Xhat^k_{\ell,\ell'}$ to the spaces of the finite element complex defined on the Greville subgrid, which we henceforth denote by $\hat{Z}^k_{\ell,\ell'}$, and a commutative diagram can also be defined for these subspaces.

\begin{remark}
Since we are assuming homogeneous boundary conditions on $\partial \Omega$, the finite element spaces $\hat{Z}^k_{\ell,\ell'}$ have homogeneous boundary conditions on $\partial \MG_{\ell,\ell'}$. Removing this assumption would cause the finite element spaces to have homogeneous boundary conditions on $\partial \MG_{\ell,\ell'} \setminus (\partial \MG_{\ell,\ell'} \cap \partial \hat \Omega)$. The results of the paper can be easily applied to this case, but writing it in detail would only add more complexity without giving more insight.
\end{remark}

\subsection{Extension to the Physical Domain and Multi-patch Geometries}

The hierarchical B-spline complex in the physical domain is built in exactly the same manner as the B-spline complex. Namely, assuming that we have a parametrization ${\bf F}: \hat \Omega \rightarrow \Omega$, we simply use the pullback operators defined in Subsection \ref{subsec:pullback}, yielding:
\begin{equation*}
W_\ell^k\left(\Omega\right) := \left\{ \omega \in X^k\left(\Omega\right) : \iota^k(\omega) \in \What^k_\ell \right\}
\end{equation*}
for all integers $0 \leq k \leq n$ and $0 \leq \ell \leq N$.

To define hierarchical differential forms in a multi-patch domain, one must take into account the presence of non-conforming meshes between patches \cite{Scott2014222}. This can be easily done following the approach used in \cite{Buchegger2016159} for 0-forms. We start by defining a multi-patch B-spline complex for each level, exactly as in Section~\ref{sec:multipatch}. Then, we apply the hierarchical algorithm of Definition~\ref{def:hierarchical} to the multi-patch spaces of the complex, taking into account that the supports of interface functions overlap two or more patches. 
We show an example of the construction for a multi-patch, maximally continuous hierarchical B-spline complex with $p_1 = p_2 = 2$ defined on the hierarchical B\'{e}zier mesh in Figure~\ref{fig:multipatch}, where the thick line represents the interface between both interfaces. The set of active basis functions for 0-forms is represented in the Greville grid in Figure~\ref{fig:multi_0forms}, for levels 0 and 1. We note the presence of an active basis function at level 0 on the interface that would not appear on the left patch when considering the patches separately. This is due to the fact that the basis function is supported on both patches. Similarly, some functions of level 1 on the interface are not active but would be active on the left patch if we considered the patches separately. A similar behavior is also encountered for the 1-forms depicted in Figure~\ref{fig:multi_1forms}.

\begin{figure}[t]
\centerline{\includegraphics[width=0.6\textwidth, trim=5cm 4cm 4cm 3cm, clip]{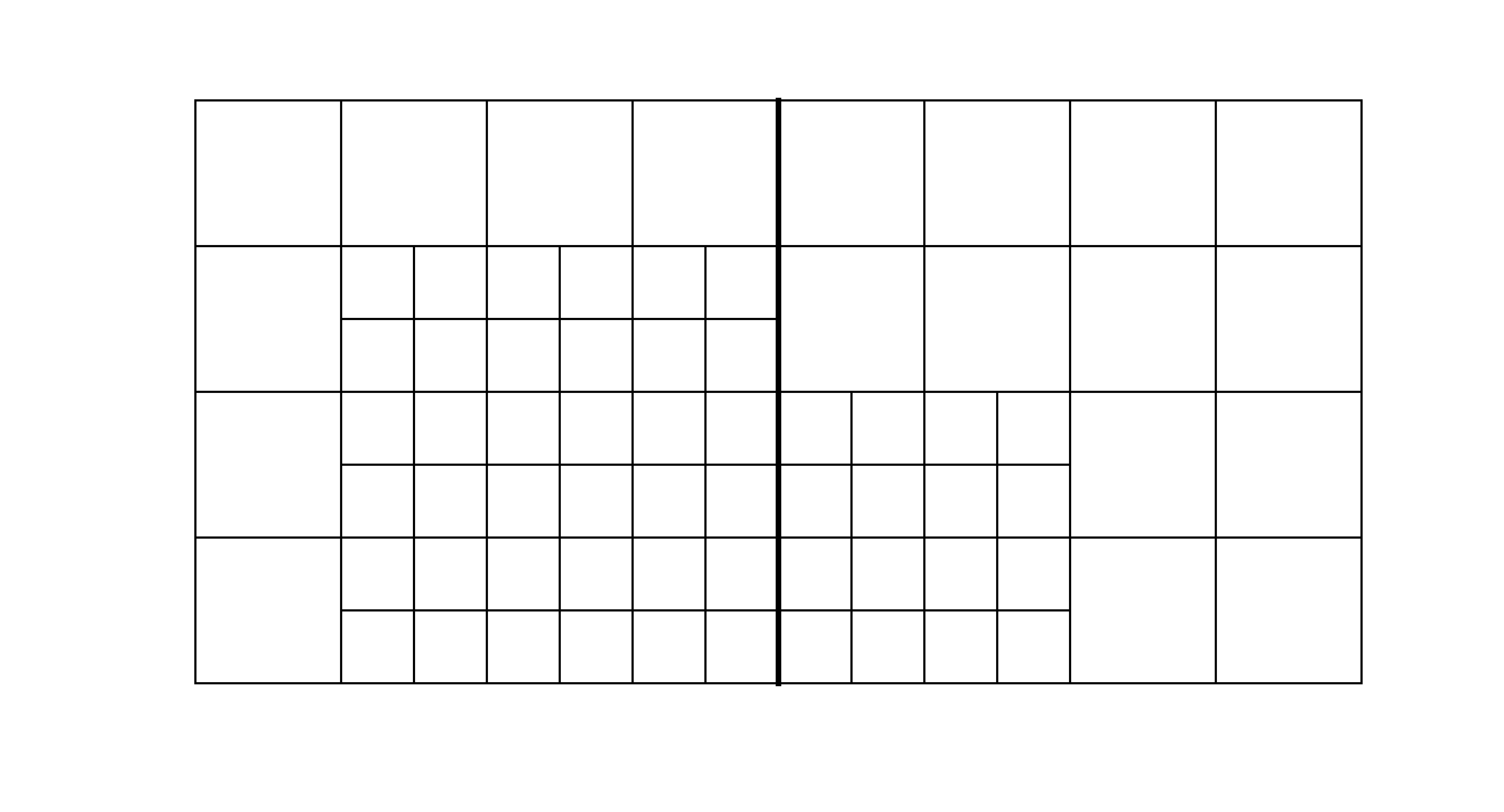}}
\caption{A hierarchical B\'{e}zier mesh for a simple multi-patch geometry with $p_1 = p_2 = 2$.}\label{fig:multipatch}
\end{figure}

\begin{figure}[t]
\begin{subfigure}[Active basis functions of level 0]{
\includegraphics[width=0.49\textwidth, trim=5cm 4cm 4cm 3cm, clip]{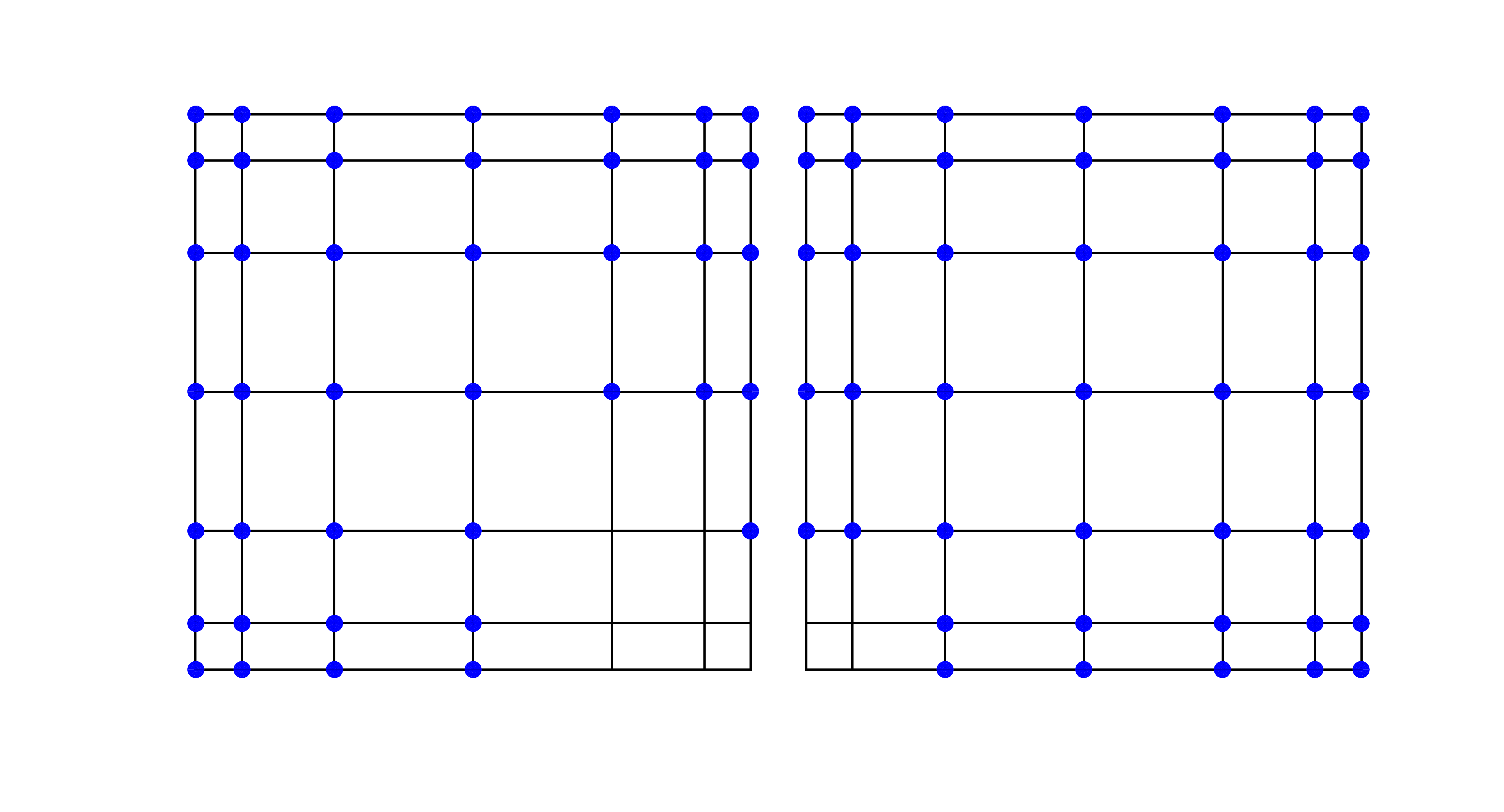}
\label{fig:multi_0forms_1lev}
}
\end{subfigure}
\begin{subfigure}[Active basis functions of level 1]{
\includegraphics[width=0.49\textwidth, trim=5cm 4cm 4cm 3cm, clip]{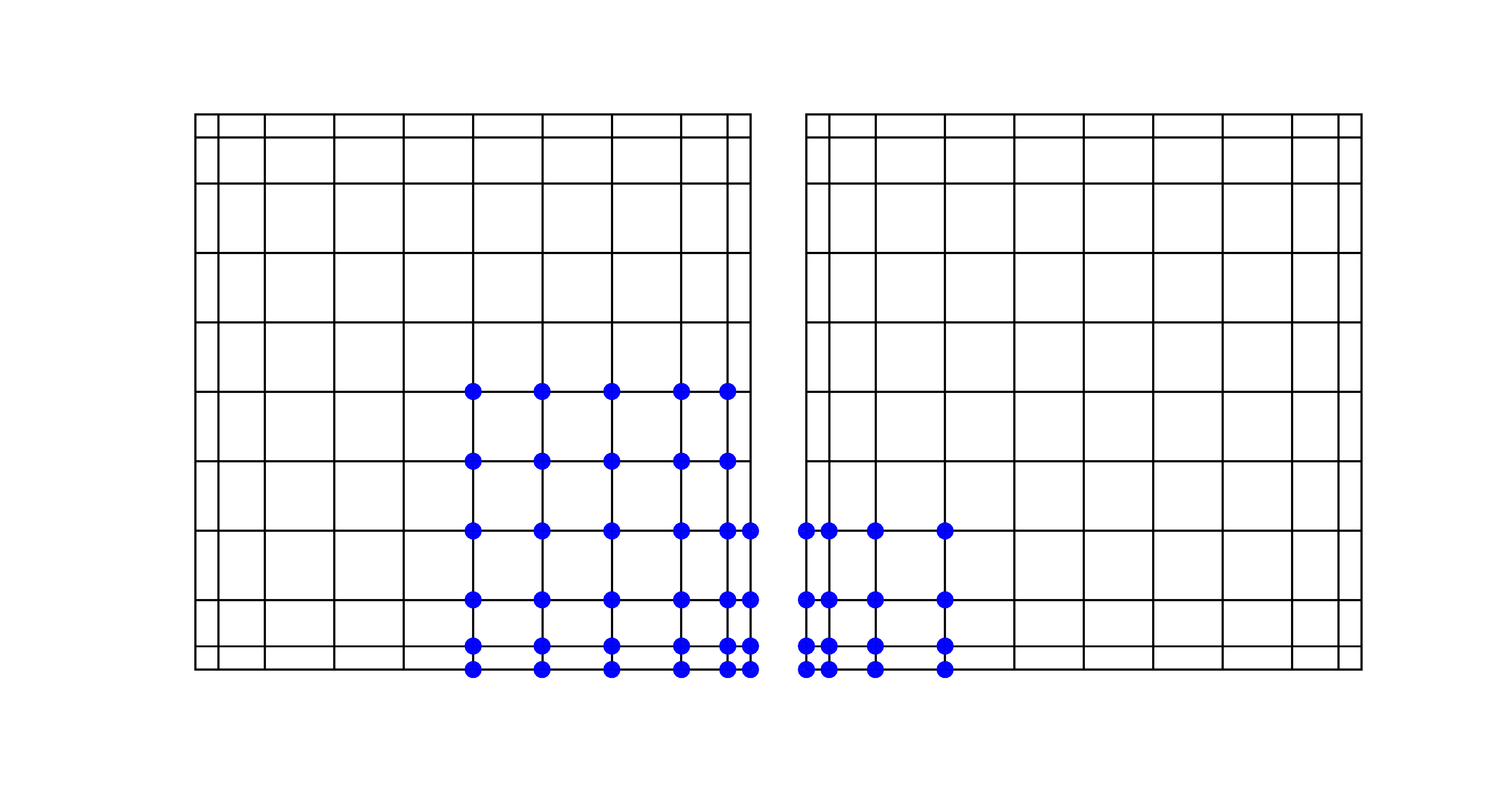}
\label{fig:multi_0forms_2lev}
}
\end{subfigure}
\caption{Geometric entities in the Greville grids associated with active basis functions for 0-forms corresponding to the multi-patch example in Figure~\ref{fig:multipatch}.}
\label{fig:multi_0forms}
\end{figure}

\begin{figure}[h!]
\begin{subfigure}[Active basis functions of level $\ell = 0$]{
\includegraphics[width=0.49\textwidth, trim=5cm 4cm 4cm 3cm, clip]{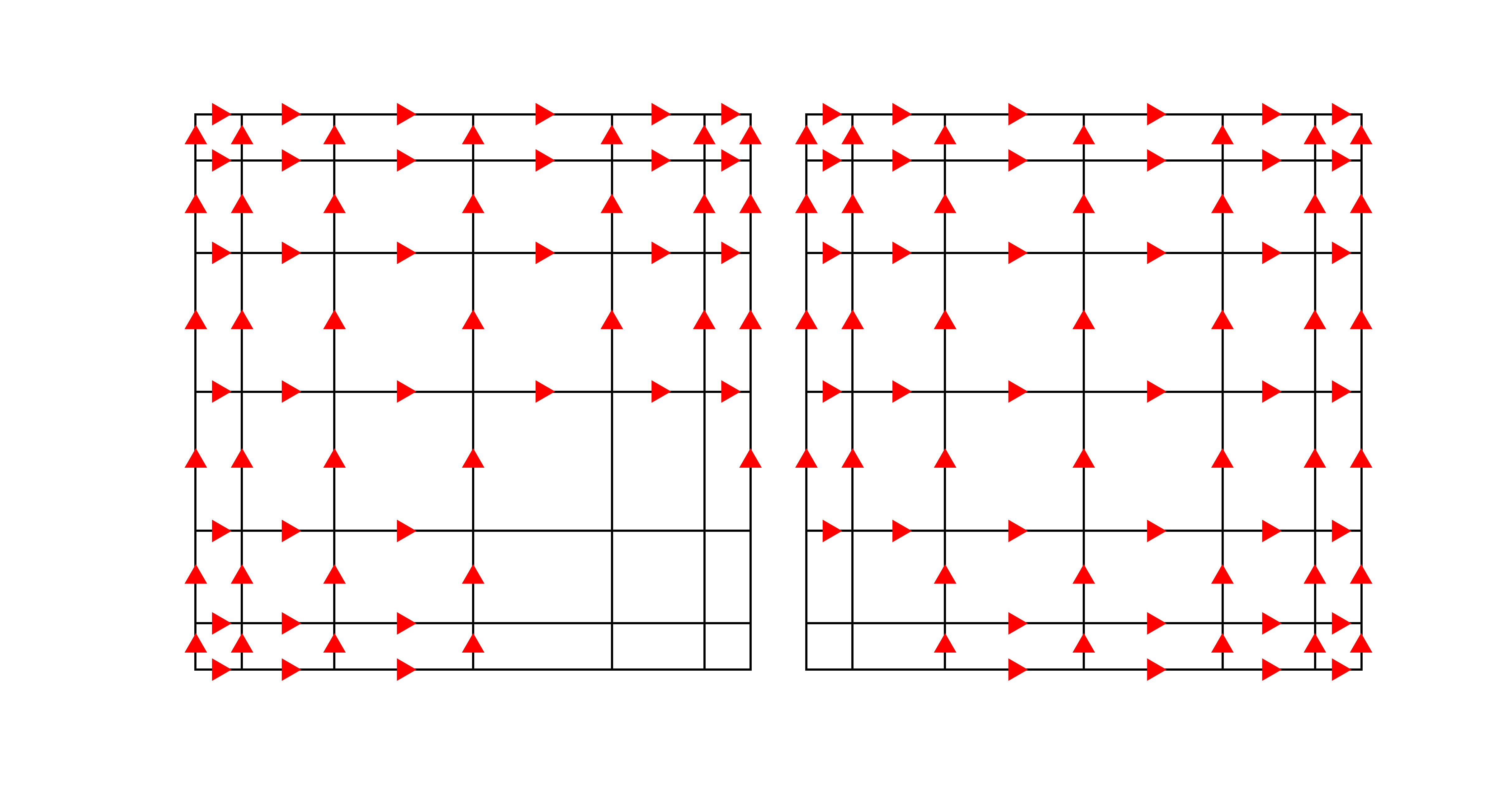}
\label{fig:multi_1forms_1lev}
}
\end{subfigure}
\begin{subfigure}[Active basis functions of level $\ell = 1$]{
\includegraphics[width=0.49\textwidth, trim=5cm 4cm 4cm 3cm, clip]{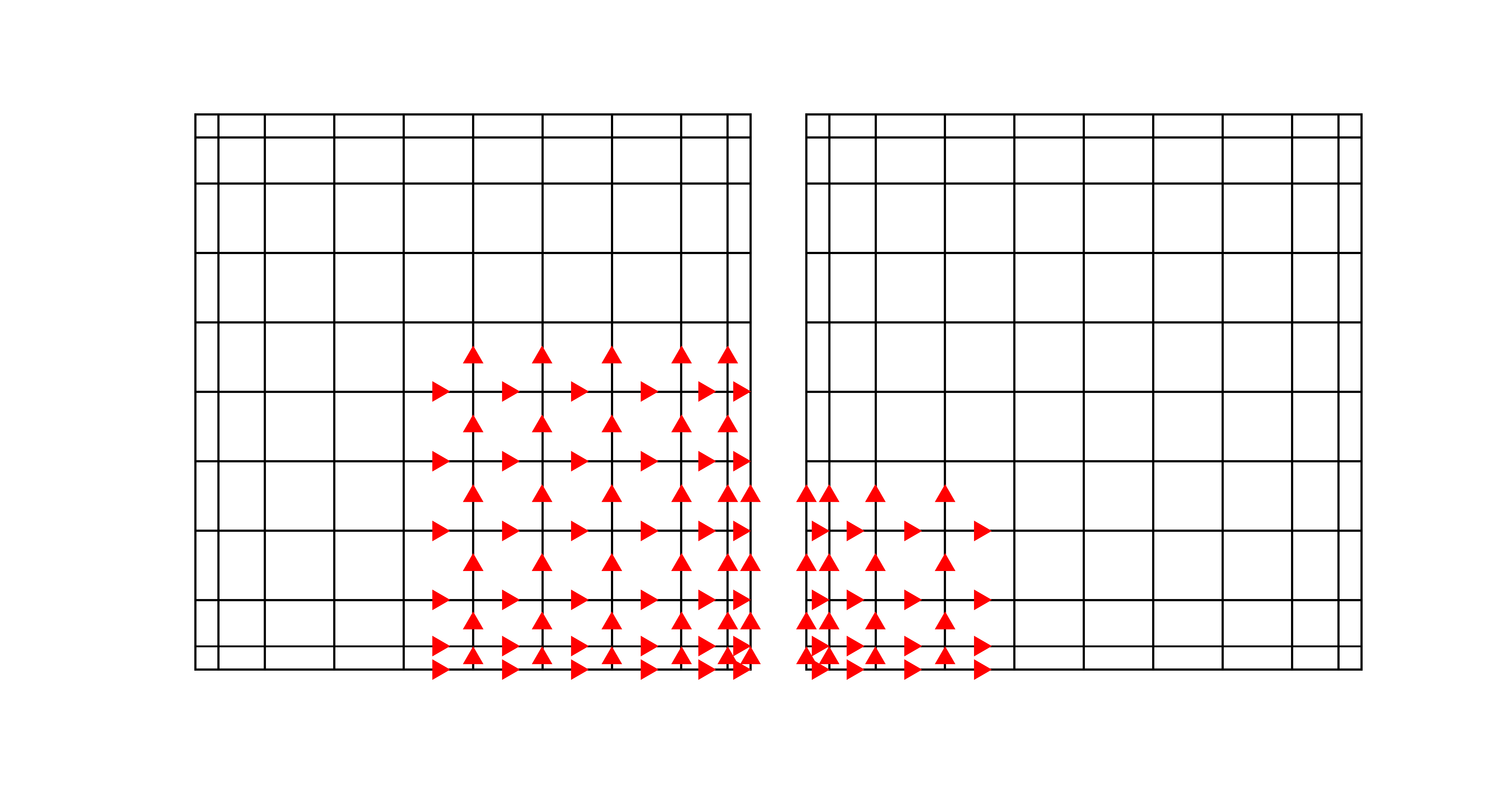}
\label{fig:multi_1forms_2lev}
}
\end{subfigure}
\caption{Geometric entities in the Greville grids associated with active basis functions for 1-forms corresponding to the multi-patch example in Figure~\ref{fig:multipatch}.}
\label{fig:multi_1forms}
\end{figure}

\section{Exactness of the Hierarchical B-spline Complex} \label{sec:exactness}
In this section, we begin the analysis of the hierarchical B-spline complex, and in particular, we study under which conditions the complex is exact. For our analysis, we make use of some powerful tools from abstract cohomology theory that we present at the beginning of the section. These results are then applied to the particular case of the hierarchical B-spline complex, from which we obtain a sufficient and necessary condition for exactness, given in Theorem~\ref{th:exact}. Since the proof is based on abstract cohomology theory, the condition we obtain is valid in a very general setting: it is dimension independent and can be also applied to other hierarchical spaces constructed with the algorithm in Definition~\ref{def:hierarchical}, including hierarchical T-splines \cite{Evans15}. We also give an interpretation of the derived condition based on Greville subgrids. 

The drawback of the aforementioned condition is that is not easily computable. For this reason, we introduce a local condition related to the support of the basis functions on the hierarchical B\'ezier mesh which is local and easier to implement. The last part of this section is devoted to proving this is a sufficient condition for exactness for the two-dimensional case. Some of the proofs in this section are rather technical, hence we will relegate the details to the appendix.


\subsection{Background Results from Cohomology Theory}\label{sec:homology}


\begin{lemma}\label{lemma:cupcap_derham}
Let $(V_h,d)$ and $(W_h,d)$ be two subcomplexes of $(V,d)$. Then $(V_h \cap W_h,d)$ and $(V_h + W_h, d)$ are also subcomplexes.
\end{lemma}
\begin{proof}
The result is trivial, since $V^k_h \cap W^k_h \subset V^k$ and $V^k_h + W^k_h \subset V^k$, and moreover we also have $d^k(V_h^k \cap W_h^k) \subset V_h^{k+1} \cap W_h^{k+1}$ and $d^k(V_h^k + W_h^k) \subset V_h^{k+1} + W_h^{k+1}$.
\end{proof}

Let $(V,d_V)$ and $(W,d_W)$ be two Hilbert complexes, and let $f^k : V^k \rightarrow W^k$ be a set of linear maps satisfying $d_W^k f^k = f^{k+1} d_V^k$. Following \cite{AFW-2} we call this a cochain map, and we denote by $f^*$ its induced map in the cohomology groups. The map induced in cohomology by a composition of cochain maps takes the form $(f \circ g)^* = f^* \circ g^*$. Given a Hilbert complex $(V,d)$ and any subcomplex $(V_h,d)$, the inclusion map $j:V_h \rightarrow V$ is a cochain map which induces a map between the cohomology spaces, $j^* : H^k(V_h) \longrightarrow H^k(V)$.

\begin{lemma}\label{lemma:MV}
Let $(V_h,d)$ and $(W_h,d)$ be two subcomplexes of the Hilbert complex $(V,d)$, with cohomology groups $H^k(V_h)$ and $H^k(W_h)$. Then we have an exact Mayer-Vietoris sequence of the form
\begin{equation*}
\begin{CD}
\ldots \rightarrow H^k(V_h \cap W_h) @>\varphi^*>> H^k(V_h) \oplus H^k(W_h)  @>\psi^*>> H^k(V_h + W_h) @>d^*>> H^{k+1}(V_h \cap W_h) \rightarrow \ldots
\end{CD}
\end{equation*}
\end{lemma}
\begin{proof}
From Lemma~\ref{lemma:cupcap_derham}, both $(V_h \cap W_h,d)$ and $(V_h + W_h,d)$ are subcomplexes. We define the inclusion maps $i_{V_h}$ and $i_{W_h}$ from $V_h \cap W_h$ into $V_h$ and $W_h$, respectively, and also the inclusion maps $j_{V_h}$ and $j_{W_h}$ into $V_h + W_h$, from $V_h$ and $W_h$, respectively. It is trivial to prove that the following sequence is exact
\begin{equation} \label{eq:short}
\begin{CD}
0 @>>> V_h \cap W_h @>\varphi>> V_h \oplus W_h  @>\psi>> V_h + W_h @>>> 0,
\end{CD}
\end{equation}
where $\varphi(u) = (u,u) \equiv (i_{V_h} u, i_{W_h} u)$, and $\psi (v,w) = v-w \equiv j_{V_h}v - j_{W_h}w$. This induces a Mayer-Vietoris sequence of the form (see Chapters~2 and~3 in \cite{Hatcher})
\begin{equation*}
\begin{CD}
\ldots \rightarrow H^k(V_h \cap W_h) @>\varphi^*>> H^k(V_h) \oplus H^k(W_h)  @>\psi^*>> H^k(V_h + W_h) @>d^*>> H^{k+1}(V_h \cap W_h) \rightarrow \ldots
\end{CD}
\end{equation*}
The map $d^*$ is defined in the following way: given a closed form  $u \in Z(V_h + W_h)$, from the exactness of \eqref{eq:short} there exist $v \in V_h$ and $w \in W_h$ such that $u = \psi (v,w) = v-w$, and since it is a closed form $du = dv - dw = 0$. Thus in the cohomology groups we can define $d^*[u] = [dv] = [dw]$, where the brackets represent the equivalence class in the cohomology groups.
\end{proof}

\begin{theorem}\label{th:exactness}
Let $(V_h,d)$ and $(W_h,d)$ be two subcomplexes of the Hilbert complex $(V,d)$. Then the map induced by the inclusion, $i_{W_h}^*: H^k(V_h \cap W_h) \rightarrow H^k(W_h)$, is an isomorphism if and only if $j_{V_h}^*: H^k(V_h) \rightarrow H^k(V_h + W_h)$ is also an isomorphism.
\end{theorem}
\begin{proof}
We will only prove that if $i_{W_h}^*$ is an isomorphism, then $j_{V_h}^*$ is also an isomorphism. The second implication is proved similarly.

If $i_{W_h}^*$ is an isomorphism, then $\varphi^*$ is injective. Thus, the Mayer-Vietoris sequence can be split into the short exact sequences
\begin{equation*}
\begin{CD}
0 @>>> H^k(V_h \cap W_h) @>\varphi^*>> H^k(V_h) \oplus H^k(W_h)  @>\psi^*>> H^k(V_h + W_h) @>>> 0,
\end{CD}
\end{equation*}
which in particular implies that $\psi^*$ is surjective.

To prove that $j_{V_h}^*$ is an isomorphism, first we prove that it is injective. Assume that $v \in V_h$ is such that $j_{V_h}^*[v] = 0$. Then the pair $([v],0) \in \ker \psi^* = {\rm Im} \, \varphi^*$, that is, $\varphi^* [u] = ([v],0)$ for some $u \in V_h \cap W_h$. Since $\varphi^* = (i_{V_h}^*, i_{W_h}^*)$ and $i_{W_h}^*$ is injective, we get $[u] = 0$ and therefore $[v] = i_{V_h}^*[u] = 0$.

To prove the surjectivity of $j_{V_h}^*$ we use that $\psi^*$ is surjective. That is, given $z \in V_h+W_h$ there exist $v \in V_h$ and $w \in W_h$ such that $\psi^*([v],[w]) = j^*_{V_h} [v] - j^*_{W_h} [w] = [z]$. Since $i_{W_h}^*$ is surjective, there exists $u \in V_h \cap W_h$ such that $i_{W_h}^*[u] = [w]$, and by the exactness of the sequence $\varphi^*[u] = (i^*_{V_h}[u], [w]) \in \ker \psi^*$. As a consequence, $\psi^* ([v], [w]) - \psi^* (i^*_{V_h}[u], [w]) = j^*_{V_h}([v] - i^*_{V_h}[u]) = [z]$, which proves that $j_{V_h}^*$ is surjective.
\end{proof}

\begin{corollary}
Let $(V_h,d)$ an exact subcomplex of the Hilbert complex $(V,d)$, and $(W_h,d)$ another subcomplex not necessarily exact. Then the subcomplex $(V_h + W_h, d)$ is exact if and only if $i^*_{W_h} : H^k(V_h \cap W_h) \rightarrow H^k(W_h)$ is an isomorphism.
\end{corollary}


\subsection{A Necessary and Sufficient Condition for Exactness} \label{sec:nec-suff}
By applying the previous results to hierarchical B-splines, we obtain a necessary and sufficient condition for the exactness of the hierarchical B-spline complex. We first present the result, and then we explain how it can be applied in practice with the help of the Greville subgrids.

\begin{theorem} \label{th:exact}
The hierarchical B-spline complex $(\What_{\ell+1},d)$ is exact for $\ell = 0, \ldots, N-1$ if and only if the inclusion $j: \Xhat^k_{\ell, {\ell+1}} \rightarrow \Xhat^k_{\ell+1, {\ell+1}}$ induces an isomorphism between the cohomology groups $H^k(\Xhat_{\ell, {\ell+1}})$ and $H^k(\Xhat_{\ell+1,{\ell+1}})$ for $k = 1, \ldots, n$.
\end{theorem}
\begin{proof}
We proceed by induction. The exactness of the sequence holds for $\ell =0$ since $\What^k_0 = \hat X^k_{0,h}$ are tensor-product spline spaces, and the exactness of the tensor-product spline sequence~\eqref{eq:spline_complex} is known from \cite{BRSV11}. Assume now that the sequence $(\What_\ell,d)$ is exact. The result is immediately proved for level $\ell+1$ by applying Theorem~\ref{th:exactness} with $V_h^k = \What^k_\ell$ and $W_h^k = \Xhat^k_{\ell+1,{\ell+1}}$, using the relations given in Lemma~\ref{lemma:cupcap}.
\end{proof}

\noindent In simple words, the above condition tells us that, at each step of the recursive algorithm in Definition~\ref{def:hierarchical}, the spaces spanned by the basis functions to be removed and by the basis functions to be added should have the {\bf \emph{same topological properties}}, and these should be preserved by the inclusion operator indicated in Theorem \ref{th:exact}.

\subsubsection{An interpretation of the exactness condition using Greville subgrids}
The condition in Theorem~\ref{th:exact} may not be easy to understand for practitioners, as it requires a basic knowledge of homology theory. We give now an intuitive interpretation of this condition with the help of Greville subgrids. 

As we have seen in Section~\ref{sec:submesh}, there exist isomorphisms that define a commutative diagram between the spline subcomplex $(\Xhat_{\ell,\ell+1},d)$ (respectively, $(\Xhat_{\ell+1,\ell+1},d)$) and the low-order finite element complex $(\hat Z_{\ell,\ell+1},d)$ (resp. $(\hat Z_{\ell+1,\ell+1}, d)$) defined on the Greville subgrid. As a consequence, the topological properties of the spline subcomplex are identical to those of the finite element complex on the Greville subgrid. One can, indeed, consider the Greville subgrids associated to the basis functions to be added and removed and check that they have the same topological properties, that is, the same number of connected components, the same number of holes and, in the three-dimensional case, the same number of cavities.  If the Greville subgrids do not have the same topological properties for every $\ell = 0, \ldots, N-1$, then the condition in Theorem~\ref{th:exact} is not met and the hierarchical B-spline complex is not exact.

\begin{figure}[t]
\includegraphics[width=0.32\textwidth,trim=2cm 1.2cm 2cm 0cm, clip]{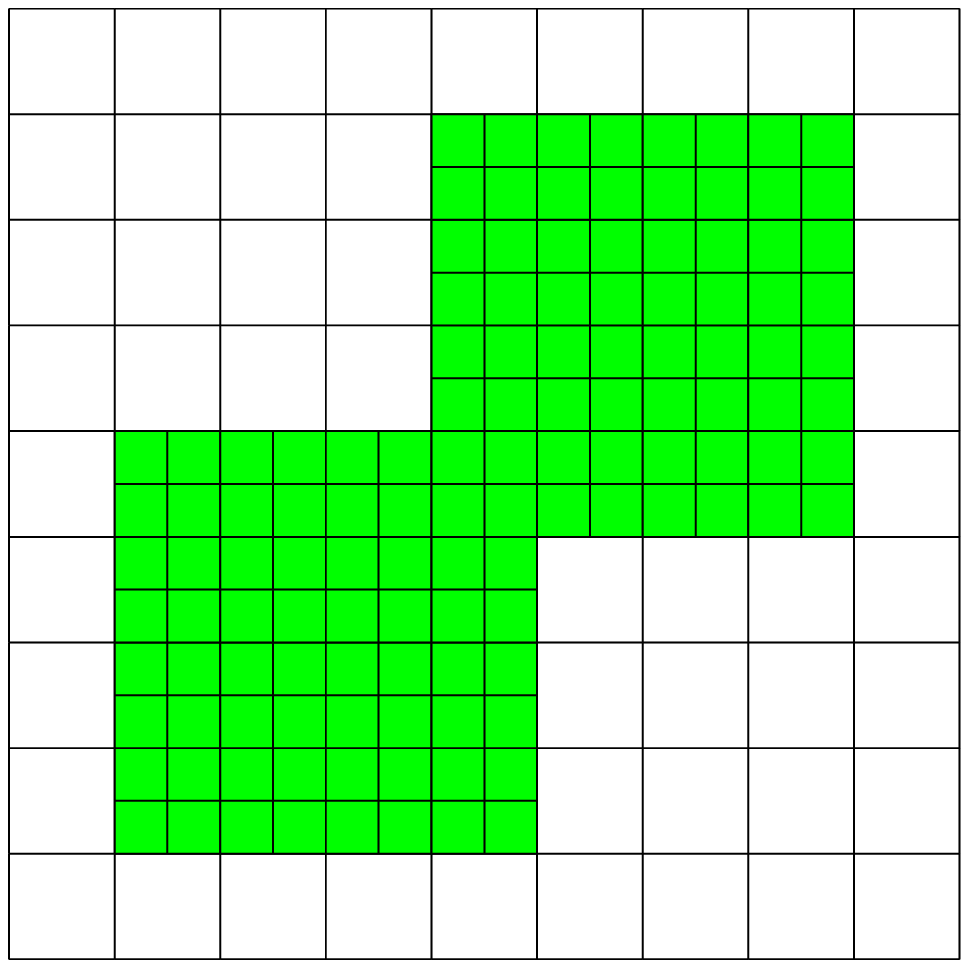}
\includegraphics[width=0.32\textwidth,trim=2cm 1.2cm 2cm 0cm, clip]{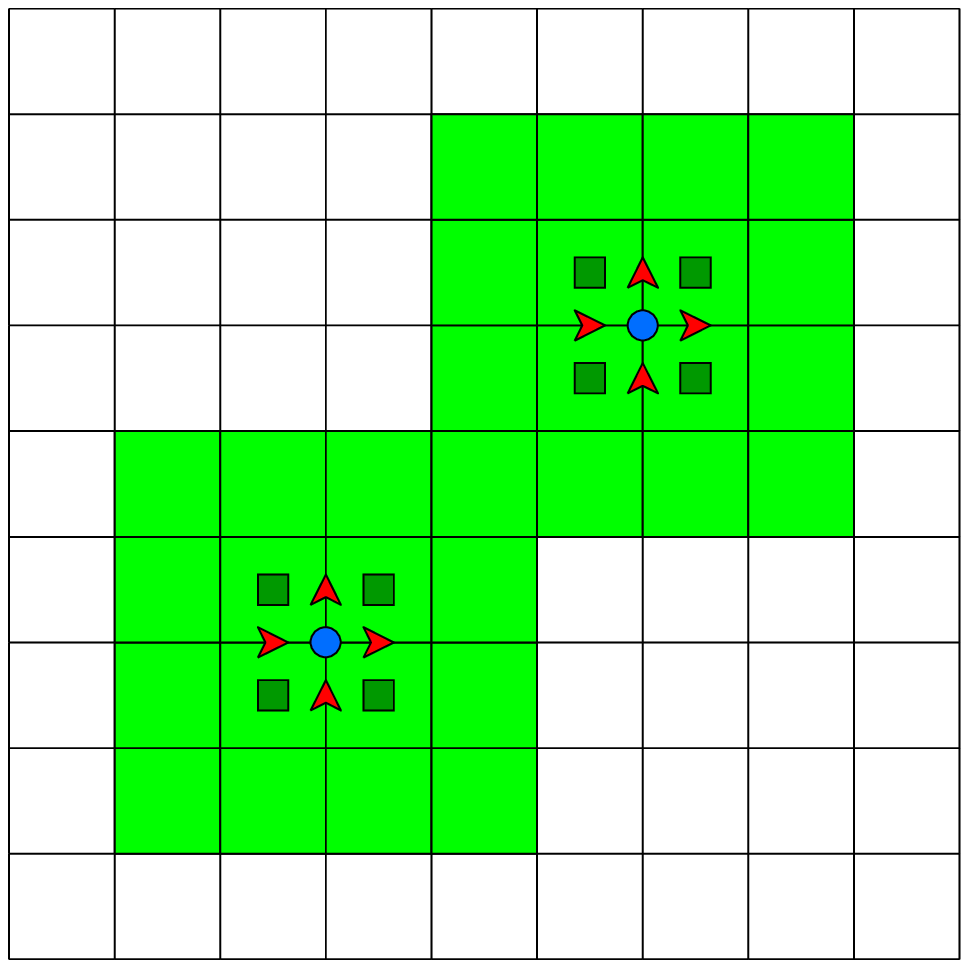}
\includegraphics[width=0.32\textwidth,trim=2cm 1.2cm 2cm 0cm, clip]{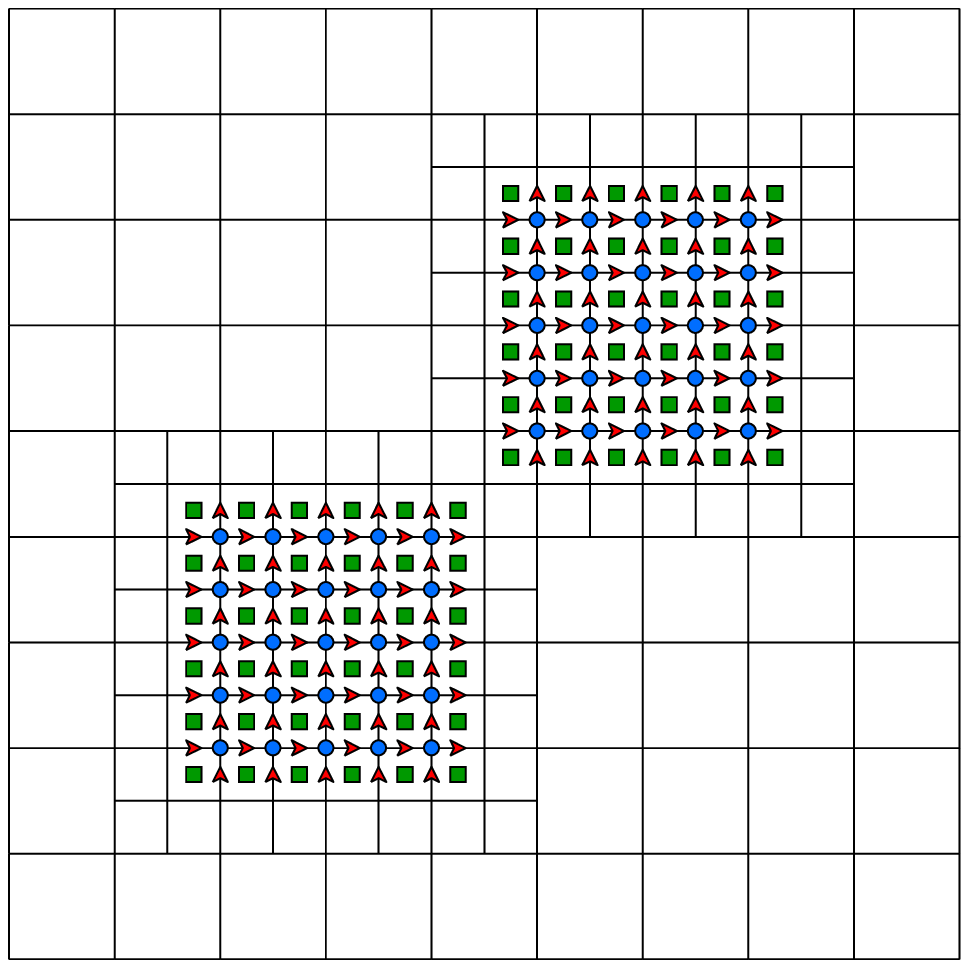}
\caption{An example of a maximally continuous hierarchical B-spline complex which is exact for $p_1=p_2=3$. On the left, the hierarchical B\'{e}zier mesh is depicted, with the refined region $\Omega_1$ colored in green. In the center and on the right, the vertices, edges, and faces on the Greville subgrids $\MG_{0,1}$ and $\MG_{1,1}$ are depicted, related to the spaces $\Xhat^k_{0,1}$ and $\Xhat^k_{1,1}$, respectively.}
\label{fig:remove-add-1x1}
\end{figure}

\begin{figure}[t!]
\includegraphics[width=0.32\textwidth,trim=2cm 1.2cm 2cm 0cm, clip]{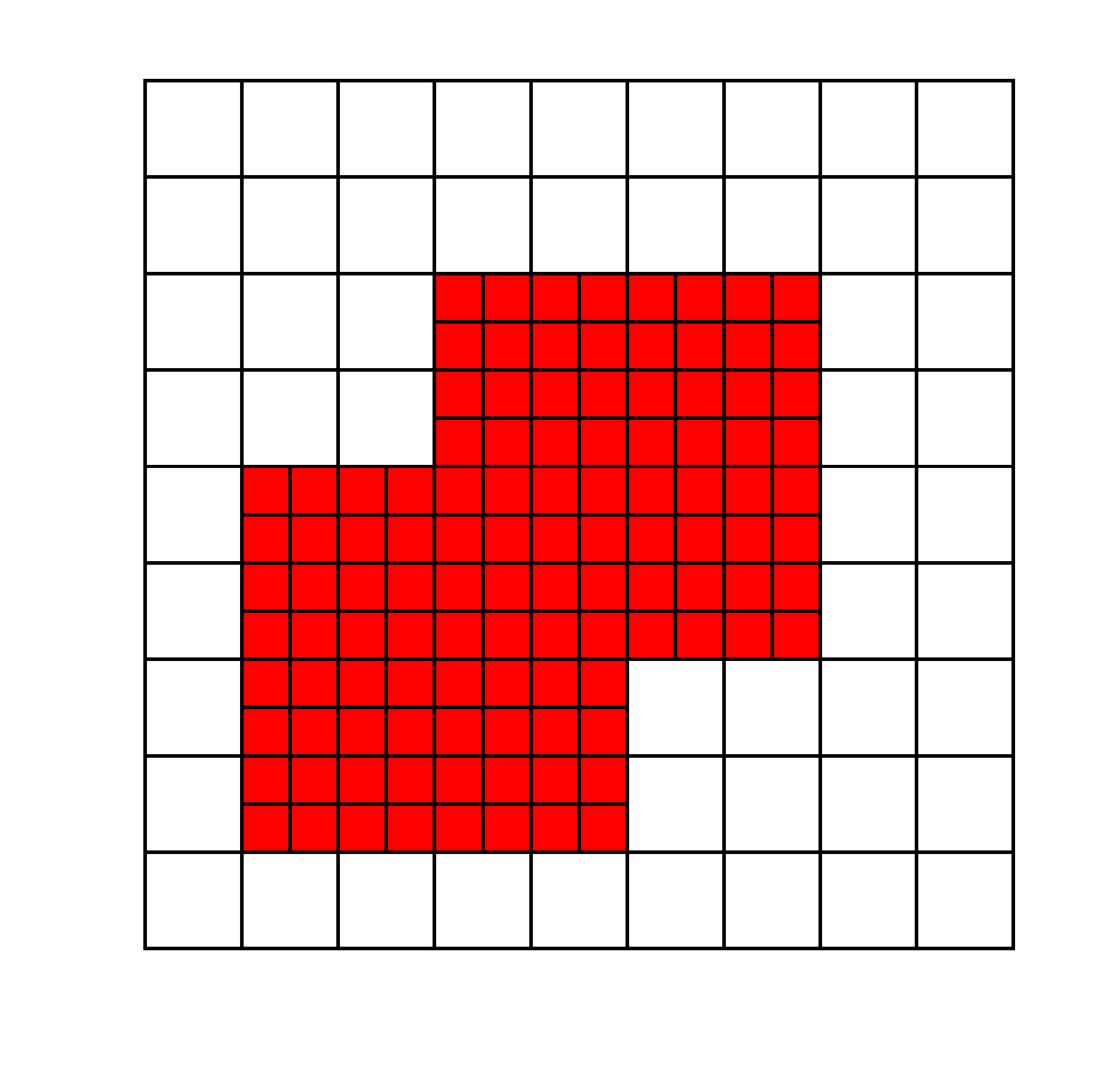}
\includegraphics[width=0.32\textwidth,trim=2cm 1.2cm 2cm 0cm, clip]{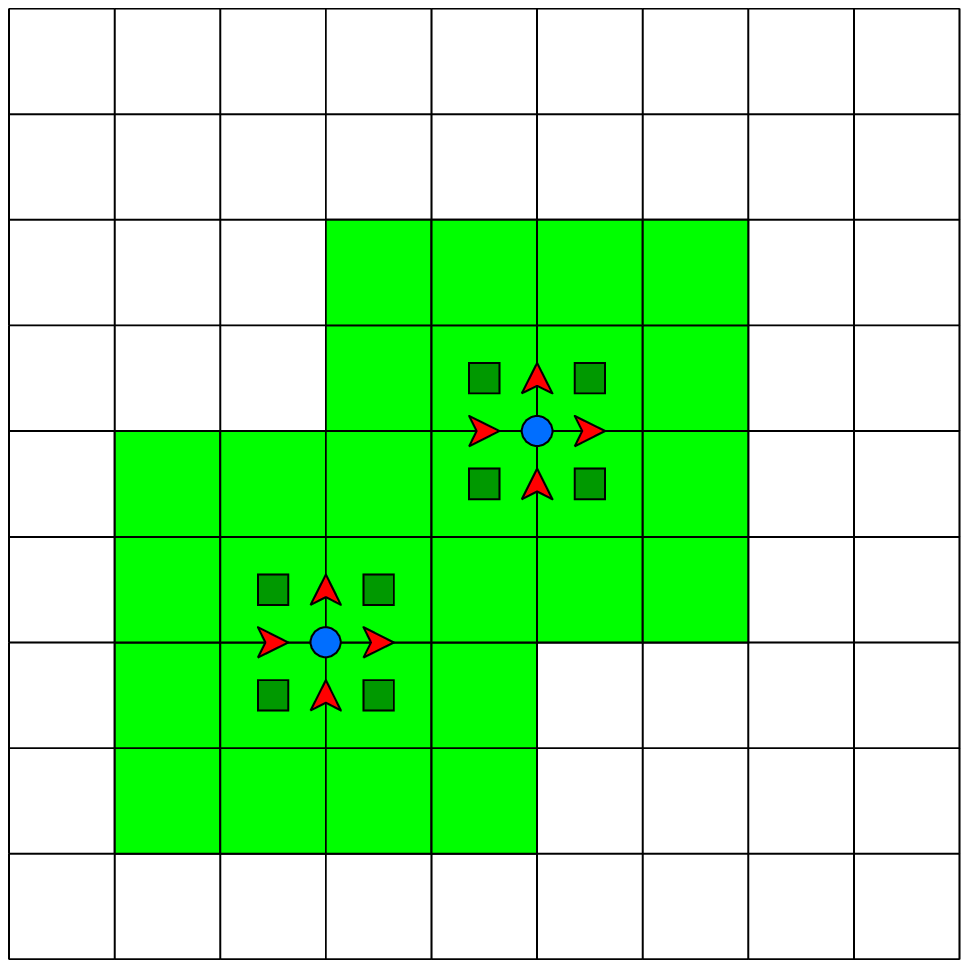}
\includegraphics[width=0.32\textwidth,trim=2cm 1.2cm 2cm 0cm, clip]{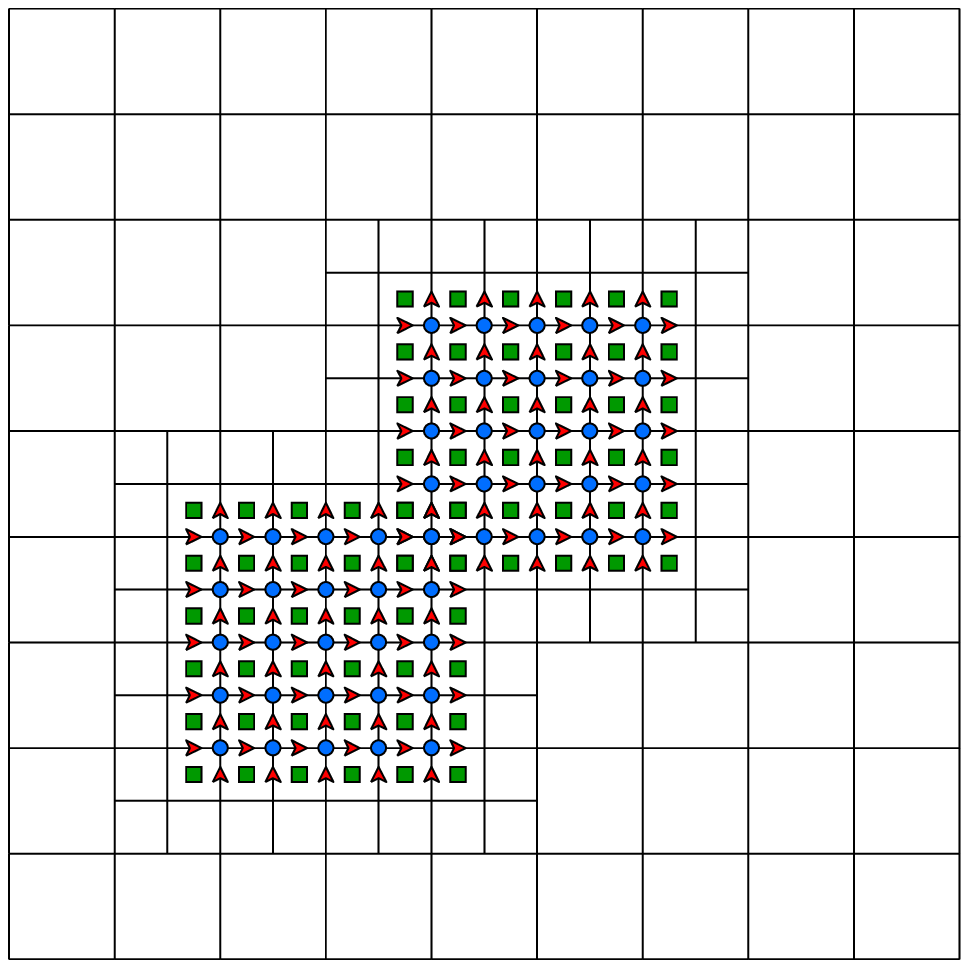}
\caption{An example of a maximally continuous hierarchical B-spline complex which is not exact for $p_1=p_2=3$. On the left, the hierarchical B\'{e}zier mesh is depicted, with the refined region $\Omega_1$ colored in red. In the center and on the right, the vertices, edges, and faces on the Greville subgrids $\MG_{0,1}$ and $\MG_{1,1}$ are depicted, related to the spaces $\Xhat^k_{0,1}$ and $\Xhat^k_{1,1}$, respectively.}
\label{fig:remove-add-2x2}
\end{figure}

Two examples are shown in Figures~\ref{fig:remove-add-1x1} and~\ref{fig:remove-add-2x2} for maximally continuous bicubic splines. The hierarchical B-spline complex associated to the mesh in Figure~\ref{fig:remove-add-1x1} is exact, because the two Greville subgrids associated to the spaces $\Xhat^k_{0,1}$ and $\Xhat^k_{1,1}$ each consist of two connected components without holes and thus have the same topology. Instead, the hierarchical B-spline complex of the example in Figure~\ref{fig:remove-add-2x2}, which is the same as that presented in Figure~\ref{fig:2x2}, does not provide an exact sequence because the number of connected components for the two Greville subgrids is different and changes from two to one.

\subsection{A Local and Sufficient Condition for Exactness} \label{sec:local}
Despite its generality, the condition for the exactness provided in Theorem~\ref{th:exact} is not easy to check in practice. In fact, in real three-dimensional applications, the Greville subgrids may become very intricate objects\footnote{For instance, the region to be refined could consist of several (topological) knots with links.}, and computation of the cohomology groups is not a simple task, requiring tools from computational (co)homology. Moreover, the comparison of the cohomology groups must be done through the map induced by the inclusion, because even if the cohomology groups of the two Greville subgrids are isomorphic, it may happen that the map induced by the inclusion is not an isomorphism, which would give a non-exact sequence. Finally, for the development of adaptive methods, the refined region $\Omega_\ell$ is chosen with respect to the B\'ezier mesh, and one would like to obtain a condition in this mesh that guarantees the exactness of the sequence. We will now introduce a \textbf{\emph{local and sufficient}} condition for the exactness of the hierarchical B-spline complex. Apart from being local, the advantage of this condition is that it is easier to check because it is defined directly with respect to the B\'ezier submeshes and not with respect to the Greville subgrids.

From now on, we will restrict ourselves to the two-dimensional case, that is, $n=2$. 
We start by introducing a mild assumption about the choice of the refined subdomains (see equation (11) in \cite{Vuong_giannelli_juttler_simeon}):
\begin{assumption}\label{ass:support}
The subdomain $\Omega_{\ell+1}$ is given as the union of supports of basis functions in the previous level, that is, for $\ell = 0, \ldots, N-1$
\begin{equation}
\Omega_{\ell+1} = \bigcup_{\beta \in S} \supp (\beta), \text{ with } S \subset {\cal B}^n_\ell.  \label{eq:supports}
\end{equation}
As a consequence, (the closure of) the subdomain $\Omega_{\ell+1}$ is given as (the closure of) the union of elements of level $\ell$.
\end{assumption}
The second assumption is closely related to the one introduced in \cite{Mokris2014}, where the authors prove that hierarchical B-splines span the space of piecewise polynomials in the corresponding hierarchical B\'{e}zier mesh. Let us first define, for each level $\ell = 0, \ldots, N-1$, the \textbf{\emph{open}} complementary region
\begin{equation*}
\Omega_{\ell+1}^c = \Omega_0 \setminus \overline \Omega_{\ell+1}.
\end{equation*}
\begin{assumption}\label{ass:overlap}
For each level $\ell = 0, \ldots, N-1$, and for any basis function $\beta^k \in {\cal B}^k_\ell$ of the tensor-product spaces $\Xhatlh{k}$, with $k = 0, 1, 2$, the overlap 
\begin{equation*}
{\rm overlap} (\beta^k) := \supp (\beta^k) \cap \Omega_{\ell+1}^c
\end{equation*}
is connected and simply connected.
\end{assumption}

\begin{figure}[t!]
\begin{subfigure}[]
{\includegraphics[width=0.3\textwidth,trim=2cm 2cm 2cm 1cm, clip]{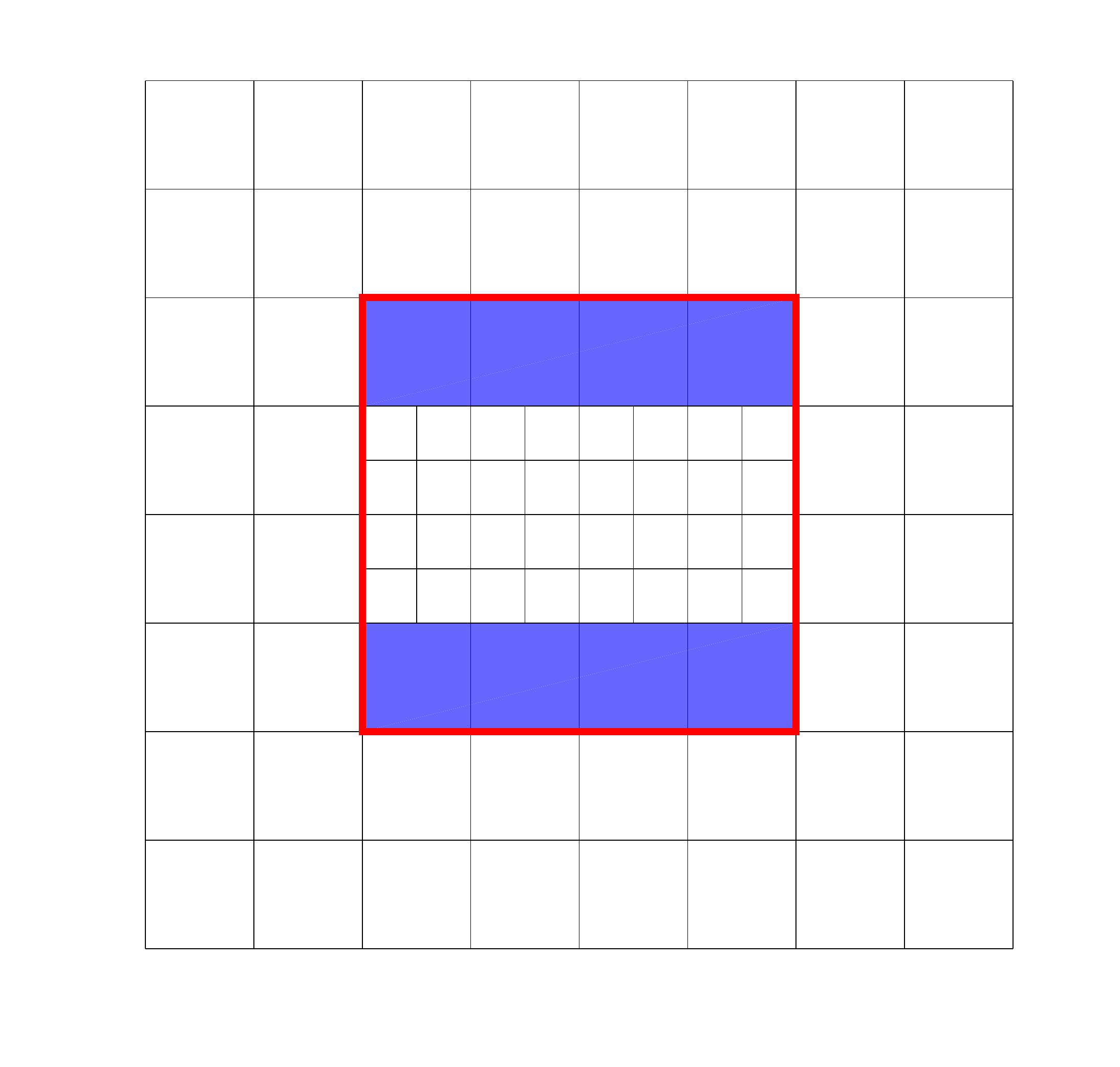} \label{fig:not-assumption1}}
\end{subfigure}
\begin{subfigure}[]
{\includegraphics[width=0.3\textwidth,trim=2cm 2cm 2cm 1cm, clip]{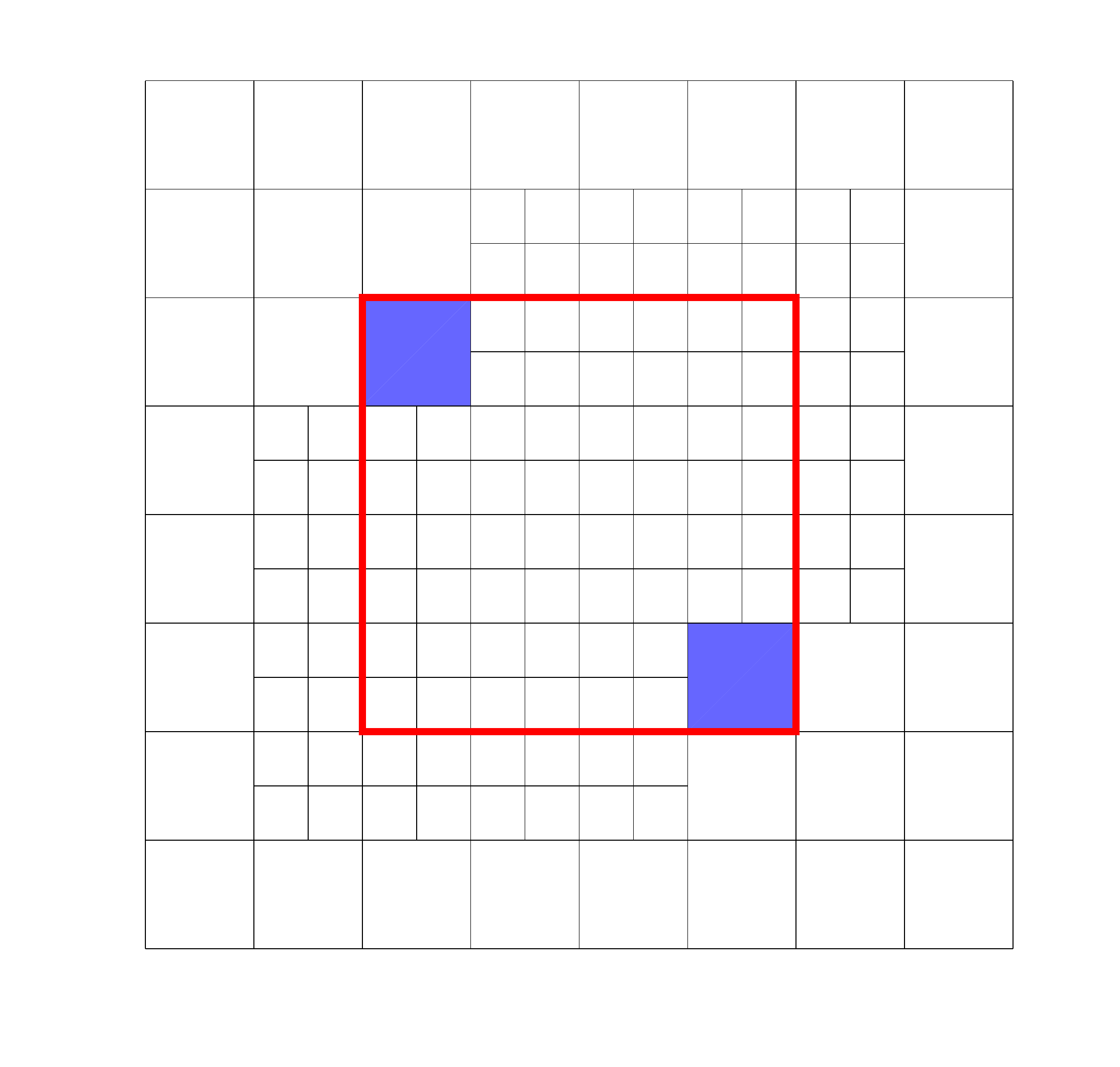} \label{fig:not-assumption2}}
\end{subfigure}
\begin{subfigure}[]
{\includegraphics[width=0.3\textwidth,trim=2cm 2cm 2cm 1cm, clip]{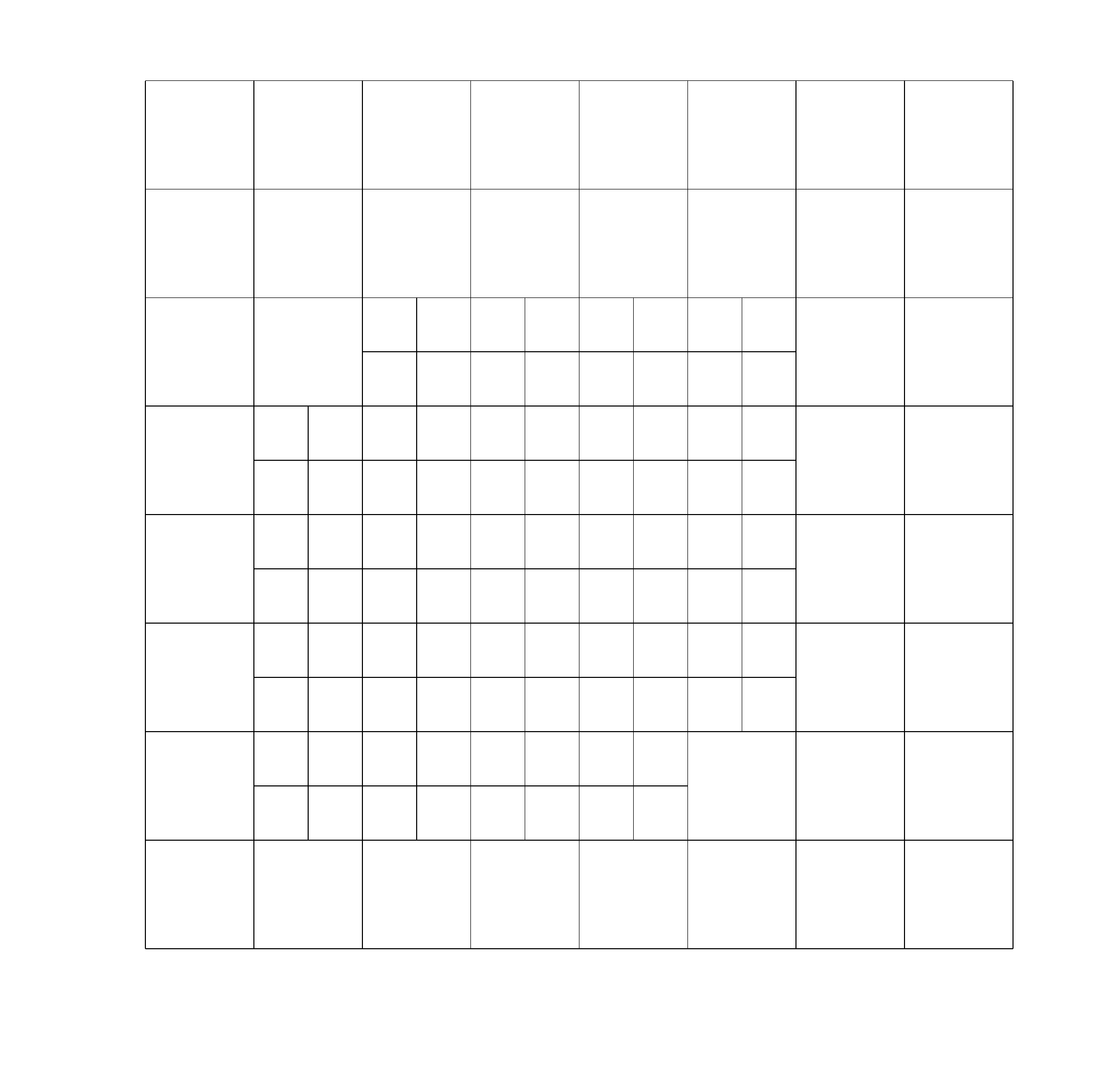} \label{fig:assumption12}}
\end{subfigure}
\caption{Different hierarchical B\'{e}zier mesh configurations for maximally continuous B-splines and $p_1=p_2=3$. The first mesh does not satisfy Assumption~\ref{ass:support}, nor Assumption~\ref{ass:overlap}, because for the function highlighted with the red square, the overlap (given by the purple colored elements) is not connected. The second mesh does not satisfy Assumption~\ref{ass:overlap}, because the overlap of the highlighted function is also not connected. The third mesh satisfies both Assumptions.}
\label{fig:assumptions-1-2}
\end{figure}

The main result of this section is the following theorem, that for $\ell = N$ gives the exactness of the hierarchical B-spline complex.
\begin{theorem} \label{th:local}
If Assumptions~\ref{ass:support} and~\ref{ass:overlap} hold, then the complex 
\begin{equation}
\begin{CD}
0 @>>> \What^0_{\ell} @>\grad>> \What^1_{\ell} @>\curl>> \What^2_{\ell} @>>> \mathbb{R} @>>> 0
\end{CD}
\end{equation}
is exact, for $\ell = 0, \ldots, N$.
\end{theorem}

The constraints posed by the two assumptions are better understood using the simple examples depicted in Figure~\ref{fig:assumptions-1-2}. For maximally continuous B-splines and $p_1=p_2=3$, the hierarchical B\'{e}zier mesh in Figure~\ref{fig:not-assumption1} violates Assumption~\ref{ass:support}, because the refined region does not contain the support of any basis function of level zero. It also violates Assumption~\ref{ass:overlap}, because there exists one function of level zero, whose support is highlighted in red in the figure, that violates the condition in the assumption. Indeed, the intersection of the support with the non-refined region $\Omega^c_1$, which is represented by the eight purple elements, is not connected. The hierarchical B\'{e}zier mesh in Figure~\ref{fig:not-assumption2} satisfies Assumption~\ref{ass:support}, but it does not satisfy Assumption~\ref{ass:overlap}, because the same function violates the overlap condition in the assumption. Finally, the hierarchical B\'{e}zier mesh in Figure~\ref{fig:assumption12} satisfies both assumptions. More examples are given later in Section~\ref{sec:stability}.

The remainder of this section is devoted to the proof of Theorem~\ref{th:local}. We recall that, as stated in Theorem~\ref{th:exact}, the condition for the exactness is given by the relation between the spaces spanned by the functions to remove, $\Xhat^k_{\ell,\ell+1} = {\rm span} \{ {\cal B}^k_{\ell,\ell+1} \}$, and the functions to add, $\Xhat^k_{\ell+1,\ell+1} = {\rm span} \{ {\cal B}^k_{\ell+1,\ell+1} \}$. We will prove the result with the help of the Greville subgrids corresponding to these spaces, that using the same notation of Section~\ref{sec:submesh} we denote by $\MG_{\ell,\ell+1}$ and $\MG_{\ell+1,\ell+1}$, respectively. With the same abuse of notation in Section~\ref{sec:submesh}, we denote in the same manner the two subdomains covered by the Greville subgrids, that we recall that are different from $\Omega_{\ell+1}$, and also different between them.

The proof follows in three steps. First, we show that the three involved subdomains, that is, the subdomain $\Omega_{\ell+1}$ and the two Greville subgrids, are manifolds with boundary and share the same homological properties. Second, using the relations in Section~\ref{sec:submesh} between spline subspaces and finite element spaces defined in the Greville subgrids, we construct a family of commutative projectors for the subspaces restricted to a subgrid. Finally, joining the above results gives the exactness of the sequence.

\begin{remark}
The analogous condition to Assumption~\ref{ass:overlap} in the three-dimensional setting is that the overlaps are connected, simply connected, and with connected boundary.
\end{remark}
\begin{remark}
In two dimensions, Assumption~\ref{ass:support} implies that the overlaps are simply connected, so this condition could be dropped in Assumption~\ref{ass:overlap}, in this case obtaining exactly the same assumption that is used in \cite{Mokris2014} to characterize hierarchical B-splines in terms of piecewise polynomials. We have preferred to maintain the simply connected condition to make more evident that topology is playing a major role in our construction. The results of this section can actually be proved without using Assumption~\ref{ass:support}, but the proofs become more tedious.
\end{remark}

\subsubsection{First Part of the Proof: Topology of the Grids}

\begin{lemma}\label{lemma:omwb}
If Assumptions~\ref{ass:support} and~\ref{ass:overlap} hold, then the subdomain $\Omega_{\ell+1}$ is a manifold with boundary, for $\ell = 0,\ldots, N-1$.
\end{lemma}
\begin{proof}
In the two-dimensional case, the only configuration for which $\Omega_{\ell+1}$ is not a manifold with boundary is the one in Figure~\ref{fig:not-manifold} (see also Figure~3 in \cite{Berdinsky201486}), in which two connected components of $\Omega_{\ell+1}$ meet in a corner. From Assumption~\ref{ass:support}, every cell in $\Omega_{\ell+1}$ belongs to the support of a basis function as in \eqref{eq:supports}. That is, in the case of maximum continuity, we have a rectangle of size $p_{1,\ell} \times p_{2,\ell}$ elements contained in $\Omega_{\ell+1}$. Therefore, we can obviously choose a basis function in ${\cal B}^0_\ell$ such that its support contains the two white elements of $\Omega^c_{\ell+1}$ in Figure~\ref{fig:not-manifold} and the overlap is not connected, which contradicts Assumption~\ref{ass:overlap}.
\end{proof}

\begin{figure}[t!]
\centerline{\includegraphics[width=0.40\textwidth]{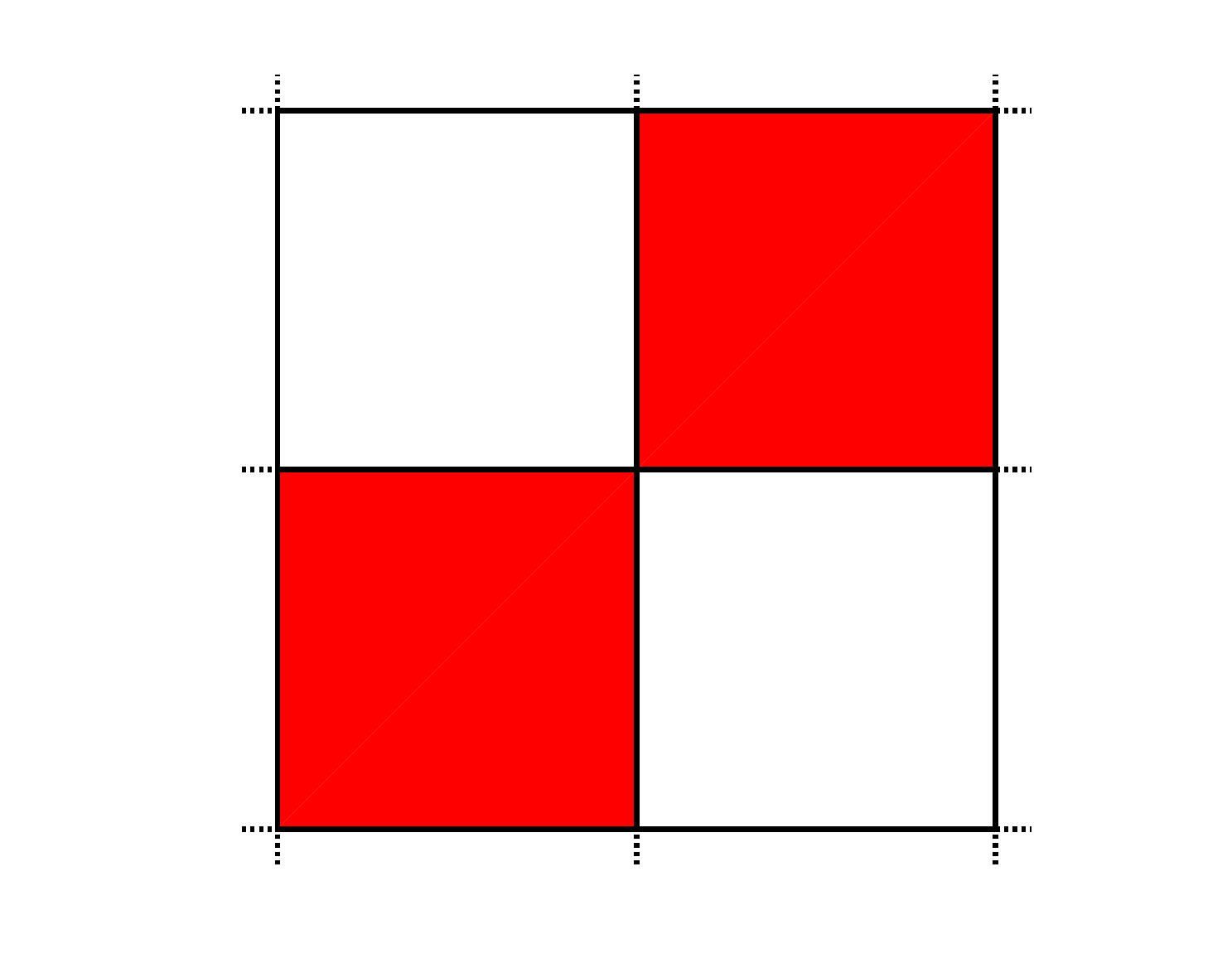}}
\caption{Configuration for which a subdomain is not a manifold with boundary. The colored elements belong to the subdomain, whereas the white elements do not.}\label{fig:not-manifold}
\end{figure}

\begin{lemma} \label{lemma:mwb}
If Assumptions~\ref{ass:support} and~\ref{ass:overlap} hold, then the regions of the Greville subgrids $\MG_{\ell,{\ell+1}}$ and $\MG_{\ell+1,{\ell+1}}$ are manifolds with boundary, for $\ell = 0,\ldots, N-1$.
\end{lemma}
\begin{proof}
Similar to the previous proof, the only configuration for which a Greville subgrid is not a manifold with boundary is the one in Figure~\ref{fig:not-manifold}, where now the elements correspond to the Greville cells. If $\MG_{\ell,{\ell+1}}$ is not a manifold with boundary, i.e., it is as in Figure~\ref{fig:not-manifold}, the support of the basis functions in ${\cal B}^2_\ell$ associated to the colored cells is completely contained in $\Omega_{\ell+1}$, and the support of the basis functions associated to the white cells is not. As a consequence, the overlap of the basis function in ${\cal B}^0_\ell$ corresponding to the middle Greville point is not connected, which contradicts Assumption~\ref{ass:overlap}.

If $\MG_{\ell+1,{\ell+1}}$ is as in Figure~\ref{fig:not-manifold} the support of the basis functions associated to the colored cells, which are now in ${\cal B}^2_{\ell+1}$, is contained in $\Omega_{\ell+1}$, and from Assumption~\ref{ass:support} each one of the cells is contained in a rectangle of $p_{1,\ell} \times p_{2,\ell}$ cells of level $\ell$ contained in $\Omega_{\ell+1}$. Since the support of the basis functions associated to the white cells is not contained in $\Omega_{\ell+1}$, analogously to the proof of Lemma~\ref{lemma:omwb} there exists a function in ${\cal B}^0_\ell$ such that its overlap is not connected, which contradicts Assumption~\ref{ass:overlap}.
\end{proof}

\begin{theorem}\label{th:homology}
If Assumptions~\ref{ass:support} and~\ref{ass:overlap} hold, then the homology groups of $\Omega_{\ell+1}$, $\MG_{\ell,{\ell+1}}$ and $\MG_{\ell+1,{\ell+1}}$ are isomorphic, that is, $H_k(\Omega_{\ell+1}) \cong H_k(\MG_{\ell,\ell+1}) \cong H_k(\MG_{\ell+1,\ell+1})$ for $k = 0, 1, 2$.
\end{theorem}
\begin{proof}
The proof of this result is presented in Appendix~\ref{sec:appendix}.
\end{proof}

\begin{corollary} \label{corol:cohomology}
If Assumptions~\ref{ass:support} and~\ref{ass:overlap} hold, then $H^k(\Xhat_{\ell,\ell+1}) \cong H^k(\Xhat_{\ell+1,\ell+1})$, for $k = 0,1,2$.
\end{corollary}
\begin{proof}
Using the definitions from Section~\ref{sec:submesh}, the spline subspaces $\Xhat_{\ell,{\ell+1}}$ and $\Xhat_{\ell+1,{\ell+1}}$ are isomorphic to the finite element spaces (with boundary conditions) defined on the Greville subgrids $\MG_{\ell,\ell+1}$ and $\MG_{\ell+1,\ell+1}$, respectively, and that we denote by $\hat Z_{\ell,\ell+1}$ and $\hat Z_{\ell+1,\ell+1}$. As it is known, the cohomology groups of finite element spaces with homogeneous boundary conditions are isomorphic to the relative cohomology groups of the domain, where the boundary plays the role of the subspace in relative cohomology. We do not need any details here, and refer the interested reader to \cite[Chapter~3]{Hatcher} for a deeper understanding of relative cohomology, and to \cite{GK} for its use in finite elements. In our case, the cohomology groups of the finite element spaces are isomorphic to the relative cohomology groups of the subdomains $\MG_{\ell,{\ell+1}}$ and $\MG_{\ell+1,{\ell+1}}$. From Lemma~\ref{lemma:mwb} they are both manifolds with boundary, and we can use the classical Lefschetz duality theorem (see \cite[Theorem~3.43]{Hatcher}), to obtain
\begin{equation*}
H^k(\Xhat_{\ell,{\ell+1}}) \cong H^k(\hat Z_{\ell,{\ell+1}})
\cong H^{k}(\MG_{\ell,\ell+1},\partial \MG_{\ell,\ell+1})
\cong H_{n-k} (\MG_{\ell,\ell+1}),
\end{equation*}
and
\begin{equation*}
H^k(\Xhat_{{\ell+1},{\ell+1}}) \cong H^k(\hat Z_{\ell+1,{\ell+1}})
\cong H^{k}(\MG_{\ell+1,\ell+1},\partial \MG_{\ell+1,\ell+1})
\cong H_{n-k} (\MG_{\ell+1,\ell+1}),
\end{equation*}
where $H^k(A,B)$ denotes relative cohomology groups. Moreover, from Theorem~\ref{th:homology} we have
\begin{equation*}
H_{n-k} (\MG_{\ell,{\ell+1}}) \cong H_{n-k} (\MG_{\ell+1,{\ell+1}}) \cong H_{n-k}(\Omega_{\ell+1}),
\end{equation*}
which implies that $H^k(\Xhat_{\ell,{\ell+1}}) \cong H^k(\Xhat_{\ell+1,{\ell+1}})$, that is, the cohomology groups are isomorphic.
\end{proof}

\subsubsection{Second Part of the Proof: Commutative Projectors}

The second step is to prove that the inclusion induces an isomorphism, as required by Theorem~\ref{th:exact}. In order to do so, we will construct a set of commutative projectors for the spline spaces, using the relation between the spline spaces $\Xhat^k_{\ell,\ell+1}$ and the low-order finite element spaces $\hat Z^k_{\ell,\ell+1}$. Before defining the projectors, we have to construct a set of generators of the cohomology groups for the finite element and the spline spaces. We explain the construction for $\hat Z_{\ell,\ell+1}$, since the construction for $\hat Z_{\ell+1,\ell+1}$ is completely analogous.

We recall that the dimension of $H^2(\hat Z_{\ell,\ell+1})$ is the number of connected components of $\MG_{\ell,\ell+1}$ minus one, that we denote by $\nc-1$. To build a generating set of $H^2(\hat Z_{\ell,\ell+1})$ we choose on each connected component (except one) of $\MG_{\ell,\ell+1}$ one Greville cell, and the associated piecewise constant basis function $z^2_j$, for $j=1, \ldots, \nc - 1$. The generating set of $H^2(\hat Z_{\ell+1,\ell+1})$ is built in an analogous way.

The dimension of $H^1(\hat Z_{\ell,\ell+1})$ is equal to the number of ``holes'' of $\MG_{\ell,\ell+1}$, $\nh$, that in the case of a planar manifold with boundary is the number of connected components of the boundary $\partial \MG_{\ell,\ell+1}$ minus the number of connected components of $\MG_{\ell,\ell+1}$. To construct a generating set of the cohomology group we first choose a reference boundary on each connected component of $\MG_{\ell,\ell+1}$, and then we pick the remaining connected components of the boundary in such a way that we have a set of boundaries $\Gamma^\MG_j$, for $j = 1, \ldots, \nh$. For each index $j$ we consider the nodal finite element function with degrees of freedom equal to one on $\Gamma^\MG_j$ and zero elsewhere, and then take the function ${\bf z}^1_j$ as its gradient, see Figure~\ref{fig:harmonic} for a graphical explanation. For a discussion on the construction of a basis of harmonic forms see Section~1.4 and Appendix~A.4 in \cite{AV-book}, and references therein.

The generators of the cohomology groups for spline spaces $H^k(\Xhat_{\ell,\ell+1})$ are constructed from the finite element ones, applying the isomorphisms \eqref{eq:cd-iga-fem}, and will be denoted by ${\boldsymbol \zeta}^1_j := \hat{I}^1_h ({\bf z}^1_j)$ and $\zeta^2_j := \hat{I}^2_h(z^2_j)$.
\begin{figure}[t!]
\includegraphics[width=0.32\textwidth]{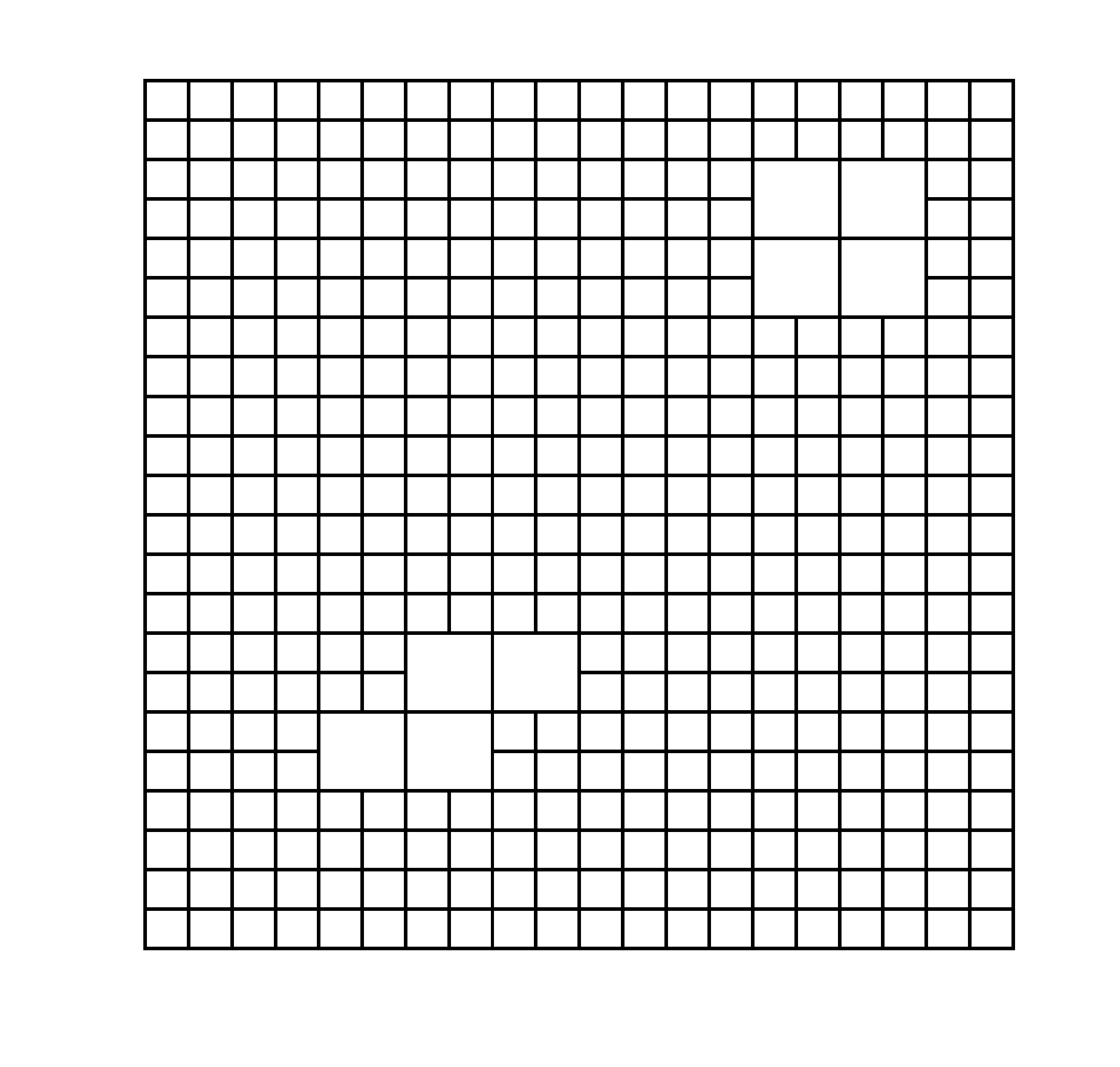}
\includegraphics[width=0.32\textwidth]{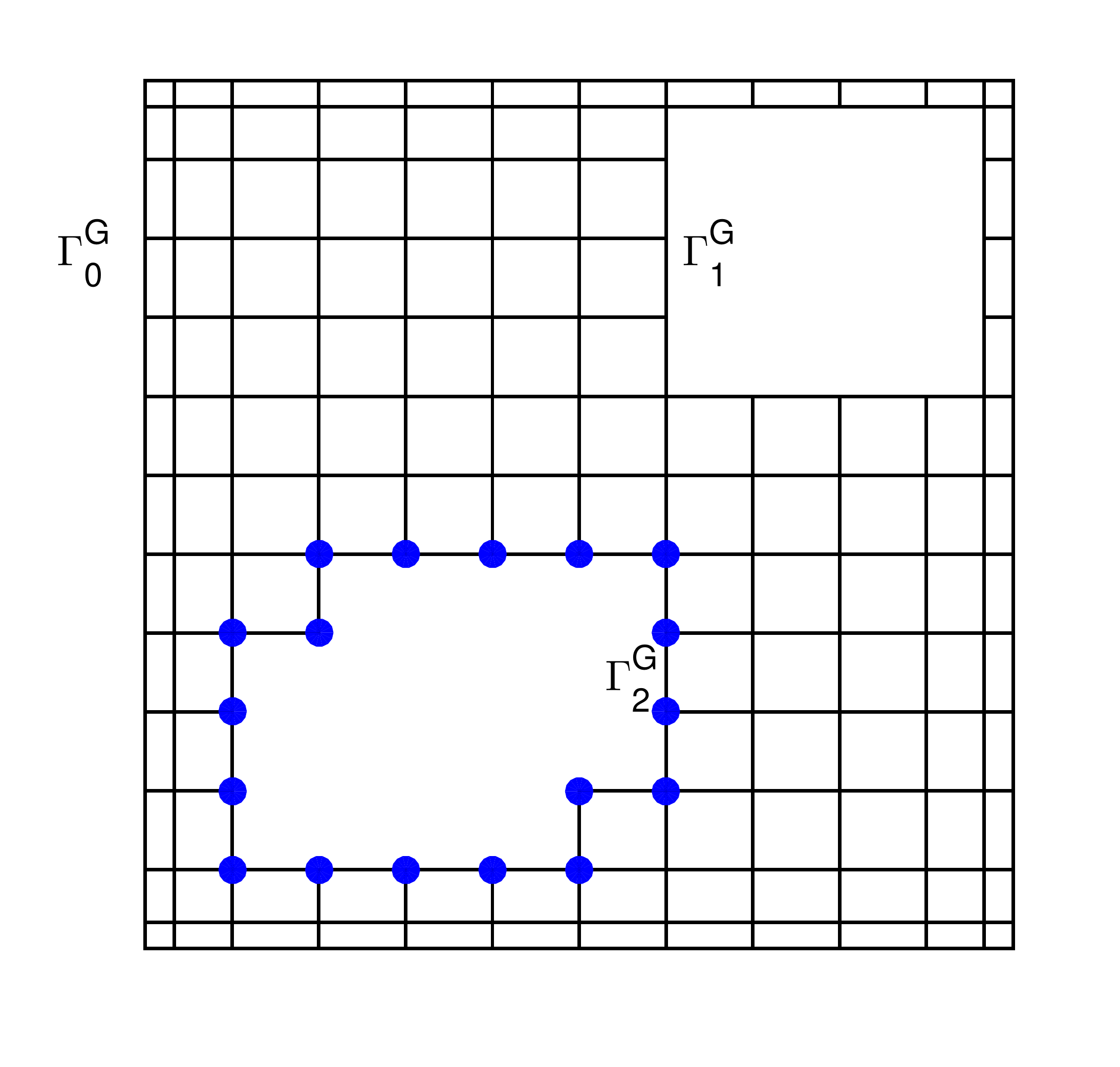}
\includegraphics[width=0.32\textwidth]{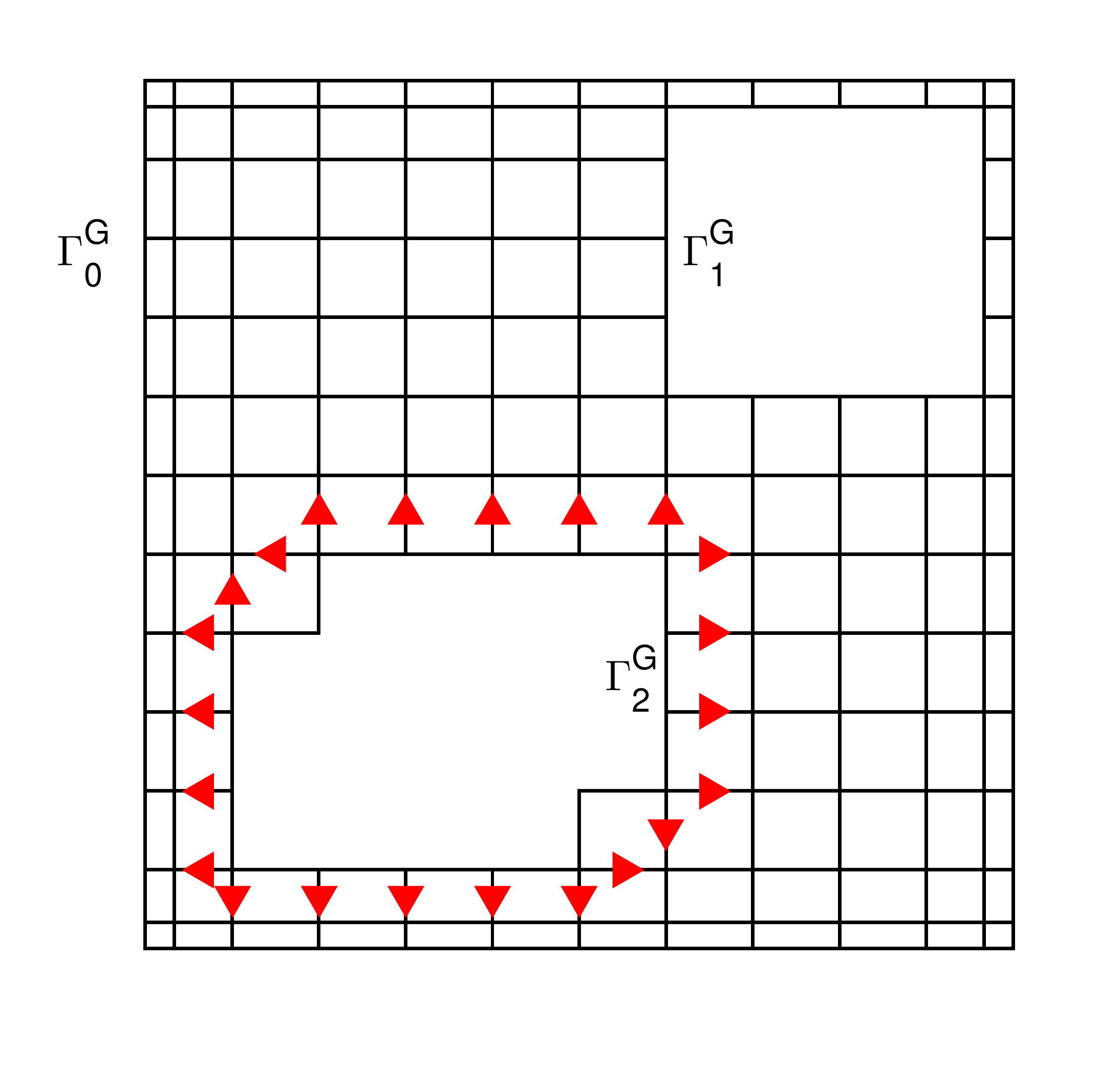}
\caption{Construction of the harmonic function for $\MG_{0,1}$ corresponding to the hierarchical B\'ezier mesh on the left. In the middle, we choose the degrees of freedom associated to $\Gamma^\MG_2$ to be equal to one and all the others equal to zero. On the right, the gradient of that function expressed in terms of the degrees of freedom of $\hat Z^1_{\ell,\ell+1}$, with coefficients equal to one or minus one depending on the arrow orientation, gives ${\bf z}^1_j$.} \label{fig:harmonic}
\end{figure}

\begin{remark} \label{rem:orthogonal}
The functions ${\bf z}^1_j$ are generators of the cohomology group for finite elements, but are not \emph{harmonic forms}, in the sense that they are not orthogonal to the exact forms. In order to obtain the harmonic forms, one could take the gradients of the functions that solve the problem
\begin{equation*}
\left \{
\begin{array}{rl}
\displaystyle 
\int_{\MG_{\ell,\ell+1}} \grad \phi_j \cdot \grad \psi = 0 & \forall \psi \in \hat Z^0_{\ell,\ell+1}, \\
\phi_j = 1 & \text{\rm  on } \Gamma^\MG_j, \\
\phi_j = 0 & \text{\rm  on } \partial \MG_{\ell,\ell+1} \setminus \Gamma^\MG_j.
\end{array}
\right.
\end{equation*}
However, orthogonality is lost when applying the isomorphism $\hat I_h^1$ to construct the spline functions ${\boldsymbol \zeta}_j^1$.
\end{remark}

\begin{lemma} \label{lemma:curves}
Under Assumptions~\ref{ass:support} and~\ref{ass:overlap}, the following statements hold.
\begin{enumerate}[i)]
\item Given a spline closed form $\psi \in Z^2(\Xhat_{\ell,\ell+1})$ (respectively, $\psi \in Z^2(\Xhat_{\ell+1,\ell+1})$), its equivalence class in $H^2(\Xhat_{\ell,\ell+1})$ (resp. $H^2(\Xhat_{\ell+1,\ell+1})$) is uniquely determined by
\begin{equation*}
\int_{\omega_j} \psi \, dx \quad \text{\rm for } j = 1, \ldots, \nc-1,
\end{equation*}
with $\omega_j \subset \Omega_{\ell+1}$ denoting the connected components of $\Omega_{\ell+1}$.
\item There exists a set of curves $\gamma_j$ in $\Omega_{\ell+1}$, with $j = 1, \ldots, \nh$, such  that given a closed form $\bv \in Z^1(\Xhat_{\ell,\ell+1})$ (respectively $\bv \in Z^1(\Xhat_{\ell+1,\ell+1})$), its equivalence class in $H^1(\Xhat_{\ell,\ell+1})$ (resp. $H^1(\Xhat_{\ell+1,\ell+1})$) is uniquely determined by
\begin{equation*}
\int_{\gamma_j} \bv \cdot d{\bf s} \quad \text{\rm for } j = 1, \ldots, \nh.
\end{equation*}
\end{enumerate}
\end{lemma}
\begin{proof}
Noting that the support of a B-spline function associated to a Greville cell of $\MG_{\ell,\ell+1}$ is contained in a connected component of $\Omega_{\ell+1}$, point i) is an immediate consequence of the definition of the forms $\zeta^2_j$, and the fact that there is a one-to-one correspondence between the connected components of $\MG_{\ell,\ell+1}$ and $\Omega_{\ell+1}$, as stated in Theorem~\ref{th:homology}, and more specifically in Proposition~\ref{prop:cc}. It is then clear that $\int_{\omega_j} \zeta^2_i$ form a non-singular diagonal matrix.

Point ii) requires a more technical proof, which is detailed in Appendix~\ref{sec:appendixb}.
\end{proof}

\begin{lemma} \label{lemma:projectors}
If Assumptions~\ref{ass:support} and~\ref{ass:overlap} hold, then it is possible to define a set of cochain projectors $\hat \Pi^k: \Xhat^k_{\ell+1,\ell+1} \rightarrow \Xhat^k_{\ell,\ell+1}$. That is, for every $\phi^k \in \Xhat^k_{\ell,\ell+1}$ it holds $\hat \Pi^k(\phi^k) = \phi^k$, and the following diagram commutes:
\begin{equation*}
\begin{CD}
0 @>>> \Xhat^0_{\ell+1,\ell+1} @>\grad>> \Xhat^1_{\ell+1,\ell+1} @>\curls>> \Xhat^2_{\ell+1,\ell+1} @>>> \mathbb{R} @>>> 0 \\
@. @V\hat \Pi^0VV @V\hat \Pi^1VV @V\hat \Pi^2VV  \\
0 @>>> \Xhat^0_{\ell,\ell+1} @>\grad>> \Xhat^1_{\ell,\ell+1} @>\curls>> \Xhat^2_{\ell,\ell+1} @>>> \mathbb{R} @>>> 0.
\end{CD}
\end{equation*}
\end{lemma}
\begin{proof}
The construction of these commutative projectors is inspired by \cite{DB05}. We recall that we are always considering spaces with homogeneous boundary conditions.

Given $q \in \Xhat^0_{\ell+1,\ell+1}$, we define $\hat \Pi^0 q \in \Xhat^0_{\ell,\ell+1}$ as the unique solution to
\begin{align*}
(\grad \hat \Pi^0 q, \grad \phi) = (\grad q, \grad \phi) \quad \forall \phi \in \Xhat^0_{\ell,\ell+1}.
\end{align*}

Given $\bv \in \Xhat^1_{\ell+1,\ell+1}$, we define $\hat \Pi^1 \bv \in \Xhat^1_{\ell,\ell+1}$ as the unique solution to 
\begin{align*}
(\curls \hat \Pi^1 \bv, \curls \bw) = (\curls \bv, \curls \bw) &\quad \forall \bw \in \Xhat^1_{\ell,\ell+1}, \\
(\hat \Pi^1 \bv, \grad \phi) = (\bv, \grad \phi) & \quad \forall \phi \in \Xhat^0_{\ell,\ell+1}, \\
\int_{\gamma_j} \hat \Pi^1 \bv \cdot d{\bf s} = \int_{\gamma_j} \bv \cdot d{\bf s} & \quad \text{ for } j = 1, \ldots, \nh,
\end{align*}
where the curves $\gamma_j$ are defined as in Lemma~\ref{lemma:curves}.

Given $\psi \in \Xhat^2_{\ell+1,\ell+1}$ we define $\hat \Pi^2 \psi \in \Xhat^2_{\ell,\ell+1}$ as the unique solution to 
\begin{align*}
(\hat \Pi^2 \psi, \curls \bw) = (\psi, \curls \bw) &\quad \forall \bw \in \Xhat^1_{\ell,\ell+1}, \\
\int_{\omega_j} \hat \Pi^2 \psi \, dx = \int_{\omega_j} \psi \, dx &\quad \text{ for } j = 1, \ldots, \nc,
\end{align*}
where $\omega_j \subset \Omega_{\ell+1}$ are the connected components of $\Omega_{\ell+1}$.

The correct definition of the projectors is a consequence of Lemma~\ref{lemma:curves}, and the fact that we are working with spaces with boundary conditions. The commutativity of the diagram comes trivially from the definition of the projectors.
\end{proof}

\subsubsection{Final Part of the Proof: The Main Result}

We can finally prove the main result of this section.
\begin{proof}[Proof of Theorem~\ref{th:local}]
The theorem is proved by an induction argument. The case of $\ell = 0$ corresponds to standard tensor-product B-splines, and the result holds from \cite{BRSV11}. Assuming that the sequence is exact for $\ell$, and using Theorem~\ref{th:exact}, in order to prove the exactness for $\ell+1$ we have to show that the inclusion $j: \Xhat_{\ell,{\ell+1}} \longrightarrow \Xhat_{\ell+1,{\ell+1}}$ induces an isomorphism between the cohomology groups $H^k(\Xhat_{\ell,{\ell+1}})$ and $H^k(\Xhat_{\ell+1,{\ell+1}})$.

From Corollary~\ref{corol:cohomology} we have $H^k(\Xhat_{\ell,{\ell+1}}) \cong H^k(\Xhat_{\ell+1,{\ell+1}})$, and since the spaces are finite-dimensional, it is enough to prove that $j^*:H^k(\Xhat^k_{\ell,\ell+1}) \rightarrow H^k(\Xhat^k_{\ell+1,\ell+1})$ is injective. This is proved using the cochain projectors in Lemma~\ref{lemma:projectors}. Since $\hat \Pi^k \circ j = {\rm Id}$, we have that $(\hat \Pi^k \circ j)^* = (\hat \Pi^k)^* \circ j^* = {\rm Id}^*$, which is the identity in cohomology. As a consequence $j^*$ is injective (and ($\hat \Pi^k)^*$ is surjective), which ends the proof.
\end{proof}

\begin{remark}
The exactness of the sequence might be also proved using the characterization of hierarchical B-splines in terms of piecewise polynomials of \cite{Mokris2014}, invoking the same arguments already used for T-splines in \cite{BSV14} and for LR-splines in \cite{Johannessen15}.
\end{remark}

\section{Stability of the Hierarchical B-spline Complex} \label{sec:stability}
In this section, we examine the stability properties of the hierarchical B-spline complex with specific reference to the Maxwell eigenproblem and the Stokes problem.  For these applications, stability is typically proven based on the construction of a set of local, stable and commutative projectors from the continuous de Rham complex to the discrete subcomplex. In the B-spline case, such projectors can be easily built exploiting the tensor-product structure of B-splines \cite{BRSV11}. The construction in the case of hierarchical B-splines is more subtle and is the subject of future work.

Due to the current lack of a complete theory based on commutative projectors, we have chosen to numerically study the stability of isogeometric discrete differential forms as applied to the Maxwell eigenproblem and the Stokes problem.  Not too unexpectedly, the numerical results show that exactness by itself is not sufficient to obtain a spurious free solution of Maxwell eigenproblem (see \cite{CFR01}), nor to obtain an inf-sup stable scheme for Stokes.  Instead, the numerical results suggest that Assumption~\ref{ass:support} is a necessary condition for stability for the considered applications.  However, our numerical results also suggest that Assumptions~\ref{ass:support} and~\ref{ass:overlap} together, plus a suitable grading of the mesh, are sufficient to obtain a stable scheme for both the Maxwell eigenproblem and the Stokes problem.

\subsection{Numerical Assessment of Spectral Correctness}

In our first group of stability tests, we consider the Maxwell eigenvalue problem.  In variational form, the problem consists of finding the eigenpair $({\bf u}_h, \lambda_h) \in W^1_N \times \mathbb{R}$, with $\lambda_h \ne 0$, such that 
\begin{equation}
(\curls {\bf u}_h, \curls {\bf v}_h) = \lambda_h ({\bf u}_h, {\bf v}_h) \quad \forall {\bf v}_h \in W^1_N, \label{eq:maxwell}
\end{equation}
where $(\cdot,\cdot)$ is the $L^2$-inner product defined over $\Omega$ and $W^1_N$ is the space of 1-forms, or curl conforming splines, in the hierarchical B-spline complex. In order to assess how the hierarchical B\'{e}zier mesh configuration affects the exactness of the complex, we also solve the mixed problems \cite{Boffi07}
\begin{equation} \label{eq:mf1}
\left \{
\begin{array}{rl}
(\curls {\bf u}_h, \curls {\bf v}_h) + (\grad p_h, {\bf v}_h) = \lambda_h ({\bf u}_h, {\bf v}_h) & \forall {\bf v}_h \in W^1_N, \\
(\grad q_h, {\bf u}_h) = 0 & \forall q_h \in W^0_N,
\end{array}
\right.
\end{equation}
and
\begin{equation} \label{eq:mf2}
\left \{
\begin{array}{rl}
({\bf u}_h, {\bf v}_h) + (\curls {\bf v}_h, \phi_h) = 0 & \forall {\bf v}_h \in W^1_N, \\
(\curls {\bf u}_h, \psi_h) = -\lambda_h (\phi_h, \psi_h) & \forall \psi_h \in W^2_N / \mathbb{R},
\end{array}
\right.
\end{equation}
which involve the other spaces of the sequence. For simplicity, the computations are conducted in the square domain $\Omega = (0,\pi)^2$ using the homogeneous Dirichlet condition ${\bf u}_h \times {\bf n}= 0$. In all the computations, the degree is taken equal to four in the first space of the discrete complex, which means that 1-forms are discretized with hierarchical splines of mixed degree three and four.

We have performed numerical tests for several hierarchical B\'{e}zier meshes, displayed in Figure~\ref{fig:meshes-maxwell1}, to assess the influence of Assumptions~\ref{ass:support} and~\ref{ass:overlap} on the eigenproblem discretization and also to see how the mesh configuration affects the exactness. The first mesh, in Figure~\ref{fig:mesh_3lines} does not satisfy any of the two assumptions for the chosen degree, and it does not satisfy the condition for exactness described in Section~\ref{sec:nec-suff}. The second mesh in Figure~\ref{fig:mesh_3lines_bulge} is a variation of the previous one, and it does not satisfy any of the assumptions but satisfies the exactness condition. The third through sixth meshes, in Figures~\ref{fig:mesh_p4_1x1}-\ref{fig:mesh_p4_4x4}, are refined along the diagonal, with a configuration inspired by Figures~\ref{fig:remove-add-1x1}--\ref{fig:assumptions-1-2}. They all satisfy Assumption~\ref{ass:support}, but only the last one satisfies Assumption~\ref{ass:overlap}. It can be easily checked that, among these last four meshes, only the ones in Figure~\ref{fig:mesh_p4_1x1} and Figure~\ref{fig:mesh_p4_4x4} satisfy the exactness condition. The properties satisfied by the six meshes, for the chosen degree, are summarized in Table~\ref{tab:maxwell-prop1}.

\begin{figure}[htp]
\begin{center}
\begin{subfigure}[]
{\includegraphics[width = 0.45\textwidth, trim=1cm 1cm 1cm 0cm, clip]{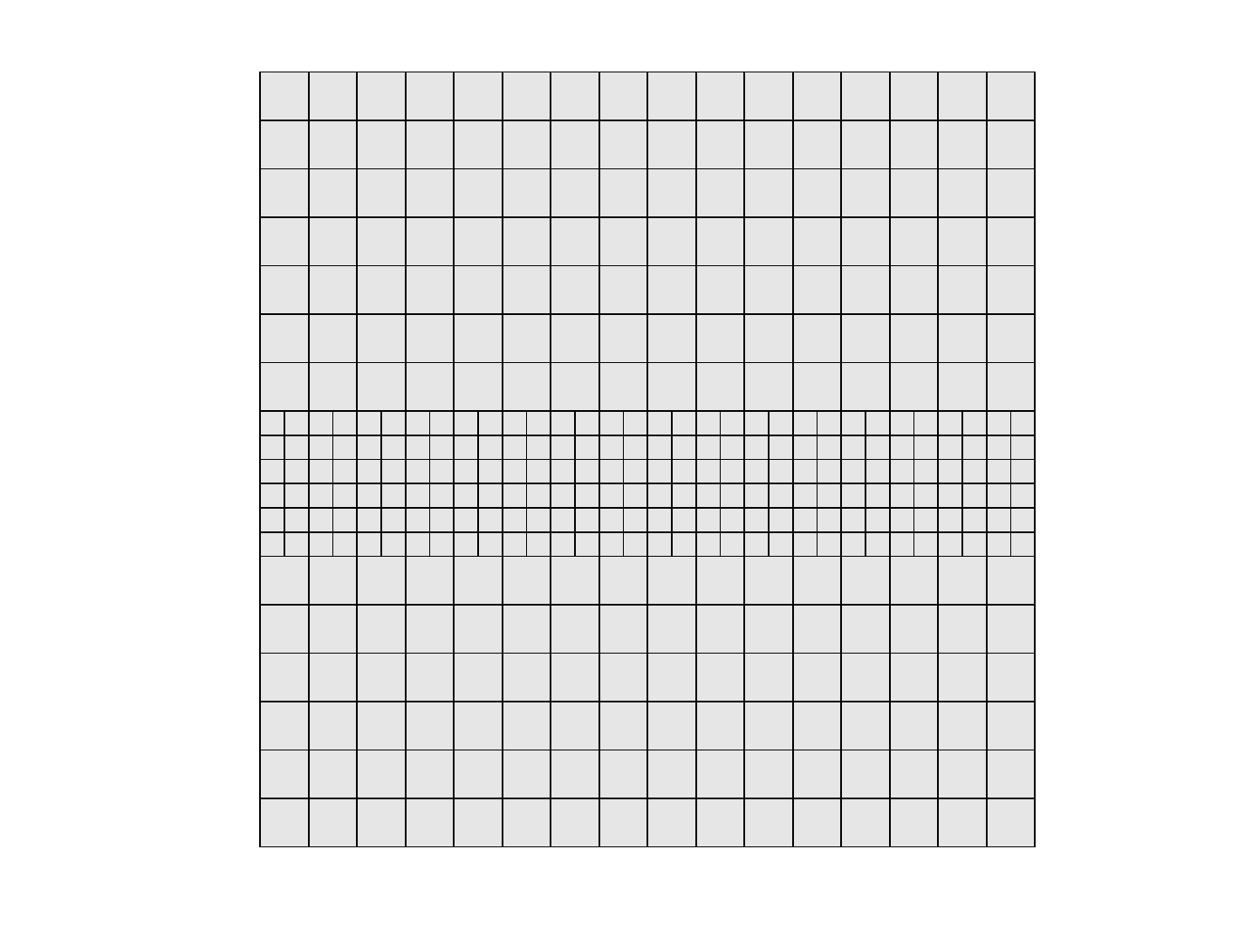}\label{fig:mesh_3lines}}
\end{subfigure}
\begin{subfigure}[]
{\includegraphics[width = 0.45\textwidth, trim=1cm 1cm 1cm 0cm, clip]{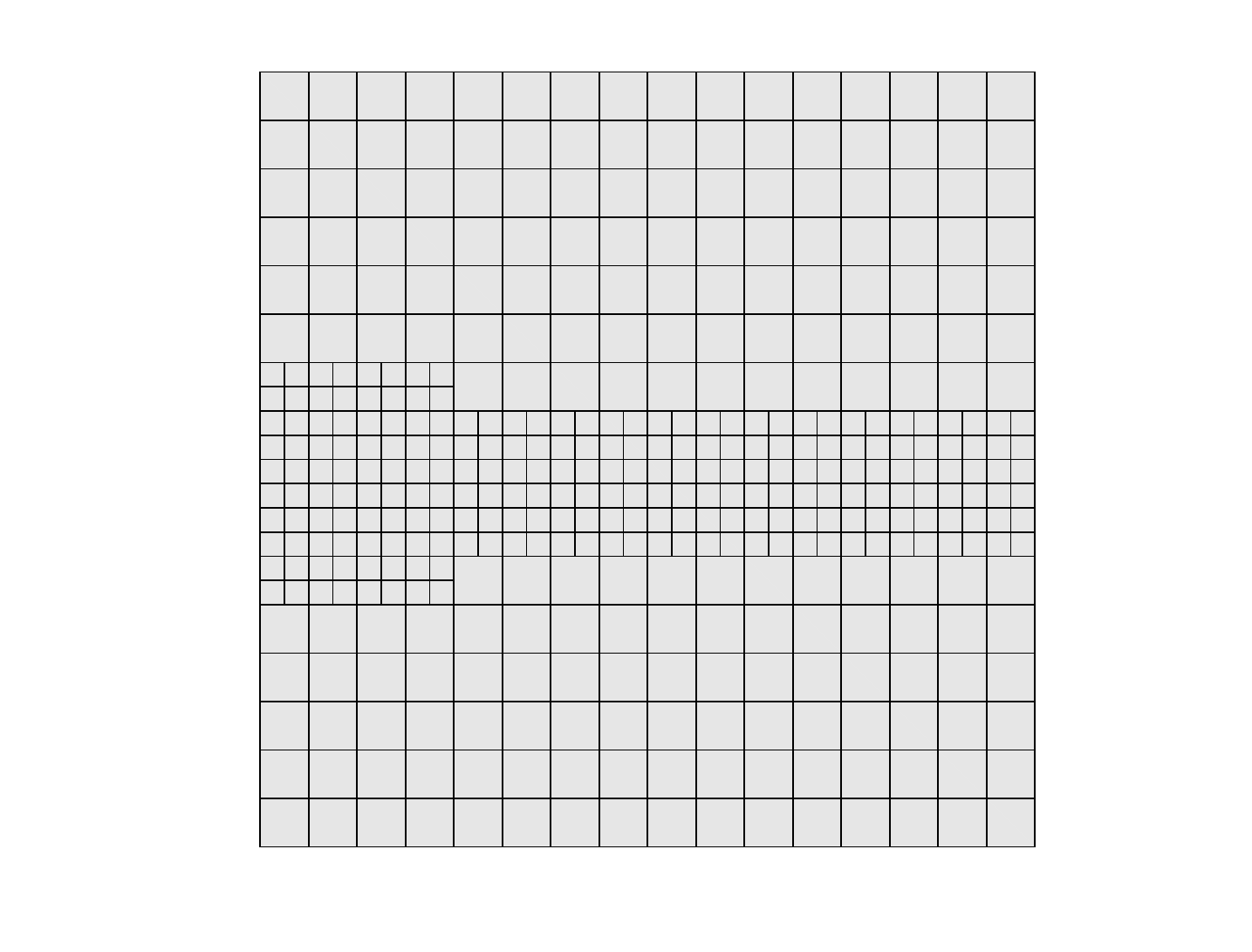}\label{fig:mesh_3lines_bulge}}
\end{subfigure}
\begin{subfigure}[]
{\includegraphics[width = 0.45\textwidth, trim=1cm 1cm 1cm 0cm, clip]{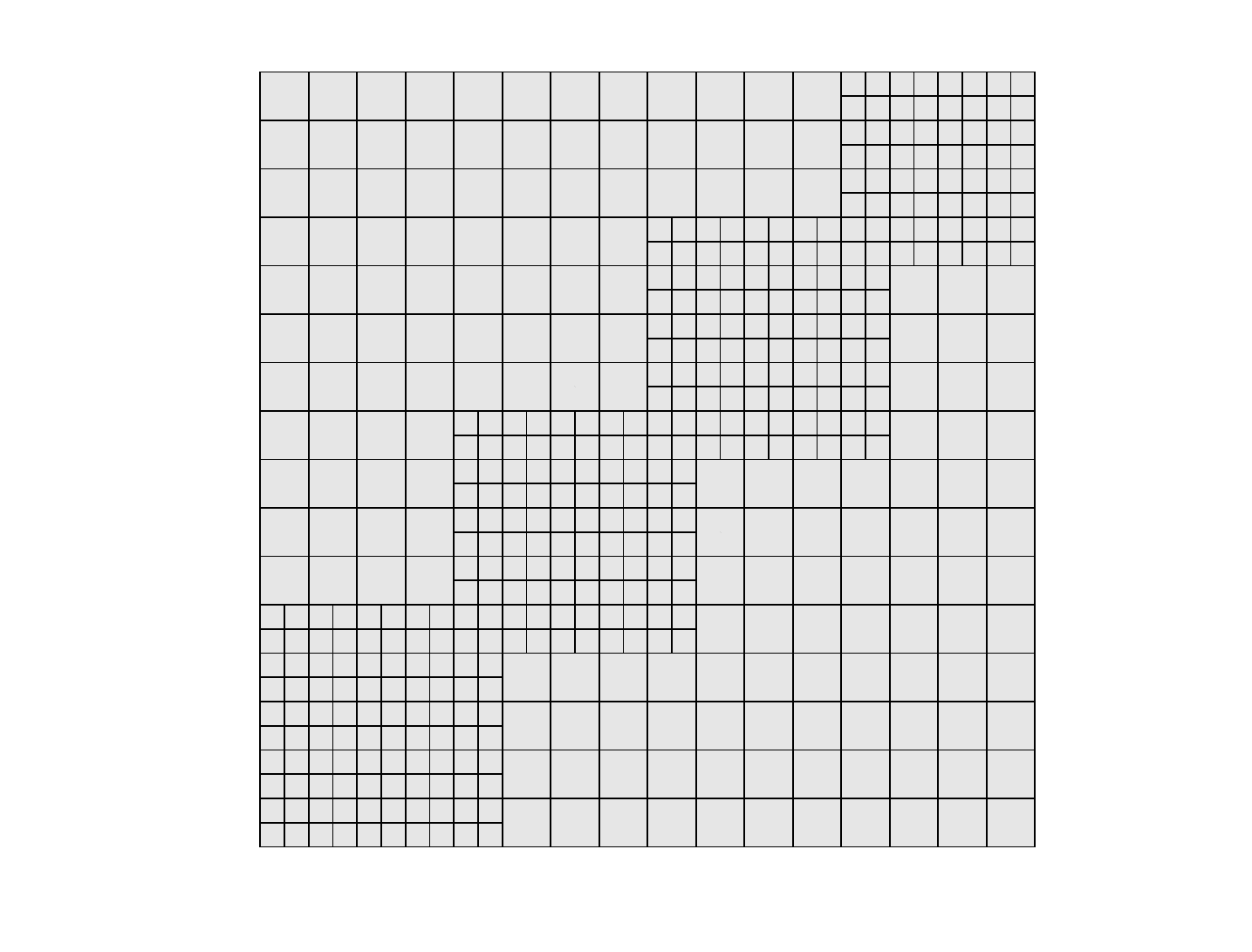}\label{fig:mesh_p4_1x1}}
\end{subfigure}
\begin{subfigure}[]
{\includegraphics[width = 0.45\textwidth, trim=1cm 1cm 1cm 0cm, clip]{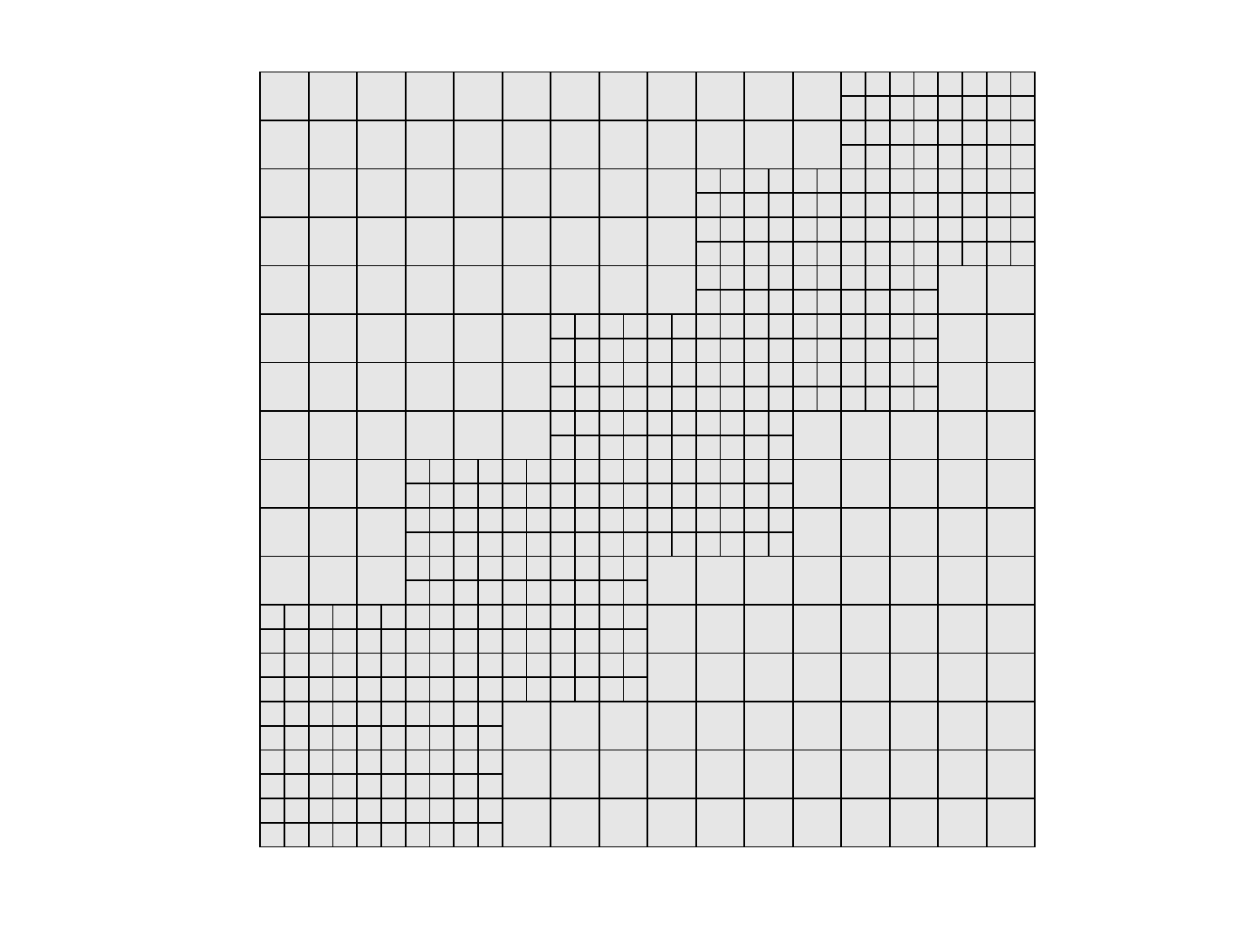}\label{fig:mesh_p4_2x2}}
\end{subfigure}
\begin{subfigure}[]
{\includegraphics[width = 0.45\textwidth, trim=1cm 1cm 1cm 0cm, clip]{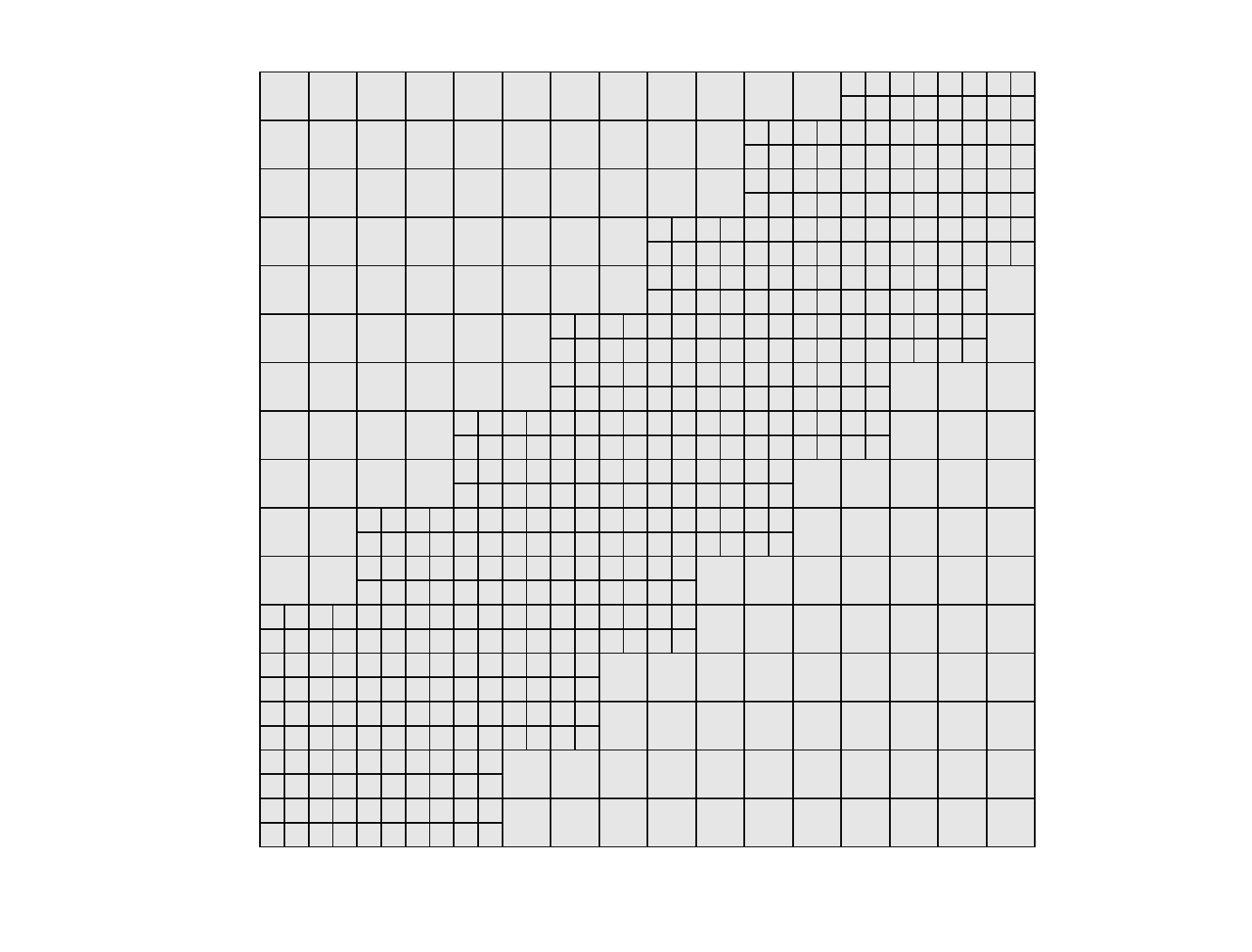}\label{fig:mesh_p4_3x3}}
\end{subfigure}
\begin{subfigure}[]
{\includegraphics[width = 0.45\textwidth, trim=1cm 1cm 1cm 0cm, clip]{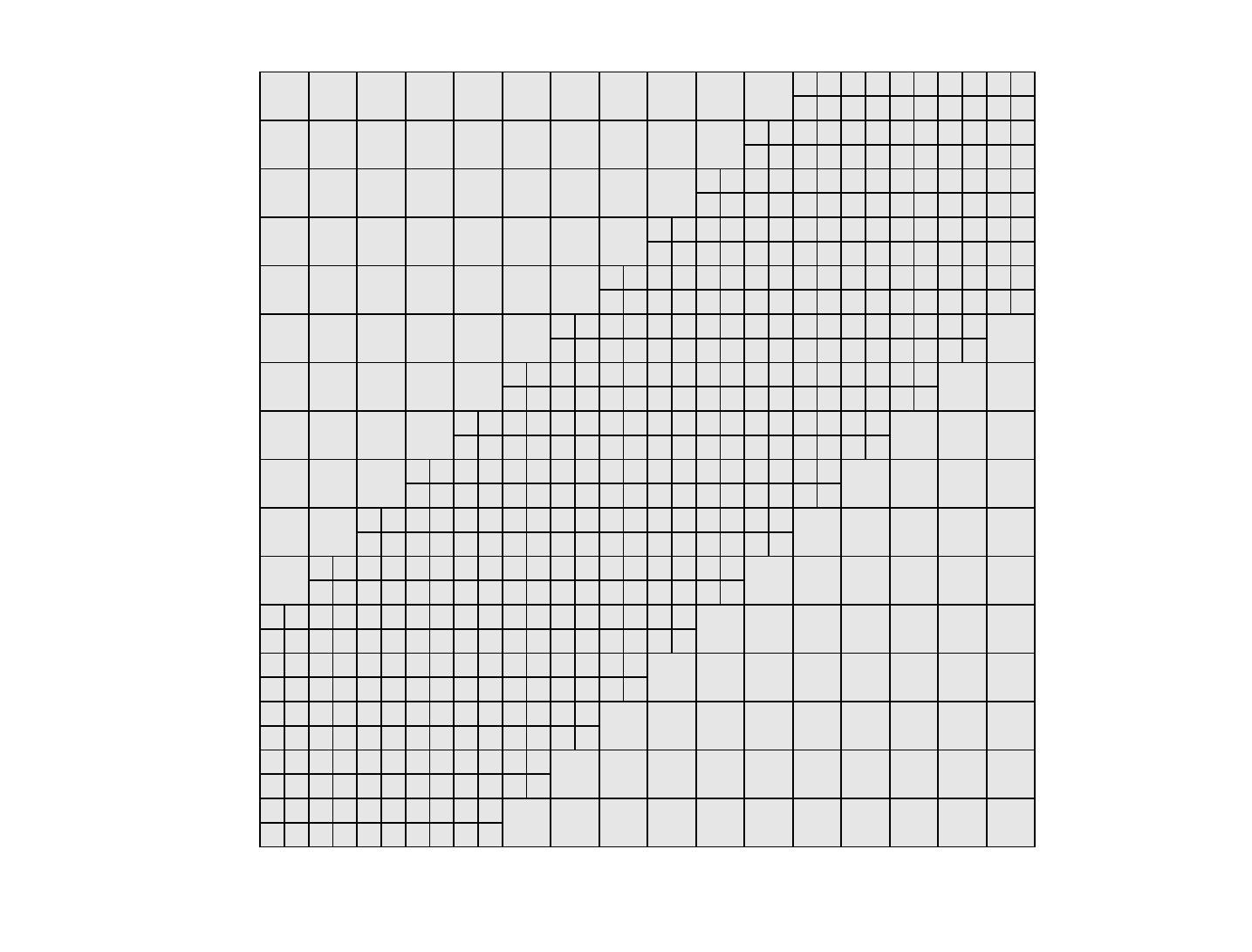}\label{fig:mesh_p4_4x4}}
\end{subfigure}
\end{center}
\caption{The six different hierarchical B\'{e}zier mesh configurations used in the Maxwell eigenproblem tests.  Maximally continuous B-splines with $p_1 = p_2 = 4$ were employed in the tests.} \label{fig:meshes-maxwell1}
\end{figure}

\begin{table}
\caption{Properties satisfied by the six meshes tested for the Maxwell eigenproblem.} \label{tab:maxwell-prop1}
\centering
\begin{tabular}{|c||c|c|c|c|c|c|}
\hline
Mesh & Fig.~\ref{fig:mesh_3lines} & Fig.~\ref{fig:mesh_3lines_bulge} & Fig~\ref{fig:mesh_p4_1x1} & Fig.~\ref{fig:mesh_p4_2x2} & Fig.~\ref{fig:mesh_p4_3x3} & Fig.~\ref{fig:mesh_p4_4x4} \\
\hline
Assumption~\ref{ass:support} & No & No & Yes & Yes & Yes & Yes\\
Assumption~\ref{ass:overlap} & No & No & No & No & No & Yes \\
Exact sequence & No & Yes & Yes & No & No & Yes \\
\hline
No. of zeros in \eqref{eq:mf1} & 0 & 0 & 0 & 4 & 6 & 0 \\
No. of zeros in \eqref{eq:mf2} & 1 & 0 & 0 & 0 & 0 & 0 \\
Spurious free & No & No & Yes & No & Yes & Yes \\
\hline
\end{tabular}
\end{table}

In Figure~\ref{fig:spectrum}, we plot the first 50 eigenvalues computed using each of the six mesh configurations and compare them with the exact ones, including the multiplicity of each eigenvalue. In the results corresponding to the first mesh, which does not satisfy Assumption~\ref{ass:support}, spurious eigenvalues appear all along the spectrum in Figure~\ref{fig:spectrum_3lines}, and from the number of zeros in the solution of the mixed formulations~\eqref{eq:mf1} and~\eqref{eq:mf2}, reported in Table~\ref{tab:maxwell-prop1}, we see that exactness is lost in the second part of the sequence (that is, $\curls W^1_N \neq W^2_N / \mathbb{R}$). The exactness issue is fixed in the second mesh, but it also presents spurious eigenvalues in Figure~\ref{fig:spectrum_3lines_bulge}. This suggests that {\bf \emph{Assumption~\ref{ass:support} is a necessary condition}} to obtain a spurious-free discrete scheme.

For last four meshes, the lack of exactness in the meshes of Figures~\ref{fig:mesh_p4_2x2} and~\ref{fig:mesh_p4_3x3} only affects the first part of the sequence, since the zero eigenvalue appears in the solution of the mixed variational problem \eqref{eq:mf1}. The results of Figure~\ref{fig:spectrum} show that the mesh in Figure~\ref{fig:mesh_p4_2x2} presents spurious eigenvalues at positions $28^{\rm th}$, $31^{\rm st}$, $35^{\rm th}$ and $40^{\rm th}$. An analysis of the associated eigenfunctions, as the one represented in Figure~\ref{fig:eigenfunction}, suggests that spurious results are related to regions where Assumption~\ref{ass:overlap} fails. Surprisingly, the mesh in Figure~\ref{fig:mesh_p4_3x3} does not seem to present spurious eigenvalues, despite the lack of exactness. Finally, the meshes in Figures~~\ref{fig:mesh_p4_1x1} and~\ref{fig:mesh_p4_4x4} yield exact sequences and the corresponding results are spurious-free.

We have conducted a number of additional tests, and in each of these, we found that hierarchical B\'{e}zier meshes satisfying both Assumptions~\ref{ass:support} and~\ref{ass:overlap} yield spurious-free approximations for the Maxwell eigenproblem.  This suggests that {\bf \emph{Assumptions~\ref{ass:support} and~\ref{ass:overlap} are sufficient conditions}} to obtain a spurious-free discrete scheme.

\begin{figure}[htp]
\begin{subfigure}[]
{\includegraphics[width = 0.49\textwidth]{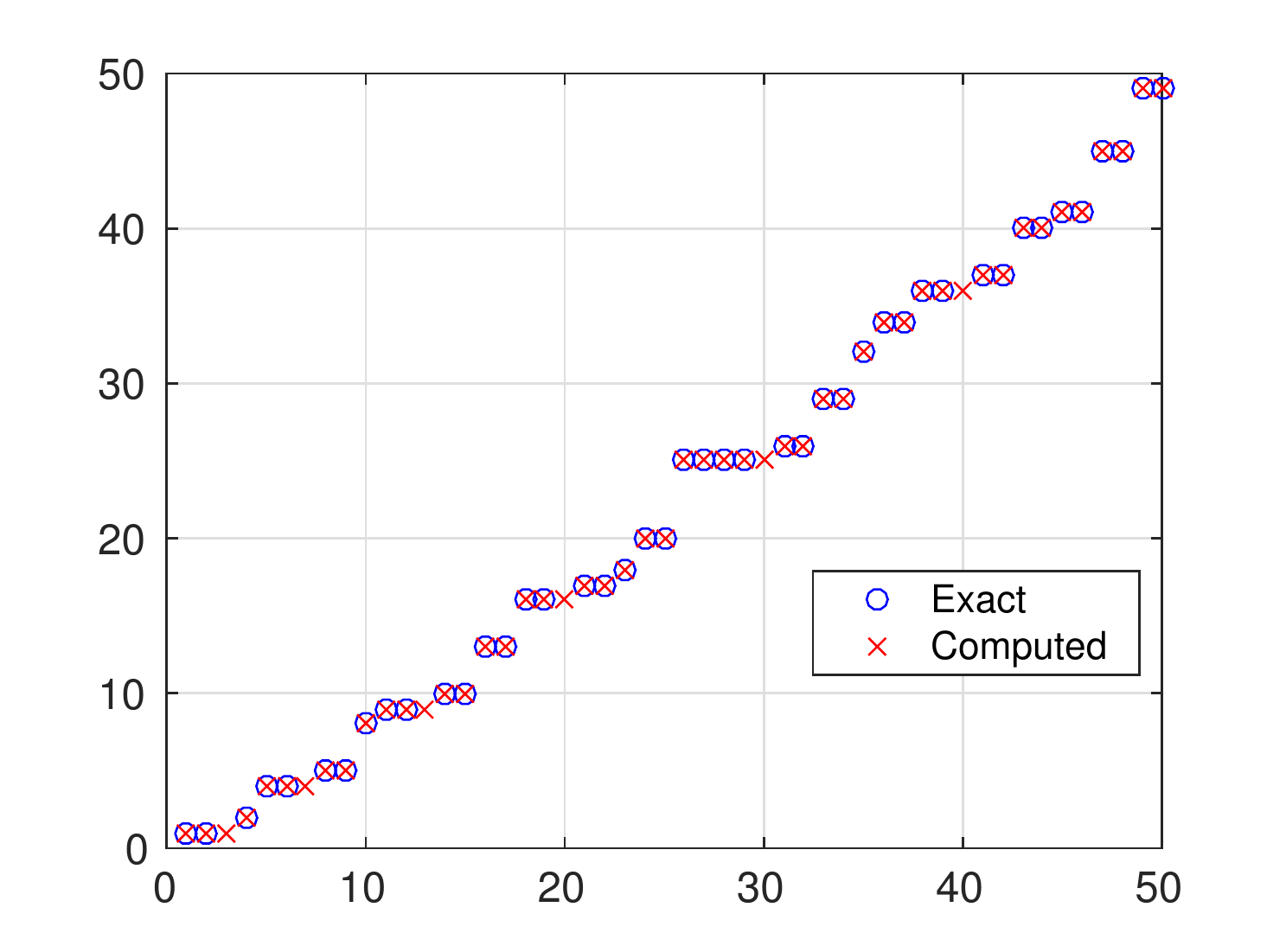}\label{fig:spectrum_3lines}}
\end{subfigure}
\begin{subfigure}[]
{\includegraphics[width = 0.49\textwidth]{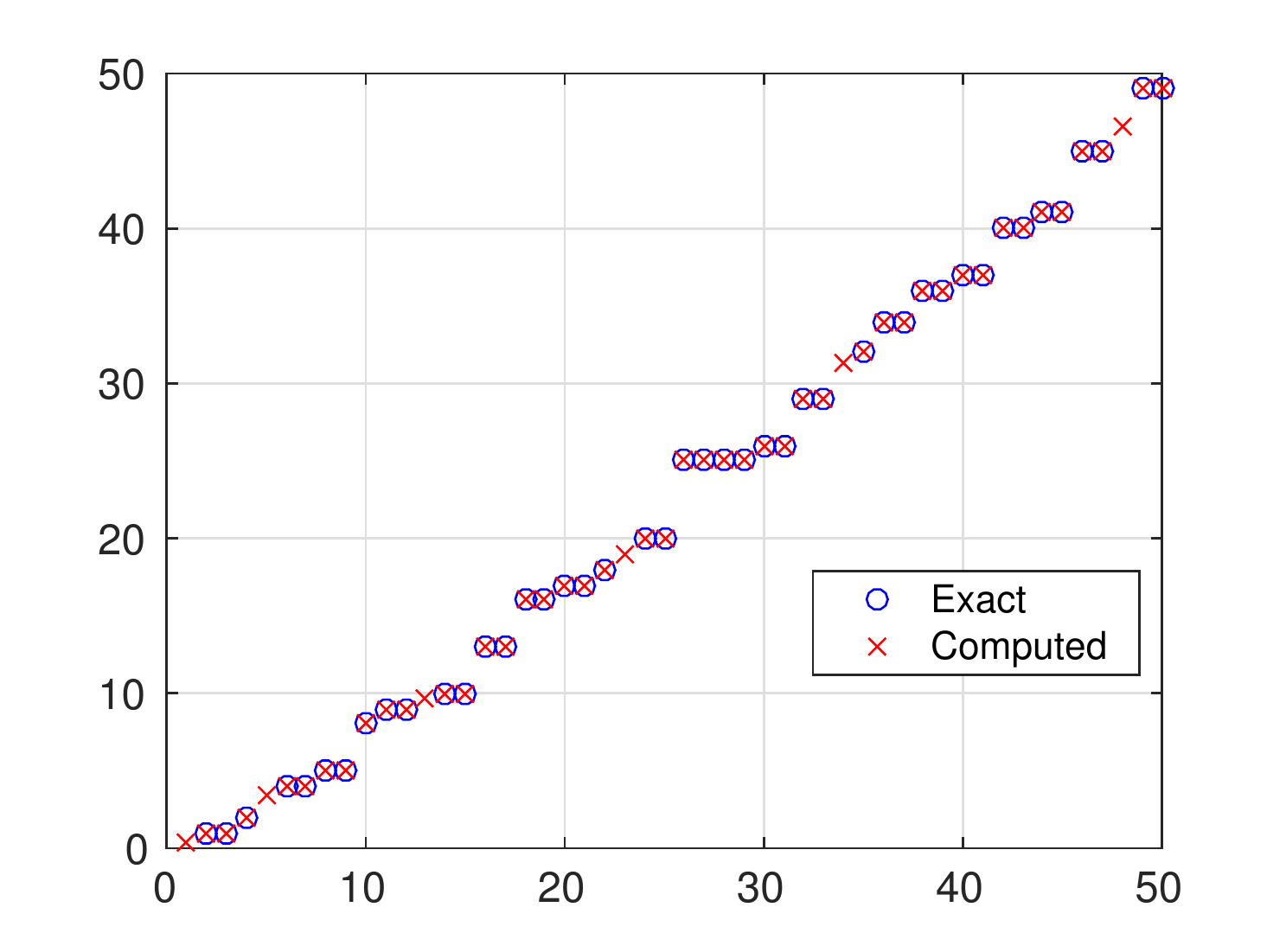}\label{fig:spectrum_3lines_bulge}}
\end{subfigure}
\begin{subfigure}[]
{\includegraphics[width = 0.49\textwidth]{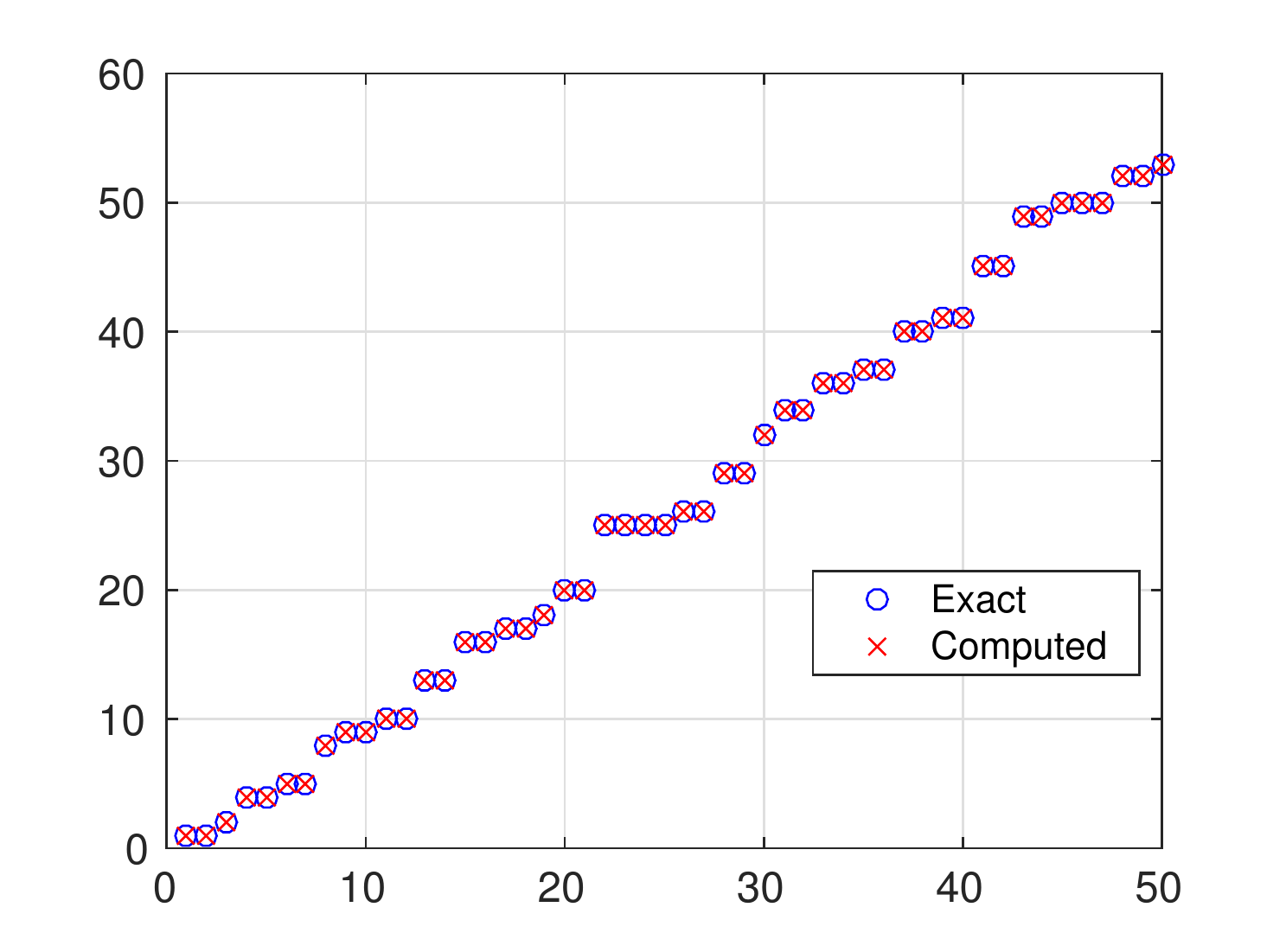}\label{fig:spectrum_p4_1x1}}
\end{subfigure}
\begin{subfigure}[]
{\includegraphics[width = 0.49\textwidth]{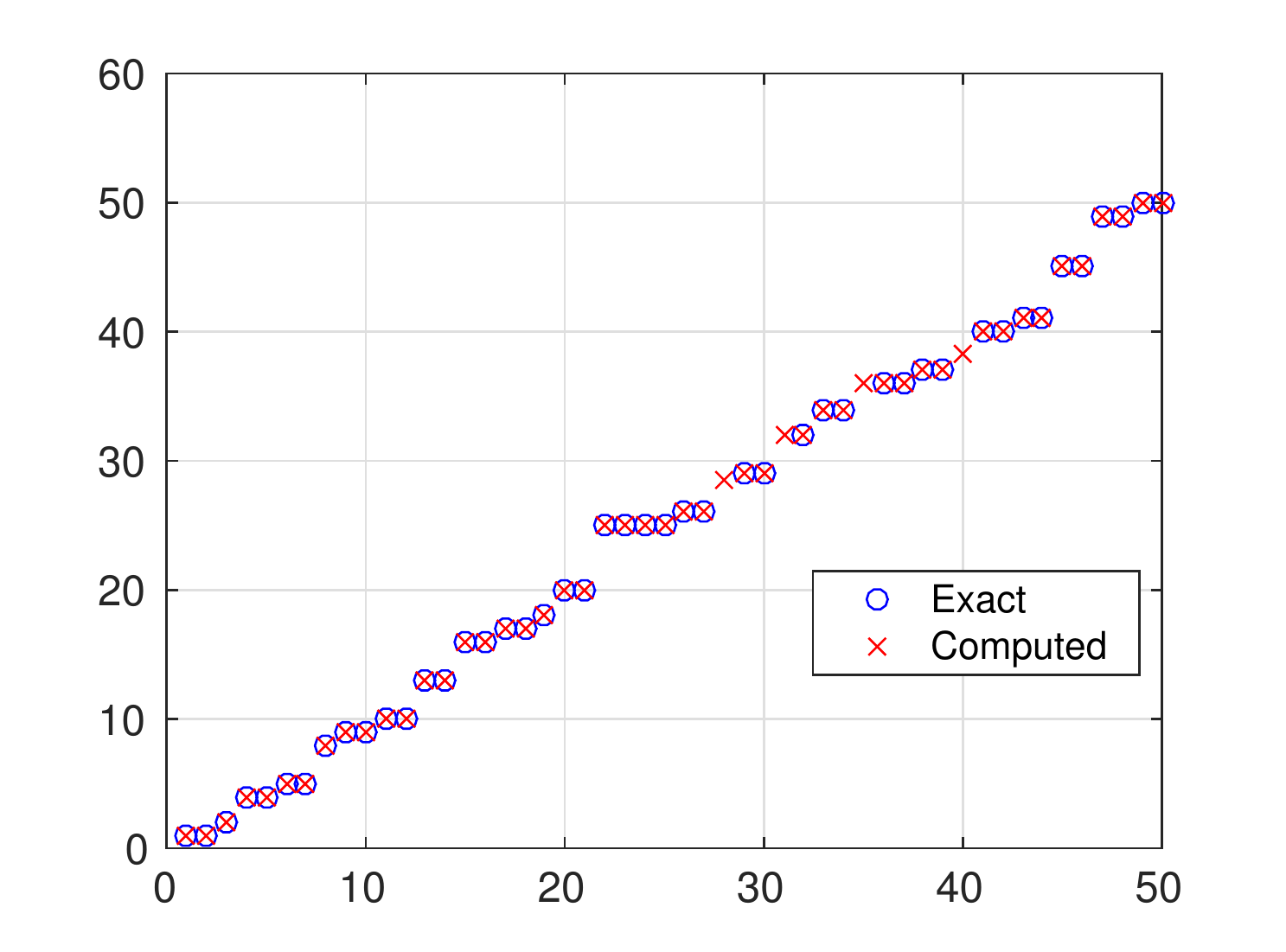}\label{fig:spectrum_p4_2x2}}
\end{subfigure}
\begin{subfigure}[]
{\includegraphics[width = 0.49\textwidth]{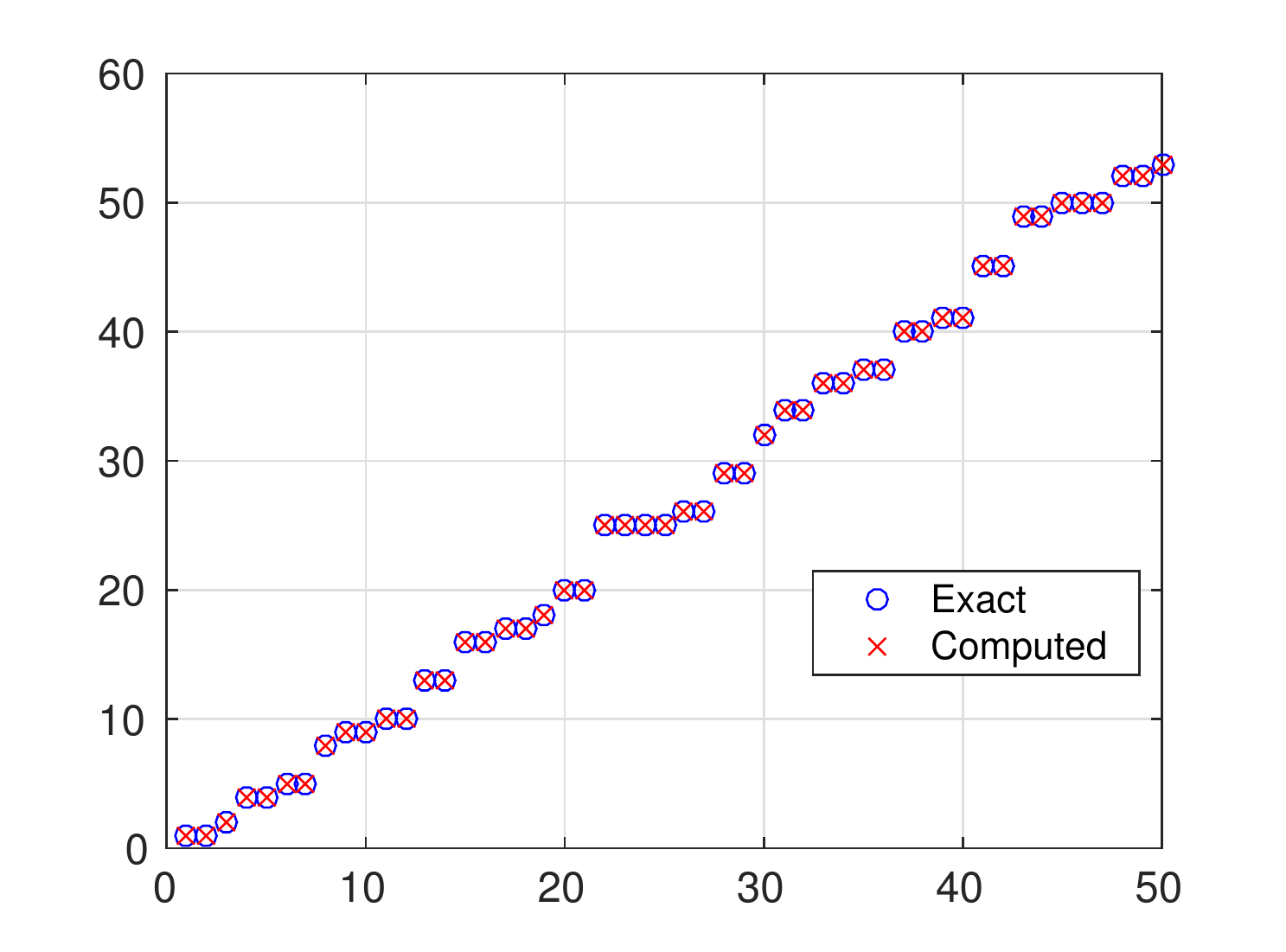}\label{fig:spectrum_p4_3x3}}
\end{subfigure}
\begin{subfigure}[]
{\includegraphics[width = 0.49\textwidth]{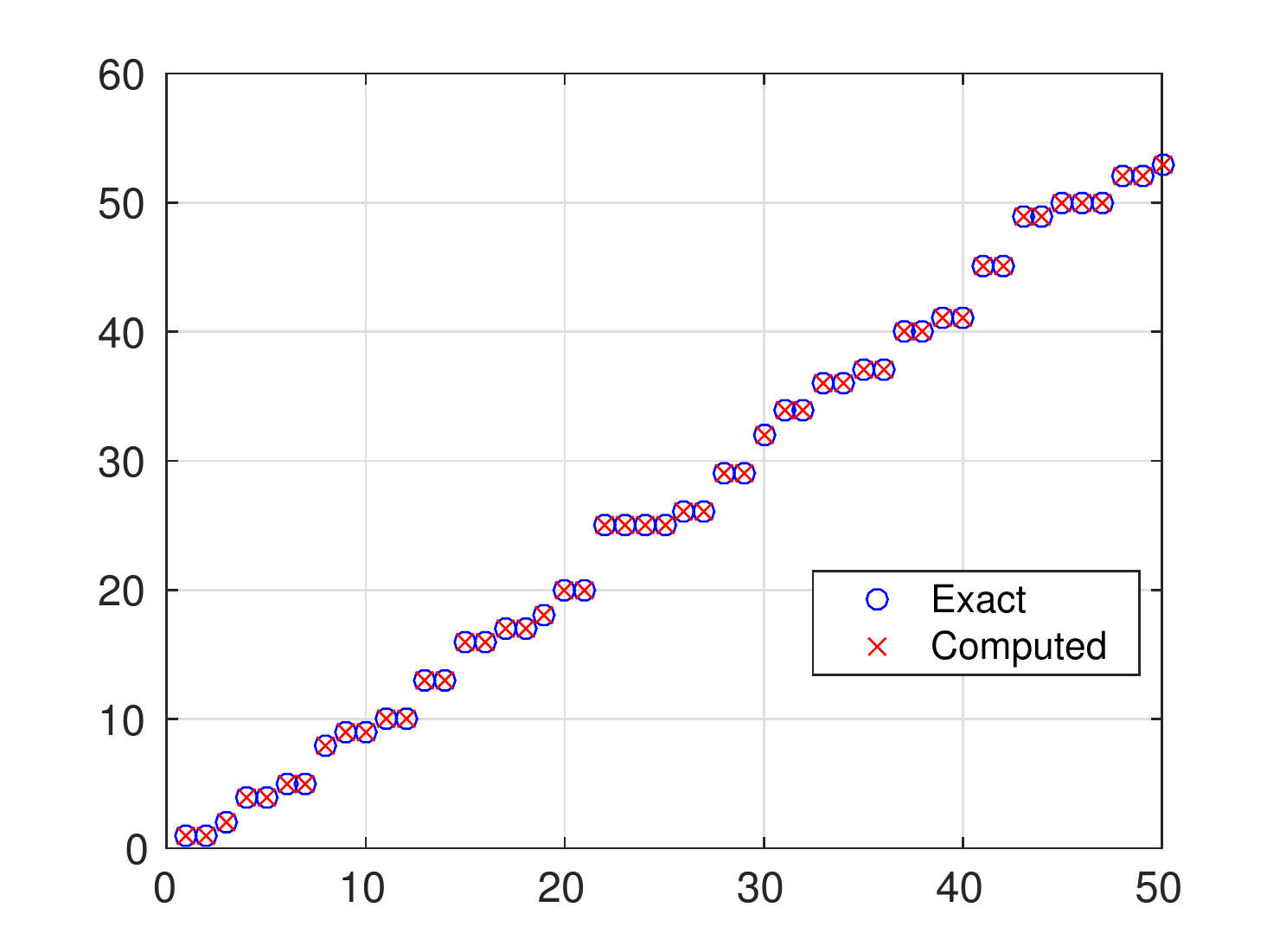}\label{fig:spectrum_p4_4x4}}
\end{subfigure}
\caption{Comparison of the exact and computed eigenvalues for the six meshes depicted in Figure~\ref{fig:meshes-maxwell1}.} \label{fig:spectrum}
\end{figure}

\begin{figure}[ht]
\centerline{\includegraphics[width=0.5\textwidth]{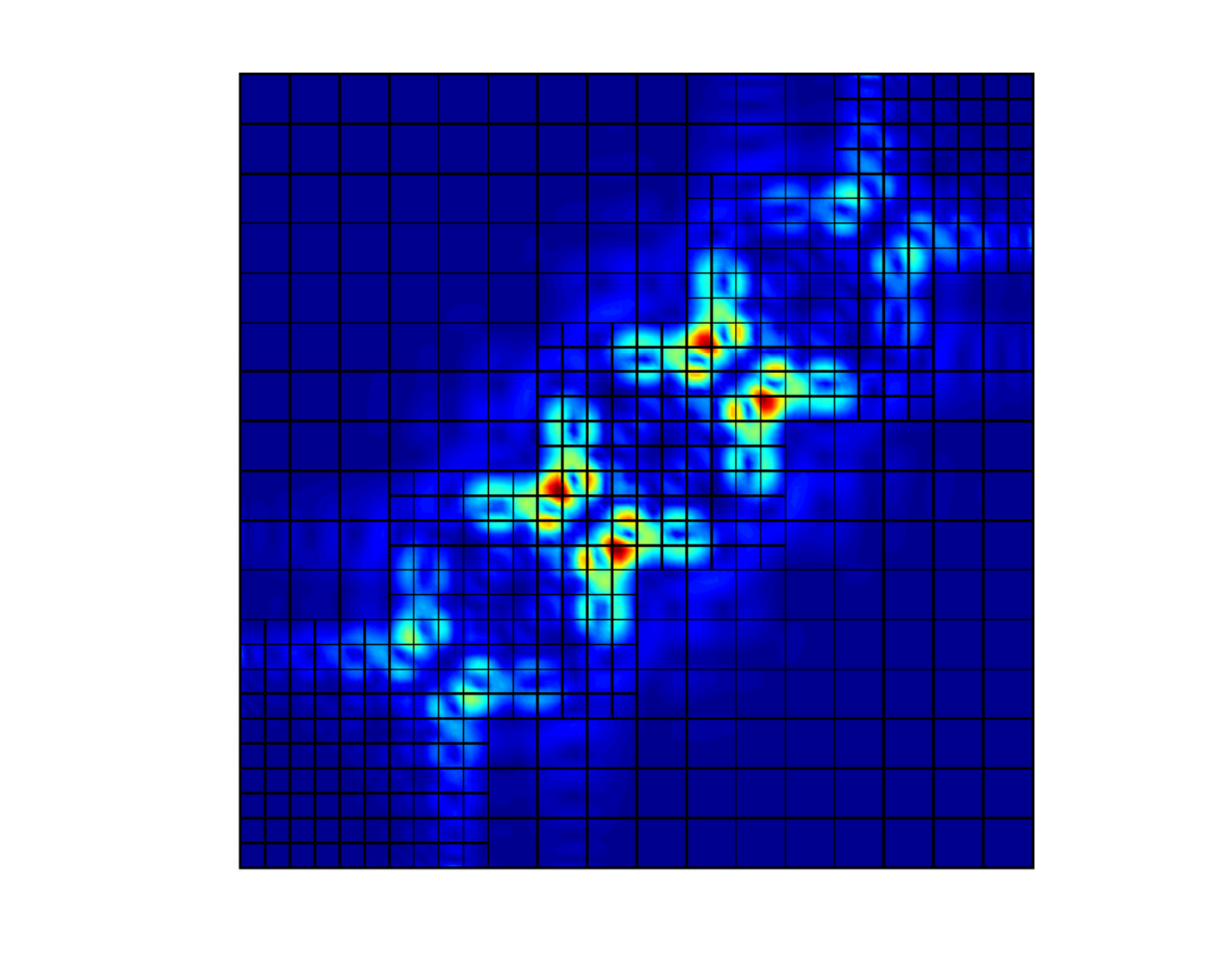}}
\caption{Magnitude of the first spurious eigenfunction for the mesh depicted in Figure~\ref{fig:mesh_p4_2x2}.}\label{fig:eigenfunction}
\end{figure}

\subsection{Numerical Assessment of Inf-Sup Stability}
Our second group of tests regards the inf-sup stability of a Stokes flow discretization based on the hierarchical B-spline complex. The variational formulation of the problem is to find a velocity ${\bf u}_h \in \textbf{V}_h$ and a pressure $p_h \in Q_h$ such that

\begin{equation*}
\left \{
\begin{array}{ll}
2 \nu ({\bf grad}^s\, {\bf u}_h, {\bf grad}^s {\bf v}_h) - c({\bf u}_h, {\bf v}_h) - (p_h, \div \, {\bf v}_h) = ({\bf f}, {\bf v}_h) - L({\bf v}_h) & \forall {\bf v}_h \in \textbf{V}_h, \\
(q_h, \div \, {\bf u}_h) = 0 & \forall q_h \in Q_h, \label{eq:stokes}
\end{array}
\right.
\end{equation*}
where ${\bf grad}^s$ is the symmetric part of the gradient, $\nu$ is a constant viscosity, ${\bf f}$ is a given source term, and the terms $c({\bf u}_h, {\bf v}_h)$ and $L({\bf v}_h)$ come from the weak imposition of tangential Dirichlet boundary conditions using Nitsche's method.  In particular:
\begin{equation*}
\begin{array}{l}
\displaystyle c({\bf u}_h, {\bf v}_h) = \int_{\Gamma_D} 2 \nu \left ((({\bf grad}^s {\bf u}_h) {\bf n}) \cdot {\bf v}_h + (({\bf grad}^s {\bf v}_h) {\bf n}) \cdot {\bf u}_h - \frac{C_{pen}}{h_F} {\bf u}_h \cdot {\bf v}_h \right), \\
\displaystyle L({\bf v}_h) = \int_{\Gamma_D} 2 \nu \left( (({\bf grad}^s {\bf v}_h){\bf n}) \cdot {\bf u}_{B} - \frac{C_{pen}}{h_F} {\bf v}_h \cdot {\bf u}_{B} \right),
\end{array}
\end{equation*}
where ${\bf u}_B$ is the imposed velocity along the Dirichlet boundary $\Gamma_D$, $h_F$ is the element size in the normal direction, and $C_{pen}$ is a sufficiently large penalization constant, see \cite{EvHu12} for details.  For the calculations shown here, $C_{pen} = 5 \max_{1 \leq i \leq n} p_i$.

In order to guarantee stability of the discretization scheme, the discrete spaces $V_h$ and $Q_h$ must satisfy the {\bf \emph{inf-sup}} stability condition
\begin{equation} \label{eq:inf-sup}
\inf_{q_h \in Q_h \setminus\{0\}} \sup_{{\bf v}_h \in \textbf{V}_h\setminus\{0\}} \frac{(q_h, \div \, {\bf v}_h)}{\|{\bf v}_h\|_{\textbf{V}_h} \|q_h\|_{L^2(\Omega)} } \ge \beta > 0,
\end{equation}
where $\beta$ is a positive constant independent of the mesh size and the number of levels and $\|\cdot \|_{\textbf{V}_h}$ is a suitable discrete norm.  Following \cite{EvHu12}, we choose:
\begin{equation}
\| \textbf{v}_h \|_{\textbf{V}_h} := \left( \| {\bf grad}^s \textbf{v}_h \|_{L^2(\Omega)}^2 + \| h^{1/2}_F \left( {\bf grad}^s \textbf{v}_h \right) {\bf n} \|^2_{L^2(\Gamma_D)} + \| \left(C_{pen}/h_F\right)^{1/2} \textbf{v}_h \|^2_{L^2(\Gamma_D)} \right)^{1/2}.
\end{equation}
We are also interested in discretization spaces that provide a pointwise divergence-free solution. A sufficient condition to satisfy the incompressibility constraint is that the divergence of $\textbf{V}_h$ is contained in $Q_h$, that is,
\begin{equation} \label{eq:equality}
\{ \div \, {\bf v}_h, {\bf v}_h \in \textbf{V}_h \} \subseteq Q_h,
\end{equation}
but this condition is in conflict with \eqref{eq:inf-sup} unless the two spaces in \eqref{eq:equality} are equal.

\begin{figure}[t]
\centering
\begin{subfigure}[]
{\includegraphics[width = 0.4\textwidth, trim=1cm 1cm 1cm 0cm, clip]{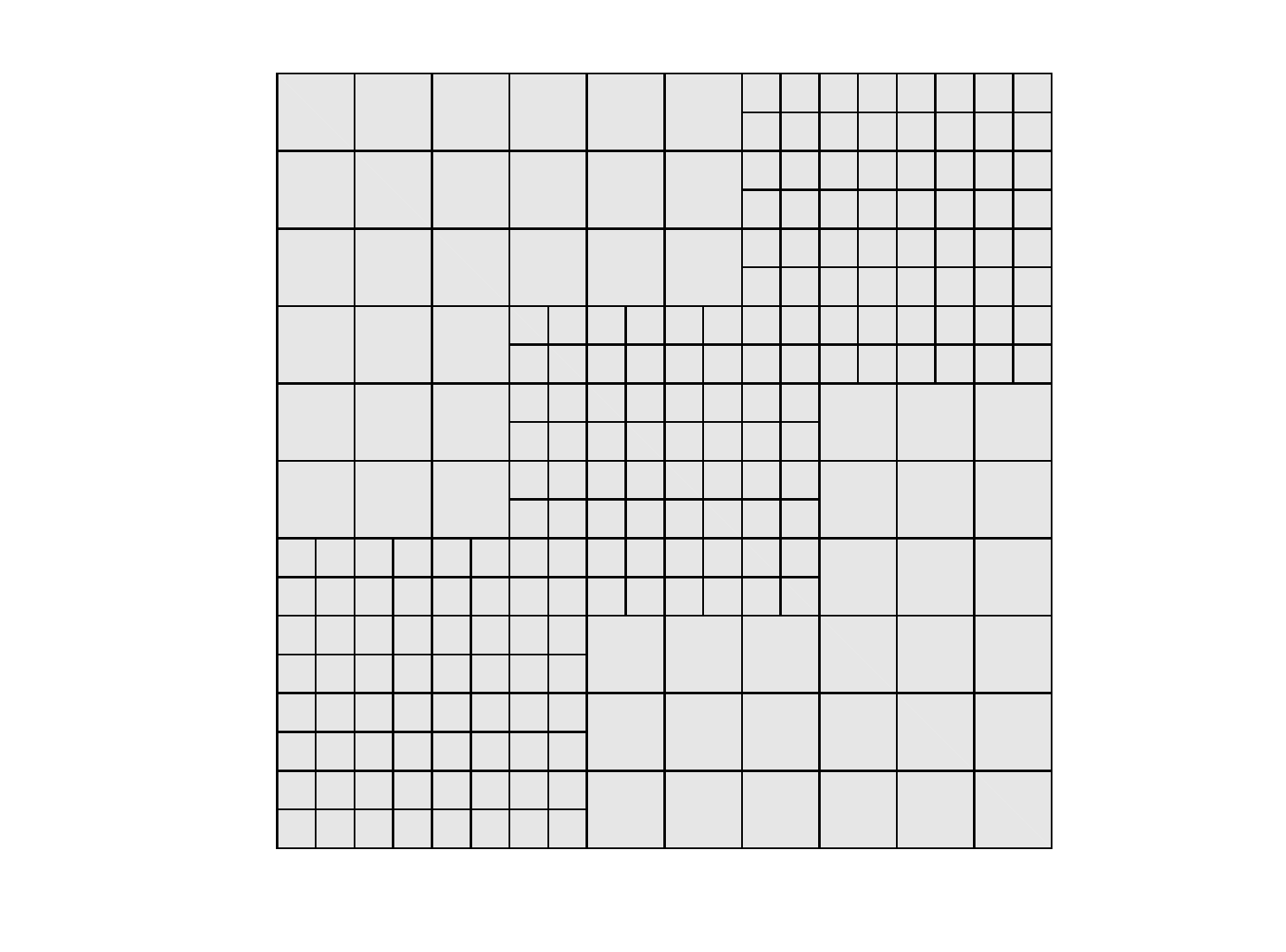}\label{fig:mesh_min_stokes}}
\end{subfigure}
\begin{subfigure}[]
{\includegraphics[width = 0.4\textwidth, trim=1cm 1cm 1cm 0cm, clip]{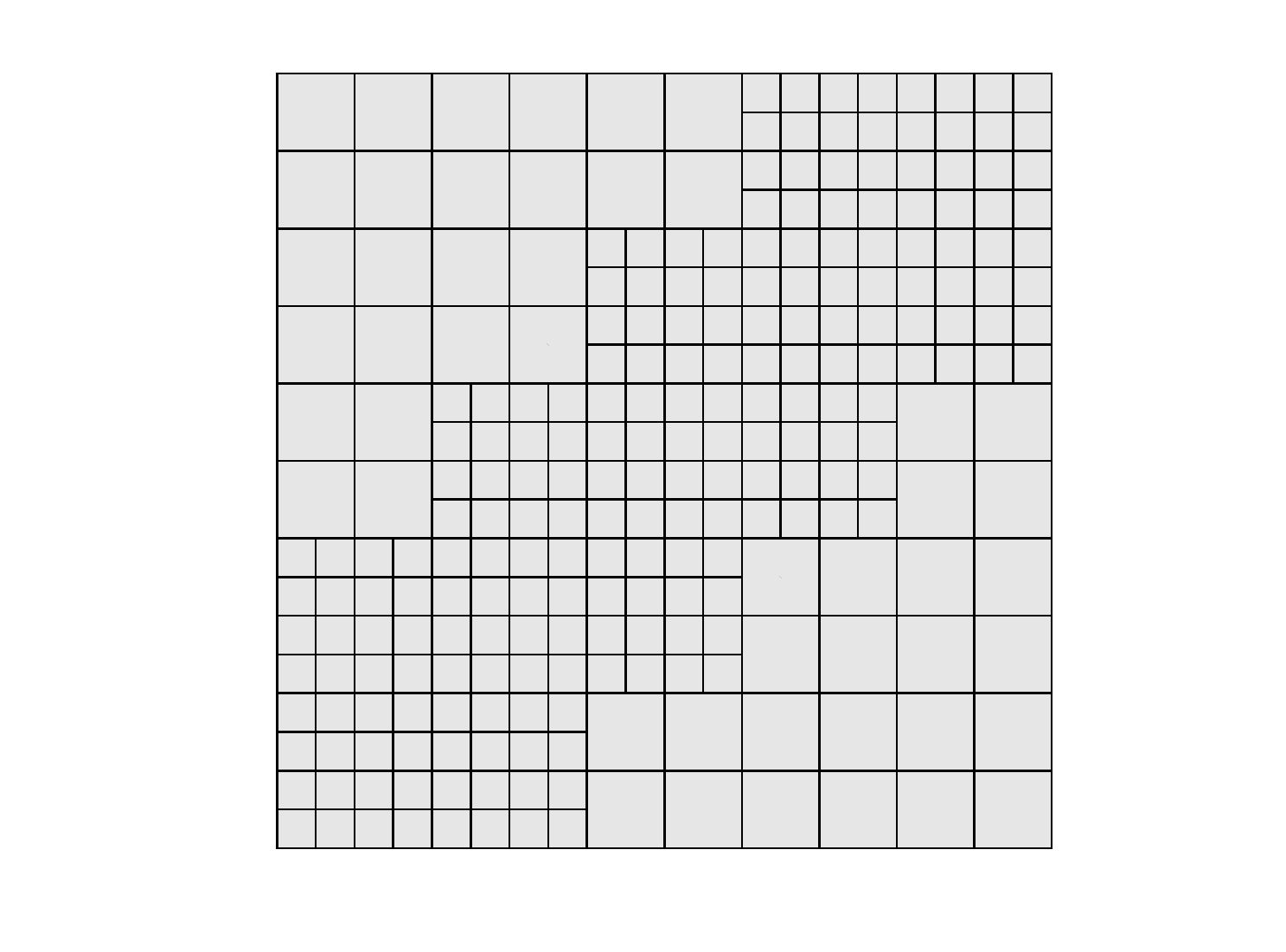}\label{fig:mesh_med_stokes}}
\end{subfigure}
\begin{subfigure}[]
{\includegraphics[width = 0.4\textwidth, trim=1cm 1cm 1cm 0cm, clip]{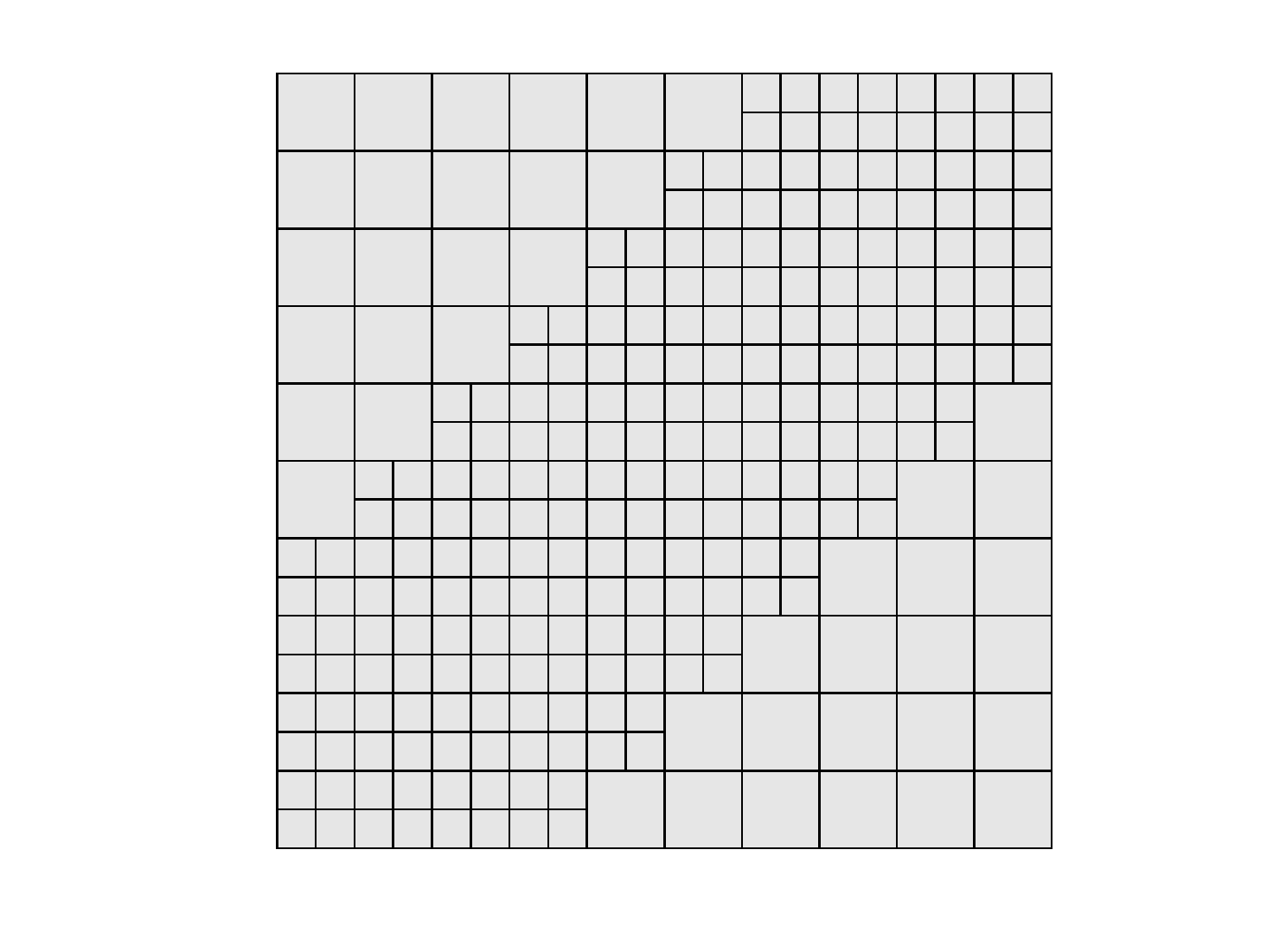}\label{fig:mesh_max_stokes} }
\end{subfigure}
\begin{subfigure}[]
{\includegraphics[width = 0.4\textwidth, trim=1cm 1cm 1cm 0cm, clip]{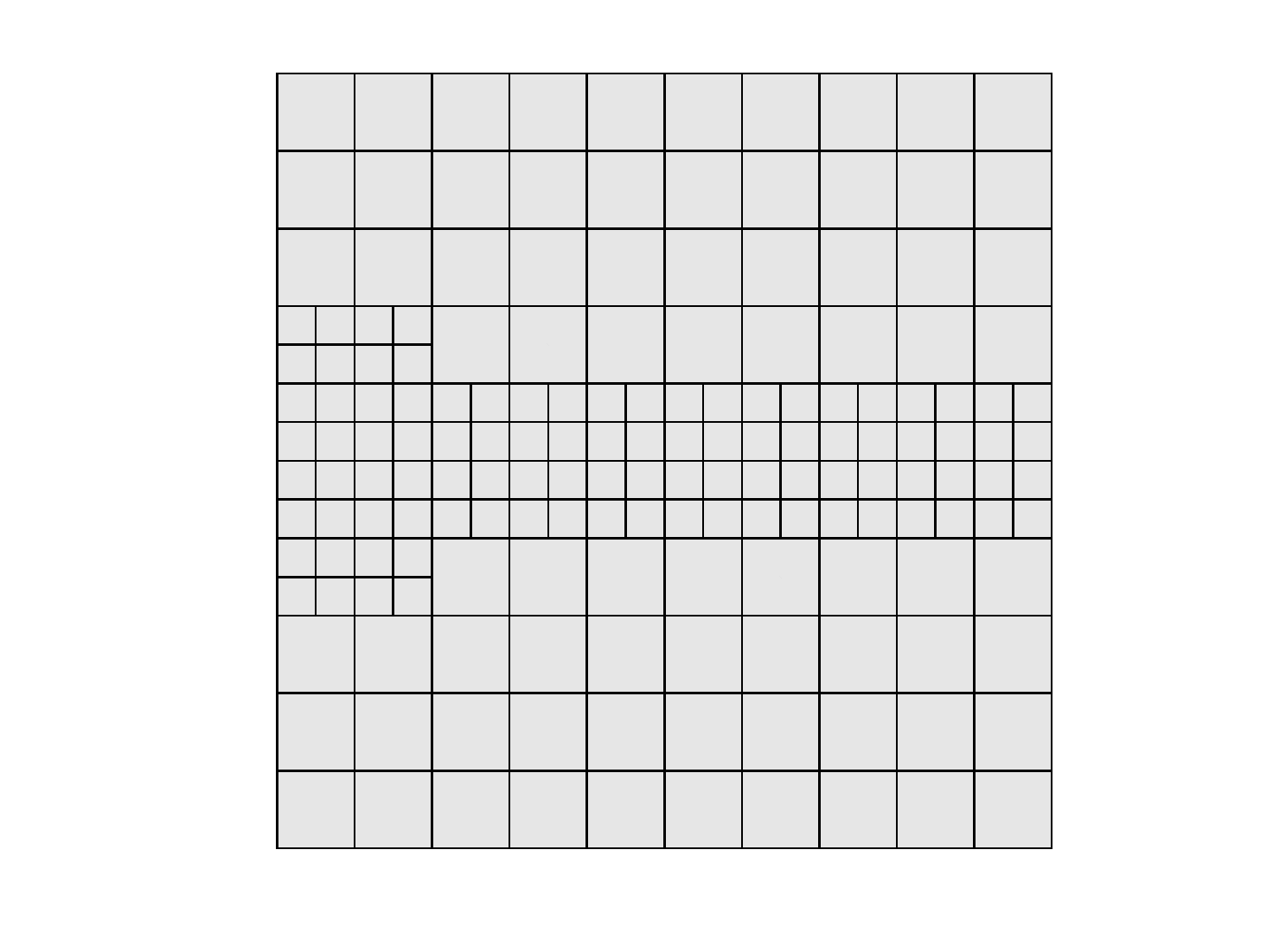}\label{fig:mesh_bulge_stokes}}
\end{subfigure}
\caption{The four different two-level hierarchical B\'{e}zier mesh configurations used in the Stokes inf-sup tests.  Maximally continuous B-splines with $p_1 = p_2 = 3$ were employed in the tests.} \label{fig:meshes-stokes}
\end{figure}

\begin{table}[t!]
\caption{Properties satisfied by the four two-level meshes considered for the Stokes inf-sup tests.} \label{tab:stokes_meshes}
\centering
\begin{tabular}{|c||c|c|c|c|}
\hline
& $1\times 1$ overlap & $2\times 2$ overlap & $3\times 3$ overlap & Bulge \\
Mesh & Fig.~\ref{fig:mesh_min_stokes} & Fig.~\ref{fig:mesh_med_stokes} & Fig.~\ref{fig:mesh_max_stokes} & Fig.~\ref{fig:mesh_bulge_stokes} \\
\hline
Assumption~\ref{ass:support} & Yes & Yes & Yes & No \\
Assumption~\ref{ass:overlap} & No & No & Yes & No \\
Exact sequence & Yes & No & Yes & Yes \\
\hline
Inf-sup stable (two levels) & Yes & Yes & Yes & No \\
Inf-sup stable (multilevel) & No & Yes & Yes & No \\
\hline
\end{tabular}
\end{table}

\begin{table}[t!]
\caption{Results of the Stokes inf-sup tests with the two-level meshes.} \label{tab:2levels}
\centering
\begin{tabular}{|c||c|c|c|c|c|}
\hline
Level 0 size & Uniform & $1\times 1$ overlap & $2\times 2$ overlap & $3\times 3$ overlap & Bulge \\
\hline
1/10 & 0.40996 & 0.40963 & 0.40963 & 0.40963 & 0.04818 \\
1/22 & 0.40957 & 0.40943 & 0.40943 & 0.40943 & 0.02046 \\
1/40 & 0.40943 & 0.40934 & 0.40934 & 0.40934 & 0.01098 \\
\hline
\end{tabular}
\end{table}

\begin{figure}[t!]
\begin{center}
\begin{subfigure}[]
{\includegraphics[width = 0.4\textwidth, trim=1cm 1cm 1cm 0cm, clip]{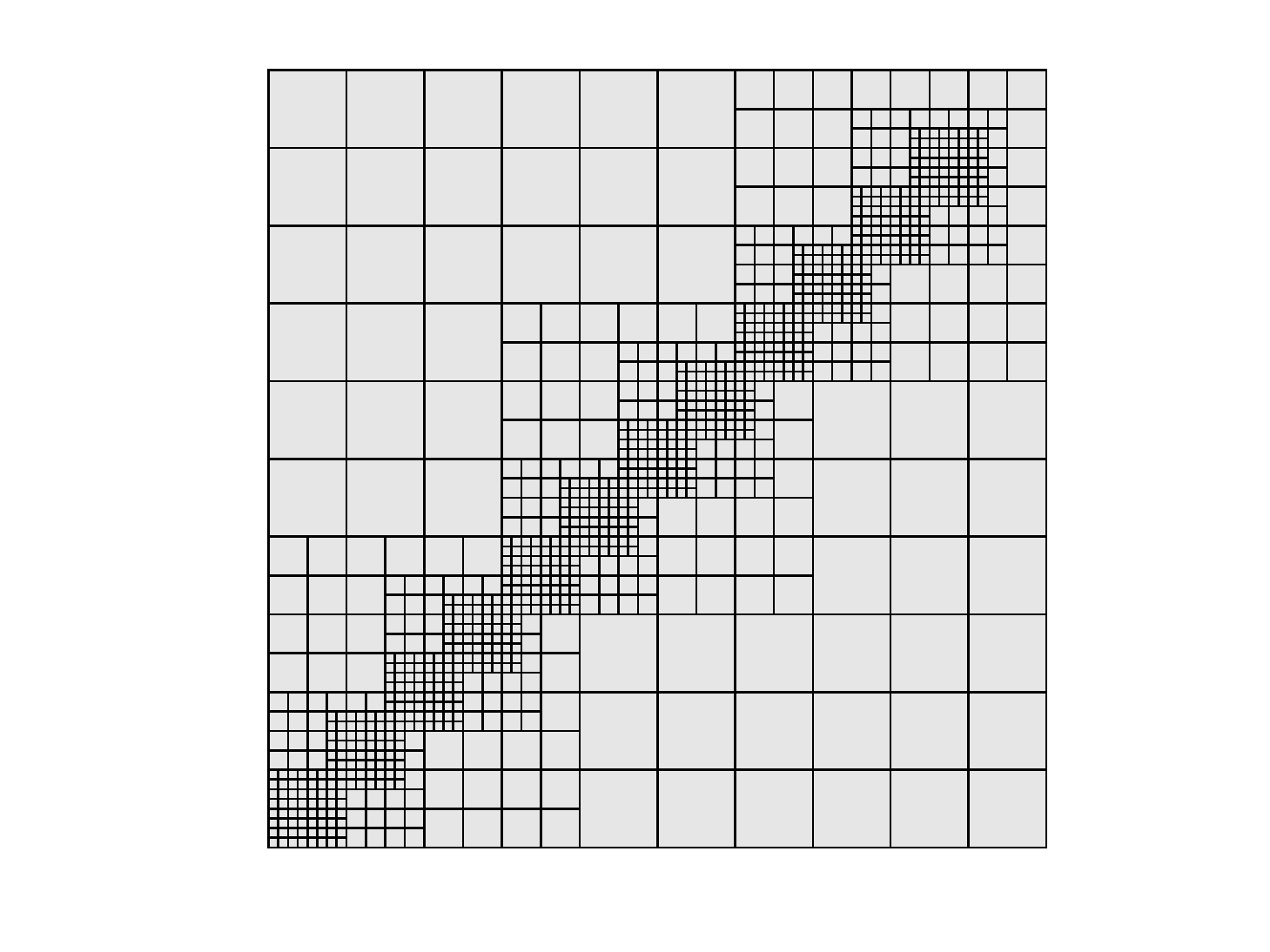}\label{fig:min_4levels}}
\end{subfigure}
\begin{subfigure}[]
{\includegraphics[width = 0.4\textwidth, trim=1cm 1cm 1cm 0cm, clip]{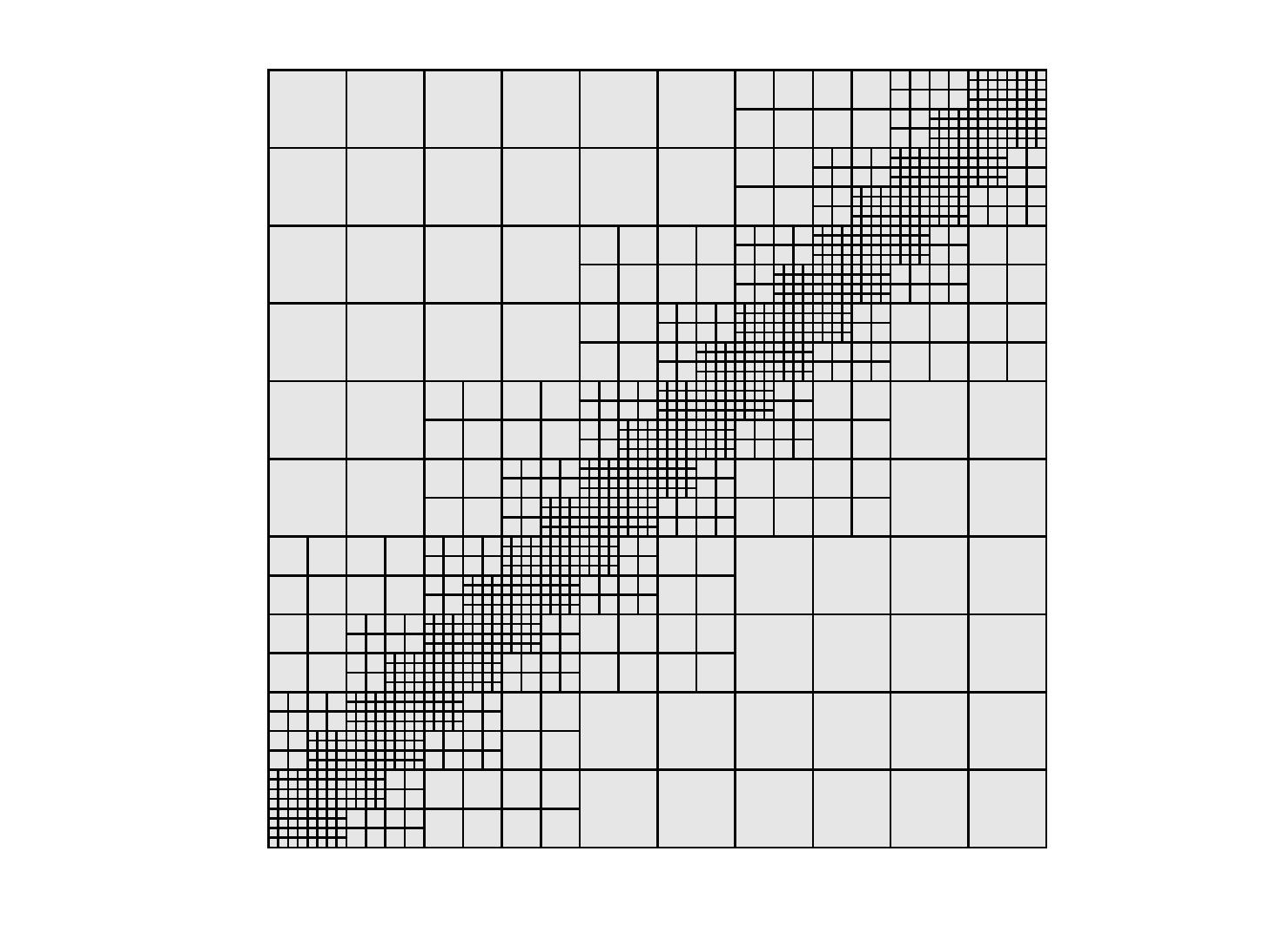}\label{fig:med_4levels}}
\end{subfigure}
\begin{subfigure}[]
{\includegraphics[width = 0.4\textwidth, trim=1cm 1cm 1cm 0cm, clip]{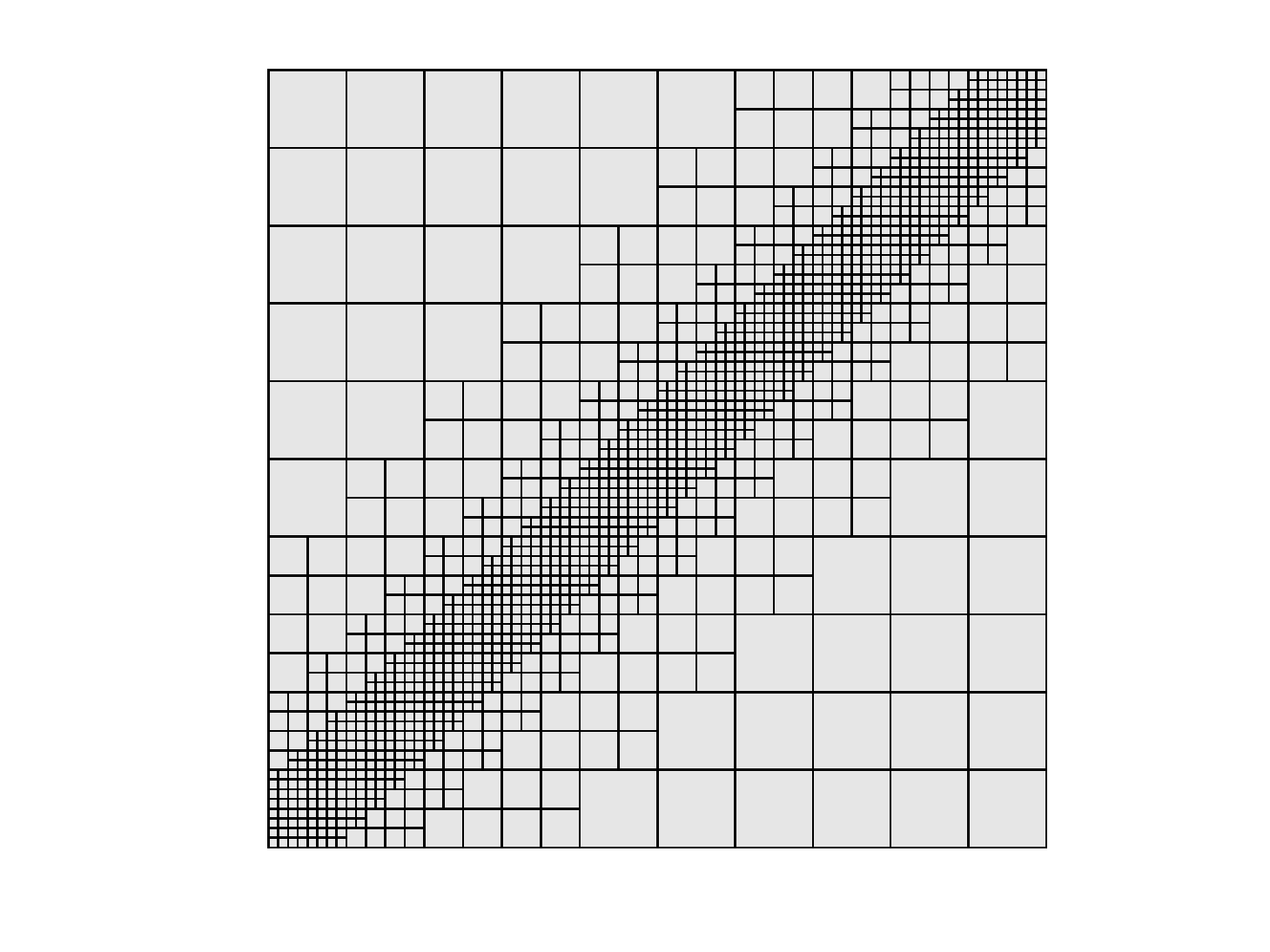}\label{fig:max_4levels}}
\end{subfigure}
\begin{subfigure}[]
{\includegraphics[width = 0.4\textwidth, trim=1cm 1cm 1cm 0cm, clip]{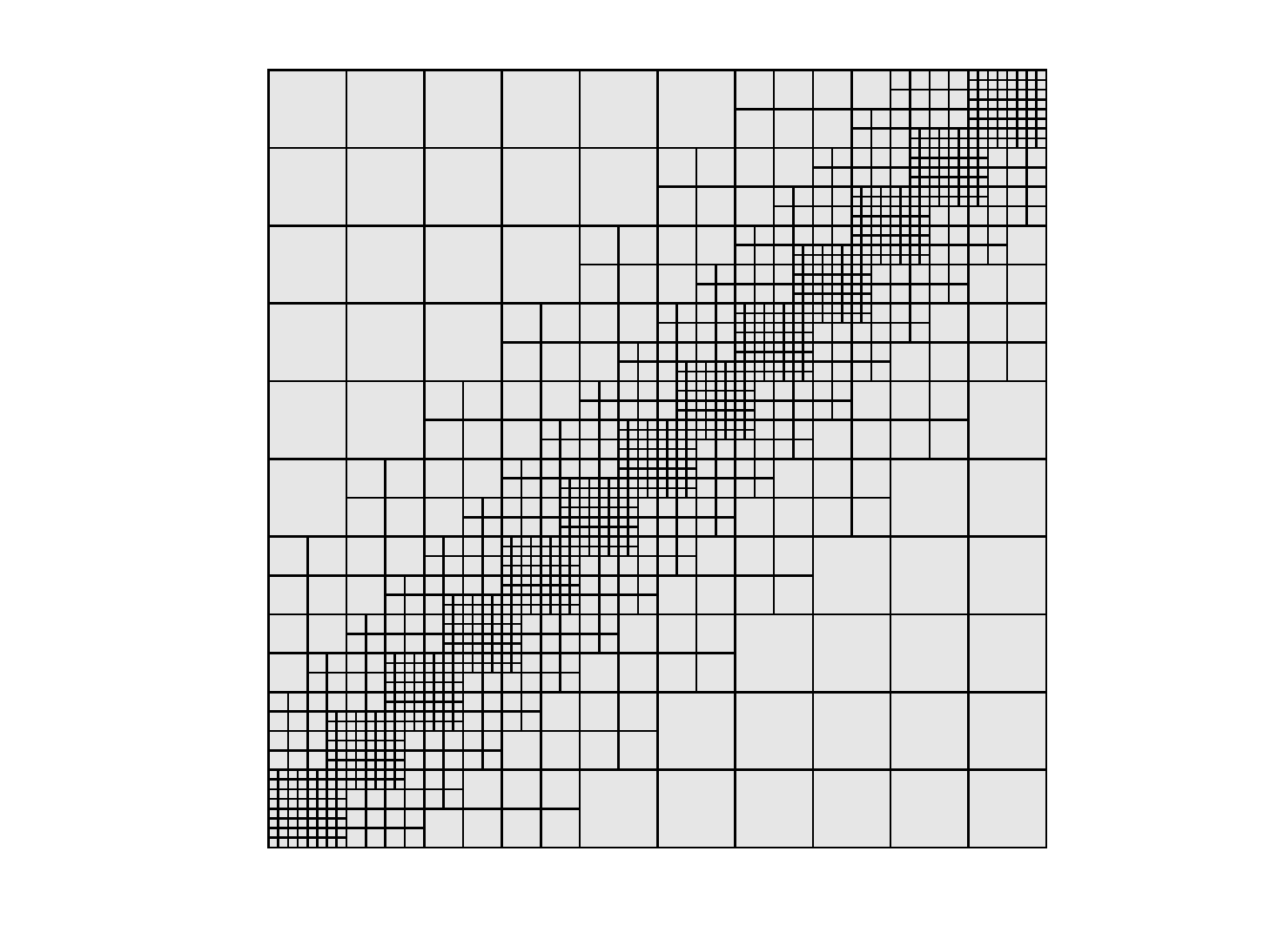}\label{fig:maxmin_4levels}}
\end{subfigure}
\end{center}
\caption{The four different multilevel hierarchical B\'{e}zier mesh configurations, with four levels, used in the Stokes inf-sup tests.  Maximally continuous B-splines with $p_1 = p_2 = 3$ were employed in the tests.} \label{fig:meshes-stokes2}
\end{figure}

\begin{table}[t!]
\caption{Results of the Stokes inf-sup tests with the multilevel meshes.} \label{tab:multilevel}
\centering
\begin{tabular}{|c||c|c|c|c|}
\hline
& $1\times 1$ overlap & $2\times 2$ overlap & $3\times 3$ overlap & Graded ($1\times 1$) overlap\\
Levels& Fig.~\ref{fig:min_4levels} & Fig.~\ref{fig:med_4levels} & Fig.~\ref{fig:max_4levels} & Fig.~\ref{fig:maxmin_4levels} \\
\hline
1 & 0.40996 & 0.40996 & 0.40996 & 0.40996 \\
2 & 0.40963 & 0.40963 & 0.40963 & 0.40963 \\
3 & 0.29883 & 0.40945 & 0.40945 & 0.40944 \\
4 & 0.22467 & 0.40935 & 0.40935 & 0.40935 \\
5 & 0.17209 & 0.40929 & 0.40929 & 0.40916 \\
6 & 0.14519 & 0.40926 & 0.40926 & 0.40926 \\
\hline
\end{tabular}
\end{table}

In \cite{Buffa_deFalco_Sangalli} and \cite{EvHu12} a discretization scheme based on spline differential forms was proposed that is both divergence-free and stable. The key point is to choose the space of divergence-conforming splines (i.e., 1-forms in the rotated hierarchical B-spline complex in the two-dimensional setting and 2-forms in the standard hierarchical B-spline complex in the three-dimensional setting) as the velocity space $\textbf{V}_h$ and the spline space of $n$-forms as the pressure space $Q_h$.  When Dirichlet boundary conditions are enforced along the entire boundary, we set $\textbf{V}_h = W^2_N$ and $Q_h = W^3_N/\mathbb{R}$ in the three-dimensional setting.  Provided the velocity space $\textbf{V}_h$ satisfies $\textbf{V}_h \subset \left(H^1(\Omega)\right)^n$, the exactness of the tensor-product B-spline complex as well as associated commutative projectors guarantee that both \eqref{eq:inf-sup} and \eqref{eq:equality} are satisfied in the case of tensor-product B-splines (see \cite{EvHu12} for details). For hierarchical B-splines, the condition \eqref{eq:equality} is satisfied by construction, but since a set of commutative projectors is not available to prove \eqref{eq:inf-sup}, we have instead numerically computed the inf-sup constant through the solution of a discrete eigenvalue problem (see \cite{inf-sup-test}) for several mesh configurations to understand whether the resulting discretization scheme is stable or not. All the examples consider quadratic splines for pressure approximation and mixed quadratic-cubic splines for the velocity approximation.

The first set of tests regards hierarchical B\'{e}zier meshes with only two levels, and the value of the constant is compared also with the one obtained for tensor-product spaces with one single level. We have considered four different kind of refinements, similar to those already tested in Maxwell eigenproblem. For the first three, displayed in Figures~\ref{fig:mesh_min_stokes}-\ref{fig:mesh_max_stokes}, we refine several regions of $4\times4$ elements along the diagonal with different levels of overlaps between them. The fourth mesh, displayed in Figure~\ref{fig:mesh_bulge_stokes}, gives an exact sequence, but it does not satisfy Assumption~\ref{ass:support}. The properties satisfied by each of the four meshes are summarized in Table~\ref{tab:stokes_meshes}.

The results of the inf-sup test are reported in Table~\ref{tab:2levels}. All the diagonal refinements show a good behavior, even if the $2\times2$ overlap does not satisfy the exactness condition (see Figure~\ref{fig:remove-add-2x2}). In fact, we have seen in the previous tests for the Maxwell eigenproblem that the lack of exactness in this latter case only affects the ``left part'' of the diagram. On the contrary, the refinement in Figure~\ref{fig:mesh_bulge_stokes}, which does not satisfy Assumption~\ref{ass:support}, gives a discrete scheme that becomes unstable with increasing levels of mesh resolution. This suggests that {\bf \emph{Assumption~\ref{ass:support} is a necessary condition}} for inf-sup stability.

To understand whether the inf-sup constant depends on the number of levels present in the hierarchical B-spline complex, we have considered spline spaces with multiple levels in a second set of tests. We consider meshes with refinement along the diagonal as before but with an increasing number of levels.  See Figure~\ref{fig:min_4levels}-\ref{fig:max_4levels} for the resulting hierarchical B\'{e}zier meshes with four levels of refinement. The results obtained for the inf-sup constant are summarized in Table~\ref{tab:multilevel}. In this case, the discrete scheme remains stable for the $2\times 2$ and the $3 \times 3$ overlap, while the stability deteriorates for the $1\times 1$ overlap. This loss of stability is caused by the presence of adjacent elements corresponding to non-consecutive levels, and the issue can be avoided with a better grading of the mesh (see, e.g., \cite{Scott2014222}) which can be obtained using the $m$-admissible meshes defined in \cite{Buffa16-1}. To show this, we have repeated the inf-sup test for hierarchical B\'{e}zier meshes similar in appearance to Figure~\ref{fig:maxmin_4levels} where only the finest level is like in the $1\times 1$ overlap, thus avoiding non-graded meshes. The results in Table~\ref{tab:multilevel} show that in this case the scheme is stable with increasing levels of mesh resolution. It is worth to note that the loss of stability on non-graded meshes does not appear when using the generalization of Taylor-Hood elements with hierarchical splines \cite{Bressan-Juttler}, but in that case the incompressibility condition is not fulfilled exactly.  We have conducted a number of additional tests, and all of these suggest that {\bf \emph{Assumptions~\ref{ass:support} and~\ref{ass:overlap}, together with a suitable grading of the mesh, are sufficient conditions}} for inf-sup stability.  However, as demonstrated above for diagonal refinements with $2\times 2$ overlap, Assumption~\ref{ass:overlap} is not necessary for inf-sup stability.  Further study is required to develop necessary and sufficient conditions for inf-sup stability.

\section{Application of the Hierarchical B-spline Complex} \label{sec:tests}
In this final section, we apply the hierarchical B-spline complex to the numerical solution of two canonical problems of interest, one stemming from electromagnetics and the other from creeping flow.  For both problems, we examine the convergence properties of the hierarchical B-spline complex and its suitability for adaptive methods, and we further examine the accuracy gains hierarchical B-splines provide as compared with uniform B-splines.

\subsection{Application to Electromagnetics: The Curved L-shaped Domain}
As a first test problem, we solve the Maxwell eigenvalue problem, given in variational form by \eqref{eq:maxwell}, over the curved L-shaped domain.  This is a challenging problem as the first few eigenfunctions exhibit a strong singularity at the reentrant corner.  The problem details are adopted from Monique Dauge's webpage of benchmark computations for Maxwell equations \cite{BMAX}. The domain, shown in Figure~\ref{fig:curvedL}, is represented by three patches, and for the computations we start with a zero level mesh of $8\times 8$ element on each patch. Then, at each refinement step we refine on each patch one quarter of the (parametric) domain near the reentrant corner, to obtain a better approximation of the singularity. The meshes obtained with this refinement procedure, as the one represented in Figure~\ref{fig:curvedL} for six levels, satisfy both Assumptions~\ref{ass:support} and~\ref{ass:overlap}.

\begin{figure}[b!]
\begin{subfigure}[Mesh with six levels]
{\includegraphics[width=0.4\textwidth,trim=3cm 2.5cm 2cm 1.8cm, clip]{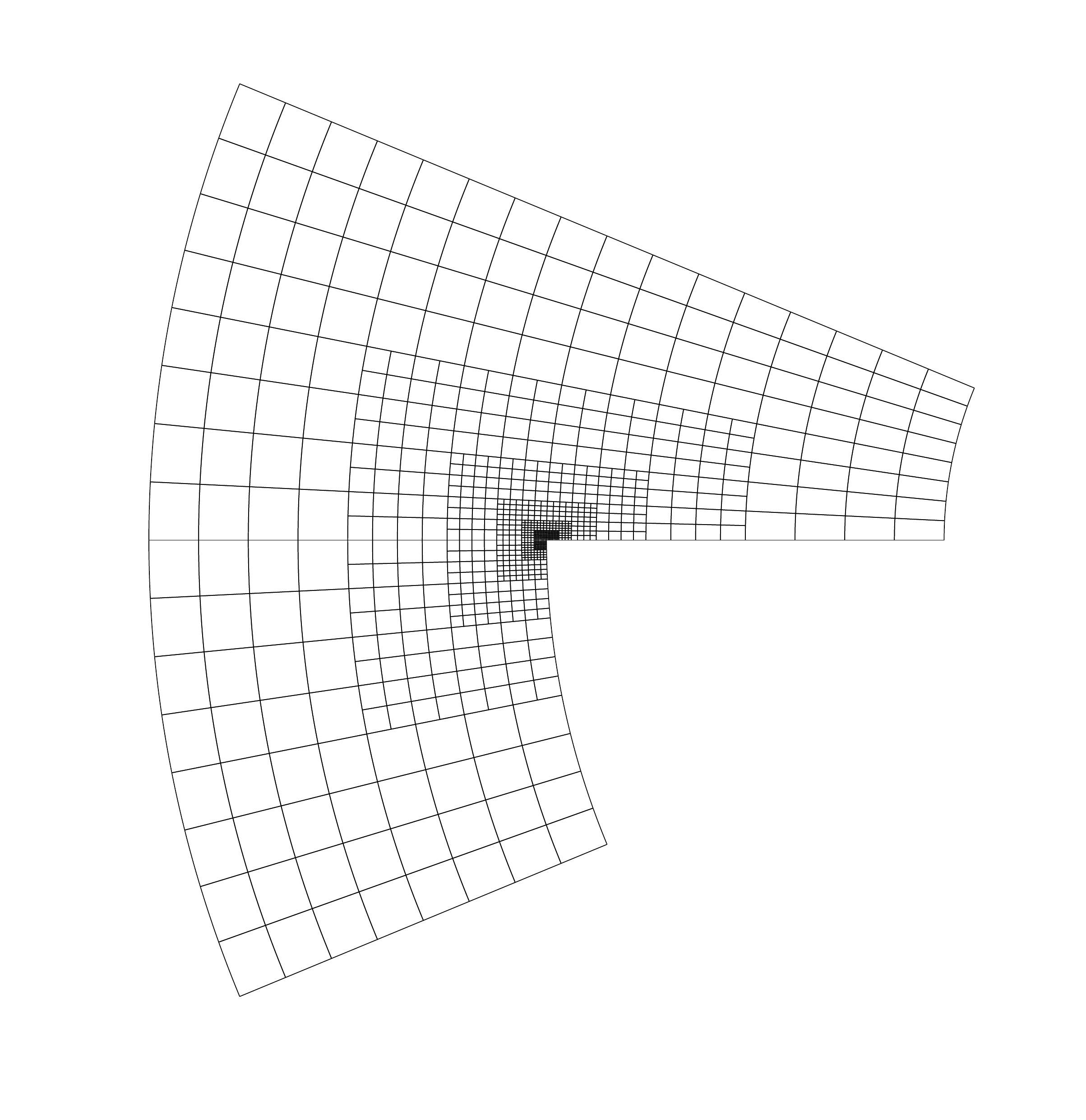} \label{fig:curvedL}}
\end{subfigure}
\begin{subfigure}[Convergence of the first eigenvalue]
{\includegraphics[width=0.58\textwidth,trim=4cm 1cm 4cm 1cm, clip]{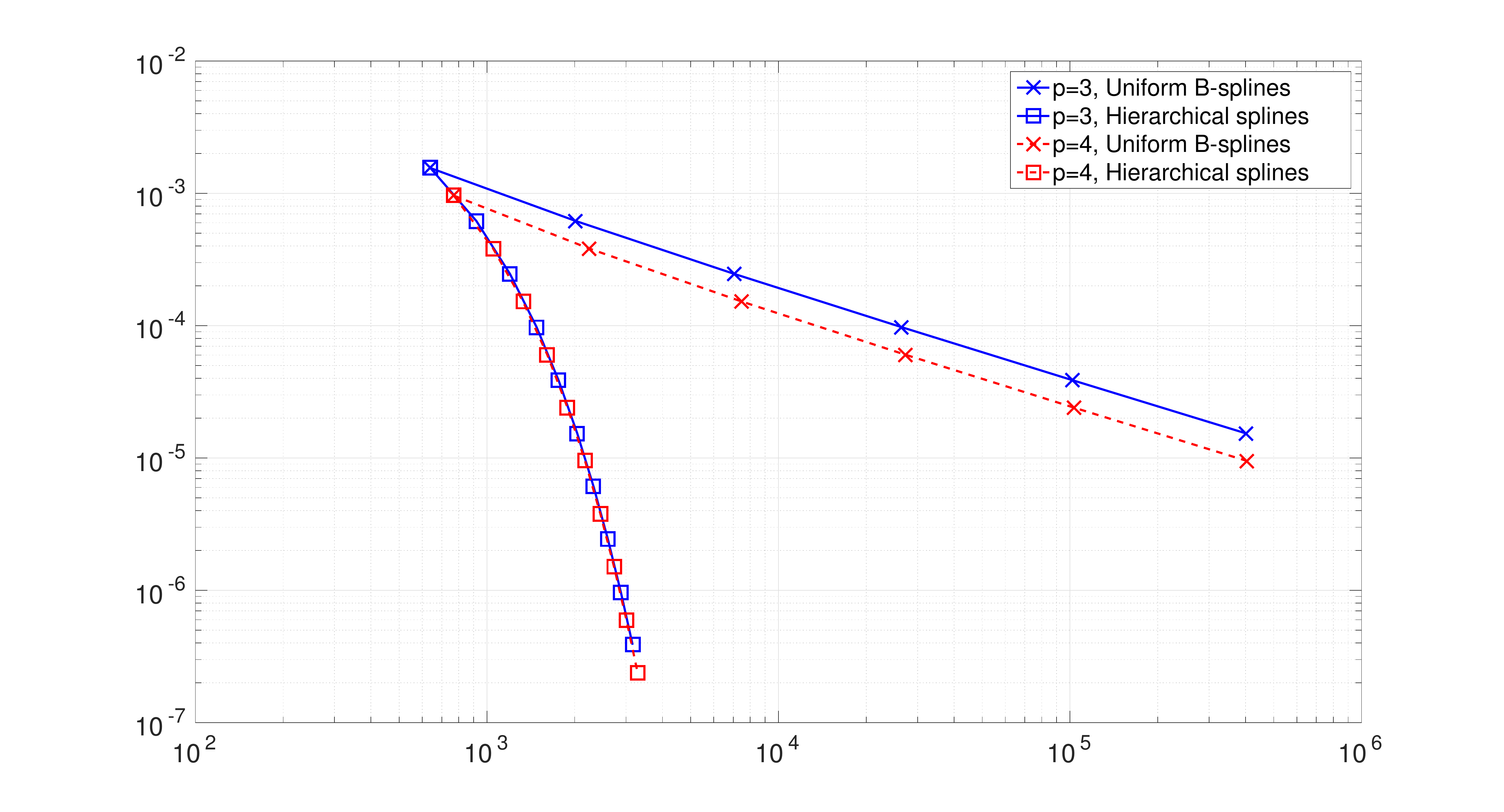} \label{fig:curvedL-results} }
\end{subfigure}
\caption{Curved L-shaped domain: Hierarchical B\'{e}zier mesh and convergence results.} 
\end{figure}

\begin{table}[b!]
\caption{Comparison of the first five computed eigenvalues of the curved L-shaped domain, for uniform tensor-product B-splines and hierarchical B-splines.} \label{tab:curvedL}
\centering
\begin{tabular}{|c||c|c|c|c|}
\hline
 & \multicolumn{2}{|c|}{$p_1 = p_2 = 3$} & \multicolumn{2}{|c|}{$p_1 = p_2 = 4$} \\
\cline{2-5}
Exact & Uniform & HB-splines & Uniform & HB-splines \\
\hline
1.818571152 & 1.818555761 & 1.818555754 & 1.818561623 & 1.818561620 \\
3.490576233 & 3.490576123 & 3.490576132 & 3.490576165 & 3.490576165 \\
10.06560150 & 10.06560150 & 10.06560272 & 10.06560150 & 10.06560150 \\
10.11188623 & 10.11188494 & 10.11188615 & 10.11188543 & 10.11188544 \\
12.43553725 & 12.43550823 & 12.43550992 & 12.43551928 & 12.43551928 \\
\hline
d.o.f. & 400416 & 2040 & 403522 & 2170 \\
\hline
\end{tabular}
\end{table}

We compare in Figure~\ref{fig:curvedL-results} the convergence results for the approximation of the first eigenvalue using hierarchical B-splines and tensor-product B-splines with uniform refinement, for degrees three and four. The associated eigenfunction only belongs to $H^{2/3-\epsilon}$ for any $\epsilon > 0$, which determines the convergence rate in the case of uniform B-splines. The results show much better approximation results for hierarchical B-splines, although the asymptotic regime is not reached after ten refinement steps. We also present, in Table~\ref{tab:curvedL}, the first five computed eigenvalues for the mesh of level six, and a comparison with the exact\footnote{The eigenvalues are taken from \cite{BMAX} and are correct up to 11 digits.} eigenvalues. The results are almost identical, but hierarchical B-splines require 200 times fewer degrees of freedom to obtain the same accuracy.

\begin{remark}
The results for tensor-product B-splines can be improved by generating a radical graded mesh, see \cite{Beirao_Cho_Sangalli} and \cite[Section~6]{BBSV-acta}. However, in this case the condition number strongly deteriorates due to the presence of anisotropic elements.
\end{remark}

\subsection{Application to Creeping Flow: Lid-Driven Cavity Flow}
As a second test problem, we solve the creeping lid-driven cavity flow problem.  Creeping flow is governed by the Stokes equations, given in variational form by \eqref{eq:stokes}, and the setup for the lid-driven cavity flow problem is elaborated in Figure \ref{f:lidcavitysetup}.  The left, right, and bottom sides of the cavity are fixed no-slip walls while the top side of the cavity is a wall which slides to the right with velocity magnitude $U$.  For the computations here, $H$ and $U$ are set to one.  Like the curved L-shaped domain problem described previously, the pressure and stress fields associated with the lid-driven cavity flow problem experience corner singularities which impede the convergence of numerical methods and expose unstable velocity/pressure pairs.  The velocity field is further characterized by a primary vortex near the center of the cavity and an infinite sequence of Moffatt eddies of decreasing size and intensity in the lower left and right corners of the cavity.  We demonstrate here that a discretization based on the hierarchical B-spline complex is able to both resolve the corner singularities as well as a select number of Moffatt eddies with a properly chosen adaptive strategy.

\begin{figure}[b!]
	\centering
	\includegraphics[height=3.25in]{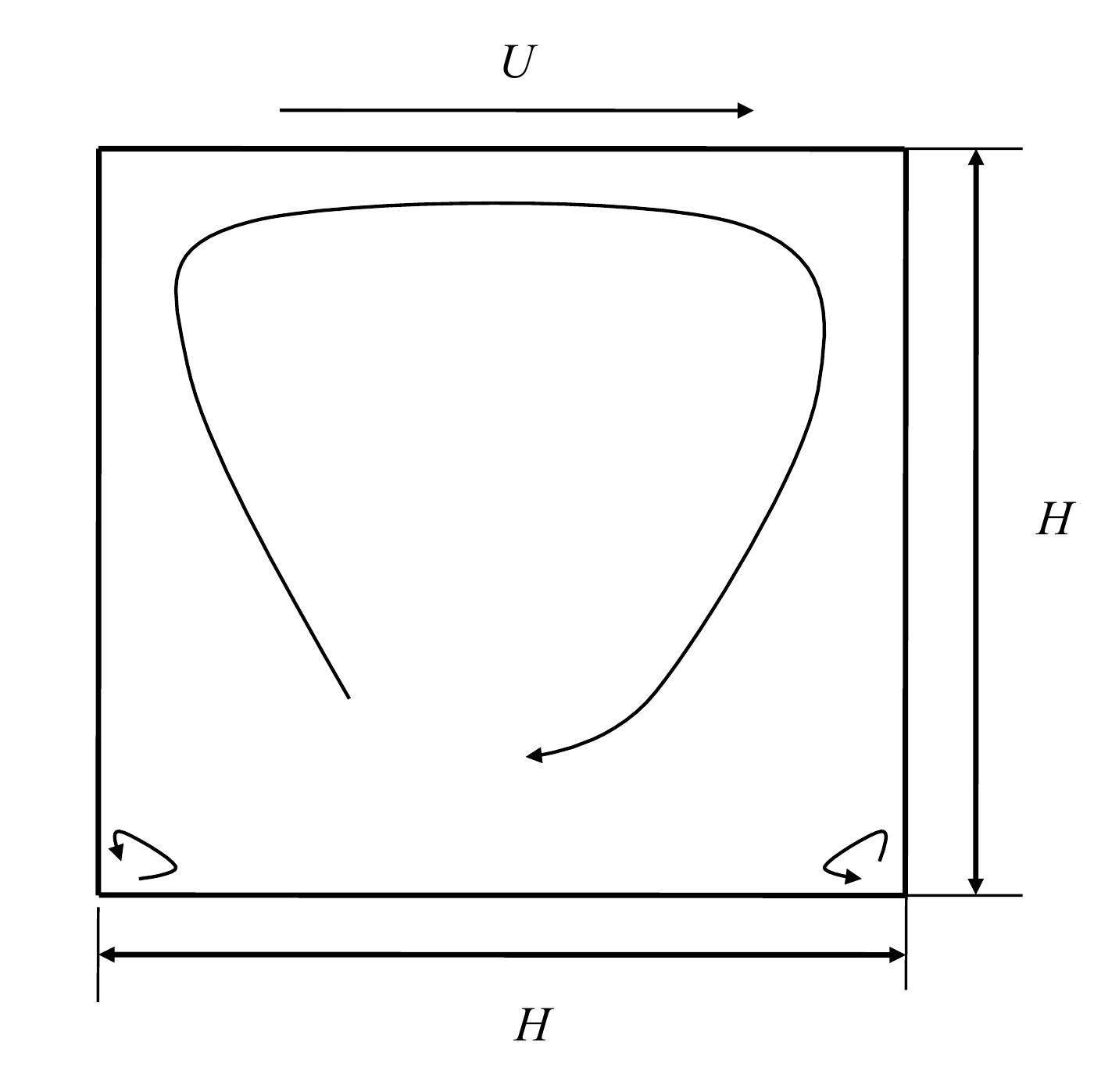}
	\caption{Problem setup for creeping lid-driven cavity flow.}
	\label{f:lidcavitysetup}
\end{figure}

We have computed approximations of lid-driven cavity flow using both hierarchical B-splines and tensor-product B-splines with varying mesh size and polynomial degrees two, three, and four.  For the case of tensor-product B-splines, we employed uniform refinement to resolve the features of the flow.  The computed streamlines for $h = 1/32$ and $h = 1/128$ with $p_1 = p_2 = 2$ are shown in Figure~\ref{f:tensorprodstreamlines}. Note that while the first Moffatt eddy is apparent in the $h = 1/32$ mesh, it is highly distorted.  The second Moffatt eddy does appear in the $h = 1/128$ mesh, but the third cannot be calculated using tensor-product B-splines until a mesh resolution of $h = 1/1024$.  For the case of hierarchical B-splines, we began with a uniform tensor-product B-spline mesh and applied three levels of hierarchical refinement in each of the corners of the domain as illustrated in Figure~\ref{f:lidcavityhierarchicalmesh}.  The computed streamlines for $h = 1/32$ and $h = 1/128$ with $p_1 = p_2 = 2$ are shown in Figure~\ref{f:hierarchicalstreamlines}, where $h$ indicates the element size for the initial tensor-product B-spline mesh and hence is the maximal element size for the hierarchical mesh.  Note that both the first and second Moffatt eddies are well-resolved for the $h = 1/32$ mesh, and the third Moffatt eddy can also be observed for the $h = 1/128$ mesh.  Consequently, while 3,153,925 degrees of freedom are necessary to capture the third Moffatt eddy with uniform tensor-product B-splines, one can capture the eddy with hierarchical B-splines using only $93,381$ degrees of freedom.  This is more than a thirty-fold reduction in total number of degrees of freedom.  It should be noted that the fourth Moffatt eddy cannot be obtained with double precision floating point arithmetic as the magnitudes of the velocities of the eddies are simply too small (below $1e{-12}$).

\begin{figure}[t]
	\centering
	\includegraphics[height=2.75in]{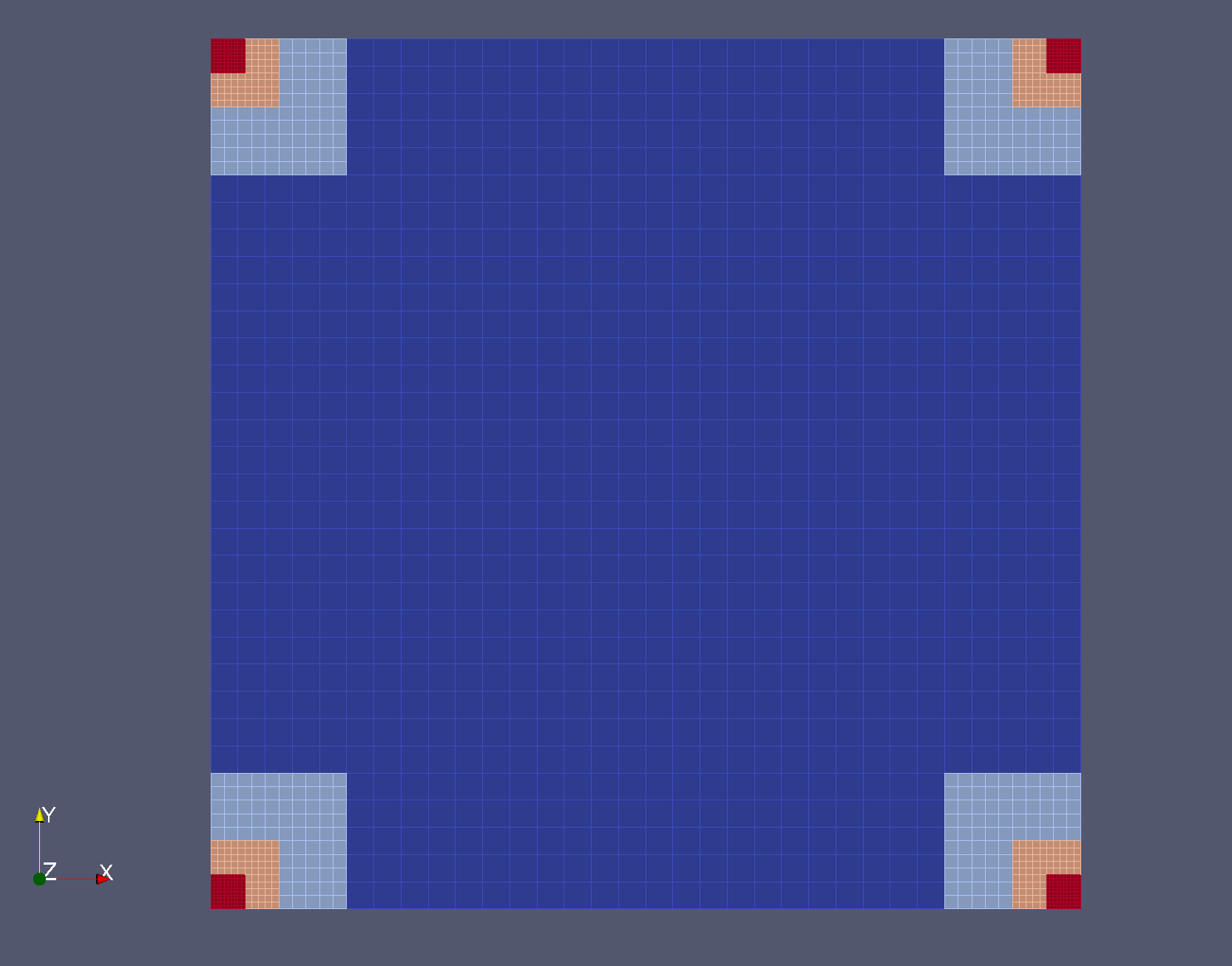}
	\caption{Hierarchical refinement strategy for the creeping lid-driven cavity flow problem.  The displayed hierarchical B\'{e}ziermesh corresponds to $h = 1/32$ where $h$ is the element size for the initial tensor-product B-spline mesh.}
	\label{f:lidcavityhierarchicalmesh}
\end{figure}

%
\begin{figure}[t]
	\centering

	\subfigure[{$[0,1] \times [0,1]$, $h = 1/32$}]{\includegraphics[height=2.3in]{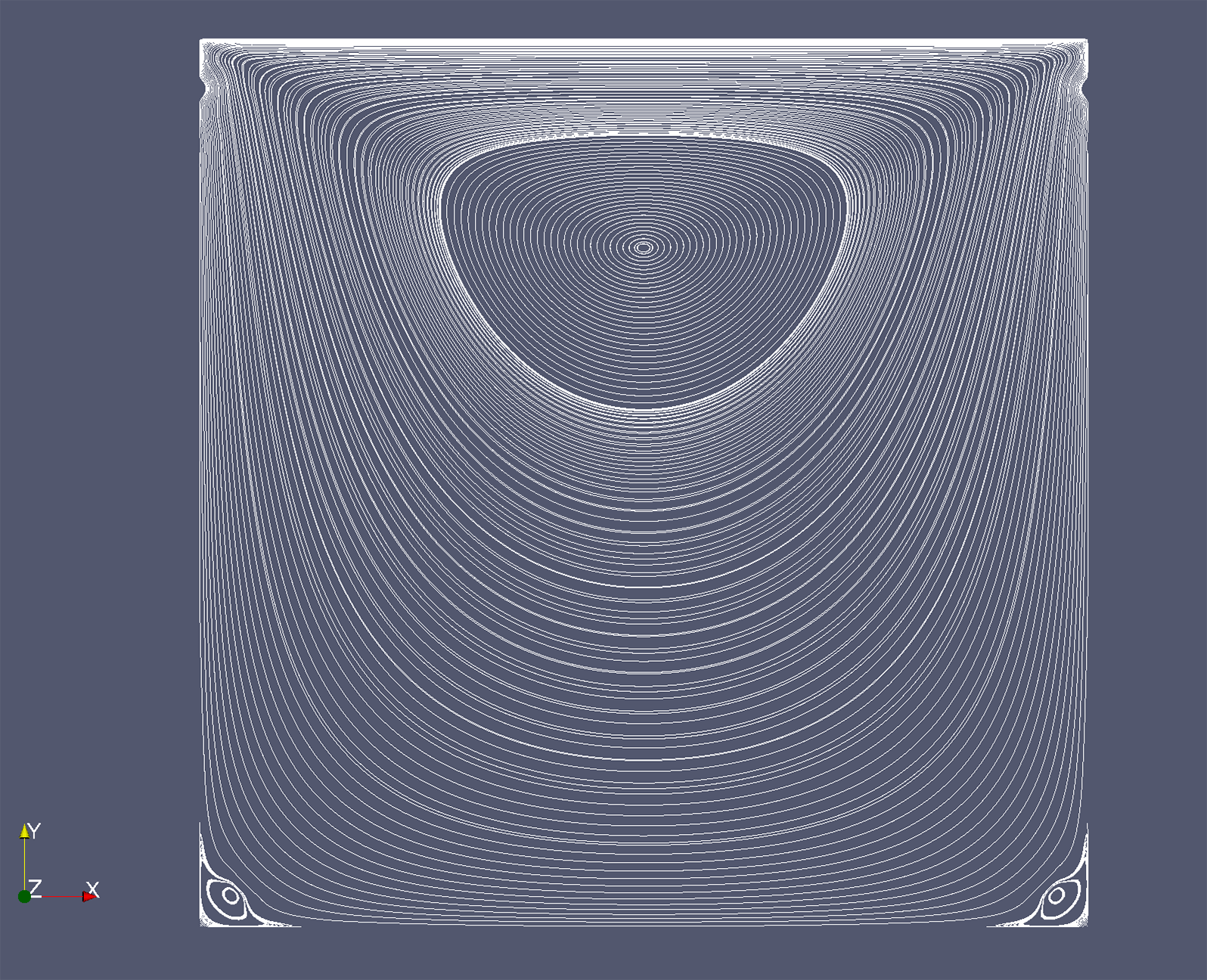}\label{f:tensorprodstreamlinesh32}}
	\qquad
	\subfigure[{$[0,0.0073] \times [0,0.0073]$, $h = 1/32$}]{\includegraphics[height=2.3in]{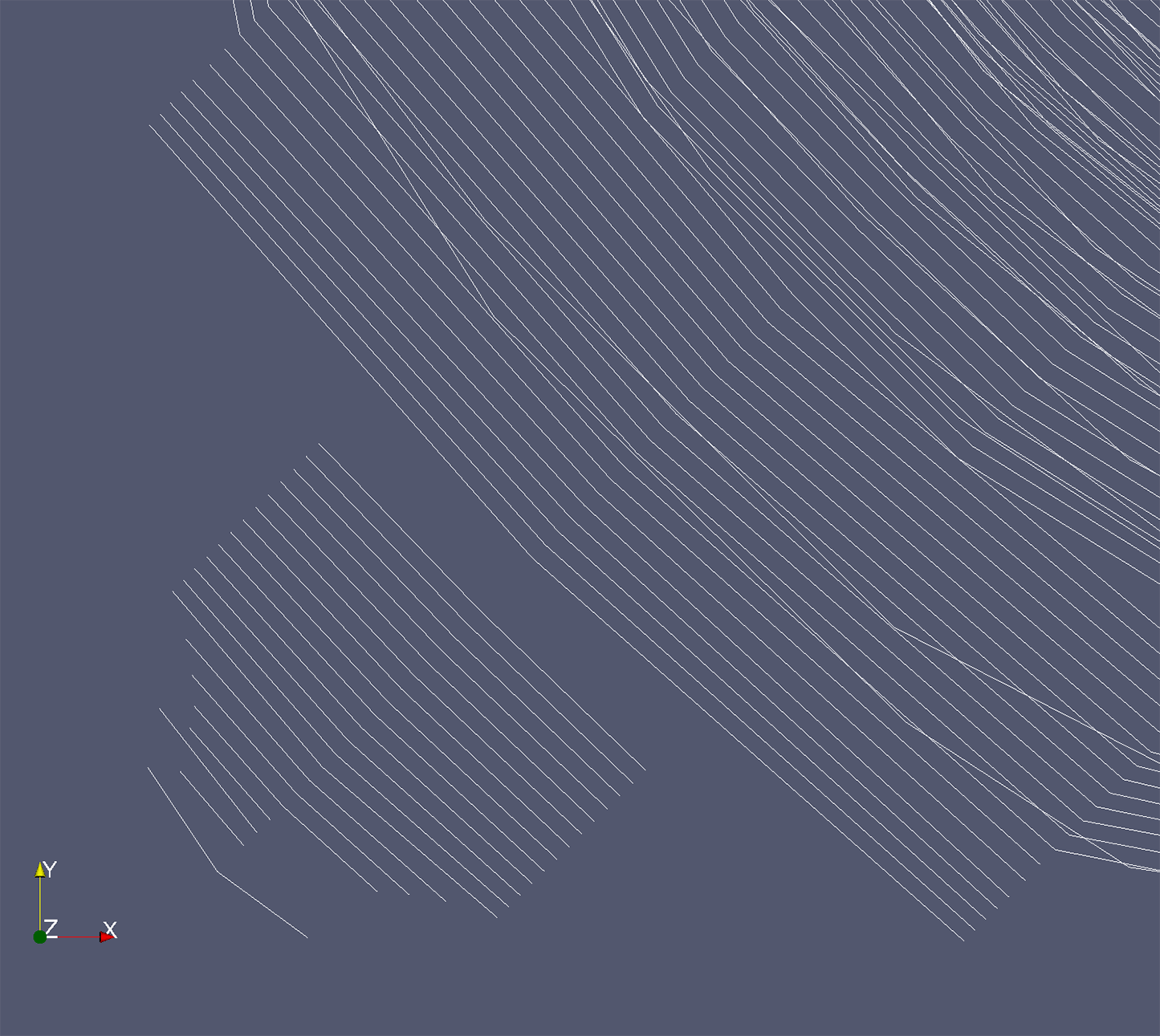}\label{f:tensorprodstreamlinesh32zoom}}
	
	\subfigure[{$[0,1] \times [0,1]$, $h = 1/128$}]{\includegraphics[height=2.3in]{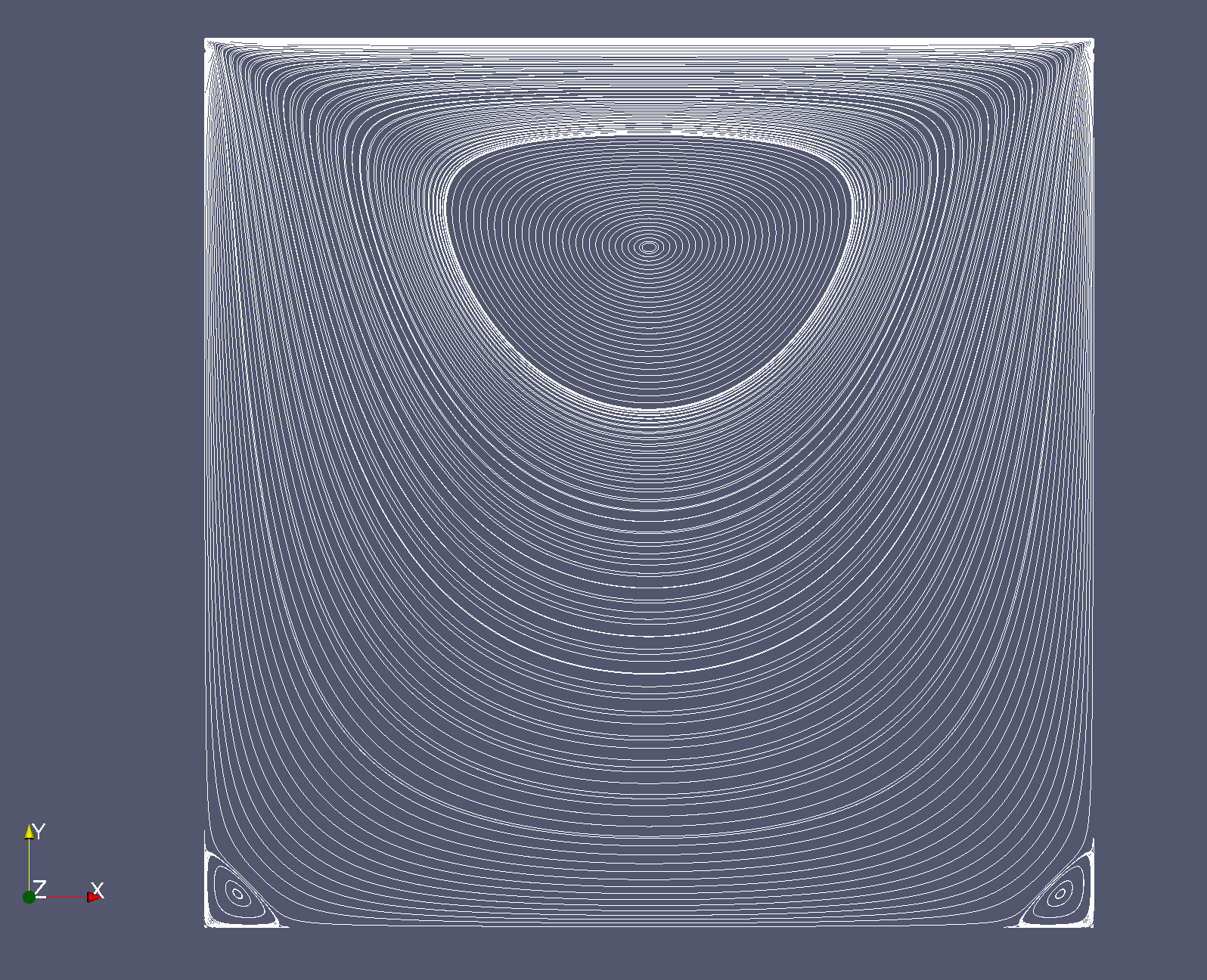}\label{f:tensorprodstreamlinesh128}}
	\qquad
	\subfigure[{$[0,0.0073] \times [0,0.0073]$, $h = 1/128$}]{\includegraphics[height=2.3in]{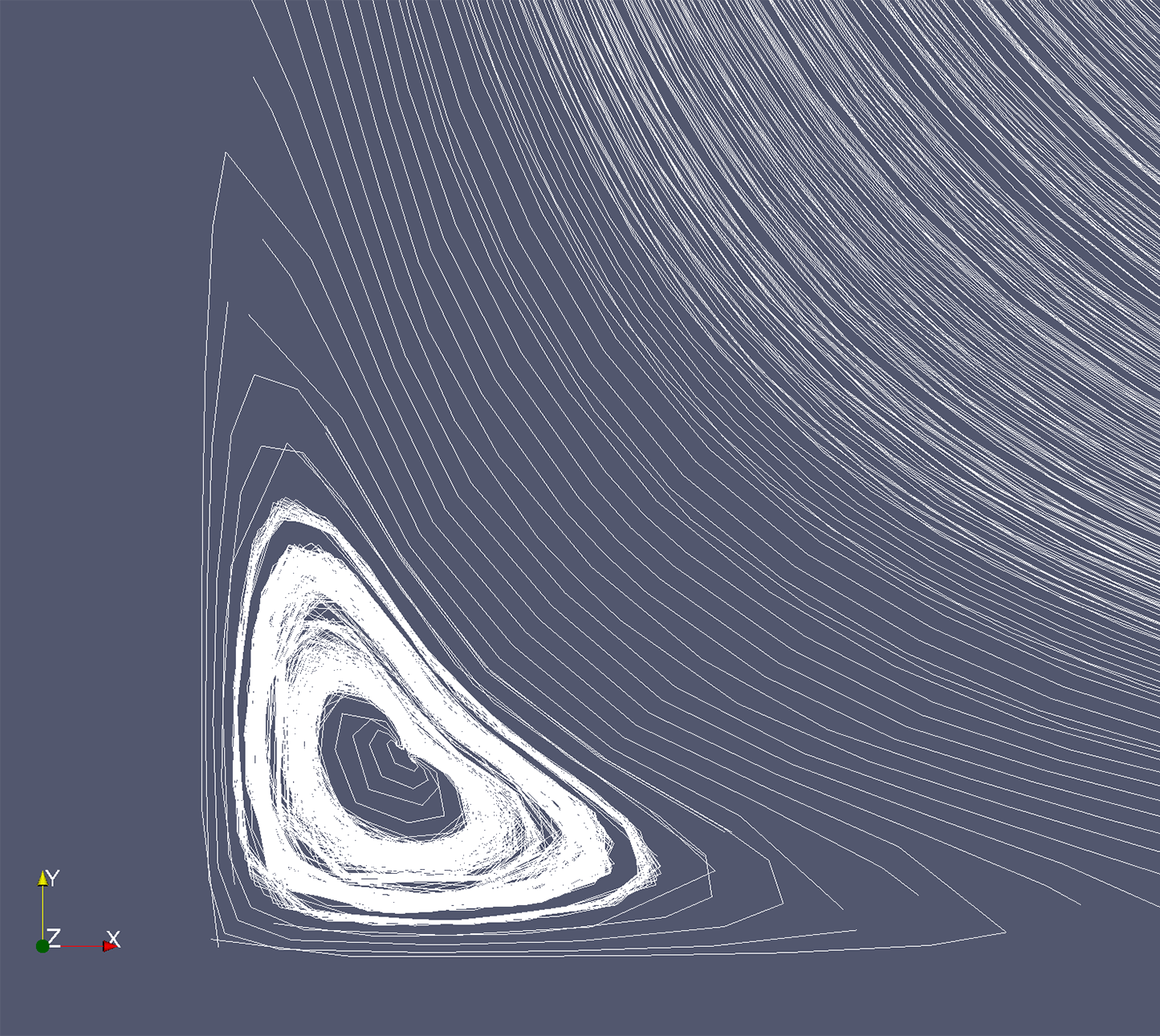}\label{f:tensorprodstreamlinesh128zoom}}

	\caption{Streamlines for the creeping lid-driven cavity problem obtained from tensor-product B-splines with $p_1 = p_2 = 2$. The first Moffatt eddies can be seen in the corners of the domain, while a second Moffatt eddy is shown by zooming in to the lower-left corner of the domain.}
	\label{f:tensorprodstreamlines}
\end{figure}

%
\begin{figure}[t]
	\centering

	\subfigure[{$[0,1] \times [0,1]$, $h = 1/32$}]{\includegraphics[height=2.3in]{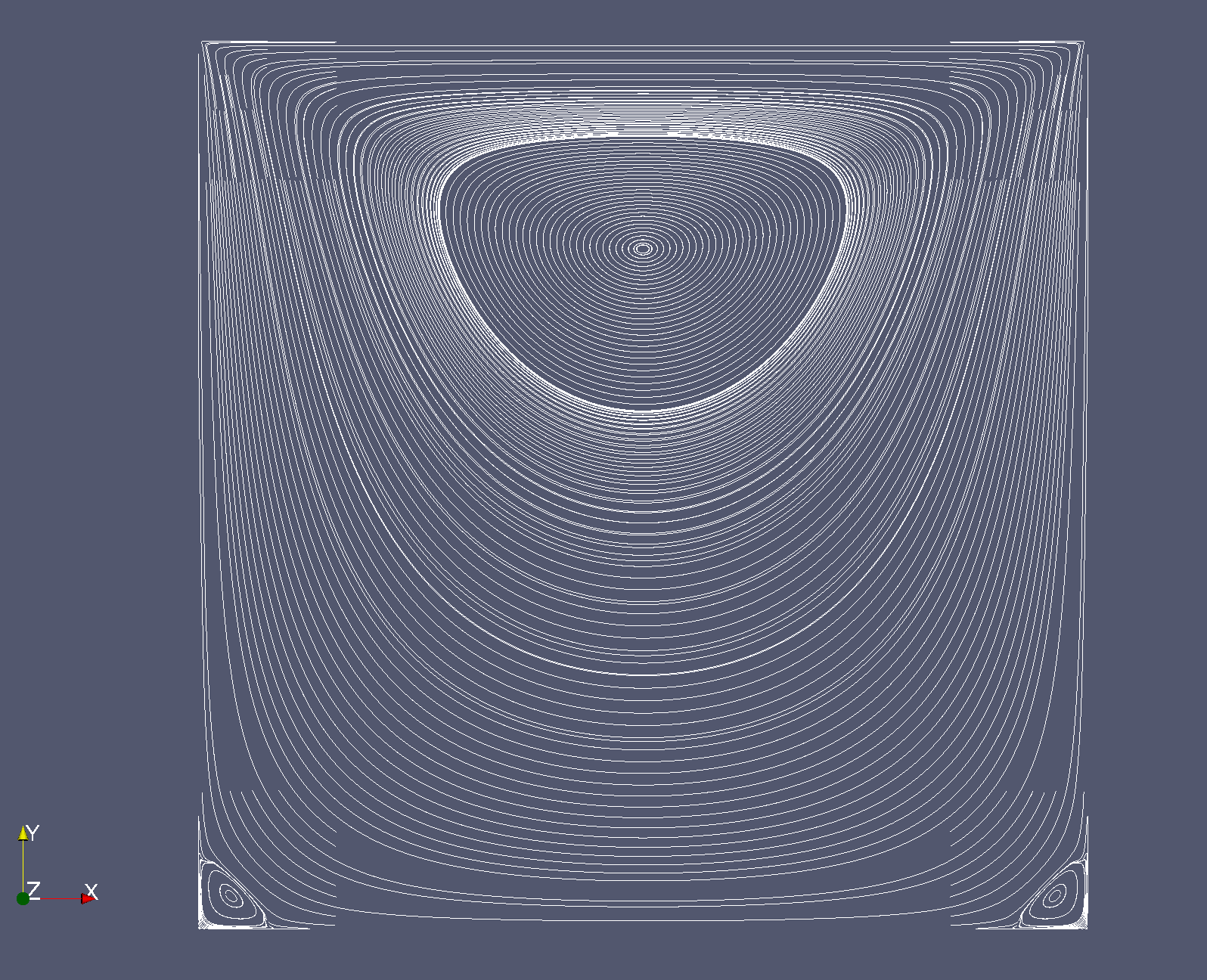}\label{f:hierarchicalstreamlinesh32}}
	\qquad
	\subfigure[{$[0,0.0073] \times [0,0.0073]$, $h = 1/32$}]{\includegraphics[height=2.3in]{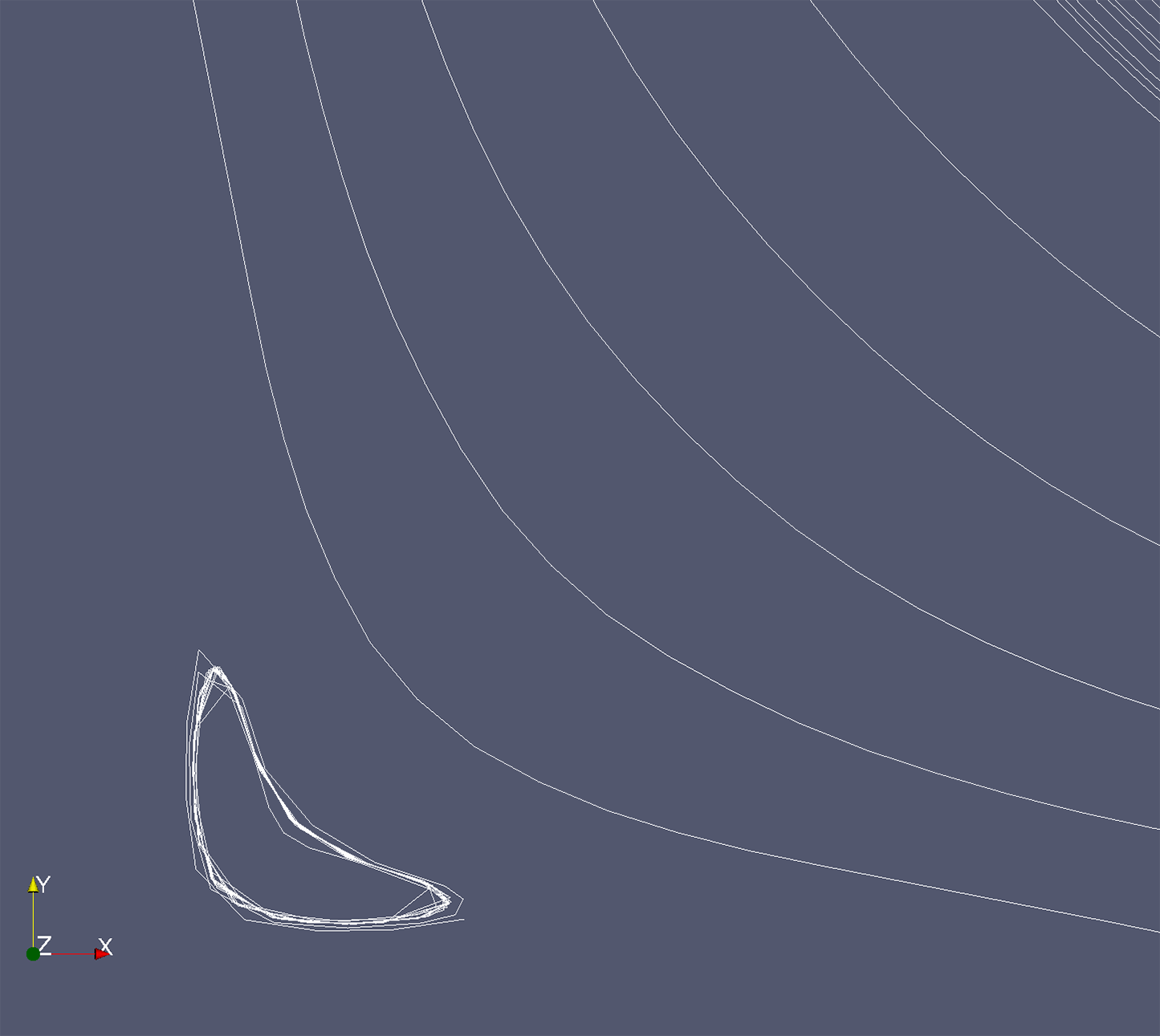}\label{f:hierarchicalstreamlinesh32zoom}}
	
	\subfigure[{$[0,1] \times [0,1]$, $h = 1/128$}]{\includegraphics[height=2.3in]{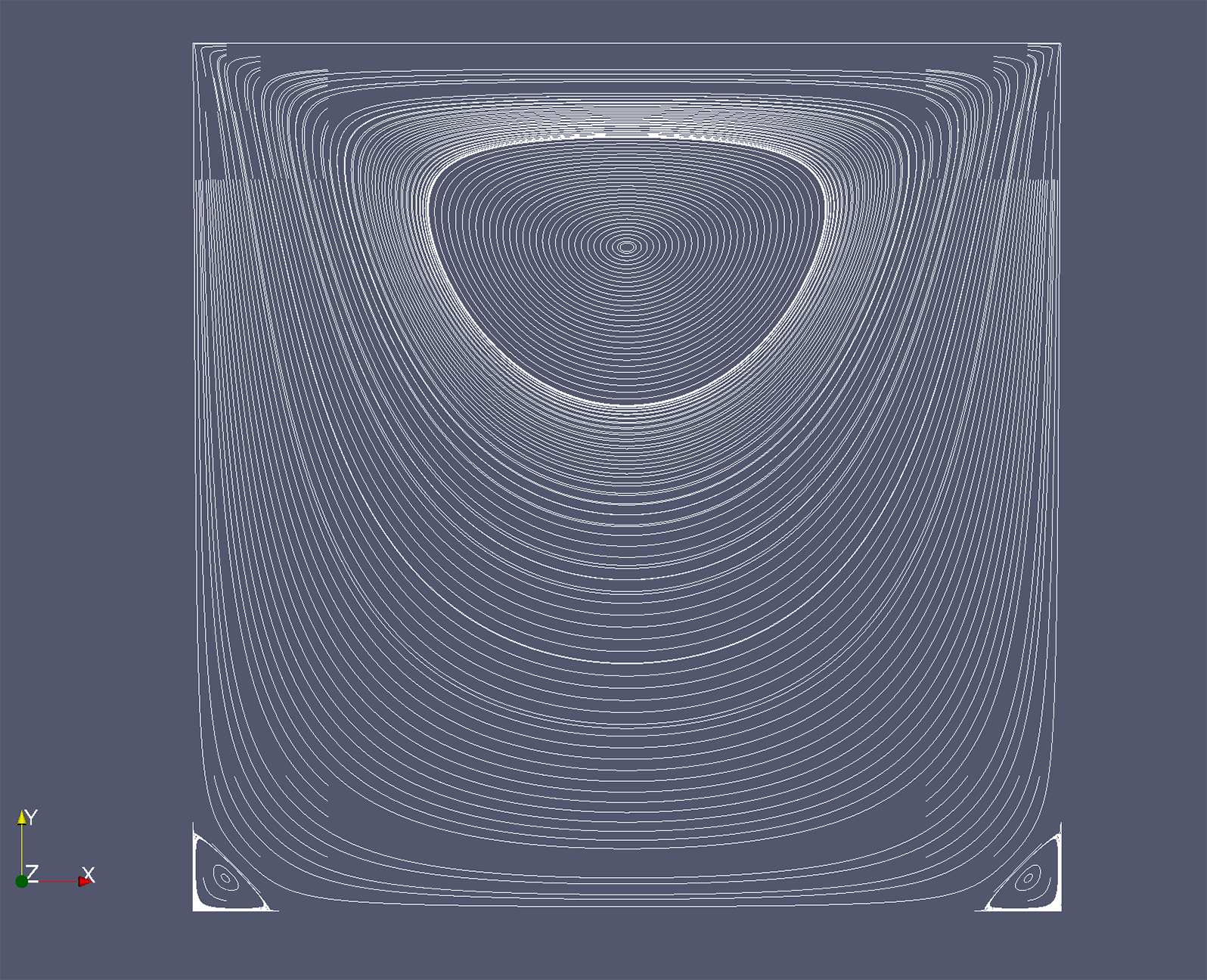}\label{f:hierarchicalstreamlinesh128}}
	\qquad
	\subfigure[{$[0,0.0073] \times [0,0.0073]$, $h = 1/128$}]{\includegraphics[height=2.3in]{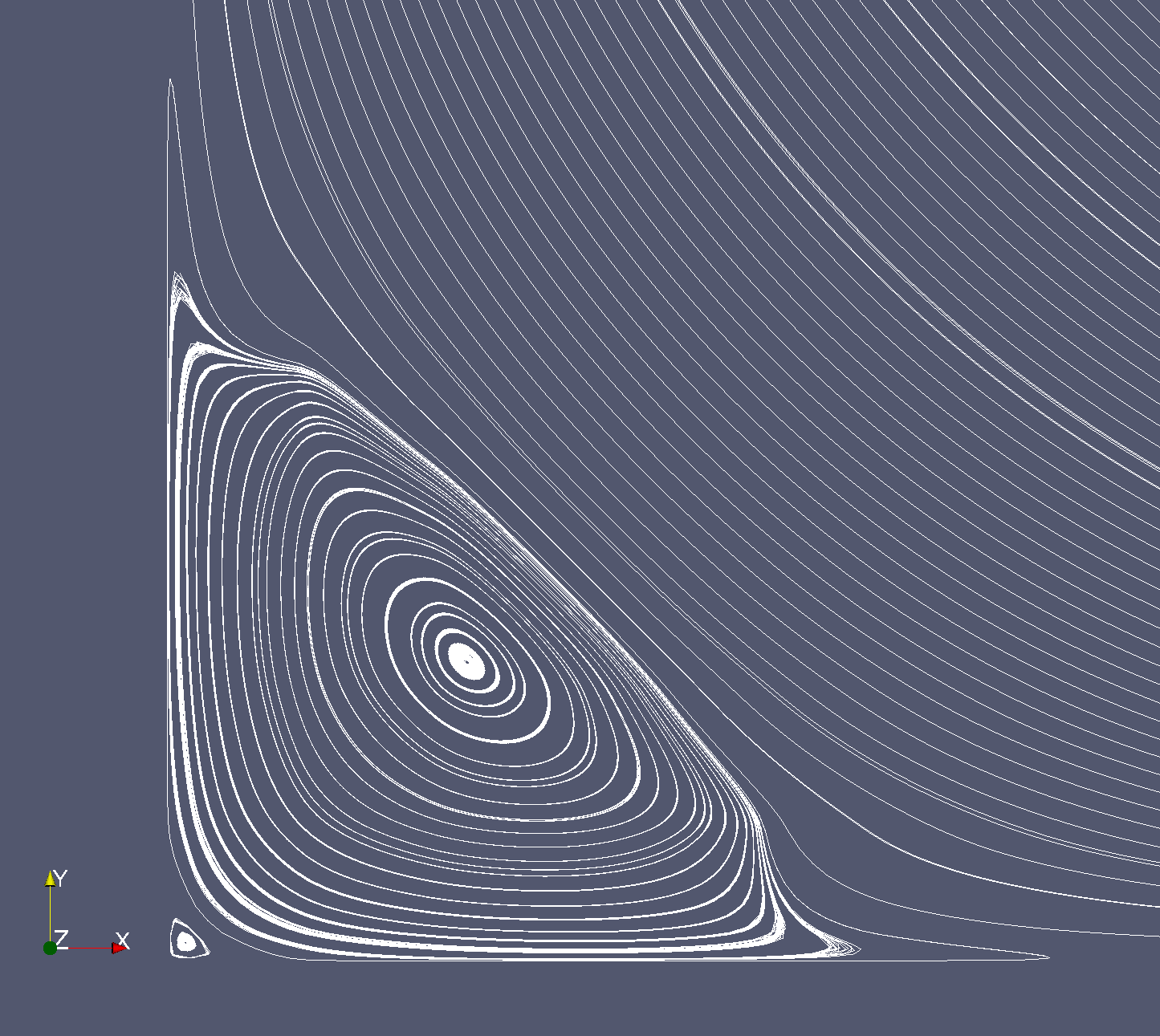}\label{f:hierarchicalstreamlinesh128zoom}}

	\caption{Streamlines for the creeping lid-driven cavity problem obtained from thrice-refined hierarchical B-splines with $p_1 = p_2 = 2$ (see Figure~\ref{f:lidcavityhierarchicalmesh}). The first Moffatt eddies can be seen in the corners of the domain, while a second and a third Moffatt eddy are shown by zooming in to the lower-left corner of the domain.}
	\label{f:hierarchicalstreamlines}
\end{figure}

To examine how well hierarchical B-splines and tensor-product B-splines resolve the corner singularities in the lid-driven cavity flow problem, we also compared the vorticity solution $(\omega = \textnormal{curl} \hspace{3pt} \mathbf{u}_{h})$ obtained by these approaches near the upper right corner of the cavity with the highly accurate pseudospectral results obtained in \cite{BoPe98}.  The results of this study are shown in Table~\ref{t:StokesLidDriven}.  From these results, it is apparent that hierarchical B-splines are much more effective in capturing the corner singularities than tensor-product B-splines.  For example, when $h = 1/256$ and $p_1 = p_2 = 4$, the hierarchical B-spline results match the pseudospectral results up to six significant digits.  This is particularly impressive when taking into consideration that the pseudospectral results were obtained using a subtraction of the leading terms of the asymptotic solution of the Stokes equations in order to exactly represent the corner singularities.

%
\begin{table}
\caption{Vorticity values at the point $\mathbf{x} = (0,0.95)$ and computational expense associated with tensor-product B-splines and hierarchical B-splines as applied to the creeping lid-driven cavity flow problem.}\label{t:StokesLidDriven}
\centering
Polynomial Degree $p_1 = p_2 = 2$ \\
\begin{tabular}{ l c c c}
\hline
Method  & $\omega$ & No. Elements & No. Unknowns \\
\hline
Tensor product B-splines, $h = 1/32$	 & 18.07580 & 1,024	 & 3,333 \\
Tensor product B-splines, $h = 1/64$ 	 & 19.18682 & 4,096	 & 12,805\\
Tensor product B-splines, $h = 1/128$ & 23.47947 & 16,384 & 50,181\\
Tensor product B-splines, $h = 1/256$ & 25.42558 & 65,536 & 198,661\\
\hline
Hierarchical B-splines, $h = 1/32$ 	& 23.59458 & 1,924 	& 6,033 \\
Hierarchical B-splines, $h = 1/64$ 	& 25.47702 & 7,696 	& 23,605 \\
Hierarchical B-splines, $h = 1/128$ 	& 26.38816 & 30,784 & 93,381 \\
Hierarchical B-splines, $h = 1/256$ 	& 26.83601 & 123,136 & 371,461 \\
\hline
Pseudospectral (Ref. \cite{BoPe98}) & 27.27901 & --- & --- \\
\hline
\end{tabular}

\vspace{12pt}
Polynomial Degree $p_1 = p_2 = 3$ \\
\begin{tabular}{ l c c c}
\hline
Method  & $\omega$ & No. Elements & No. Unknowns \\
\hline
Tensor product B-splines, $h = 1/32$	 & 33.27731 & 1,024	 & 3,536 \\
Tensor product B-splines, $h = 1/64$ 	 & 35.01708 & 4,096	 & 13,200 \\
Tensor product B-splines, $h = 1/128$ & 25.84879 & 16,384 & 50,960 \\
Tensor product B-splines, $h = 1/256$ & 27.34235 & 65,536 & 200,208\\
\hline
Hierarchical B-splines, $h = 1/32$ 	& 27.54476 & 1,924 	& 6,236 \\
Hierarchical B-splines, $h = 1/64$ 	& 27.33604 & 7,696 	& 24,000\\
Hierarchical B-splines, $h = 1/128$ 	& 27.29035 & 30,784 & 94,160\\
Hierarchical B-splines, $h = 1/256$ 	& 27.28284 & 123,136 & 373,008\\
\hline
Pseudospectral (Ref. \cite{BoPe98}) & 27.27901 & --- & --- \\
\hline
\end{tabular}

\vspace{12pt}
Polynomial Degree $p_1 = p_2 = 4$ \\
\begin{tabular}{ l c c c}
\hline
Method  & $\omega$ & No. Elements & No. Unknowns \\
\hline
Tensor product B-splines, $h = 1/32$	 & 22.52289 & 1,024	 & 3,745 \\
Tensor product B-splines, $h = 1/64$ 	 & 30.29182 & 4,096	 & 13,601 \\
Tensor product B-splines, $h = 1/128$ & 29.32455 & 16,384 & 51,745 \\
Tensor product B-splines, $h = 1/256$ & 27.64269 & 65,536 & 201,761 \\
\hline
Hierarchical B-splines, $h = 1/32$ 	& 27.25594 & 1,924 	& 6,445\\
Hierarchical B-splines, $h = 1/64$ 	& 27.28393 & 7,696 	& 24,401\\
Hierarchical B-splines, $h = 1/128$ 	& 27.27935 & 30,784 & 94,945\\
Hierarchical B-splines, $h = 1/256$ 	& 27.27905 & 123,136 & 374,561\\
\hline
Pseudospectral (Ref. \cite{BoPe98}) & 27.27901 & --- & --- \\
\hline
\end{tabular}

\end{table}

\section{Conclusions and Future Work} \label{Sec:conclusion}
In this paper, we introduced the hierarchical B-spline complex of isogeometric discrete differential forms for arbitrary spatial dimension.
By appealing to abstract cohomology theory, we derived a sufficient and necessary condition guaranteeing exactness of the hierarchical B-spline complex for arbitrary spatial dimension, and we presented an intuitive interpretation of this condition based on the topologies of the Greville grids associated with various levels of a hierarchical B-spline discretization.  We further derived a set of local, easy-to-compute, and sufficient exactness conditions for the two-dimensional setting, and we demonstrated via example that the hierarchical B-spline complex yields stable approximations of both the Maxwell eigenproblem and Stokes problem provided these local exactness conditions are satisfied.  We concluded by providing numerical results in the context of electromagnetics and creeping flow.

In future work, we plan to continue the analysis of the hierarchical B-spline complex.  First, we plan to extend our local, easy-to-compute, and sufficient exactness conditions to the three-dimensional setting.  We anticipate that this will be an exceptionally technical exercise if we follow the method of proof employed in this paper.  However, we expect that a more intelligent use of cohomology theory may yield a much simpler proof that also extends to arbitrary spatial dimension.  Second, we seek to construct commuting projection operators which make the spaces of hierarchical B-spline complex conform to a commutative de Rham diagram.  It is expected that not all hierarchical B-spline complexes will admit such commuting projection operators, though based on the results presented here, we believe that hierarchical B-spline complexes satisfying the local exactness conditions presented in this paper will.  This belief is further strengthened by the fact that hierarchical B-spline spaces satisfying our local exactness conditions are complete in that they contain all piecewise polynomial functions on the hierarchical grid with the smoothness specified by the grid and underlying polynomial degree \cite{Mokris2014}.

We also plan to further explore application of the hierarchical B-spline complex in the adaptive isogeometric solution of vector field problems such as those arising in electromagnetics and incompressible fluid flow.  In this direction, we first seek to develop refinement algorithms which yield hierarchical B-spline complexes satisfying the local exactness conditions presented in this paper with a minimal level of added degrees of freedom.   We additionally seek to develop appropriate error estimators for use in vector field problems, drawing inspiration from finite element exterior calculus \cite{Demlow14}.

\section*{Acknowledgements}
J.A. Evans and M.A. Scott were partially supported by the Air Force Office of Scientific Research under Grant No. FA9550-14-1-0113. J.A. Evans was also partially supported by the Defense Advanced Research Projects Agency under Grant No. HR0011-17-2-0022. K. Shepherd was supported by the National Science Foundation Graduate Research Fellowship under Grant No. DGE-1610403. Any opinion, findings, and conclusions or recommendations expressed in this material are those of the authors and do not necessarily reflect the views of the National Science Foundation. R. V\'azquez was partially supported by the European Research Council through the FP7 ERC Consolidator Grant no.~616563 HIGEOM (PI: Giancarlo Sangalli) and the H2020 ERC Advanced Grant no.~694515 CHANGE (PI: Annalisa Buffa). 

\begin{appendices}


\section{Proof of Theorem~\ref{th:homology}}\label{sec:appendix}

Without loss of generality, we assume that the polynomial degree is the same in both parametric directions (i.e., $p = p_1 = p_2$) as this greatly simplifies our exposition. However, the subsequent analysis also holds in the mixed polynomial degree setting.

Let us begin by recalling notation from the main text. We denote the B\'ezier meshes of two consecutive levels $\ell$ and $\ell+1$ by $\M_{\ell}$ and $\M_{\ell+1}$, respectively, and we denote the Greville subgrids of two consecutive levels by $\MG_{\ell}$ and $\MG_{\ell+1}$.  For the subdomain $\Omega_{\ell+1}$ appearing in the hierarchical B-spline construction, we denote the corresponding B\'ezier submeshes associated with levels $\ell$ and $\ell+1$ by $\M_{\ell,{\ell+1}} \subset \M_\ell$ and $\M_{\ell+1,{\ell+1}} \subset \M_{\ell+1}$. We further denote the Greville subgrids associated with functions whose support is completely contained in $\Omega_{\ell+1}$ for levels $\ell$ and $\ell+1$ by $\MG_{\ell,\ell+1} \subset \MG_\ell$ and $\MG_{\ell+1,\ell+1} \subset \MG_{\ell+1}$. With some abuse of notation, the subdomains occupied by the Greville subgrids are also denoted by $\MG_{\ell,\ell+1}$ and $\MG_{\ell+1,\ell+1}$.

We now establish some additional definitions that will be required in our subsequent analysis.  Given a B\'ezier element $Q \in \M_{\ell}$, we define the extended support\footnote{This has been traditionally denoted by $\tilde Q$ in the IGA literature.}, denoted as $\M_{\ell,\tilde Q} \subset \M_\ell$, to be the union of the support of basis functions in ${\cal B}^k_\ell$ that do not vanish in $Q$, for $k = 0, 1, 2$. We define a corresponding Greville subgrid, denoted as $\MG_{\ell,\tilde Q} \subset \MG_\ell$, to be that formed by the vertices, edges, and cells of the Greville grid $\MG_\ell$ associated with basis functions that do not vanish in $Q$. 
Finally, we define an additional Greville subgrid $\MG'_{\ell,\tilde Q}$ to be the extension of $\MG_{\ell,\tilde Q}$ obtained by adding one cell in each direction (see Figure~\ref{fig:qtilde}), noting that when $\MG_{\ell,\tilde Q}$ is adjacent to the boundary the extension in the corresponding direction is not necessary.  It should be noted that the Greville subgrids $\MG_{\ell,\tilde Q}$ and $\MG'_{\ell,\tilde Q}$ are not related to the previously defined Greville subgrids $\MG_{\ell,\ell+1}$ and $\MG_{\ell+1,\ell+1}$.


\begin{figure}[t]
\begin{subfigure}[Element $Q$ and its extended support $\M_{\ell,\tilde Q}$.]{
\includegraphics[width=0.48\textwidth]{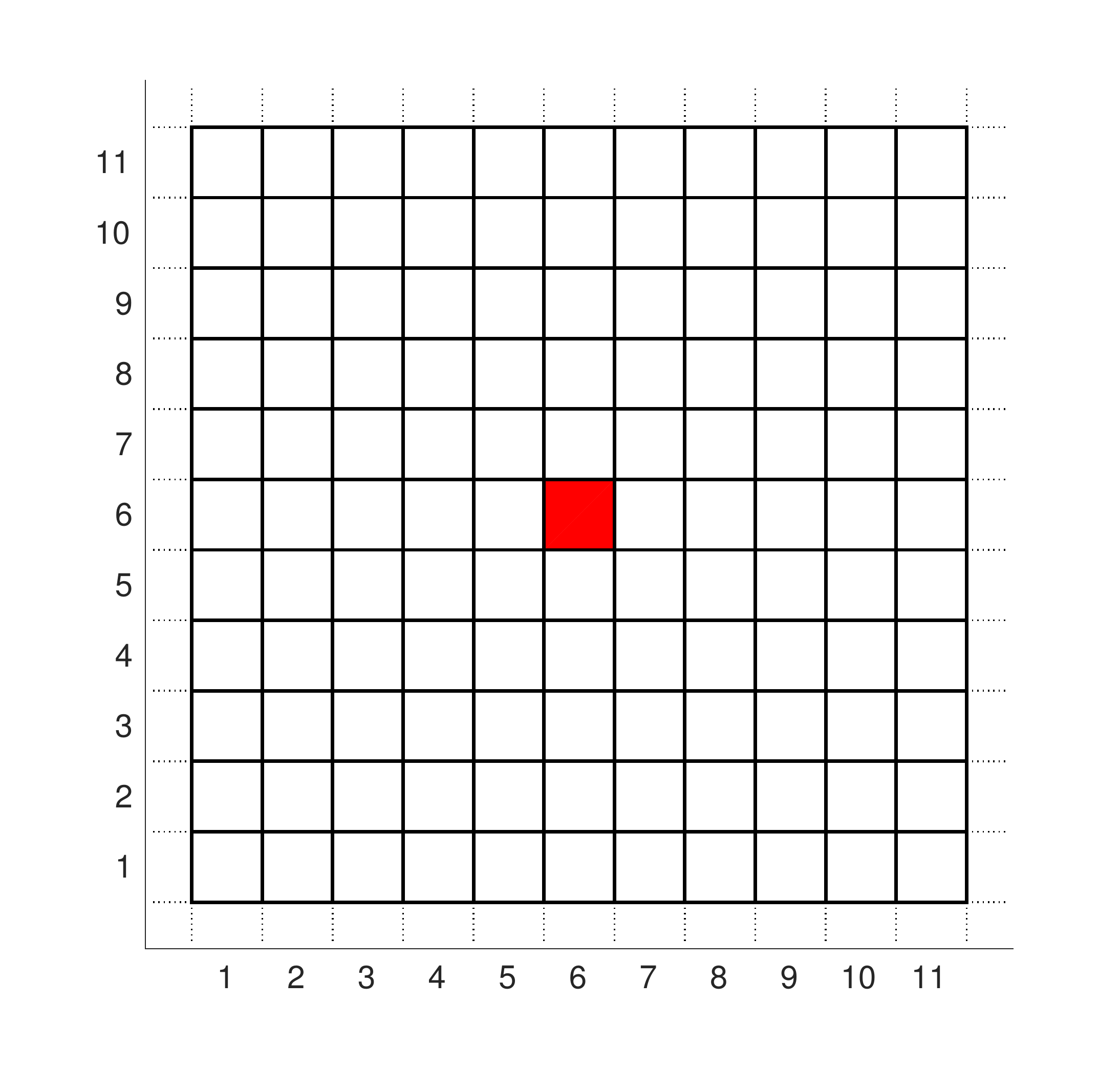}
}
\end{subfigure}
\begin{subfigure}[Greville subgrids $\MG_{\ell,\tilde Q}$ (solid) and $\MG'_{\ell,\tilde Q}$ (solid and dashed).]{
\includegraphics[width=0.48\textwidth]{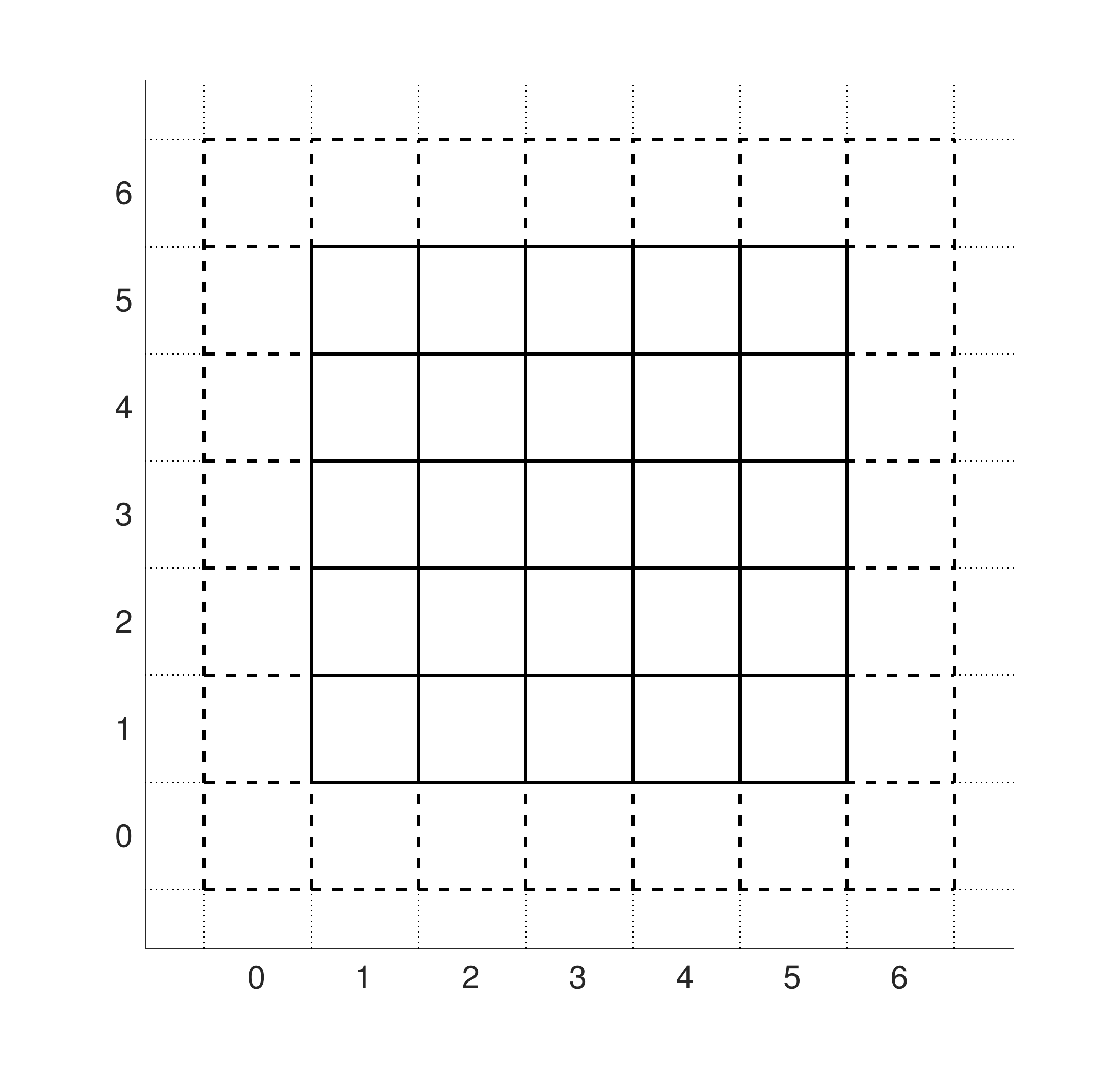}
}
\end{subfigure}
\caption{Extended support for an element in the B\'ezier mesh and the corresponding Greville subgrids for degree $p=5$.} \label{fig:qtilde}
\end{figure}

The following lemma states two geometrical results for B\'ezier and Greville grids.
\begin{lemma} \label{lemma:one-cell} The following statements hold:
\begin{enumerate}[i)]
\item For any Greville cell $K \in \MG_{\ell}$, with $\beta_K \in {\cal B}^2_\ell$ the corresponding basis function, it holds that $K \subset \supp(\beta_K)$.
\item For any B\'ezier element $Q \in \M_{\ell}$, it holds that $Q \subset \MG_{\ell,\tilde Q}$.
\end{enumerate}
\end{lemma}
\begin{proof}
It is sufficient to prove the result in the univariate setting of Section~\ref{sec:univariate} since the multivariate case is proved by Cartesian product.
\begin{enumerate}[i)]
\item Given two consecutive Greville points $\gamma_i = \frac{\xi_i + \ldots + \xi_{i+p-1}}{p} \ge \xi_i$ and $\gamma_{i+1} = \frac{\xi_{i+1} + \ldots + \xi_{i+p}}{p} \le \xi_{i+p}$, the basis function associated to the edge between them is $D_{i,p-1}$, which is supported in the interval $(\xi_i,\xi_{i+p})$.  Since $(\xi_i,\xi_{i+p}) \subseteq (\gamma_i,\gamma_{i+p})$, the desired result follows.
\item Given the non-empty element $(\xi_i, \xi_{i+1})$, its extended support is the interval $(\xi_{i-p}, \xi_{i+p+1})$, and the first and last basis functions that do not vanish in the element are associated to the Greville points $\gamma_{i-p+1} = \frac{\xi_{i-p+1} + \ldots + \xi_i}{p} \le \xi_i$ and $\gamma_{i+1} = \frac{\xi_{i+1} + \ldots + \xi_{i+p}}{p} \ge \xi_{i+1}$.
\end{enumerate}
\end{proof}

In what follows, we index the elements in the extended support $\M_{\ell,\tilde Q}$ from $1$ to $2p+1$ in each direction (including empty elements), the cells in the Greville subgrid $\MG_{\ell,\tilde Q}$ from 1 to $p$ in each direction, and the cells in $\MG'_{\ell,\tilde Q}$ from 0 to $p+1$ in each direction.  This is graphically depicted in Figure~\ref{fig:qtilde}. The support of the function in ${\cal B}^2_\ell$ corresponding to the cell $K_{i,j}$ in $\MG'_{\ell,\tilde Q}$ is given by the elements $\{Q_{k_1,k_2} \in \M_{\ell,\tilde Q} : i+1 \le k_1 \le i+p, j+1 \le k_2 \le j+p\}$. The following lemma states that, if two cells of $\MG_{\ell,\tilde Q}$ in the same row (column) belong to $\MG_{\ell,\ell+1}$, all the cells of the row (column) between them also belong to $\MG_{\ell,\ell+1}$.  That is, the supports of all the functions associated with cells in the row (column) are contained in $\Omega_{\ell+1}$. The proof, which we omit, is an immediate consequence of the support characterization above.

\begin{lemma} \label{lemma:supp}
Let $Q \in \M_{\ell,\ell+1}$ and further let $K_{i,j} \in \MG_{\ell,\tilde Q} \cap \MG_{\ell,\ell+1}$ where $i$ and $j$ are local indices as discussed above. If $K_{i,j_0} \in \MG'_{\ell,\tilde Q} \cap \MG_{\ell,\ell+1}$, then $K_{i,k} \in \MG_{\ell,\ell+1}$ for $\min\{j,j_0\} \le k \le \max\{j,j_0\}$. Analogously, if $K_{i_0,j} \in \MG'_{\ell,\tilde Q} \cap \MG_{\ell,\ell+1}$, then $K_{k,j} \in \MG_{\ell,\ell+1}$ for $\min\{i,i_0\} \le k \le \max\{i,i_0\}$.
\end{lemma}

With the previous definitions and results, we are now in a position to prove Theorem~\ref{th:homology}. The principal idea is to establish one-to-one correspondences between the connected components and holes of the subdomain $\Omega_{\ell+1}$ and those of the regions defined by the Greville subgrids $\MG_{\ell,\ell+1}$ and $\MG_{\ell+1,\ell+1}$. We begin with the connected components.

\begin{proposition}\label{prop:cc}
The following statements hold:
\begin{enumerate}[i)]
\item Every connected component of $\MG_{\ell,\ell+1}$ is contained in a connected component of $\MG_{\ell+1,\ell+1}$.
\item Every connected component of $\MG_{\ell+1,\ell+1}$ is contained in a connected component of $\Omega_{\ell+1}$.
\item Under Assumptions~\ref{ass:support} and~\ref{ass:overlap}, every connected component of $\Omega_{\ell+1}$ contains exactly one connected component of $\MG_{\ell,\ell+1}$.
\end{enumerate}
\end{proposition}
\begin{proof}
We prove each statement independently.
\begin{enumerate}[i)]
\item It suffices to show that each cell in $\MG_{\ell,\ell+1}$ is contained in a connected component of $\MG_{\ell+1,\ell+1}$. Let $K \in \MG_{\ell,\ell+1}$ denote a Greville cell and let $\beta_K \in {\cal B}^2_\ell$ its associated basis function.  By the nestedness of spline refinement, we note that $\beta_K$ can be expressed in terms of basis functions associated with level $\ell + 1$.  That is, we can write $\beta_K$ uniquely as: $$\beta_K = \sum_{\beta' \in {\cal B}^2_{\ell+1}} c_{\beta'} \beta'$$
where $c_{\beta'} \in \mathbb{R}$ for all $\beta' \in {\cal B}^2_{\ell+1}$.  We refer to those functions $\beta' \in {\cal B}^2_{\ell+1}$ for which $c_{\beta'} \neq 0$ as the children of $\beta_K$.  As the B-spline basis functions in ${\cal B}^2_{\ell+1}$ are non-negative and form a partition of unity, it holds that the coefficients $c_{\beta'}$ are all non-negative, and the coefficients associated with the children are all positive.  It follows that the children of $\beta_K$ are fully supported within the support of $\beta_K$.  Since $\beta_K$ is completely supported in $\Omega_{\ell+1}$, all the children of $\beta_K$ are also supported in $\Omega_{\ell+1}$. Consequently, the Greville cells associated with the children of $\beta_K$ belong to $\MG_{\ell+1,\ell+1}$. It is easily shown that the closure of the union of the Greville cells associated with the children of $\beta_K$ contains $K$ and is connected, thus the desired result follows.

\item The result is a trivial consequence of point i) in Lemma~\ref{lemma:one-cell}.
\item The arguments of the proof work for each connected component of $\Omega_{\ell+1}$, so we can assume without loss of generality that $\Omega_{\ell+1}$ is connected. The existence of a connected component of $\MG_{\ell,\ell+1}$ contained in $\Omega_{\ell+1}$ is an immediate consequence of Assumption~\ref{ass:support} and the previous two points of this proposition. The proof of uniqueness, instead, is rather technical.

Let us assume that uniqueness does not hold and, without loss of generality, that there are only two connected components of $\MG_{\ell,\ell+1}$ which we denote by $C_1$ and $C_2$. With some abuse of notation, we denote by $\supp(C_k)$ the union of the supports of the basis functions associated to the entities (vertices, edges, cells) in $C_k$ and by
 $\supp (\overline{C_k})$ the union of the supports considering also the entities on the boundary of $C_k$.  By construction, $\supp (\overline{C_k})$ contains one more element in each direction than $\supp(C_k)$.
By Assumption~\ref{ass:support}, we know that $\Omega_{\ell+1}$ is covered by $\supp(C_1) \cup \supp(C_2)$, and since $\supp(C_1) \cup \supp(C_2)$ is connected, there is at least one B\'ezier element $Q \subset \supp(C_1) \cap \supp(\overline{C_2})$\footnote{With more connected components, for each $C_i$ there exists a $C_j$ such that the intersection in one element holds true.}. As a consequence, there exists at least one Greville cell $K_1 \subset C_1$ such that $K_1 \subset \MG_{\ell,\tilde Q}$ and another Greville cell $K_2 \subset C_2$ such that $K_2 \subset \MG'_{\ell,\tilde Q}$. We will prove that Lemma~\ref{lemma:supp} and Assumption~\ref{ass:overlap} prevent the coexistence of two such cells, resulting in a contradiction.

To proceed forward, we establish some terminology using the index notation introduced earlier.  Let $i_{left}$ and $i_{right}$ denote the column indices associated with the leftmost and rightmost Greville cells, respectively, of $C_1 \cap \MG_{\ell,\tilde Q}$, and let $j_{bottom}$ and $j_{top}$ denote the row indices associated with the bottommost and topmost Greville cells of $C_1 \cap \MG_{\ell,\tilde Q}$.  Then, either one of the following two properties hold:
\begin{enumerate}
\item The Greville cell $K_2 \subset C_2$ satisfies $K_2 = K_{i_2,j_2}$ with $i_{left} \leq i_2 \leq i_{right}$ or $j_{bottom} \leq j_2 \leq j_{top}$.
\item The Greville cell $K_2 \subset C_2$ satisfies $K_2 = K_{i_2,j_2}$ with either $i_2 < i_{left}$ or $i_2 > i_{right}$, and also $j_2 < j_{bottom}$ or $j_2 > j_{top}$.
\end{enumerate}


Let us assume first that property (a) listed above holds, and in particular that $K_2 = K_{i_2,j_2}$ with $i_{left} \leq i_2 \leq i_{right}$, the other case being analogous.  Since $i_{left} \leq i_2 \leq i_{right}$ and $C_1 \cap \MG_{\ell,\tilde Q}$ is connected, we know that there exists a $j*$ such that $j_{bottom} \leq j* \leq j_{top}$ and $K_{i_2,j*} \in C_1 \cap \MG_{\ell,\tilde Q}$.  By Lemma~\ref{lemma:supp}, we can then form a column of Greville cells in $\MG_{l,l+1}$ from $K_2 = K_{i_2,j_2}$ to $K_{i_2,j*}$, and as a consequence, $C_1$ and $C_2$ are connected in contradiction with our earlier assumption.


Let us now assume that property (b) holds.  More specifically, let us assume that $K_2 = K_{i_2,j_2}$ with $i_2 < i_{left}$ and $j_2 < j_{bottom}$, the other three cases being analogous.  The Greville cells of $C_1 \cap \MG_{\ell,\tilde Q}$ whose row index is $i_{left}$ are the leftmost cells of $C_1 \cap \MG_{\ell,\tilde Q}$, and by Lemma~\ref{lemma:supp}, these cells form a single column.  Consequently, we may denote these cells by $K_{i_{left},j}$ with $j_* \le j \le j^*$, and in particular $j_2 < j_{bottom} \le j_*$.


Consider the cells $K_{i_{left},j}$ for $j_* \leq j \leq j^*$.  By construction, the cells to the left of these, namely $K_{i_{left}-1,j}$, cannot be in $C_1$ except when $i_{left} = 1$.  However, even in this latter case, we can use the support condition as in Lemma~\ref{lemma:supp} to show $C_1$ and $C_2$ would be connected if there were both cells $K_{0,j}$ in $C_1$ and a cell $K_2 = K_{0,j_2}$ in $C_2$, yielding a contradiction.  Therefore, we can safely assume that the cells $K_{i_{left}-1,j}$ are not in $C_1$ for $j_* \leq j \leq j^*$.  Since $K_{i_{left}-1,j_*}$ is not in $C_1$, there exists at least one B\'ezier element $Q_L \equiv Q_{i_{left},j}$, with $j_*+1 \le j \le j_*+p$, within the support of its associated basis function such that it is contained in $\M_{\ell,\tilde Q}$ but not in $\Omega_{\ell+1}$.  We may repeat the same argument to show that $K_{i_{left},j_*-1}$ is not in $C_1$, and there is another element $Q_D \equiv Q_{i,j_*}$, with $i_{left}+1 \le i \le i_{left}+p$, within the support of the basis function associated with $K_{i_{left},j_*-1}$ such that it is in $\M_{\ell,\tilde Q}$ but not in $\Omega_{\ell+1}$. We have graphically depicted the previous steps of the proof in Figure~\ref{fig:point-i)} for the worst case configuration, in which the distance between the cells $K_{i_{left},j_*}$ and $K_{i_2,j_2}$ inside is the maximum possible.

\begin{figure}[ht]
\includegraphics[width=0.49\textwidth,trim=1cm 1cm 1cm 1cm, clip]{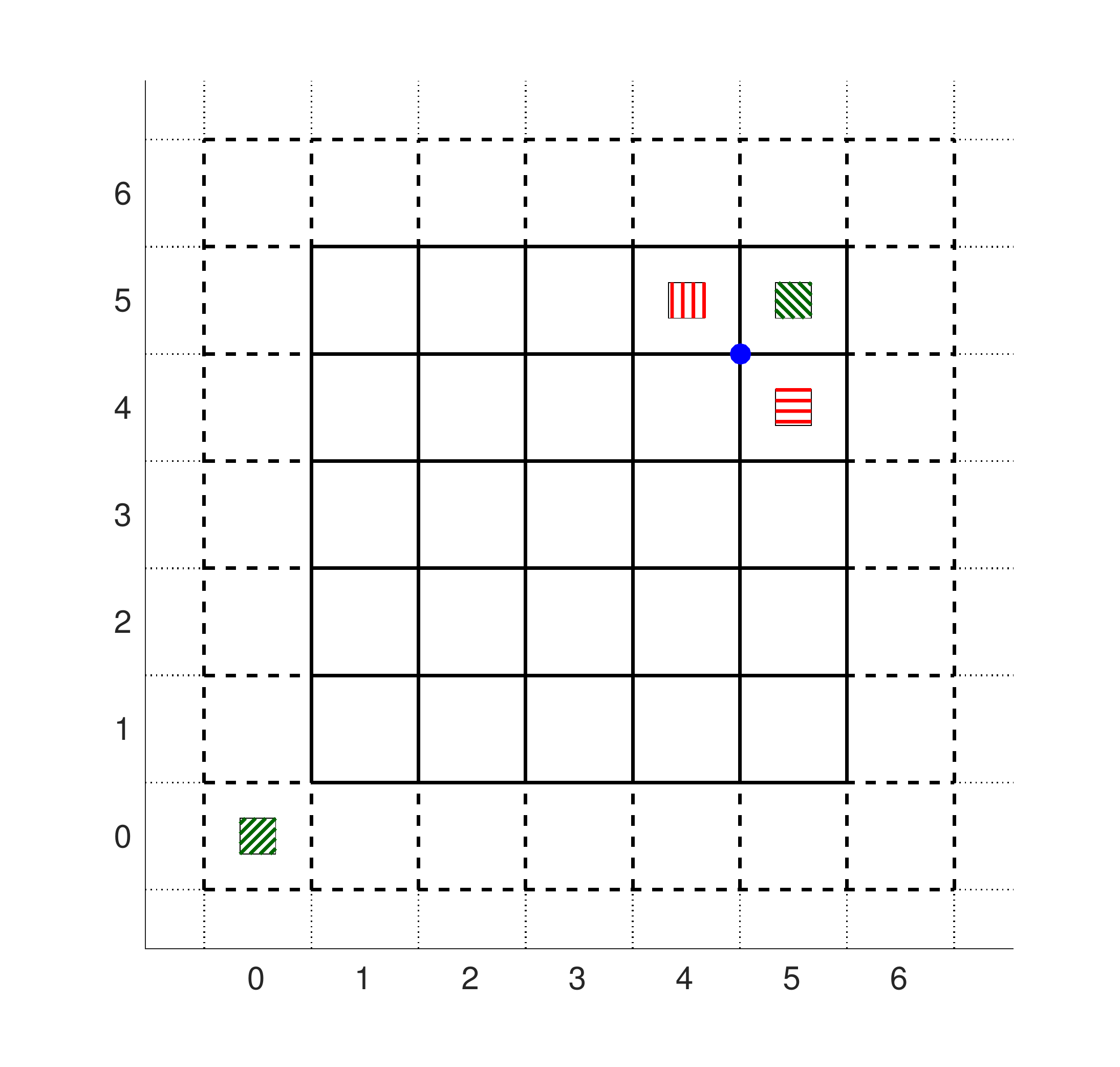}
\includegraphics[width=0.49\textwidth,trim=1cm 1cm 1cm 1cm, clip]{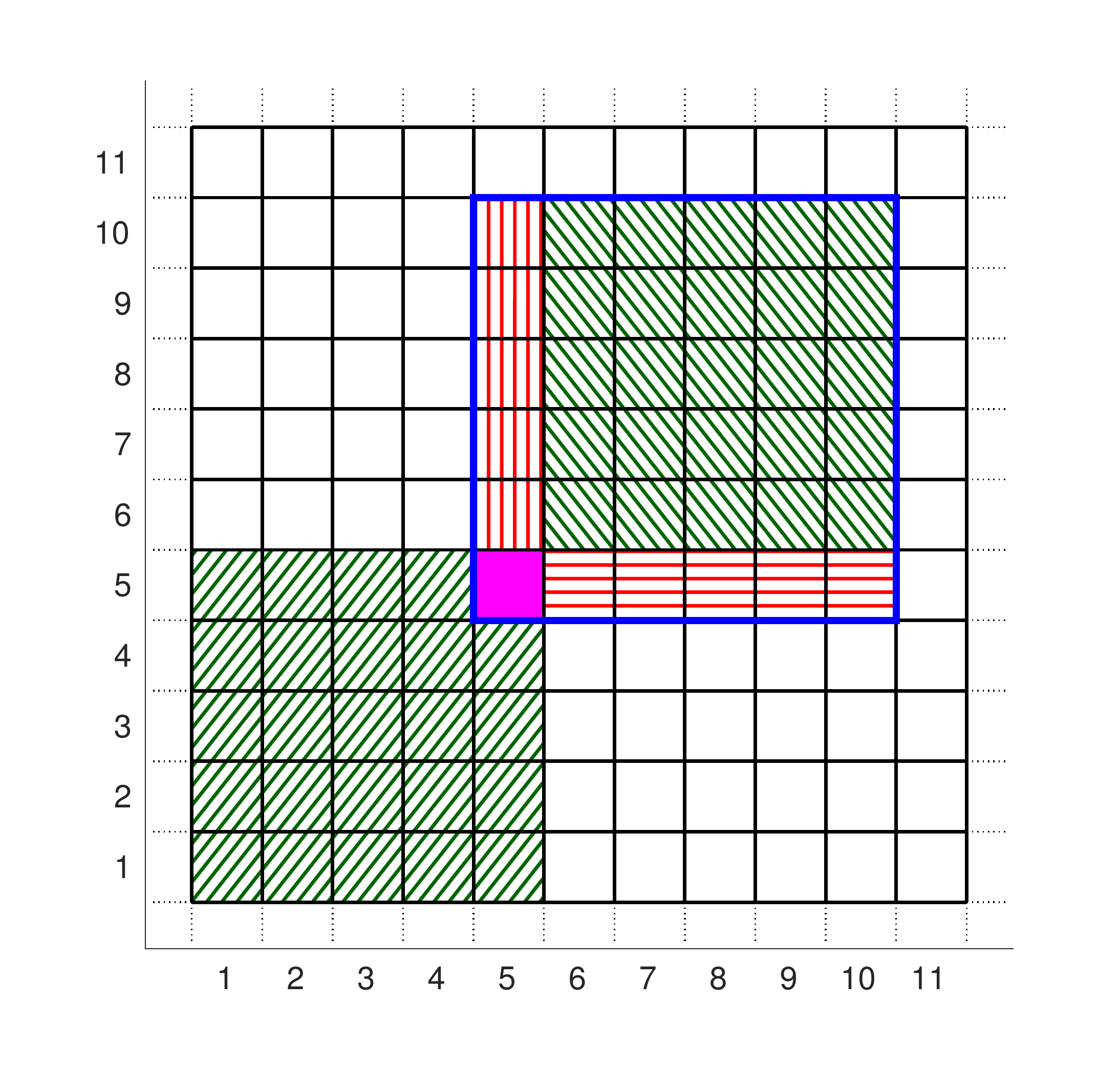}
\caption{Graphical representation of the proof of Proposition~\ref{prop:cc}, with a configuration corresponding to degree $p=5$, as in Figure~\ref{fig:qtilde}. The Greville mesh $\MG'_{\ell,\tilde Q}$ is shown on the left, while on the right we show the B\'ezier submesh $\M_{\ell,\tilde Q}$. The two green diagonally hatched Greville cells in the left figure correspond to $K_{i_{left},j_*}$ (upper-right) and $K_{i_2,j_2}$ (bottom-left), while the supports of the associated 2-form basis functions, which are contained in $\Omega_{\ell+1}$, are also diagonally hatched in the right figure. The other two Greville cells correspond to $K_{i_{left}-1,j_*}$ (red vertically hatched) and $K_{i_{left},j_*-1}$ (red horizontally hatched), while on the right we have used the same pattern to draw the region of the support of the associated basis functions that is not completely contained in $\Omega_{\ell+1}$, and where the elements $Q_L$ and $Q_D$ lie. The purple element on the right corresponds to $Q_{i_{left},j_*}$, which is contained in $\Omega_{\ell+1}$. Finally, the blue vertex in the left figure is associated to a 0-form basis function, whose support is highlighted by the blue rectangle in the right figure. Since the diagonally hatched regions and $Q_{i_{left},j_*}$ are contained in $\Omega_{\ell+1}$, its overlap is not connected, which contradicts Assumption~\ref{ass:overlap}.}
\label{fig:point-i)}
\end{figure}

Now consider the basis function associated to the lower-left vertex of $K_{i_{left},j_*}$.  The support of this basis function subsumes the supports of the basis functions associated with cells $K_{i_{left}-1,j_*}$ and $K_{i_{left},j_*-1}$, thus the support contains both $Q_L$ and $Q_D$.  By Assumption ~\ref{ass:overlap}, it follows that $Q_{i_{left},j_*}$ must not be in $\Omega_{\ell+1}$, otherwise the overlap of the basis function associated to the lower-left vertex of $K_{i_{left},j_*}$ will not be connected (see Figure~\ref{fig:point-i)}).  This is in contradiction with the presence of a cell $K_2 \equiv K_{i_2,j_2}$ in $C_2$ with $0\le i_2 < i_{left}$ and $0 \le j_2 < j_*$ since the element $Q_{i_{left},j_*}$ belongs to the support of the basis function associated with this cell and therefore should be in $\Omega_{\ell+1}$.

We may repeat the same arguments as above to show that a contradiction arises if we assume that $i_2 > i_{right}$, $j_2 < j_{bottom}$, or $j_2 > j_{top}$.  Thus, we have arrived at the desired result.
\end{enumerate}
\end{proof}

We now prove the analogous result for the holes. Before proceeding, let us denote the complementary regions of the subdomains of interest by $\Omega^c_{\ell+1} := \Omega_0 \setminus \overline \Omega_{\ell+1}$, $\MG^c_{\ell,\ell+1} := \Omega_0 \setminus \overline \MG_{\ell,\ell+1}$ and $\MG^c_{\ell+1,\ell+1} := \Omega_0 \setminus \overline \MG_{\ell+1,\ell+1}$. Since we are in the planar two-dimensional case, we can define a hole of $\Omega_{\ell+1}$ as a connected component $\hole \subset \Omega^c_{\ell+1}$ such that $\partial \hole \cap \partial \Omega_0 = \emptyset$. The holes for the Greville subgrids are defined analogously.

\begin{proposition}\label{prop:holes}
The following statements hold:
\begin{enumerate}[i)]
\item Every hole of $\MG_{\ell+1,\ell+1}$ contains at least one hole of $\Omega_{\ell+1}$.
\item Every hole of $\MG_{\ell,\ell+1}$ contains at least one hole of $\MG_{\ell+1,\ell+1}$.
\item Under Assumptions~\ref{ass:support} and~\ref{ass:overlap}, every hole of $\MG_{\ell,\ell+1}$ contains exactly one hole of $\Omega_{\ell+1}$.
\item Under Assumptions~\ref{ass:support} and~\ref{ass:overlap}, every hole of $\Omega_{\ell+1}$ is contained in a hole of $\MG_{\ell,\ell+1}$.
\end{enumerate}
\end{proposition}
\begin{proof} We prove each statement independently.
\begin{enumerate}[i)]
\item As a consequence of point ii) in Proposition~\ref{prop:cc}, any connected component of $\Omega_{\ell+1}^c$ is contained in a connected component of $\MG^c_{\ell+1,\ell+1}$. Let $\hole^\MG \subset \MG^c_{\ell+1,\ell+1}$ be a hole, which is a connected component of $\MG^c_{\ell+1,\ell+1}$ with $\partial \hole^\MG \cap \partial \Omega_0 = \emptyset$.  Thus if $\hole^\MG$ contains a connected component of $\Omega_{\ell+1}^c$, that component must also be a hole. We will show that there is at least one.

Let us consider a function associated to a Greville cell in $\hole^\MG$. Since the cell is not in $\MG_{\ell+1,\ell+1}$, there exists at least one element $Q \in \M_{\ell+1}$ in the support of the associated basis function that is not contained in $\Omega_{\ell+1}$. As a consequence, $\MG_{\ell+1,\tilde Q} \subset \MG^c_{\ell+1,\ell+1}$, and by point ii) in Lemma~\ref{lemma:one-cell} and since $\MG_{\ell+1,\tilde Q}$ is connected, we have $Q \subset \MG_{\ell+1,\tilde Q} \subset \hole^\MG$, which proves the result.
\item Let $\hole^\MG_\ell \subset \MG^c_{\ell,\ell+1}$ be a hole. Reasoning as in point i), and using point i) in Proposition~\ref{prop:cc}, it is readily shown that any connected component of $\MG^c_{\ell+1,\ell+1}$ contained in $\hole^\MG_\ell$ must be a hole. To show that there is at least one, consider a Greville cell in $\hole^\MG_\ell$ and its associated basis function. The support of this function must contain an element $Q \subset \Omega^c_{\ell+1}$, and the function must have a child in ${\cal B}^2_{\ell+1}$ such that its support also intersects $Q$. The Greville cell associated with the child is in $\MG^c_{\ell+1,\ell+1}$, and it must be contained in $\MG_{\ell,\tilde Q} \subset \hole^\MG_\ell$.  This proves the result.
\item Assume that a hole $\hole^\MG \subset \MG^c_{\ell,\ell+1}$ contains more than one hole of $\Omega_{\ell+1}$. For each hole $\hole \subset \Omega_{\ell+1}^c$ we can consider a basis function in ${\cal B}^2_\ell$, with associated Greville cell in $\hole^\MG$, such that its support intersects $\hole$. Since $\hole^\MG$ is connected, we can construct a ``path'' of Greville cells in $\hole^\MG$ between the cells of two such functions corresponding to two different holes in $\Omega_{\ell+1}^c$. The support of every function associated to a Greville cell in this path intersects $\Omega_{\ell+1}^c$, and by Assumption~\ref{ass:overlap}, this intersection is contained in one connected component of $\Omega_{\ell+1}^c$. But since the holes are connected components, there exist two adjacent Greville cells such that the supports of the associated basis functions intersect two different connected components. Then, the support of the basis function in ${\cal B}^1_\ell$ associated to the intermediate edge intersects both, which is a contradiction with Assumption~\ref{ass:overlap}.  This proves the result.
\item Let $\hole \subset \Omega^c_{\ell+1}$ be a hole. Let us take all the B\'ezier elements in $\hole$ and define $\tilde \hole^\MG = \bigcup_{Q \subset \hole} \MG_{\ell,\tilde Q}$, which is connected and contained in $\MG^c_{\ell,\ell+1}$.  By Lemma~\ref{lemma:one-cell}, $\hole \subset \tilde \hole^\MG$. Consider the connected component of $\MG^c_{\ell,\ell+1}$ that contains $\tilde \hole^\MG$ which we henceforth denote as $\hole^\MG$. We want to prove that $\hole^\MG$ is a hole, that is, $\partial \hole^\MG \cap \partial \Omega_0 = \emptyset$.

Assume it is not true, i.e., there exists a Greville cell in $\hole^\MG$ adjacent to the boundary of $\Omega_0$. Since we are using an open knot vector, the support of the corresponding basis function is contained in B\'ezier elements adjacent to the boundary, and since the Greville cell is in $\hole^\MG \subset \MG^c_{\ell,\ell+1}$ and $\hole$ is a hole, one of these elements must be in $\Omega^c_{\ell+1}$ and contained in a connected component $\hole' \ne \hole$. Moreover, there exists a Greville cell in $\tilde \hole^\MG$ such that the support of the associated function in ${\cal B}^2_\ell$ intersects $\hole$, and reasoning as in point iii) we can build a path of Greville cells between this cell and the one adjacent to the boundary, which leads to a contradiction with Assumption~\ref{ass:overlap}, because $\hole$ and $\hole'$ are not connected.
\end{enumerate}
\end{proof}

The proof of Theorem~\ref{th:homology} is a direct consequence of Propositions~\ref{prop:cc} and~\ref{prop:holes}.

\begin{remark}
It is likely that the proof of Theorem~\ref{th:homology} can be simplified, and generalized to arbitrary dimension, with a more intelligent use of homology theory. All our attempts in this direction have failed due to the difficulties related to having three different grids defined on different subdomains. That is the reason why we restricted ourselves to the planar two-dimensional case.
\end{remark}

\section{Proof of Lemma~\ref{lemma:curves}, point ii)}\label{sec:appendixb}

For the proof we can assume, without loss of generality, that $\Omega_{\ell+1}$ is connected. If it is not, the procedure of the proof can be applied to each connected component separately.

Let $\Gamma_0$ be the most external boundary of $\Omega_{\ell+1}$, and let $\Gamma_1, \ldots, \Gamma_{\nh}$ be the other connected components of the boundary. By Proposition~\ref{prop:holes}, the sets $\Gamma_0, \ldots, \Gamma_{\nh}$ are in a one-to-one correspondence with the boundaries $\Gamma^\MG_j$ in the Greville subgrid of Figure~\ref{fig:harmonic}. Let us denote by ${\bf x}_j \in \Gamma_j$ their respective lowermost (and leftmost) corners, and assume that the $\Gamma_j$ are ordered from the position of ${\bf x}_j$ starting from the lowermost and leftmost, moving first horizontally and then vertically. Since we are in the planar case, and $\Omega_{\ell+1}$ is a manifold with boundary from Lemma~\ref{lemma:omwb}, each $\Gamma_j$ is associated to a hole $\hole_j$. We define $\gamma_j$, for $j = 1, \ldots, \nh$, as the straight line that, starting from ${\bf x}_j$, goes downwards until it encounters another boundary $\Gamma_k$, with $0 \le k < j$. Recalling that ${\boldsymbol \zeta}^1_j$ are the generators of the first cohomology group for splines, $H^1(\Xhat_{\ell,\ell+1})$, to prove the result it is enough to show that:
\begin{enumerate}[a)]
\item $\displaystyle \int_{\gamma_j} {\boldsymbol \zeta}^1_j \cdot d{\bf s}> 0$, \quad for $j = 1, \ldots, \nh$,
\item $\displaystyle \int_{\gamma_j} {\boldsymbol \zeta}^1_k \cdot d{\bf s} = 0$, \quad for $1 \le j < k \le \nh$,
\end{enumerate}
since in this case the matrix given by the coefficients $\left\{\int_{\gamma_j} {\boldsymbol \zeta}^1_k \cdot d{\bf s} \right\}_{j,k=1}^{\nh}$ is a lower triangular matrix with positive diagonal entries and hence non-singular.

To prove a), let us consider a curve $\gamma_j$ and the associated harmonic function ${\boldsymbol \zeta}^1_j$ which is constructed as in Figure~\ref{fig:harmonic}. Note that the B\'ezier element above and to the right of ${\bf x}_j$ is not contained in $\Omega_{\ell+1}$.  Moreover, note the support of any basis function associated to a degree of freedom pointing upwards in Figure~\ref{fig:harmonic} (right) is contained in $\Omega_{\ell+1}$ and located above the aforementioned element.  Hence, such a basis function does not intersect $\gamma_j$. Alternatively, it is trivial to see that there is at least one function associated to a degree of freedom pointing downwards such that its support intersects $\gamma_j$. Since both $\gamma_j$ and the function are oriented downwards, we have obtained a).

To prove b), let us consider the same curve $\gamma_j$ but a different harmonic function ${\boldsymbol \zeta}^1_k$ with $k > j$ built as in Figure~\ref{fig:harmonic} for a different hole. From the ordering of $\Gamma_j$, and reasoning as in point a), the support of the functions associated to degrees of freedom pointing upwards cannot intersect $\gamma_j$. Assume now that there is at least one degree of freedom pointing downwards such that the support of its associated basis functions intersects $\gamma_j$. In this case, consider the Greville point above the chosen Greville edge, which is in $\Gamma^\MG_k \subset \partial \MG_{\ell,\ell+1}$, and its associated basis function, which is in $\Xhat^0_\ell$ but not in $\Xhat^0_{\ell,\ell+1}$. Obviously, the support of this function intersects the hole $\hole_k$, and since it contains the support of the function associated to the Greville edge, it also intersects $\gamma_j$. Finally, since $k > j$, from the ordering of the boundaries and the construction of $\gamma_j$, the support of this function must also intersect the hole $\hole_j$.  Since the holes are connected components of $\Omega_{\ell+1}^c$ this contradicts Assumption~\ref{ass:overlap}, which ends the proof.

\end{appendices}

\bibliographystyle{plain}
\bibliography{biblio}

\def\cprime{$'$}
\begin{thebibliography}{10}

\bibitem{AV-book}
A.~Alonso~Rodr{\'{\i}}guez and A.~Valli.
\newblock {\em Eddy current approximation of {M}axwell equations}, volume~4 of
  {\em MS\&A. Modeling, Simulation and Applications}.
\newblock Springer-Verlag Italia, Milan, 2010.
\newblock Theory, algorithms and applications.

\bibitem{AFW06}
D.N. Arnold, R.S. Falk, and R.~Winther.
\newblock Finite element exterior calculus, homological techniques, and
  applications.
\newblock {\em Acta Numer.}, 15:1--155, 2006.

\bibitem{AFW-2}
D.N. Arnold, R.S. Falk, and R.~Winther.
\newblock Finite element exterior calculus: from {H}odge theory to numerical
  stability.
\newblock {\em Bull. Amer. Math. Soc. (N.S.)}, 47(2):281--354, 2010.

\bibitem{BBSV-acta}
L.~Beir{\~a}o~da Veiga, A.~Buffa, G.~Sangalli, and R.~V{\'a}zquez.
\newblock Mathematical analysis of variational isogeometric methods.
\newblock {\em Acta Numer.}, 23:157--287, 2014.

\bibitem{Beirao_Cho_Sangalli}
L.~Beir{\~a}o~da Veiga, D.~Cho, and G.~Sangalli.
\newblock Anisotropic {NURBS} approximation in {I}sogeometric {A}nalysis.
\newblock {\em Comput. Methods Appl. Mech. Engrg.}, 209-212:1--11, 2012.

\bibitem{Berdinsky201486}
D.~Berdinsky, T.~Kim, C.~Bracco, D.~Cho, B.~Mourrain, M.~Oh, and
  S.~Kiatpanichgij.
\newblock Dimensions and bases of hierarchical tensor-product splines.
\newblock {\em J. Comput. Appl. Math.}, 257:86 -- 104, 2014.

\bibitem{Boffi01}
D.~Boffi.
\newblock {A note on the de Rham complex and a discrete compactness property}.
\newblock {\em Appl. Math. Lett.}, 14(01):33--38, 2001.

\bibitem{Boffi07}
D.~Boffi.
\newblock Approximation of eigenvalues in mixed form, discrete compactness
  property, and application to {$hp$} mixed finite elements.
\newblock {\em Comput. Methods Appl. Mech. Engrg.}, 196(37-40):3672--3681,
  2007.

\bibitem{BOS88}
A.~Bossavit.
\newblock Whitney forms: {A} class of finite elements for three-dimensional
  computations in electromagnetism.
\newblock {\em IEE Proc. A}, 135(8):493--500, 1988.

\bibitem{BoPe98}
O.~Botella and R.~Peyret.
\newblock Benchmark spectral results on the lid-driven cavity flow.
\newblock {\em Comput. Fluids}, 27:421--433, 1998.

\bibitem{Bressan-Juttler}
A.~Bressan and B.~J\"uttler.
\newblock Inf-sup stability of isogeometric {T}aylor-{H}ood and {S}ub{G}rid
  methods for the {S}tokes problem with hierarchical splines.
\newblock Technical report, 2016.
\newblock Submitted to IMA J. Numer. Anal.

\bibitem{Buchegger2016159}
F.~Buchegger, B.~J{\"u}ttler, and A.~Mantzaflaris.
\newblock Adaptively refined multi-patch {B}-splines with enhanced smoothness.
\newblock {\em Appl. Math. Comput.}, 272, Part 1:159 -- 172, 2016.
\newblock Subdivision, Geometric and Algebraic Methods, Isogeometric Analysis
  and Refinability.

\bibitem{Buffa_deFalco_Sangalli}
A.~Buffa, C.~de~Falco, and G.~Sangalli.
\newblock Isogeometric {A}nalysis: {S}table elements for the 2{D} {S}tokes
  equation.
\newblock {\em Internat. J. Numer. Methods Fluids}, 65(11-12):1407--1422, 2011.

\bibitem{Buffa16-1}
A.~Buffa and C.~Giannelli.
\newblock {Adaptive isogeometric methods with hierarchical splines: Error
  estimator and convergence}.
\newblock {\em Math. Models Methods in Appl. Sci.}, 26(01):1--25, 2016.

\bibitem{BGi17}
A.~Buffa and C.~Giannelli.
\newblock Adaptive isogeometric methods with hierarchical splines: {O}ptimality
  and convergence rates.
\newblock Technical Report 08.2017, MATHICSE, Institute of Mathematics, Ecole
  Polytechnique F\'ed\'erale de Lausanne, 2017.

\bibitem{Buffa16-2}
A.~Buffa, C.~Giannelli, P.~Morgenstern, and D.~Peterseim.
\newblock {Complexity of hierarchical refinement for a class of admissible mesh
  configurations}.
\newblock {\em Comput. Aided Geom. Des.}, 47:83--92, 2016.

\bibitem{BRSV11}
A.~Buffa, J.~Rivas, G.~Sangalli, and R.~V\'{a}zquez.
\newblock Isogeometric discrete differential forms in three dimensions.
\newblock {\em SIAM J. Numer. Anal.}, 49(2):818--844, 2011.

\bibitem{Buffa_Sangalli_Vazquez}
A.~Buffa, G.~Sangalli, and R.~V{\'a}zquez.
\newblock Isogeometric analysis in electromagnetics: B-splines approximation.
\newblock {\em Comput. Methods Appl. Mech. Engrg.}, 199(17-20):1143 -- 1152,
  2010.

\bibitem{BSV14}
A.~Buffa, G.~Sangalli, and R.~V\'azquez.
\newblock Isogeometric methods for computational electromagnetics: {B}-spline
  and {T}-spline discretizations.
\newblock {\em J. Comput. Phys.}, 257, Part B:1291 -- 1320, 2014.

\bibitem{CFR01}
S.~Caorsi, P.~Fernandes, and M.~Raffetto.
\newblock Spurious-free approximations of electromagnetic eigenproblems by
  means of {N}edelec-type elements.
\newblock {\em M2AN Math. Model. Numer. Anal.}, 35(2):331--354, 2001.

\bibitem{inf-sup-test}
D.~Chapelle and K.-J. Bathe.
\newblock The inf-sup test.
\newblock {\em Comput. \& Structures}, 47(4-5):537--545, 1993.

\bibitem{Cohen01}
A.~Cohen, W.~Dahmen, and R.~DeVore.
\newblock {Adaptive wavelet methods for elliptic operator equations:
  Convergence rates}.
\newblock {\em Math. Comput.}, 70(233):27--75, 2001.

\bibitem{Corno20161}
J.~Corno, C.~de~Falco, H.~De Gersem, and S.~Sch{\"o}ps.
\newblock Isogeometric simulation of {L}orentz detuning in superconducting
  accelerator cavities.
\newblock {\em Comput. Phys. Commun.}, 201:1 -- 7, 2016.

\bibitem{IGA-book}
J.A. Cottrell, T.J.R. Hughes, and Y.~Bazilevs.
\newblock {\em Isogeometric {A}nalysis: {T}oward integration of {CAD} and
  {FEA}}.
\newblock John Wiley \& Sons, 2009.

\bibitem{BMAX}
M.~Dauge.
\newblock {Benchmark computations for {M}axwell equations for the approximation
  of highly singular solutions}.
\newblock {\em http://perso.univ-rennes1.fr/monique.dauge/benchmax.html}, 2014.

\bibitem{deBoor}
C.~de~Boor.
\newblock {\em {A Practical Guide to Splines}}, volume~27 of {\em Applied
  Mathematical Sciences}.
\newblock Springer-Verlag, New York, revised edition, 2001.

\bibitem{DB05}
L.~Demkowicz and A.~Buffa.
\newblock {$H^1$}, {$H({\rm curl})$} and {$H({\rm div})$}-conforming
  projection-based interpolation in three dimensions. {Q}uasi-optimal
  {$p$}-interpolation estimates.
\newblock {\em Comput. Methods Appl. Mech. Engrg.}, 194(2-5):267--296, 2005.

\bibitem{Demlow14}
A.~Demlow and A.N. Hirani.
\newblock {A posteriori error estimates for finite element exterior calculus:
  The de Rham complex}.
\newblock {\em Found. Comput. Math.}, 14(6):1337--1371, 2014.

\bibitem{DEC-arxiv}
M.~{Desbrun}, A.N. {Hirani}, M.~{Leok}, and J.E. {Marsden}.
\newblock {Discrete {E}xterior {C}alculus}.
\newblock {\em ArXiv Mathematics e-prints}, August 2005.

\bibitem{Dettmer16}
W.G. Dettmer, C.~Kadapa, and D.~Peri{\'c}.
\newblock {A stabilised immersed boundary method on hierarchical {B}-spline
  grids}.
\newblock {\em Comput. Methods in Appl. Mech. Engrg.}, 311:415--437, 2016.

\bibitem{Evans15}
E.J. Evans, M.A. Scott, X.~Li, and D.C. Thomas.
\newblock {Hierarchical T-splines: Analysis-suitability, B{\'e}zier extraction,
  and application as an adaptive basis for isogeometric analysis}.
\newblock {\em Comput. Methods in Appl. Mech. Engrg.}, 284:1--20, 2015.

\bibitem{Evans12}
J.A. Evans and T.J.R. Hughes.
\newblock {Discrete spectrum analyses for various mixed discretizations of the
  Stokes eigenproblem}.
\newblock {\em Comput. Mech.}, 50:667--674, 2012.

\bibitem{EvHu12}
J.A. Evans and T.J.R. Hughes.
\newblock Isogeometric divergence-conforming {B}-splines for the
  {D}arcy-{S}tokes-{B}rinkman equations.
\newblock {\em Math. Models Methods Appl. Sci.}, 23(04):671--741, 2013.

\bibitem{EvHu12-2}
J.A. Evans and T.J.R. Hughes.
\newblock Isogeometric divergence-conforming {B}-splines for the {S}teady
  {N}avier-{S}tokes {E}quations.
\newblock {\em Math. Models Methods Appl. Sci.}, 23(08):1421--1478, 2013.

\bibitem{EvHu12-3}
J.A. Evans and T.J.R. Hughes.
\newblock Isogeometric divergence-conforming {B}-splines for the {U}nsteady
  {N}avier-{S}tokes {E}quations.
\newblock {\em J. Comput. Phys.}, 241:141 -- 167, 2013.

\bibitem{Garau16}
E.M. Garau and R.~V{\'a}zquez.
\newblock Algorithms for the implementation of adaptive isogeometric methods
  using hierarchical splines.
\newblock Technical Report 16-08, IMATI-CNR, 2016.

\bibitem{Grinspun02}
E.~Grinspun, P.~Krysl, and P.~Schr{\"o}der.
\newblock {CHARMS: A simple framework for adaptive simulation}.
\newblock {\em ACM Transactions on Graphics (TOG)}, 21(3):281--290, 2002.

\bibitem{GK}
P.W. Gross and P.R. Kotiuga.
\newblock {\em {Electromagnetic Theory and Computation: A Topological
  Approach}}, volume~48 of {\em Mathematical Sciences Research Institute
  Publications}.
\newblock Cambridge University Press, Cambridge, 2004.

\bibitem{Hatcher}
A.~Hatcher.
\newblock {\em Algebraic topology}.
\newblock Cambridge University Press, Cambridge, 2002.

\bibitem{Hennig16}
P.~Hennig, S.~M{\"u}ller, and M.~K{\"a}stner.
\newblock {B{\'e}zier extraction and adaptive refinement of truncated
  hierarchical NURBS}.
\newblock {\em Comput. Methods in Appl. Mech. Engrg.}, 305:316--339, 2016.

\bibitem{HIP02a}
R.~Hiptmair.
\newblock Finite elements in computational electromagnetism.
\newblock {\em Acta Numer.}, 11:237--339, 2002.

\bibitem{Hirani}
A.N. Hirani.
\newblock {\em Discrete {E}xterior {C}alculus}.
\newblock PhD thesis, California Institute of Technology, Pasadena, CA, 2003.

\bibitem{Hirani-darcy}
A.N. Hirani, K.B. Nakshatrala, and J.H. Chaudhry.
\newblock Numerical method for {D}arcy flow derived using discrete exterior
  calculus.
\newblock {\em Int. J. Comput. Methods Eng. Sci Mech.}, 16(3):151--169, 2015.

\bibitem{Johannessen15}
K.A. Johannessen, M.~Kumar, and T.~Kvamsdal.
\newblock {Divergence-conforming discretization for Stokes problem on locally
  refined meshes using LR B-splines}.
\newblock {\em Comput. Methods in Appl. Mech. Engrg.}, 293:38--70, 2015.

\bibitem{Kadapa16}
C.~Kadapa, W.G. Dettmer, and D.~Peri{\'c}.
\newblock {A fictitious domain/distributed Lagrange multiplier based
  fluid--structure interaction scheme with hierarchical B-Spline grids}.
\newblock {\em Comput. Methods in Appl. Mech. Engrg.}, 301:1--27, 2016.

\bibitem{Kamensky17}
D.~Kamensky, M.-C. Hsu, Y.~Yu, J.A. Evans, M.S. Sacks, and T.J.R. Hughes.
\newblock {Immersogeometric cardiovascular fluid-structure interaction analysis
  with divergence-conforming B-splines}.
\newblock {\em Comput. Methods in Appl. Mech. Engrg.}, 314:408--472, 2017.

\bibitem{Kraft}
R.~Kraft.
\newblock Adaptive and linearly independent multilevel {$B$}-splines.
\newblock In {\em Surface Fitting and Multiresolution Methods
  ({C}hamonix--{M}ont-{B}lanc, 1996)}, pages 209--218. Vanderbilt Univ. Press,
  Nashville, TN, 1997.

\bibitem{Kuru14}
G.~Kuru, C.V. Verhoosel, K.G. Van~der Zee, and E.H. van Brummelen.
\newblock {Goal-adaptive isogeometric analysis with hierarchical splines}.
\newblock {\em Comput. Methods in Appl. Mech. Engrg.}, 270:270--292, 2014.

\bibitem{Matthies_Tobiska07}
G.~Matthies and L.~Tobiska.
\newblock {Mass conservation of finite element methods for coupled
  flow-transport problems}.
\newblock {\em Int. J. Comput. Sci. Appl. Math.}, 1(02-04):293--307, 2007.

\bibitem{Mokris2014}
D.~Mokri{\v{s}}, B.~J{\"u}ttler, and C.~Giannelli.
\newblock On the completeness of hierarchical tensor-product {$B$}-splines.
\newblock {\em J. Comput. Appl. Math.}, 271:53--70, 2014.

\bibitem{Schillinger12}
D.~Schillinger, L.~Ded\`e, M.A. Scott, J.A. Evans, M.J. Borden, E.~Rank, and
  T.J.R. Hughes.
\newblock {An isogeometric design-through-analysis methodology based on
  adaptive hierarchical refinement of NURBS, immersed boundary methods, and
  T-spline CAD surfaces}.
\newblock {\em Comput. Methods in Appl. Mech. Engrg.}, 249:116--150, 2012.

\bibitem{Scott2014222}
M.A. Scott, D.C. Thomas, and E.J. Evans.
\newblock Isogeometric spline forests.
\newblock {\em Comput. Methods Appl. Mech. Engrg.}, 269(0):222 -- 264, 2014.

\bibitem{Van17}
T.M. van Opstal, J.~Yan, C.~Coley, J.A. Evans, T.~Kvamsdal, and Y.~Bazilevs.
\newblock {Isogeometric divergence-conforming variational multiscale
  formulation of incompressible turbulent flows}.
\newblock {\em Comput. Methods in Appl. Mech. Engrg.}, 316:859--879, 2017.

\bibitem{Vuong_giannelli_juttler_simeon}
A.-V. Vuong, C.~Giannelli, B.~J\"uttler, and B.~Simeon.
\newblock A hierarchical approach to adaptive local refinement in isogeometric
  analysis.
\newblock {\em Comput. Methods Appl. Mech. Engrg.}, 200(49-52):3554--3567,
  2011.

\end{thebibliography}

\end{document}